\documentclass[12pt,leqno]{article}
\usepackage{amsmath}
\usepackage{amssymb}
\usepackage{theorem}
\usepackage{curves}
\usepackage{epic}
\usepackage{rotating}
\usepackage{epsf}
\usepackage[scanall]{psfrag}
\usepackage{graphicx}
\usepackage{esint} 
\usepackage{dsfont}
\usepackage{hyperref} 
\hypersetup{backref,  pdfpagemode=FullScreen,  colorlinks=true} 

\numberwithin{equation}{section}

\usepackage{pgf}
\usepackage{color}

\newcommand{\Bk}{\color{black}}

\newtheorem{lem}{Lemma}[section]
\newtheorem{pro}[lem]{Proposition}
\newtheorem{defi}[lem]{Definition}
\newtheorem{thm}[lem]{Theorem}
\newtheorem{cor}[lem]{Corollary}
\newtheorem{rem}[lem]{Remark}

\newenvironment{proof}{\medskip\noindent{\bf Proof.}}{\medskip}

\newcommand{\ms}{\medskip}
\newcommand{\pari}{\par\noindent}
\newcommand{\R}{\mathbb{R}}

\renewcommand{\H}{\mathcal H}

\renewcommand{\O}{{\cal O}}
\renewcommand{\d}{\partial}
\newcommand{\dist}{\,\mathrm{dist}\,}
\newcommand{\sm}{\setminus}
\newcommand\qed{\hfill\vrule height8pt width6pt depth0pt}

\renewcommand{\i}{\subset}
\newcommand{\wt}{\widetilde}

\renewcommand{\u}{{\bf u}}
\renewcommand{\v}{{\bf v}}
\newcommand{\V}{{\bf V}}
\newcommand{\W}{{\bf W}}
\newcommand{\F}{{\cal F}}
\newcommand{\1}{{\mathds 1}}

\begin{document}

\title{A free boundary problem for the localization of eigenfunctions}
\author{G. David,  
M. Filoche,  
D. Jerison, 
and S. Mayboroda.}
\newcommand{\Addresses}{{
  \bigskip
  \vskip 0.08in \noindent --------------------------------------
\vskip 0.10in

  \footnotesize

  G.~David, \textsc{Univ Paris-Sud, Laboratoire de Math\'{e}matiques, 
UMR 8658 Orsay, F-91405; CNRS, Orsay, F-91405; Institut Universitaire de France}\par\nopagebreak
  \textit{E-mail address}: \texttt{Guy.David@math.u-psud.fr}

  \medskip

 M.~Filoche , \textsc{Physique de la Mati\`ere Condens\'ee, Ecole Polytechnique, CNRS, Palaiseau, France; CMLA, ENS Cachan, UniverSud, Cachan, France}\par\nopagebreak
  \textit{E-mail address}: \texttt{marcel.filoche@polytechnique.edu}
  
  \medskip

D.~Jerison, \textsc{Massachusetts Institute of Technology,
77 Massachusetts Ave,
Cambridge, MA 02139-4307
USA}\par\nopagebreak
  \textit{E-mail address}: \texttt{jerison@math.mit.edu}

  \medskip

S.~Mayboroda, \textsc{School of Mathematics, University of Minnesota, 206 Church St SE, Minneapolis, MN 55455 USA}\par\nopagebreak
  \textit{E-mail address}: \texttt{svitlana@math.umn.edu}

}}
\date{}
\maketitle

\ms\noindent{\bf Abstract.}
We study a variant of the Alt, Caffarelli, and Friedman free boundary problem, with many phases and a slightly different volume term, which we originally designed to guess the localization of eigenfunctions of a Schr\"odinger operator in a domain. We prove Lipschitz bounds for the functions and some nondegeneracy and regularity properties for the domains.

\ms\noindent{\bf R\'esum\'e en Fran\c cais.}
On \'etudie une variante du probl\`eme de fronti\`ere libre de Alt, Caffarelli, et Friedman,
avec plusieurs phases et un terme de volume l\'eg\`erement diff\'erent, que l'on a choisie pour
deviner la localisation des fonctions propres d'un op\'erateur de  Schr\"odinger dans un domaine.
On d\'emontre des estimations Lipschitziennes pour les fonctions associ\'ees \`a un minimiseur,
et des propri\'et\'es de nond\'eg\'erescence et de r\'egularit\'e pour les fronti\`eres libres. 

\ms\noindent{\bf Key words/Mots cl\'es.}
Free boundary problem ; localization of eigenfunctions.

\ms\noindent
AMS classification: 49Q20 35B65.

\tableofcontents

\section{Introduction}\label{intro}

The initial motivation for this paper was to describe the localization of 
eigenfunctions for an operator $\cal L$ on a domain $\OmegaÊ\i \R^n$. Let us assume that $|\Omega|$, the measure of $\Omega$, is finite. 
The typical operator that we consider is the positive Laplacian 
${\cal L} = -\Delta$, or a Schr\"odinger operator ${\cal L} = -\Delta + {\cal V}$, 
with ${\cal V}$ bounded and nonnegative;  
in a forthcoming paper we may also consider the bilaplacian $\Delta^2$. 

In \cite{FM}, a pointwise estimate for eigenfunctions for $\cal L$ is found,
which bounds them in terms of a single function $w_{0}$,
namely the solution of ${\cal L}w_{0} = 1$ on $\Omega$, with the Dirichlet
condition $w_{0} = 0$ on $\R^n \sm \Omega$. Our first goal  
is to derive an automatic way, using $w_{0}$, to find subdomains 
$W_{j}$, $1\leq j \leq N$, of $\Omega$, where the eigenfunctions of ${\cal{L}}$
are more likely to be supported.
The work in \cite{FM} indicates that, roughly speaking, one seeks a collection of disjoint $W_{j}\subset \Omega$, $1\leq j \leq N$, such that 
$w_{0}$ is small
on the boundaries of the $W_{j}$, and it is natural to try to measure ``smallness" in terms of
the operator $\cal L$ itself. 
Even though many handmade or numerical decompositions of $\Omega$ based on $w_0$ seem to 
give very good predictions of the localization of eigenfunctions,
we would like to have a more systematic way to realize the decomposition. 

The functional described below is intended to give such a good partition
of $\Omega$ into subdomains, and it turns out to be an interesting 
variant of functionals introduced by Alt and Caffarelli \cite{AC}, and studied 
by many others. In the present paper, we shall mainly study the theoretical
properties of our functional (existence and regularity of the minimizers and
regularity of the corresponding free boundaries).

Let us now describe the main free boundary problem that we shall
study here; the relation with $w_{0}$ and our original localization problem
will be explained more in Section \ref{motivation}.

We are given a domain $\Omega \i \R^n$, and (for instance) an operator 
${\cal L} = -\Delta + {\cal V}$; assumptions on the potential ${\cal V}$, or other 
functions associated to a similar problem, will come later.
We are also given an integer $N \geq 1$,  and we want to cut $\Omega$ 
into subregions $W_{i}$, $1 \leq i \leq N$, 
according to the geometry associated to $\cal L$.
For this, we want to define and minimize a functional $J$. 
But let us first define the set of admissible pairs $(\u,W)$ for which $J(\u,W)$ is defined.
\begin{defi}\label{d1.1} 
Given the open set $\Omega \i \R^n$ and the integer $N \geq 1$, we
denote by $\F = \F(\Omega)$ the set of admissible pairs 
$(\u,\W)$, where $\W = (W_{i})_{1 \leq i \leq N}$ is a $N$-uple of pairwise
disjoint Borel-measurable sets $W_{i} \i \Omega$, and 
$\u = (u_{i})_{1 \leq i \leq N}$ is a $N$-uple of real-valued 
functions $u_{i}$, such that
\begin{equation}\label{e1.1}
u_{i} \in W^{1,2}(\R^n)
\end{equation}
and
\begin{equation}\label{e1.2}
u_{i}(x) = 0 \hbox{ for almost every } x\in \R^n \sm W_{i}.
\end{equation}
\end{defi}

\ms
Here $W^{1,p}(\R^n)$, $1 \leq p < +\infty$, denotes the set of functions 
$f\in L^p_{loc}(\R^n)$ whose derivative, computed in the distribution sense, 
lies in $L^p(\R^n)$.
We chose a definition for which we do not need to assume any regularity
for the sets $W_{i}$, nor to give a precise meaning to the Sobolev
space $W^{1,2}_{0}(W_{i})$, but under mild assumptions, the functions
$u_{i}$ associated to minimizers will be continuous, and we will be able to take
$W_{i} = \big\{ x\in \Omega \, ; \, u_{i}(x) > 0 \big\}$. 
For the moment we took real-valued functions, 
but what we will say will systematically apply when some of the functions
$u_{i}$ are required to be nonnegative. In addition, some of our results will
only work under this constraint (that $u_{i} \geq 0$).

Our functional $J$ will have three main terms. The first one is the energy
\begin{equation}\label{e1.3}
E(\u) =  \sum_{i=1}^N \int |\nabla u_{i}(x)|^2 dx,
\end{equation}
where we denoted by $\nabla u_{i}$ the distributional gradient of $u_{i}$, 
which is an $L^2$ function. It does not matter whether we integrate on $\R^n$, 
$\Omega$, or $W_{i}$, because one can check that if $u_{i} \in W^{1,2}(\R^n)$
vanishes almost everywhere on $\R^n \sm W_{i}$, then $\nabla u_{i} = 0$
almost everywhere on $\R^n \sm W_{i}$.
Indeed, by the Rademacher-Calder\'on theorem, $u_{i}$ is differentiable at almost every
point of $\R^n$, with a differential that coincides almost everywhere with the distribution
$Du_{i}$. It is then easy to check that $Du_{i}(x) = 0$ when $x$ is a point of Lebesgue
differentiability of $\R^n \sm W_{i}$, hence, for almost every $\R^n \sm W_{i}$.

The second term of our functional will be
\begin{equation}\label{e1.4}
M(\u) = \sum_{i=1}^N \int [u_{i}(x)^2 f_{i}(x) - u_{i}(x) g_{i}(x)] dx,
\end{equation}
where the $f_{i}$ and the $g_{i}$, $1 \leq i \leq N$,
are given functions on $\Omega$
that we may choose, depending on our problem. We are slightly abusing notation
here, because $M(\u)$ also depends on the $W_{i}$ through the choice of
functions we integrate against $u_{i}^2$ or $u_{i}$, at least if the $f_{i}$ and
$g_{i}$ depend on $i$. The convergence of the integrals in \eqref{e1.4}
will follow from our assumptions on the $f_{i}$ and the $g_i$.
For our initial localization problem, all the $f_{i}$ will be equal to the potential ${\cal V}$, 
and $g_{i}= 2$ for all $i$. 

We put a negative sign before $u_{i}(x) g_{i}(x)$
so that the $f_{i}$ and $g_{i}$ will be nonnegative
in the most interesting cases, but we won't always need 
both assumptions. When this happens, the functions $u_{i}$
will also naturally be nonnegative, which will allow some specific arguments.

The last term of the functional is a function $F(\W)$ that depends on the sets
$W_i$. We do not need to be too specific for the moment, but our main example 
is to take a continuous function of the Lebesgue measures $|W_i|$ of the $W_i$.
Thus our functional is 
\begin{eqnarray}\label{e1.5}
J(\u,\W) &=& E(\u) + M(\u) + F(\W)
\nonumber\\
&=& \sum_{i=1}^N \int |\nabla u_{i}(x)|^2 dx 
+ \sum_{i=1}^N \int [u_{i}(x)^2 f_{i}(x) - u_{i}(x) g_{i}(x)] dx 
+ F(W_{1}, \ldots W_{N}). 
\end{eqnarray}
We add the term $F(\W)$ to the functional to avoid some trivial solutions
(typically, where $W_1 =\Omega$ and all the other $W_i$
are empty). For some of our main results we shall put some
monotonicity assumptions on $F$, 
but for the moment let us merely say that a typical choice would be
\begin{equation}\label{e1.6}
F(\W) = \sum_{i=1}^N a|W_i| + b|W_i|^{1+\alpha}
\end{equation}
for some $\alpha > 0$ and suitable positive constants $a$ and $b$, where the second term 
tends to make us choose roughly equal volumes when we minimize the functional 
(see Section~\ref{number-p}) and the linear one ensures a certain non-degeneracy 
(see, e.g., a discussion after \eqref{e2.12}).  

The reader may also think about the case
when we choose nonnegative bounded functions $q_i$, $1 \leq i \leq N$,
and we set
\begin{equation}\label{e1.7}
F(\W) = \sum_{i=1}^N \int_{W_i} q_i
\end{equation}
which is the typical choice that people use for the 
functionals of \cite{AC} and \cite{ACF}.

\ms
We shall first prove that if the Lebesgue measure of $\Omega$ is finite,
the $f_{i}$ are bounded and nonnegative, and the $g_{i}$ lie in $L^2$,
there exist minimizers for $J$ in the class $\cal F$. See Section \ref{existence}.

But the main goal of the paper is to consider a minimizer $(\u,\W)$ 
for the functional and establish, under slightly stronger assumptions, 
some regularity results for $u$ (for instance, Lipschitz bounds), 
and even the sets $W_{i}$ (all the way up to $C^1$-regularity 
almost everywhere when we are very lucky). 
Maybe we should say right away that although many of our results also
hold when $n=1$ (but may not be interesting), we shall concentrate our attention
on $n \geq 2$ and can't even promise that all our statements will make sense
when $n=1$. 

Variants of our functional $J$ have been studied extensively, 
starting with the very important papers \cite{AC} and \cite{ACF}, 
where the authors wanted to study the regularity of the
boundary of the positive set $\{ u>0 \}$ of some PDE solution; 
see also  
\cite{CJK} for further results on these free boundary problems.
Similar problems were also raised, for instance to study optimal shapes;
see \cite{HP}, 
\cite{BV}.

There is an interesting difference between the present context and the situation of 
the aforementioned papers, which is that we want to allow a number $N \geq 3$ 
of domains. When $N = 2$ and the $u_i$ are nonnegative, 
as in most of the aforementioned papers, 
we can regroup $u_{1}$ and $u_{2}$ into one single real-valued 
function $u_{1}-u_{2}$, and this is very convenient to produce competitors for $u$, 
for instance by taking the harmonic extension of the restriction of $u_{1}-u_{2}$ 
to a sphere. This trick will not be available here, and will force us to be slightly 
more imaginative in the construction of our competitors. 
See our description of harmonic competitors in Section \ref{favorites}
for a more detailed discussion.
Nonetheless, we shall still be able to use a monotonicity formula, 
which was first formulated in \cite{ACF}, and which is a major tool
in the aforementioned papers. This will allow us to prove that $u$
is Lipschitz, for instance, and try to follow the same route as in these
papers to prove some regularity results for the $W_i$.

Let us describe the plan of this paper and the results that it contains. 
In Section \ref{motivation} we try to explain why we introduced the functional $J$ above, 
what choices of the parameters seem more interesting to us, and what the regularity
results may mean in the context of localization of eigenfunctions.

Recall that the existence of minimizers is proved in 
Section \ref{existence}, in a quite general setting; the main point of the proof is 
the fact that, by Poincar\'e's inequalities, the energy term $E(\u)$
controls the second term $M(\u)$. When $F(\W)$ is a continuous
function of the volumes $|W_i|$, as in Theorem \ref{t3.1},
we extract from a minimizing sequence a subsequence for which these 
volumes converge, and we get a minimizer rather easily 
(without having to make the sets $W_i$ themselves converge); the proof also
works when $F(\W)$ is a continuous and nondecreasing function of the $W_i$, 
(hence, for instance, when $F$ is as in \eqref{e1.7} with nonnegative 
integrable functions $q_i$). See Corollary \ref{t3.5}. 

After a short Section \ref{poincare} where we derive simple consequences
of the Poincar\'e inequalities, we check in Section \ref{bounded} that
if $|\Omega| < +\infty$, the $f_i$ are nonnegative, and the $f_i$ and
the $g_i$ lie in $L^p(\Omega)$ for some $p>n/2$, and 
$(\u,\W)$ is a minimizer for $J$ in the class $\F$, then the $u_i$
are bounded. The proof uses simple bounds on the fundamental 
solution of $-\Delta$. We put his intermediate result here because it makes
it easier to estimate various error terms later.

Many of our subsequent estimates will be obtained by comparing $(\u,\W)$ 
with two competitors that we introduce in Section \ref{favorites}. The first one
is simply obtained by multiplying some of the $u_i$ by a cut-off function that
vanishes in a ball; the interest is that we may save on the energy or volume
terms. The second one is our substitute for the harmonic extension. We want to
define a smooth competitor with the same values of $\u$ on a sphere, 
and since we cannot extend harmonically each component $u_i$, we cut them off 
as above, except one for which we can use a harmonic extension because we
just created some space. 

The competitors of Section \ref{favorites} are used in Section~\ref{holder}
to prove that the $u_i$ are locally H\"older-continuous on $\Omega$
if in addition to the assumptions of Section~\ref{bounded},
$F$ is a H\"older-continuous function of the $W_i$, with an exponent
$\beta > {n-2 \over n}$.
Here and below, the distance between $\W$ and $\W'$ is defined as the sum of 
the measures of the symmetric differences $W_i \Delta W'_i$; see \eqref{e7.1}.
The proof is arranged like Bonnet's monotonicity argument for the Mumford-Shah
functional, but the estimates are not sharp and we only get a very small H\"older exponent.

We show in Section \ref{boundary} that if $\Omega$ is smooth 
(but in fact much less is needed), 
$\u$ is also H\"older-continuous all the way to the boundary $\d \Omega$.

Then we turn to the monotonicity formula. From now on, let us assume that the
$f_i$ and the $g_i$ are bounded, and that the $f_i$ are nonnegative.
This formula concerns products of two functions $\Phi_\varphi$ that are
defined as follows. 
Choose an origin $x_0$, and denote by $I$ the set of pairs
$\varphi = (i,\varepsilon)$, where $i \in [1,N]$ and 
$\varepsilon \in \{ -1, 1 \}$. To each $\varphi \in I$ we associate the
function $v_\varphi =  (\varepsilon u_i)_+ = \max(0,\varepsilon u_i)$
(we shall often call $v_\varphi$ a phase of $\u$)
and the function $\Phi_\varphi$ defined by 
\begin{equation}
\label{e1.8}
\Phi_\varphi(r) = {1 \over r^2} 
\int_{B(x_0,r)} {|\nabla v_\varphi|^2 \over |x-x_0|^{n-2}} \, dx
\end{equation}
for $r > 0$. For $\varphi \neq \psi \in I$, set 
$\Phi_{\varphi,\psi}(r) = \Phi_\varphi(r)\Phi_\psi(r)$. 
This is just a minor variant of the functional introduced by Alt, Caffarelli, and 
Friedman in \cite{ACF}, and we show that $\Phi_{\varphi,\psi}$ is nearly 
nondecreasing near the origin when $(\u,\W)$ minimizes $J$ in $\F$, 
at least if we assume that  $|\Omega| <+\infty$, 
and $F$ is H\"older-continuous as in \eqref{e7.1}.
See Theorem \ref{t9.1} for a precise statement.
The proof consists in checking that the functions $v_\varphi$ and $v_\psi$ satisfy
the assumptions of a near monotonicity formula that was 
established in \cite{CJK}, 
and for this the H\"older estimates of Section \ref{holder} are useful.
This result is also valid when $x_0$ lies on $\d \Omega$ or close to $\d \Omega$, 
if $\Omega$ is smooth (as in \eqref{e8.1}), and then we use the results of  
Section~\ref{boundary} to check the assumptions of \cite{CJK}. 

Once we have a control on the functionals $\Phi$, we can use a bootstrap argument
to show that $\u$ is locally Lipschitz. For local Lipschitz bounds in $\Omega$,
we just need to assume that (in addition to the previous assumptions) 
$F$ is a Lipschitz function of the $W_i$ (as in \eqref{e10.2}); 
See Theorem \ref{t10.1}.

If in addition $\Omega$ is bounded and has a $C^{1+\alpha	}$ boundary
for some $\alpha > 0$, we show in Section~\ref{global} that $\u$ is also
Lipschitz near $\d \Omega$.

Notice that in general, we do not expect the sets $W_i$,
$1 \leq i \leq N$, to cover, or almost cover $\Omega$. In fact,
we will often make sure that $F(\W)$ is a sufficiently large, or increasing
function of each $W_i$, so that if some part of $W_i$ is not really
useful to make $E(\u)+M(\u)$ small, we may as well remove it and
save more on the $F(\W)$ term. 
In Section \ref{positive}, we take the opposite approach and 
find conditions on $F$ and the $g_i$ that imply that the
$W_i$ associated to a minimizer $(\u,\W)$
almost cover $\Omega$ and that $\u \neq 0$ almost everywhere.
Typically, this means some decay for $F$ (so that adding
a missing piece to the $W_i$ does not cost anything) and
the positivity of $g_i$.  
See Propositions \ref{t12.3} and \ref{t12.4} in particular.

In Section \ref{number-p} we show that if 
the $f_i$ are bounded and nonnegative, the $g_i$ lie in $L^2$
and at least one of them is nonzero, and
$F$ is given by \eqref{e1.6} with $\alpha > {2 \over n}$, $b$ is large enough,
and $a \geq 0$ small enough, then the mininizers for $J$ are such that
$\u \neq 0$ and $|W_i| < |\Omega|/10$ for $1 \leq i \leq N$.

In Section \ref{number} we show that under mild conditions on
$F$ (where we say that volume is not too cheap), the number
of indices $i$ for which $|W_i| > 0$ is bounded, even if we allowed much more
components by taking $N$ very large.

\ms
We then return to a general scheme in the study of free boundary problems.
In Section~\ref{good}, we consider for instance $u_{1,+}$, the positive part of $u_1$,
and we want to show that under suitable non degeneracy conditions, 
it behaves roughly like the distance to the 
free boundary $\d_1 = \d\big\{x \, ; \,  u_1(x) > 0 \big\}$.
The nondegeneracy condition that we will use is that, as far as the volume term $F(\W)$
is concerned, we can always sell small parts of $W_1$ and get a proportional profit. 
That is, if $A \i W_1$ has a small enough measure, we can remove $A$ from
$W_1$, and maybe distribute some part of it to the $W_i$, $i \neq 1$,
in such a way as to make $F(\W)$ smaller by at least $\lambda |A|$
for some fixed $\lambda > 0$.
See \eqref{e13.1} for the precise condition.
For instance, if $F$ is given by \eqref{e1.7}, we get this condition as soon
as $q_1(x) \geq \lambda+\min\big\{ 0, q_2(x), \cdots, q_N(x) \big\}$
almost-everywhere on $\Omega$ (see \eqref{e13.4}); when we have
\eqref{e1.6}, we just need to take $a \geq \lambda$ and $b \geq 0$.

If we add this nondegeneracy condition to the other assumptions above,
we get the desired rough behavior of $u_{1,+}$; see
\eqref{e13.6} and \eqref{e13.7} in Theorem \ref{t13.1},
\eqref{e13.10} in Theorem \ref{t13.2}, and 
\eqref{e13.40} in Theorem \ref{t13.3}.
We also get that the complementary region $\big\{ u_1 \leq 0 \big\}$
is not too thin near a point of $\d_1$; see Theorem \ref{t13.4}.

In Section \ref{recti} we show that if $W_1$ satisfies our nondegeneracy condition
(we'll also say that $\Omega_1 = \big\{ u_1 > 0 \big\}$ is a good region),
then $\d_1 = \d \Omega_1$ is locally Ahlfors-regular and (uniformly) rectifiable.
The argument goes nearly as in \cite{AC}, and is based on the fact that
when $(\u,\W)$ is a minimizer, $\Delta u_1 +C$ is a positive measure
(that we can also estimate). 
We use the non degeneracy results of Section \ref{good}
to show that the restriction of this measure to 
$\d_1$ is locally Ahlfors-regular (Proposition \ref{t17.1}),  
compare it to the total variation measure $D \1_{\Omega_1}$,
show that $\Omega_1$ is a set of finite perimeter, and get the rectifiability
of $\d_1$ almost for free (as the reduced boundary of $\Omega_1$).
We also deduce from this a representation formula for 
$\Delta u_{1,+}$ in terms of the density of that measure with respect 
to the restriction of $\H^d$ to $\d_1$; see Proposition \ref{t17.2}.
Finally we check that $\d \Omega_1$ is locally uniformly rectifiable,
with ``Condition B'' and big pieces of Lipschitz graphs, because this 
follows rather easily from the results of Section \ref{good}. 
See Proposition~\ref{t17.3}.

The Lipschitz bounds in Sections \ref{lip} and \ref{global} allow us to define
the blow-up limits of $(\u,\W)$ at a point, and the non degeneracy results
of Section \ref{good} will often be the best way to make sure that these limits are nontrivial.
Before we really get to that, we need a theorem about limits.
We take care of this in Section \ref{limits}. We introduce a notion of
local minimizer for a functional $J$ in an open set $\O$,
and prove that under reasonable assumptions, if we have
pairs $(\u_k,\W_k)$ of local minimizers in $\cal O$ of functionals $J_k$
(associated to domains $\Omega_k$, and defined as $J$ above), 
and if the $\u_k$ converge to a limit $\u$,
then we can find $\W$ such that $(\u,\W)$ is a 
local minimizer in $\O$ of the natural limit functional $J$.
See Theorem \ref{t14.1} and Corollary \ref{t14.5}.

We use this in Section \ref{blowup} to prove that if $(\u,\W)$
is a minimizer for $J$ and $\u_\infty$ is a blow-up limit of $\u$
at some point $x_0$ such that $\u(x_0)=0$, then (under some
reasonable smoothness assumptions, in particular on the volume term $F$
and on $\Omega$ if $x_0 \in \Omega$) 
we can find $\W_\infty$ such that $(\u_\infty,\W_\infty)$ 
is a local minimizer (in an infinite domain, which is $\R^n$
when $x_0$ is an interior point of $\Omega$, and otherwise
is a blow-up limit of $\Omega$) 
of a simpler functional $J_\infty$.
The functional $J_\infty$ is simpler because it does not have 
an $M$-term, and its $F$ term is like \eqref{e1.7}, with
constant functions $q_i$. See Theorem~\ref{t15.1} or Corollary \ref{t15.3}.

In Sections \ref{2phases} and \ref{1phase} we find various situations where
the blow-up limits of $\u$ at a point are given by a simple formula with affine 
functions. The main result of Section \ref{2phases} is Corollary~\ref{t16.4},
where we get such an expression as soon as we can find two phases
$\varphi \neq \psi \in I$ (as above)
such that the Alt-Caffarelli-Friedman functional $\Phi_{\varphi,\psi}(r)$
that we use in Section \ref{mono} has a nonzero limit at $r=0$. 
This is the case, for instance, if the nondegeneracy condition of Section \ref{good}
is satisfied for (the indices $i$ that come from) $\varphi$ and $\psi$,
and the origin lies in $\d \big\{ v_\varphi > 0 \big\} \cap \d \big\{ v_\psi > 0 \big\}$.
The proof is not surprising (all the ingredients were prepared in the previous section)
and is based on a careful study of the case of equality in the monotonicity theorem
of \cite{ACF}; see Theorem \ref{t16.3}.

The main result of  Section \ref{1phase} is Corollary \ref{t18.3}, which concerns the case 
when we cannot find $\varphi$ and $\psi$ as above, but the origin lies in 
$\d \big\{ v_\varphi > 0 \big\}$ for some $\varphi \in I$ that satisfies the 
nondegeneracy condition of Section \ref{good}.
Then we use a result of \cite{CJK2} to show that, in dimensions $n \leq 3$,
some blow-up limits of $\u$ are composed of just one phase $v$, which is the
positive part of an affine function.

We also show (in any dimension) that the conclusion of  Corollary \ref{t16.4}
or Corollary~\ref{t18.3} holds when the origin is a point of 
$\d_\varphi = \d \big\{ v_\varphi > 0 \big\}$
where $\d_\varphi$ has a tangent (and the nondegeneracy condition holds). 
By Proposition \ref{t17.3}, this happens almost everywhere on $\d_\varphi$.
See Proposition \ref{t18.5} and Remark \ref{t18.6}.

In Section \ref{good2} we summarize the situation: when all the regions satisfy
the nondegeneracy condition of Section \ref{good}, for instance, a given point
$x_0 \in \Omega$ can only lie in at most two sets $\d_\varphi$ (at most one if
$x_0 \in \d \Omega$, and so there is a small neighborhood of $x_0$ where
$(\u,\W)$ is a minimizer of a variant of the Alt-Caffarelli-Friedman functional
with only one or two phases. See Lemmas \ref{t19.1}, \ref{t19.4}, and \ref{t19.5}.
While we were preparing this manuscript, we learned about a recent result
of Bucur and Velichkov \cite{BV}, with a version of Lemma \ref{t19.1},
obtained by a completely different method (in particular a monotonicity formula 
with three phases). We discuss this a little more in Section \ref{good2}.

At the end of  Section \ref{good2}, we are left with a series of good sufficient
conditions for $\u$ to have blow-up limits composed of affine functions. These
conditions imply the asymptotic flatness of the free boundaries $\d_\varphi$,
which holds everywhere inside $\Omega$ under nondegeneracy conditions 
and if $n \leq 3$, and only almost everywhere when $n > 3$; 
see Proposition \ref{t16.5} and Lemma \ref{t18.4}.
We do not continue the regularity study of the $\d_\varphi$, that would
normally lead to local $C^{1,\alpha}$-regularity, probably under additional
nondegeneracy assumptions, but Lemma~\ref{t19.1} says that this now is a problem
about functionals with at most $2$ phases. 

Section \ref{vari} can be seen  as an appendix. We complete some of the proofs of 
Sections \ref{2phases} and \ref{1phase} with a standard computation of first variation, 
which we only do for the affine blow-up limits of our problem, 
and then use the representation formula of 
Proposition \ref{t17.2} to show that, at points where $\u$ has a nice blow-up limit, 
there are Euler-Lagrange relations between the values of the normal derivative 
of $\u$ on both sides of the free boundary, and the multipliers that comes 
from the derivatives of $F$ in the corresponding directions. These formulas are 
deduced from the corresponding formulas for the blow-up limits, which we derived
with the first variation argument. See Proposition \ref{t20.1}.

\ms\noindent
{\bf Acknowledgements.}
The authors wish to thank D. Bucur, A. Henrot, M. Pierre, T. Toro, B. Velichkov, 
for helpful discussions during the preparation of this paper.

The first author acknowledges the generous support of the Institut Universitaire de France, 
and of the ANR (programme blanc GEOMETRYA, ANR-12-BS01-0014). 
The second author was partially supported by the CNRS grant PEPS-PTI Thorgal.
The third author was partially supported by the NSF grant DMS 1069225,
the Bergman Trust, and wishes to thank Universit\'e de Paris Sud (Orsay) for 
a monthlong visit during  June-July 2012.
The fourth author was partially supported by the NSF grants DMS 1220089 (CAREER), DMS 1344235 (INSPIRE), DMR 0212302 (UMN MRSEC Seed grant), and the the Alfred P. Sloan Fellowship. The fourth author would also like to acknowledge the support of the CNRS and hospitality of Universit\'e Paris-Sud, Orsay, France, in Oct-Dec of 2012 and April of 2013,  when parts of this work were completed.

\ms\noindent
{\bf Frequently used notation.}
\pari
$B(x,r)$ is the open ball centered at $x$, with radius $r$
\pari 
$|A|$ denotes the Lebesgue measure of a set $A$
\pari
$C$ and $c$ are positive constants that change from line to line; usually $C$
is large and $c$ is small.
\pari
$W^{1,2}(\R^n)$ is the Sobolev space (one derivative in $L^2$)
\pari
$J(\u,\W) = E(\u) + M(\u) + F(\W)$ is our main functional; see \eqref{e1.3}, \eqref{e1.4}, \eqref{e1.5}
\pari
$\F = \F(\Omega)$ is our class of acceptable pairs $(\u,\W)$;
see Definition \ref{d1.1}
\pari
$\F(\O,\Omega)$ is its variant for local minimizers in a domain $\O$;
see the start of Section \ref{limits}
\pari
$S_r = \d B(0,r)$
\pari
${\cal W}(\Omega)$ is the set of acceptable $N$-uples $\W$;
see the beginning of Section \ref{positive}
\pari
$A \Delta B$ is the symmetric difference $(A \sm B) \cup (B \sm A)$
\pari
$\Omega_i = \big\{ x\in \R^n \, ; \, u_i(x) > 0 \big\}$ is usually strictly smaller than $\Omega$  and may be smaller than $W_i$
\pari
$I = [1,N] \times \{ -1, 1 \}$ denotes our set of phases, and we set
\pari
$v_\varphi = [\varepsilon u_i]_+ = \max(0, \varepsilon u_i)$ and 
$\Omega_\varphi = \big\{ x\in \R^n \, ; \, v_\varphi(x) > 0 \big\}$
for $\varphi = (i,\varepsilon) \in I$
\pari
$\Phi_j$ and $\Phi$ define our Alt-Caffarelli-Friedman functional; see \eqref{e8.4}
\pari
$\Phi_\varphi^0$ and $\Phi_{\varphi_1,\varphi_2}^0$ are their new names
in Sections \ref{2phases} and later; see \eqref{e16.11} and \eqref{e16.17}
\pari 
$\Phi_{\varphi,k}$ and $\Phi_{\varphi_1,\varphi_2,k}$ are the same ones,
along a blow-up sequence; see \eqref{e16.13} and \eqref{e16.18}.

\section{Motivation for our main functional}  \label{motivation} 

In this section we try to explain the connection between the functional defined
in (\ref{e1.5}) and our initial localization problem.
Let us start with a domain $\Omega \i \R^n$, which we may assume to be bounded, 
and the operator ${\cal L} = -\Delta+{\cal V}$, 
with ${\cal V}$ bounded and nonnegative. For simplicity of exposition, let us also assume that $\Omega$ is smooth (only for the duration of this section: 
the relevant regularity conditions on $\Omega$, when necessary, are carefully tracked in the remainder of the paper). 
Let $w_{0}$ be the solution of ${\cal L}w_{0} = 1$ on $\Omega$, 
with the Dirichlet condition $w_{0} = 0$ on $\R^n \sm \Omega$. 
It is shown in \cite{FM} that the eigenfunctions of $\cal L$
can be estimated pointwise by the single function $w_{0}$, so we can expect
$w_0$ to give useful information on the localization of the eigenfunctions.

Given an integer $N \geq 2$, we want to  to split $\Omega$ into $N$ 
disjoint subregions $W_{j}$,
preferably with a nice boundary $\Gamma = \cup_{j} \d W_{j}$, so that the 
values of $w_{0}$ on $\Gamma$ give the best control (on $w_{0}$ and the eigenfunctions),
and for this we want the restriction of $w_{0}$ to $\Gamma$ to be as small as possible,
in a sense that will be discussed soon. Numerical experiments suggest that when we do this,
the eigenfunctions tend to be localized inside the regions $W_{j}$, even though the
precise mechanism why this happens is not clear to us. 
See \cite{FM}. 

One interesting way to  require this smallness is be to minimize something like
$\int_{\Gamma} w_{0}$, plus the same volume term $F(\W)$ as in the present paper.
We shall not do this here, and instead we shall encode smallness in a slightly more 
subtle and perhaps more natural way: via the energy corresponding to the governing operator ${\cal L}$. 
Let us associate, to any function $w\in W^{1,2}(\R^n)$ such that $w(x) = 0$ 
almost everywhere on $\R^n \sm \Omega$,  the energy
\begin{equation}\label{e2.1}
E_{0}(w) = \int_{\Omega} |\nabla w|^2+ {\cal V}\, w^2.
\end{equation}
Given the closed set $\Gamma \i \Omega$, we want to define
\begin{eqnarray}\label{e2.2}
J_{0}(\Gamma) = 
\inf\big\{E_{0}(w) \, ; \, w\in W^{1,2}(\R^n),
w(x) = 0  \text{ almost everywhere on } \R^n \sm \Omega, &\,&
\nonumber
\\
\text{ and }
w = w_{0} \text { on } \Gamma \big\}, &\,&
\end{eqnarray}
and for the moment let us assume that $\Gamma$ is composed of smooth 
hypersurfaces. Then it is possible to define the traces on $\Gamma$
of the functions $w$ and $w_{0}$ (just because they lie in $W^{1,2}(\R^n)$), 
and thus give a sense to the phrase ``$w=w_{0}$ on $\Gamma$''.

Of course $J_{0}(\emptyset) = 0$ (take $w=0$ in (\ref{e2.2})), and so $J_{0}(\Gamma)$
can be seen as the minimal amount of  energy that we need to pay, 
when we add the constraint that $w=w_{0}$ on $\Gamma$ to the basic condition that 
$w(x) = 0$ a.e. on  $\R^n \sm \Omega$\Bk. 

It does not make sense to minimize $J_{0}(\Gamma)$ without any constraint 
on $\Gamma$, because the infimum would be when $\Gamma = \emptyset$, 
so we add a term $F(\W)$ to our functional, that depends on $\Gamma$
(typically through the volumes of the components 
of $\Omega \sm \Gamma$), just to compensate and avoid this case.
Then we try to minimize $J_{0}(\Gamma) +F(\W)$. Using a reasonably smooth 
term $F$ that depends on the volumes seems to be the mildest way to avoid degeneracy; 
for instance requiring that the connected components 
of $\Omega \sm \Gamma$ have prescribed volumes looks a little too violent,
even though it gives an interesting (but harder) mathematical problem.
The definition of $F(\W)$ will be discussed later.  Assuming that this procedure can be justified and that a minimizer exists, we denote the latter by $w$.

Anyway, it is expected that picking sets $\Gamma$ on which $w_{0}$ is small helps
making $J_{0}(\Gamma)$ small, and our definition of smallness of $w_{0}$
on $\Gamma$ will be through the smallness of $J_{0}(\Gamma)$.
An apparently simpler choice of $E_0(w)$ would have been $\int_{\Omega} |w|^2$
(after all, we started the discussion with pointwise estimates on the functions themselves),
but the constraint that $w = w_{0}$ on $\Gamma$ does not really make sense
with the weaker $L^2$ norm, or in other words the infimum in (\ref{e2.2}) (say, on smooth
functions $w$) would be zero. Working in the class $W^{1,2}(\R^n)$ and with the energy
in (\ref{e2.1}) then seems to be the simplest reasonable choice, with an obvious connection
with our operator $-\Delta+ {\cal V}$.

\ms
Next we continue with ${\cal L} = -\Delta+ {\cal V}$, and show how to restate the minimization 
problem for $J_{0}$ more simply, in terms of the difference $v=w_{0}-w$.  Roughly speaking, the idea is that viewing $w_0$ as a solution to ${\cal L}w_{0} = 1$ on $\Omega$, 
with the Dirichlet condition $w_{0} = 0$ on $\R^n \sm \Omega$ and, trivially, $w_0=w_0$ on $\Gamma$, and viewing $w$ as a solution to ${\cal L}w = 0$ on $\Omega$, 
with $w = 0$ on $\R^n \sm \Omega$ and $w=w_0$ on $\Gamma$, we would have  ${\cal L}v= 1$ on $\Omega$, 
with the Dirichlet condition $v = 0$ on $\R^n \sm \Omega$ and $v=0$ on $\Gamma$. The latter problem is more natural to formalize and address in our context, but let us first convert this reasoning to the language of minimization.

By (\ref{e2.1}),
\begin{multline}\label{e2.3}
E_{0}(w) = E_{0}(w_{0}-v) = \int_{\Omega} |\nabla (w_{0}-v)|^2 +V(w_0-v)^2
\\[4pt]
= E_{0}(w_{0}) + E_{0}(v) - 2 \int_{\Omega} \langle\nabla v,\nabla w_{0}\rangle- 2 \int_{\Omega} {\cal V} \,vw_0,
\end{multline}
where for us $E_{0}(w_{0})$ is just a constant. 
Let us integrate by parts brutally; some justifications will come soon.
This yields 
\begin{equation}\label{e2.4}
-2\int_{\Omega} \langle\nabla v,\nabla w_{0}\rangle 
= 2 \int_{\Omega} v \Delta w_{0} - 2\int_{\d \Omega} v {\d w_{0} \over \d n}
=  2 \int_{\Omega}{\cal V}\, v  w_{0} - 2 \int_{\Omega} v 
\end{equation} 
because the boundary term $\int_{\d \Omega} v {\d w_{0} \over \d n}$
vanishes since $v = 0$ on $\R^n \sm \Omega$, and where we use the fact that
$ (-\Delta+{\cal V}) w_{0} = {\cal L}w_{0} = 1$ on $\Omega$.
Then (\ref{e2.3}) yields
\begin{equation}\label{e2.5}
E_{0}(w) = E_{0}(w_{0}) +  E_{0}(v) - 2\int_{\Omega} v.
\end{equation}
Concerning the integration by parts in \eqref{e2.4}, one way not to do
it is to use the variational definition of the function $w_{0}$, i.e., the fact
that it is the function $f\in W^{1,2}(\R^n)$ such that $f(x) = 0$
almost everywhere on $\R^n \sm \Omega$, and which minimizes
$\int_{\Omega} |\nabla f|^2  +{\cal V}f^2 - 2 f$ under these constraints.
Since for all $\lambda \in \R$, the function $w_{0}+\lambda v$ also
satisfies these constraints, we get that
\begin{equation}\label{e2.6}
\int_{\Omega} |\nabla w_{0}|^2  +{\cal V} w_0^2- 2  w_{0}
\leq \int_{\Omega} |\nabla (w_{0}+\lambda v)|^2 +{\cal V} (w_0+\lambda v)^2 - 2   (w_{0}+\lambda v),
\end{equation}
which implies the result of \eqref{e2.4} (just compute the derivative at $\lambda =0$).

So $J_{0}(\Gamma)$ is the same, modulo adding the constant $E_{0}(w_{0})$,
as
\begin{eqnarray}\label{e2.7}
J_{1}(\Gamma) &= &
\inf\Big\{ \int_{\Omega} |\nabla v|^2 +{\cal V} v^2 - 2 v  \, ; \, v\in W^{1,2}(\R^n),
\nonumber
\\
&\,&\hskip 1cm
v(x) = 0 \text{ almost everywhere on } \R^n \sm \Omega, \text{ and }
v = 0 \text { on } \Gamma \Big\}.
\end{eqnarray}
We can further simplify this and define $J_1(\Gamma)$ without our smoothness assumption on
$\Gamma$ (so far implicit in the understanding of what $v=0$ on $\Gamma$ means). Denote by $W_{i}$, $i\in I$, the connected components of 
$\Omega \sm \Gamma$, set $u_{i} = \1_{W_{i}} v$. 
If $\Omega$ and $\Gamma$ are smooth and $I$ is finite, then the functions $u_i$ have traces from both sides of $\Gamma$, and,
by a simple welding lemma (and the definition of the traces), our constraint that 
$v = 0$ on $\Gamma$ implies that $u_{i} \in W^{1,2}(\R^n)$,
i.e., that the distribution derivative of $u_{i}$ does not catch an extra piece
along $\Gamma$. See for instance Chapters 10-13 in \cite{D}, or rather
the discussion near \eqref{e4.13}-\eqref{e4.18} where we do similar manipulations
near spheres. 

Conversely, if the $u_{i} \in W^{1,2}(\R^n)$ are such that $u_{i} = 0$
almost everywhere on $\R^n \sm W_{i}$, it is easy to see that 
$v=\sum_{i} u_{i}$ satisfies the constraints in \eqref{e2.7}
(in order to prove that $v = 0$ on $\Gamma$, just compute the trace of $u_{i}$
from the side that does not lie in $W_{i}$). In other words,
\begin{equation}\label{e2.8}
J_{1}(\Gamma) =
\inf\Big\{ \sum_{i}\int_{W_{i}} |\nabla u_{i}|^2 +{\cal V} u_i^2 - 2 u_{i}  \, ; \, u_{i}\in W^{1,2}(\R^n),
u_{i}(x) = 0 \text{ a.e. on } \R^n \sm W_{i} \}.
\end{equation}
This is exactly the term $E(\u)+M(\u)$ of our functional (see \eqref{e1.3}
and \eqref{e1.4}), with $f_{i} = {\cal V}$ and $g_{i} = 2$ for $i\in I$. 
There is a small difference in the class of competitors. 
Here the $W_i$, $i\in I$, are the connected components
of $\Omega \sm \Gamma$, while in our description of the functional, we are
allowed to regroup some of them into a single set. In principle, we shall take
functions $F(\W)$ that are convex, in such a way that regrouping two regions
makes the functional larger, because $E(\u)+M(\u)$ does not change, but
$F(\W)$ increases. So the only case when we expect two components 
of $\Omega \sm \Gamma$ to be merged is when $I$ has more than $N$ elements, 
and we need to regroup some of them because we added the constraint that we 
do not use more than $N$ sets.
This difference should not disturb us much. If $\Omega \sm \Gamma$
has more than $N$ components, this may just be a sign that we chose 
$N$ a little too small, and anyway putting a constraint on the number of sets
seems less brutal than putting a constraint on the number of components
(which we could have done instead).

The advantage of this new definition is that we don't need to know that 
$\Gamma$ is smooth to define $J_{1}(\Gamma)$.
This is why we used this definition in the introduction, 
and intend to keep it in this paper. We still hope that, under mild
conditions on our data, the minimizers will provide smooth enough boundaries $\Gamma$
so that $J_{0}$ also makes sense, but we shall not attempt to check this and do 
the backward translation.

Notice that in the present situation we can restrict to nonnegative functions:
our special function $w_{0}$ is nonnegative, and all the interesting competitors 
for $w$ or the $u_{i}$ are nonnegative, because replacing $u_{i}$ with its
positive part always makes our functionals smaller.

\ms
Recall that we also want to add a volume term $F(\W)$, 
that for instance depends on the volumes $|W_{i}|$. 
Our main goal for this is to encourage the functional to choose nontrivial 
minimizers for which the $W_{i}$ have roughly comparable volumes, 
and in this respect the precise choice
of $F$ should mostly be a matter of experience. But we have no good reason
to make $F$ complicated, or to treat one of the component differently,
so choosing 
\begin{equation} \label{e2.12}
F(\W) = \sum_{i=1}^N f(|W_i|)
\end{equation}
looks like a good idea. We may even brutally decide to use \eqref{e1.6} for simplicity.

It is probably a good idea to choose $f$ convex, so that 
the functional will prefer to choose sets $W_i$ with roughly comparable volumes.
At this time, we also prefer to choose $f(t) \geq a t$ for some $a>0$;
this way, the non degeneracy assumptions of Sections \ref{number} and \ref{good} are
automatically satisfied for every $i$, so we get an upper bound on the number of
nonzero functions $u_i$ (even if we took $N = +\infty$) and a better description
of the sets $\Omega_i = \big\{ x\in \Omega \, ; \, u_i(x) > 0 \big\}$.
Recall that here the $u_i$ are nonnegative, so we do not care for
$\big\{ u_i < 0 \big\}$, and that $\Omega_i$ is open because $\u$
is continuous when $(\u,\W)$ is a minimizer.
For instance, we get that the $\d\Omega_i$ are uniformly rectifiable
(as in Section \ref{recti}), and that a given point $x_0$ cannot lie
on more than two sets $\d\Omega_i$ (one if $x_0 \in \d \Omega$).
In dimension $n \leq 3$, we even expect the boundary $\d\Omega_i$
to be smooth, but we only show that $\d\Omega_i$ has flat blow-up limits 
at every point.

More directly, choosing $f$ such that $f(t) \geq a t$ is a way of making sure
that if a piece of $W_i$ is not so useful to make $E(\u)+M(\u)$ smaller, 
we may as well remove it and save on the term $F(\W)$. 
In terms of localization of eigenfunctions,
this means that we introduce a black zone $\Omega \sm \bigcup_i W_i$,
where $\u = 0$, and we bet that the eigenfunctions will not live in that zone.

Another option would be to try to force the $W_i$ to cover $\Omega$,
for instance by choosing $f$ decreasing. Then, by the results of Section \ref{positive},
we expect no black zone. If $f$ is strictly convex, the regions $W_i$ that do not have
the smallest volume satisfy the degeneracy assumptions of Section \ref{good},
and we get a better description for them. But not for the regions with the smallest
volume, 
just as in the previous case we do not get a good description of the black zone.

We should insist on the fact that the regularity results that we get
for the free boundaries $\d \big\{ x\in \Omega \, ; \, \pm u_i(x) > 0 \big\}$ are
mostly a consequence of our choice of function $F$, not a corollary
of any localization property for the eigenfunctions. However the first evidence
that we have 
suggests that the eigenfunctions have a tendency to live in the regions computed
by the functional, and away from the black zone.

\ms
Let us also comment on the choice of $N$. If for instance 
$F(\W) = \sum_i f(|W_i|)$ of some $f$ such that $f(t) \geq a t$
(as above), the non degeneracy assumptions of Section \ref{good} are
satisfied, and we even get that there is a constant $\tau > 0$, that does
not depend on $N$, such that $|W_i| \geq \tau$ when $|W_i| >0$;
see  Proposition \ref{tn1}.
Thus there is a bounded number of nontrivial sets $W_i$, regardless of
our initial choice of $N$. In other words, even though we select $N$ in
advance in our theory, as soon as $N$ is large enough (depending
in particular on  $a$ above and $|\Omega|$), the minimizers
will no longer depend on $N$ and will not have too many pieces.

There are some sort of Euler-Lagrange conditions on the minimizers,
that we could obtain with a (heuristic) computation of first variation of
the domains. For instance, along a smooth boundary between $\Omega_1$
and $\Omega_2$, the normal derivatives ${\d u_1 \over \d n}$ and 
${\d u_2 \over \d n}$ would need to satisfy the relation
\begin{equation} \label{e2.13}
\Big({\d u_1 \over \d n}\Big)^2 -   \Big({\d u_2 \over \d n}\Big)^2 = 
f'(|W_1|) - f'(|W_2|).
\end{equation}
We do not do this first variation computation here, but in Section \ref{vari},
we do it for the nice blow-up limits of $u$, and get an almost-everywhere
variant of \eqref{e2.13} and similar formulas. 
See near \eqref{e20.14} (where the fact that $\lambda_i = f'(|W_i|)$ 
comes from \eqref{e15.11} and \eqref{e2.12}). 

\ms
Starting with the next section, we shall forget about $\cal L$ and
the localization of eigenfunctions and return to a more general expression
for $J$, but of course the regularity results that we shall obtain can be seen 
as an encouragement for the use of our functional in this context.

\section{Existence of minimizers }  \label{existence} 

In this section we prove the existence of minimizers for our
functional $J$ in the class $\cal F$, under the following mild assumptions.
First we shall assume that
\begin{equation}\label{e3.1}
\Omega \text{ is a borel set, with } |\Omega| < +\infty,
\end{equation}
where we use the notation $|E|$ for the Lebesgue measure of any Borel
set $E \i \R^n$. Thus we don't need $\Omega$ to be bounded, or even open.
We also assume that for $1 \leq i \leq N$,
\begin{equation}\label{e3.2}
f_{i} \in L^\infty(\Omega),  f_{i}(x) \geq 0 \text{ almost-everywhere, and }
 g_{i} \in L^2(\Omega).
\end{equation}
We could be a little more general and  only assume that $f_{i}$ and $g_{i}$ 
are in slightly larger $L^p$-spaces, and also allow $f_{i}$ to be slightly negative,
but we don't expect to use that generality, and also we shall soon have
more restrictive conditions on the $g_{i}$.
See Remark \ref{r3.4} at the end of the section. 
Concerning the volume term, let us first assume that
\begin{equation}\label{e3.3}
F(W_1, \ldots, W_N) = \wt F(|W_1|, \ldots, |W_N|),
\text{ where }
\wt F : [0,|\Omega|]^N \to \R
\text{ is continuous.}
\end{equation}
We shall see at the end of the section how to deal with
other types of functions $F$, including the more classical
definition of $F(\W)$ by \eqref{e1.7}; 
see Corollary \ref{t3.5}. 
Now define $\cal F$ and $J$ as in Definition \ref{d1.1} and (\ref {e1.3})-(\ref {e1.5}).

\begin{thm}\label{t3.1} 
Under the assumptions (\ref {e3.1})-(\ref {e3.3}),
we can find $(\u_{0},W_{0}) \in {\cal F}$ such that
\begin{equation}\label{e3.4}
J(\u_{0},\W_{0}) \leq J(\u,\W)
\text{ for all } (\u,\W) \in {\cal F}.
\end{equation}
\end{thm}

\begin{proof}
The proof will rather easily follow from the compactness of some Sobolev injections.
Let $\Omega$, $F$, and the $f_{i}$ be as in the statement, and set
\begin{equation}\label{e3.5}
m = \inf_{(\u,\W) \in {\cal F}} J(\u,\W).
\end{equation}
We shall see soon that $m$ is bounded, but let us first observe that $m < + \infty$:
just pick $\u=0$ and any acceptable disjoint collection $\W$ of subsets, 
and observe that $J(\u,\W) < +\infty$.
For $(\u,\W) \in {\cal F}$, set
\begin{equation}\label{e3.6}
E(\u) = \sum_{i=1}^N \int_{\Omega} |\nabla u_{i}|^2
= \sum_{i=1}^N \int_{W_{i}} |\nabla u_{i}|^2
\end{equation}
as in (\ref{e1.3}),
where the last identity comes from the fact that $u_{i} = 0$ a.e. outside of
$W_{i}$, hence $\nabla u_{i} = 0$ a.e. on $\Omega \sm W_{i}$ too.
The next lemma will show that $E(\u)$ controls the potentially negative terms 
$\int_{\Omega} u_{i}(x) g_{i}(x)$.

\begin{lem}\label{l3.2} 
If $E$ is a measurable set such that $|E| < +\infty$ and 
$u\in W^{1,2}(\R^n)$ is such that $u(x) = 0$ a.e. on $\R^n \sm E$, then
\begin{equation}\label{e3.7}
\int_{E} u^2 \leq C |E|^{2/n} \int_{E} |\nabla u|^2,
\end{equation}
where $C$ depends only on $n$.
\end{lem}

\begin{proof}
The lemma is a fairly easy consequence of the standard Poincar\'e's inequality
(on balls and where we subtract mean values); 
for the sake of completeness, we shall add a short Section \ref{poincare}
where we recall this inequality and prove Lemma \ref{l3.2} and its analogue on
a sphere. In the mean time let us refer to \cite{HP}, Lemma 4.5.3.
\qed
\end{proof}

We are ready to check that $m > -\infty$. For $(\u,\W) \in {\cal F}$
and $1 \leq i \leq N$, observe that
\begin{equation}\label{e3.8}
\Big|\int_{\Omega} u_{i}(x) g_{i}(x) dx \Big|
\leq ||u_{i}||_{2} ||g_{i}||_{2} 
\leq C |\Omega|^{1/n} \Big\{\int_{\Omega} |\nabla u|^2\Big\}^{1/2} ||g_{i}||_{2}
\end{equation}
by Lemma \ref{l3.2}, and then
\begin{equation}\label{e3.9}
\Big|\sum_{i} \int_{\Omega} u_{i}(x) g_{i}(x) dx \Big| \leq C E(\u)^{1/2},
\end{equation}
where $C$ now depends on the data $\Omega$ and the $||g_{i}||_{2}$. 
Since $F(\W)$ is bounded by (\ref {e3.3}), we get that
\begin{eqnarray}\label{e3.10}
J(\u,\W) &=& E(\u) + \sum_{i} \int [u_{i}(x)^2 f_{i}(x) - u_{i}(x) g_{i}(x)] dx + F(\W)
\nonumber \\
&\geq& E(\u) - \Big|\sum_{i} \int u_{i}(x) g_{i}(x) dx \Big| - C 
\geq E(\u) - C E(\u)^{1/2} - C,
\end{eqnarray}
because $f_{i} \geq 0$.
This is bounded from below; so $m > -\infty$.

\ms
Return to the proof of Theorem \ref{t3.1}. Let $\{(\u_{k},\W_{k})\}$, $k \geq 0$,
be a minimizing sequence in $\cal F$, which means that 
\begin{equation}\label{e3.11}
\lim_{k \to +\infty} J(\u_{k},\W_{k}) = m.
\end{equation}
We want to extract a converging subsequence and show that the limit is a minimizer
in $\cal F$. The following compactness lemma will help.

\begin{lem}\label{l3.3} 
Let $\Omega$ be a measurable set, with finite measure $|\Omega|$,
and denote by $U_{\Omega}$ the set of functions 
$u\in W^{1,2}(\R^n)$ such that $u(x) = 0$ a.e. on $\R^n \sm \Omega$
and $\int_{\Omega} |\nabla u|^2 \leq 1$. Then $U_{\Omega}$ is contained 
in a compact subset of $L^2(\R^n)$.
\end{lem}

\begin{proof}
We start with the case when $\Omega$ is bounded, because then we can
use the compactness of the Sobolev injection, as stated for instance
in \cite{Z}, Theorem 2.5.1. Let $B$ be a ball that contains $\Omega$;
it is well known, and easy to check, that 
$U_{\Omega} \i W^{1,2}_{0}(2B)$. By Theorem 2.5.1 in \cite{Z},
it is contained in a compact subset of $L^2(\R^n)$, and the result follows.

Let us now treat the general case when $|\Omega| < +\infty$, which
is probably almost equally well known. Since $L^2(\R^n)$ is complete, 
it is enough to show that $U_{\Omega}$ is completely bounded in $L^2(\R^n)$. 
That is, for each given $\varepsilon > 0$ we want to show
that $U_{\Omega}$ is contained in the union of a finite number of $L^2$-balls
of radius $\varepsilon$.

Let $R \geq 1$ be large, to be chosen soon (depending on $\varepsilon$),
and let $\varphi$ be a smooth cut-off function such that 
$0 \leq \varphi(x) \leq 1$ for $x\in \R^n$, 
$\varphi(x) = 1$ for $0 \leq |x| \leq R$, $\varphi(x) = 0$ for $|x| \geq 3R$,
and $|\nabla\varphi(x)| \leq 2/R$ everywhere. Then write any function 
$u \in U_{\Omega}$ as $u = \varphi u + (1-\varphi) u$. Set $v = (1-\varphi) u$,
and notice that $v \in W^{1,2}(\R^n)$, with 
$\nabla v = (1-\varphi) \nabla u - u \nabla \varphi$
(as a distribution), so that $\int_{\Omega} |\nabla v|^2 \leq 2\int_{\Omega} |\nabla u|^2
+ 2 \int_{\Omega} |u|^2|\nabla \varphi|^2 \leq 2 \int_{\Omega} |\nabla u|^2
+ 8 R^{-2} \int_{\Omega} |u|^2$. By Lemma \ref{l3.2} and because $u \in U_{\Omega}$,
we get that $\int_{\Omega} |u|^2 \leq C \int_{\Omega} |\nabla u|^2$ 
(where $C$ depends on $\Omega$), and hence
$\int_{\Omega} |\nabla v|^2 \leq 3 \int_{\Omega} |\nabla u|^2$ if $R$ is large enough.

But now $v = 0$ almost everywhere on $B(0,R)$, 
so Lemma \ref{l3.2} applies with $E = \Omega \sm B(0,R)$ and yields
$\int_{\Omega} |v|^2 \leq 3C |\Omega \sm B(0,R)|^{2/n} \leq \varepsilon^2/4$
if $R$ is large enough. That is, $||(1-\varphi) u||_{2} = ||v||_{2} \leq \varepsilon/2$.

Because of this, it is now enough to cover the set 
$U' = \big\{ u\varphi \, ; \, u\in U_{\Omega}\big\}$ by a finite number of
$L^2$-balls of radius $\varepsilon/2$. But $u\varphi \in W^{1,2}(\R^n)$,
with $\int_{\Omega} |\nabla (u\varphi) |^2 \leq 3$ by the same proof as above,
and in addition $u\in W^{1,2}_{0}(B(0,4R))$, so Theorem 2.5.1 in \cite{Z}
says that $U'$ is relatively compact in $L^2(\R^n)$, hence totally
bounded, and the relative compactness of $U_{\Omega}$ follows.
\qed
\end{proof}

Observe that by (\ref {e3.10}), $E(\u_{k})$ stays bounded along our 
minimizing sequence,  so by Lemma \ref{l3.3} the components 
$u_{i,k}$, $1 \leq i \leq N$, of $\u_k$ stay in a compact subset of $L^2(\R^n)$. 
This allows us to replace $\{(\u_{k},\W_{k})\}$ with a subsequence, 
which for convenience we still denote by $\{(\u_{k},\W_{k})\}$, 
for which $\u_{k}$ converges to a limit $\u$ in $L^2(\R^n)$. 
This just means that for $1 \leq i \leq N$, 
the component $u_{i,k}$ converges to $u_{i}$ in $L^2$. 
By extracting a new subsequence if needed, we can also 
assume that $u_{i,k}$ converges to $u_{i}$ pointwise almost everywhere.
This is more pleasant, because it will help in the definition of the
$W_{i}$. In the mean time, observe that $u_{i}$ automatically satisfies
our constraint that $u_{i}(x) =0$ a.e. on $\R^n \sm \Omega$. Similarly,
if in our definition of $F$ we added the constraint that some $u_{i}$
be nonnegative, this stays true as well for our limit. 
Notice that
\begin{eqnarray}\label{e3.12}
E(\u) &=& \sum_{i} \int_{\Omega} |\nabla u_{i}|^2
\leq \sum_{i} \liminf_{k \to +\infty} \int_{\R^n} |\nabla u_{i,k}|^2
\nonumber \\
&\leq& \liminf_{k \to +\infty} \sum_{i} \int_{\R^n} |\nabla u_{i,k}|^2
= \liminf_{k \to +\infty} E(\u_{k}) 
\end{eqnarray}
by the lower semicontinuity of  $\int_{\R^n} |\nabla u|^2$ when $u$ converges
(even weakly in $L^2$) to some limit. More easily,
\begin{equation}\label{e3.13}
\sum_{i} \int_{\Omega} u_{i}^2 f_{i} 
= \sum_{i}  \lim_{k \to +\infty} \int_{\Omega} u_{i,k}^2 \, f_{i}
\end{equation}
because $f_{i} \in L^\infty$ and the $u_{i,k}$ converge in $L^2$, and
\begin{equation}\label{e3.14}
\sum_{i} \int_{\Omega} u_{i} g_{i} 
= \sum_{i}  \lim_{k \to +\infty} \int_{\Omega} u_{i,k} \, g_{i}
\end{equation}
because $g_{i} \in L^2$, so, with the notation of (\ref{e1.4}),
\begin{equation}\label{e3.15}
M(\u) = \lim_{k \to +\infty} M(\u_{k}).
\end{equation}
We also need to take care of the volumes. We don't know how to make the 
characteristic functions $\1_{W_{i,k}}$ converge, but at least we can replace 
$\{(\u_{k},\W_{k})\}$ with a subsequence for which each $V_{i,k} = |W_{i,k}|$
converges to a limit $l_{i}$. We want to associate to $\u$ a collection
of sets $W_{i}$ such that
\begin{equation}\label{e3.16}
|W_{i}| = \lim_{k \to +\infty} V_{i,k} = l_i
\end{equation}
because then we will get that
\begin{equation}\label{e3.17}
F(\W) = \wt F(|W_{1}|,\ldots,|W_{N}|) = 
\lim_{k \to +\infty} \wt F(V_{1,k},\ldots,V_{N,k})
= \lim_{k \to +\infty} F(\W_{k})
\end{equation}
because $\wt F$ is continuous.

Let $Z$ denote the set of $x\in \Omega$ for which 
$\u(x)$ is not the limit of the $\u_{k}(x)$, or
for some choice of $k$ and $i$, $x\in \Omega \sm W_{i,k}$ but
$u_{i,k}(x) \neq 0$; then $|Z| = 0$ by definitions; it will be more convenient
to work in $\Omega' = \Omega \sm Z$.

Also set $W'_{i} = \big\{ x\in \Omega' \, ; \, u_{i}(x) \neq 0 \big\}$ and similarly,
for each $k$, $W'_{i,k} = \big\{ x\in \Omega' \, ; \, u_{i,k}(x) \neq 0 \big\}$.
By definition of $\Omega'$, $W'_{i,k} \i W_{i,k}$, and hence the $W'_{i,k}$,
$1 \leq i \leq N$, are disjoint (recall that $(\u_{k},\W_{k}) \in \F)$).

If $x\in W'_{i}$, then $u_{i}(x) \neq 0$ and hence $u_{i,k}(x) \neq 0$ 
and $x\in W'_{i,k}$ for $k$ large.
In particular, 
\begin{equation}\label{e3.18}
\text{the $W'_{i}$, $1 \leq i \leq N$, are disjoint},
\end{equation}
because the $W'_{i,k}$ are disjoint for each $k$. Also,
$\1_{W'_{i}} \leq \liminf_{k \to +\infty} \1_{W'_{i,k}}$ everywhere, and
by Fatou
\begin{equation}\label{e3.19}
|W'_{i}| = \int \1_{W'_{i}} \leq \liminf_{k \to +\infty} \int\1_{W'_{i,k}}
= \liminf_{k \to +\infty} |W'_{i,k}| \leq \lim_{k \to +\infty} |W_{i,k}| = l_{i}.
\end{equation}

Notice that $\sum_{i} l_{i} \leq |\Omega|$ because for each $k$,
$\sum_{i} V_{i,k} \leq |\Omega|$ (recall that the $W_{i,k}$ are disjoint 
when $(\u_{k},\W_{k}) \in \F)$).
Hence we can add disjoint measurable sets to the $W'_{i}$ if needed, 
and get new sets $W_{i}$, $1 \leq i \leq N$, such that $W'_{i} \i W_{i} \i \Omega$,
the $W_{i}$ are still disjoint (see \eqref{e3.18}), and $|W_{i}| = l_{i}$, as 
required for (\ref{e3.16}).

We now check that with this choice of $\W = (W_{1}, \ldots, W_{N})$,
the pair $(\u,\W)$ lies in $\F$. The fact that $u_{i} \in W^{1,2}(\R^n)$
comes from the convergence of $u_{i,k}$ in $L^2$, and our uniform bound
for $E(\u_{k})$. We have that $u_{i}(x) = 0$ for a.e. $x\in \R^n \sm \Omega$
by pointwise convergence, and then $u_{i}(x) = 0$ for a.e. $x\in \R^n \sm W_{i}$
by definition of $W_i \supset W'_i$. So $(\u,\W)$ lies in $\F$. Since in addition
\begin{equation}\label{e3.20}
J(\u,\W) \leq \liminf_{k \to +\infty} J(\u_{k},\W_{k}) \leq m
\end{equation}
by (\ref{e3.12}), (\ref{e3.15}), (\ref{e3.17}), and (\ref{e3.11}),
$(\u,\W)$ is the desired minimizer.
\qed
\end{proof}

\begin{rem}\label{r3.4} 
In Theorem \ref{t3.1}, we can replace the assumption (\ref{e3.2}) with the
following weaker set of assumptions: for $1 \leq i \leq N$,
\begin{equation}\label{e3.21}
f_i \in L^r(\Omega) \ \hbox{ and } g_i \in L^p(\Omega),
\end{equation}
as soon as we pick $r > n/2$ and $p > {2n \over n+2}$ if $n \geq 2$,
and $r \geq 1$ and $p \geq 1$ if $n=1$, and
(instead of our assumption that $f_i \geq 0$)
\begin{equation}\label{e3.22}
||f_{i,-}||_{L^r(\Omega)} \leq  \varepsilon(r,n,\Omega, F)
\end{equation}
with the same $r$, where $f_{i,-}$ denotes the negative part of $f_i$,
and where  $\varepsilon(r,n,\Omega, F)$
is a small positive constant that depends only on $n$, $r$, $F$, and $|\Omega|$.
\end{rem}

\begin{proof}
We just sketch the proof because this remark is not so important
(and probably \eqref{e3.22} is hard to use). For the different exponents,
we just want to use better Sobolev embeddings. We claim that 
Lemmas \ref{l3.2} and \ref{l3.3} still hold when we replace
$||u||_2$ with $||u||_q$, as long as we take $q < {2n \over n-2}$
when $n \geq 2$ and $q \leq +\infty$ when $n=1$; we just need to
know that the Sobolev embedding is still compact with such exponents, and
follow the proof above.

Thus, in the argument above, we can get the $u_k$ to converge in 
$L^q$. We can choose $q$ to be the conjugate exponent of $p$, 
and this gives enough control on the integrals $\int u_{i,k} g_i$
(as in \eqref{e3.14}). Or we choose $q/2$ to be the conjugate exponent
of $r$, and we get a good control on $\int u_{i,k}^2 f_i$, as in \eqref{e3.13}.
The other estimates are as as above, except that we also need to replace 
\eqref{e3.10} because \eqref{e3.22} is weaker than our earlier assumption
that $f_i \geq 0$. So we keep $q/2 = r'$, the conjugate exponent of $r$, and 
say that
\begin{eqnarray}\label{e3.23}
J(\u,\W) &=& E(\u) + \sum_{i} \int [u_{i}(x)^2 f_{i}(x) - u_{i}(x) g_{i}(x)] dx + F(\W)
\nonumber \\
&\geq& E(\u) - \sum_{i} \int u_{i}^2(x) f_{i,-}(x) dx - \Big|\sum_{i} \int u_{i}(x) g_{i}(x) dx \Big| - C 
\nonumber \\
&\geq& E(\u) - C \sum_{i} ||u_i^2||_{r'}  ||f_{i,-}||_{r} - C E(\u)^{1/2} - C
\nonumber \\
&=& E(\u) - C \sum_{i} ||u_i||_{q}^2  ||f_{i,-}||_{r} - C E(\u)^{1/2} - C
\\
&\geq& E(\u) - C \sum_{i}  ||\nabla u_{i}||_{2}^2 \, ||f_{i,-}||_{r} - C E(\u)^{1/2} - C
\nonumber\\
&\geq& E(\u) - C E(\u) \sup_i ||f_{i,-}||_{r} - C E(\u)^{1/2} - C
\geq {1 \over 2} E(\u)  - C E(\u)^{1/2} - C
\nonumber
\end{eqnarray}
if $||f_{i,-}||_{r}$ is small enough; then we can continue the argument as above,
and the remark follows. \qed
\end{proof}

\ms
We mentioned earlier that Theorem \ref{t3.1} is enough for the application
that we have in mind, but since we feel bad about the case when our volume
term $F(\W)$ is given by the more standard formula \eqref{e1.7}, let us mention
the following variant. We want to replace \eqref{e3.3} with the 
monotonicity condition
\begin{equation}\label{e3.24}
F(\W) \leq F(\W') \ \text{ whenever  $W_i \i W'_i$ for $1 \leq i \leq N$,}
\end{equation}
whose effect is that we should try to take the $W_j$ as small as we can 
(all things being equal otherwise). Notice that \eqref{e3.24} is satisfied
when $F(\W) = \sum_i \int_{W_i} q_i$, as in \eqref{e3.24}, as soon as
the $q_i$ are nonnegative and integrable on $\Omega$. 
The integrability of the $q_j$ also gives the continuity of $F$, which we define
as follows.

Denote by $A \Delta B = (A \sm B) \cup (B \sm A)$ the symmetric difference  
between two sets $A$ and $B$; if $\W = (W_1, \ldots, W_N)$ and 
$\W' = (W'_1, \ldots, W'_N)$ are $N$-uples of disjoint subsets of $\Omega$,
set
 \begin{equation}\label{e3.25}
\W \Delta \W' = \bigcup_i  W_i \Delta W'_i
= \bigcup_i [W_i \sm W'_i] \cup [W_i' \sm W_i],
\end{equation}
and then
\begin{equation}\label{e3.26}
\delta(\W,\W') = |\W \Delta \W'|
\simeq \sum_i |W_i \Delta W'_i|.
\end{equation}
We say that the sequence $\{ \W_k \}$ tends to $\W$ when 
$\lim_{k\to +\infty} \delta(\W,\W_k) = 0$, and then we say that $F$ 
is continuous at $\W$ when $F(\W) = \lim_{k \to +\infty} F(\W_k)$ when 
$\{ \W_k \}$ tends to $\W$. In other words, $F$ is continuous when
\begin{equation}\label{e3.27}
\lim_{k \to +\infty} F(\W_k) = F(\W)
\text{ whenever } \lim_{k \to +\infty} |W_i \Delta W_{i,k}| = 0
\text{ for } 1 \leq i \leq N.
\end{equation}

Here is the variant of Theorem \ref{t3.1} that we want to prove. 

\begin{cor}\label{t3.5} 
Assume \eqref{e3.1}, \eqref{e3.2}, \eqref{e3.24}, and that $F$
is continuous (as in \eqref{e3.27}). Then we can find 
$(\u_{0},W_{0}) \in {\cal F}$ such that
\begin{equation}\label{e3.28}
J(\u_{0},\W_{0}) \leq J(\u,\W)
\text{ for all } (\u,\W) \in {\cal F}.
\end{equation}
\end{cor}

\begin{proof}
We proceed as in the proof of Theorem \ref{t3.1}, except that instead 
of completing the $W_i'$ to get sets $W_i$ such that \eqref{e3.18} holds, 
we just keep $W_i = W'_i$.
The sets $W_{i,k} \cap W'_i$ converge to $W'_i$
(that is, each $|(W_{i,k} \cap W'_i) \Delta W'_i| = |W'_i \sm W_{i,k}|$ tends
to $0$), because every $x\in W'_i$ lies in $W_{i,k}$ for $k$ large,
and then
\begin{equation}\label{e3.29}
F(W'_1, \ldots, W'_N) = \lim_{k \to +\infty} F(W_{1,k} \cap W'_i, \ldots, W_{N,k} \cap W'_i) 
\leq \liminf_{k \to +\infty} F(W_{1,k}, \ldots, W_{N,k}) 
\end{equation}
by \eqref{e3.27} and \eqref{e3.24}. The rest of the proof is the same.
\qed
\end{proof}

\section{Poincar\'e inequalities and restriction to spheres}  \label{poincare} 

We record here some easy properties of functions in the Sobolev spaces
$W^{1,p}$, $1\leq p < +\infty$. In particular, some consequences of
the Poincar\'e inequalities (like Lemma \ref{l3.2} above),
the fact that the restriction of $u \in W^{1,p}$ to almost every sphere
$S_{r} = \d B(0,r)$ lies in $W^{1,p}(S_{r})$, and a way to glue
two Sobolev functions along a sphere.

We first recall the standard Poincar\'e inequalities. We shall find it convenient 
to use the notation
\begin{equation}\label{e4.1}
m_{B} u = \fint_{B} u = {1 \over |B|} \int_{B} u
\end{equation}
for the average of a function $u$ on a set $B$, which will often be a ball.
With this notation, the standard Poincar\'e inequality for a ball says that 
if $B = B(x,r) \i \R^n$ and $u\in W^{1,p}(B)$ for some $p \in [1,+\infty)$, 
then $u\in L^p(B)$ and
\begin{equation}\label{e4.2}
\fint_{B} |u-m_{B}u|^p \leq C_{p} r^p\fint_{B} |\nabla u|^p.
\end{equation}

We shall also use Poincar\'e inequalities on spheres, for which we want to use similar
notation. On the spheres $S_{r} = \d B(0,r)$, we shall systematically use the
surface measure, which we denote by $\sigma$; we could also have used the equivalent
Hausdorff measure $\H^{n-1}$, but $\sigma$ will be simpler and in particular we
won't need to worry about normalization in the Fubini-like formula
\begin{equation}\label{e4.3}
\int_{B(0,r)} f(x) dx = \int_{0}^r \int_{S_{t}} f(x) d\sigma(x) dt,
\end{equation}
which we shall use from time to time.
We will use the notation  $\sigma(E) = |E|_{\sigma}$ for the surface measure 
of a measurable set $E \i S_{r}$, and
\begin{equation}\label{e4.4}
m_{E}^\sigma u = \fint_{E} u d\sigma= {1 \over |E|_{\sigma}} \int_{E} u(x) d\sigma(x)
\end{equation}
when $E \i S_{r}$ is such that $\sigma(E) > 0$
and $u$ is a measurable function at least defined on $E$.

It is easy to define the Sobolev spaces $W^{1,p}(S_{r})$, $1 \leq p < +\infty$,
for instance by making smooth local changes of variable that sent surface disks
in $S_{r}$ to disks in $\R^{n-1}$. We can then define the distribution gradient
$\nabla_{t} u$ for $u\in W^{1,p}(S_{r})$, in such a way as to coincide with the
usual notion when $u$ is smooth. We shall use the Poincar\'e inequalities on
surface disks, which says that if $1 \leq p < +\infty$,
$u \in W^{1,p}(S_r)$, and $D = B(x,s) \cap S_r$ for some
choice of $x\in S_r$ and $0 < s \leq 2r$, then
\begin{equation}\label{e4.5}
\fint_{D} |u-m_{B}u|^p d\sigma \leq C_{p} s^p\fint_{D} |\nabla_{t} u|^p d\sigma.
\end{equation}
This can be proved just like \eqref{e4.2} above; here $C_{p}$
does not depend on $r$, $x$, $s$, or $u$.

We shall now prove an analogue of Lemma \ref{l3.2} on the sphere $\d B(0,r)$.
The proof will also apply on $\R^n$, and give a proof of Lemma \ref{l3.2}.
We state the result for all $p$, but we are mostly interested in $p=1$ or $2$.

\begin{lem}\label{l4.1} 
Set $S_{r} = \d B(0,r)$, let $E \i S_{r}$ be a measurable set such that 
$\sigma(S_{r}\sm E) > 0$, and suppose $u\in W^{1,p}(S_{r})$ 
for some $p\in [1,+\infty)$ is such that $u(x) = 0$ for $\sigma$-a.e.  
$x\in S_{r} \sm E$. Then
\begin{equation}\label{e4.6}
\int_{E} |u|^p \leq C_{p} r^p {\sigma(S_{r}) \over \sigma(S_{r} \sm E)} 
\int_{E} |\nabla_{t} u|^p
\end{equation}
and, in the special case when $\sigma(E) < {1 \over 2} \sigma(S_{r})$,
\begin{equation}\label{e4.7}
\int_{E} |u|^p \leq C_{p} \sigma(E)^{p \over n-1} \int_{E} |\nabla_{t} u|^p.
\end{equation}
Here $C_{p}$ depends only on $n$ and $p$.
\end{lem}

\begin{proof}
We start with the easier \eqref{e4.6}. 
Notice that $|u(x)-m_{B}^\sigma u|^p = |m_{B}^\sigma u|^p$ almost everywhere 
on $S_{r} \sm E$ (just because $u(x) = 0$). We average and get that
\begin{eqnarray}\label{e4.8}
|m_{B}^\sigma u|^p 
&=& \sigma(S_{r} \sm E)^{-1} 
\int_{S_{r} \sm E} |u(x)-m_{B}^\sigma u|^p d\sigma(x)
\nonumber
\\
&\leq& {\sigma(S_{r}) \over \sigma(S_{r} \sm E)}
 \fint_{S_{r}} |u-m_{B}^{\sigma} u|^p d\sigma
\leq C_{p} r^p {\sigma(S_{r}) \over \sigma(S_{r} \sm E)}
 \fint_{S_{r}} |\nabla_{t} u|^p d\sigma
\end{eqnarray}
by \eqref{e4.5}; then 
\begin{eqnarray}\label{e4.9}
\int_{E} |u|^p &\leq& 
C_{p } \int_{S_{r}} \big[|u-m_{B}^{\sigma} u|^p +  |m_{B}^{\sigma} u|^p \big]
\leq C'_{p} r^p \int_{S_{r}} |\nabla_{t} u|^p + C_{p } |m_{B}^{\sigma} u|^p \sigma(S_{r})
\nonumber
\\
&\leq&  C''_{p} r^p {\sigma(S_{r}) \over \sigma(S_{r} \sm E)}
 \int_{S_{r}} |\nabla_{t} u|^p d\sigma
\end{eqnarray}
by \eqref{e4.5} and \eqref{e4.8}; \eqref{e4.6} follows.

Now we prove \eqref{e4.7}. The proof will also work for Lemma \ref{l3.2}.
We shall use a covering. Let $x\in E$ be given, and consider the density ratio 
$d(x,t) = \sigma(E \cap B(x,t))/\sigma(S_{r}\cap B(x,t))$.
Notice that $d(x,2r) = \sigma(E) /\sigma(S_{r}) < 1/2$
if $\sigma(E) < \sigma(S_{r})/2$. If $x$ is a Lebesgue density
point of $E$ (on the sphere), $\lim_{t \to 0} d(x,t) = 1$, 
and by continuity we can find $t = t(x) \in (0,2r)$ such that $d(x,t) = 1/2$. 
Then use the Besicovitch covering lemma (see for instance \cite{F}) to find a set $X$ 
such that the $B(x,r(x))$, $x\in X$, cover the set of Lebesgue points of $E$
(and hence, $\sigma$-almost all $E$), but $\sum_{x\in X} \1_{B(x,r(x))} \leq C$. 

Fix $x\in X$ and set $D = B(x,t(x)) \cap S_r$
The point of choosing $t(x)$ as we did is that, as in \eqref{e4.8},
\begin{eqnarray}\label{e4.10}
|m_{D}^{\sigma} u|^p 
&=& \sigma(D \sm E)^{-1} \int_{D \sm E} |u-m_{D}^{\sigma} u|^p 
\nonumber
\\
&=& 2 \sigma(D)^{-1} \int_{D \sm E} |u-m_{D}^{\sigma} u|^p 
\leq 2 \fint_{D} |u-m_{D}^{\sigma} u|^p 
\leq 2C_{p} t(x)^p  \fint_{D} |\nabla_{t} u|^p 
\end{eqnarray}
because $u(x) = 0$ almost everywhere on $S_{r} \sm E$,
and by \eqref{e4.5}. Then, as in \eqref{e4.9},
\begin{equation}\label{e4.11}
\int_{D} |u|^p \leq 
C_{p } \int_{D} |u-m_{D}^{\sigma} u|^p +  C_{p} |m_{D}^{\sigma} u|^p \sigma(D)
\leq C'_{p} t(x)^p \int_{D} |\nabla_{t} u|^p
\end{equation}
by \eqref{e4.5} and \eqref{e4.10}. Observe that 
$t(x)^{n-1} \leq C \sigma(D) = C\sigma(S_{r} \cap B(x,t))
= 2C \sigma(E \cap B(x,t)) \leq 2C\sigma(E)$
by various definitions. We now sum over $x\in X$ and get that
\begin{eqnarray}\label{e4.12}
\int_{E} |u|^p &\leq& \sum_{x\in X} \int_{B(x,t(x))\cap S_r} |u|^p
\leq C_{p} \sum_{x\in X} t(x)^p\int_{B(x,t(x))\cap S_r} |\nabla_{t}u|^p
\nonumber
\\
&\leq& C_{p} \sigma(E)^{p \over n-1} \sum_{x\in X} 
\int_{S_{r}} \1_{B(x,t(x))} |\nabla_{t}u|^p
\leq C_{p} \sigma(E)^{p \over n-1} \int_{S_{r}} |\nabla_{t}u|^p
\end{eqnarray}
because the $B(x,t(x))$, $x\in X$, have bounded covering.
This completes our proof of \eqref{e4.7}; Lemma \ref{l4.1} follows.
\qed
\end{proof}

\ms
Next we talk about the restriction of $u\in W^{1,p}$ to almost
every sphere $S_{r}$. Let $u \in W^{1,p}_{loc}(\R^n)$ be given, 
with $1 \leq p < +\infty$. It is possible to define the trace of $u$ 
on every sphere $S_{r}$; this looks fine, but the trace is a less regular object 
that will be difficult to control.
What we shall do instead is use arguments based on Fubini, and control
the restriction of $u$ on almost every sphere, with often a better control.

The following description is not hard to get, for instance by changing
variables (so that spheres become hyperplanes) and using Fubini;
see for instance Chapter 10-13 in \cite{D}. 
First, for almost every half line $L$ through the origin, $u$ is 
equal almost everywhere on $L$ to an absolutely continuous,
i.e., the integral of some locally integrable function on $L$.

Let us modify $u$ on a set of measure zero, so that it becomes
absolutely continuous (and hence, continuous) on almost every half line 
through the origin. This way, for each $r > 0$, we have a radial limit description
of the restriction, as
\begin{equation}\label{e4.13}
u(x) = \lim_{t \to 1} u(tx)
\ \text{ for $\sigma$-almost every } x\in S_{r}.
\end{equation}
Next, for almost every $r > 0$,
\begin{equation}\label{e4.14}
\text{the restriction of $u$ to $S_{r}$ lies in $W^{1,p}(S_{r})$,} 
\end{equation}
and with tangential partial derivatives that coincide almost everywhere on $S_{r}$
with (or we should rather say, can be naturally computed in terms of) 
the restriction of the derivative $Du$ to $S_{r}$.
This will be pleasant, for instance because we immediately get, 
by \eqref{e4.3} and Fubini, that
\begin{equation}\label{e4.15}
 \int_{\rho=0}^r \int_{S_{r}} |\nabla_{t} u|^p d\sigma d\rho
 = \int_{B(0,r)} |\pi_{t}(\nabla u)|^p \leq \int_{B(0,r)} |\nabla u|^p,
\end{equation}
where for the sake of the argument we introduced the tangential part
$\pi_{t}(\nabla u)$ of $\nabla u$.

\ms
It will also be useful to glue two Sobolev functions along a sphere.
Suppose we have two functions $u_{1} \in W^{1,p}(B(0,r))$ and
$u_{2} \in W^{1,p}(B(0,2r)\sm \overline B(0,r))$. Again modulo changing
$u_{1}$ and $u_{2}$ on sets of measure zero, and by the same proof
as for \eqref{e4.13}, we have the existence of boundary values
\begin{equation}\label{e4.16}
\overline u_{1}(x) = \lim_{t \to 1^-} u_{1}(tx)
\ \text{ and } \overline u_{2}(x) = \lim_{t \to 1^+} u_{2}(tx).
\end{equation}
for $\sigma$-almost every $x\in S_{r}$. Now suppose that
\begin{equation}\label{e4.17}
\overline u_{1}(x) = \overline u_{2}(x)
\ \text{ for  $\sigma$-almost every } x\in S_{r}. 
\end{equation}
Then set $u(x) = u_{1}(x)$ for $x\in B(0,r)$ and 
$u(x) = u_{2}(x)$ for $x\in B(0,2r)\sm B(0,r)$. It is not hard to check that
\begin{equation}\label{e4.18}
u\in W^{1,p}(B(0,2r)),
\end{equation}
and of course its derivative can be computed locally, so it coincides
with the derivative of $u_{i}$ on the corresponding domain.
That is, because of the absence of jump, the distribution derivative of $u$
does not have an extra piece on $S_{r}$.
The proof is not hard: for the existence of the radial derivative, for instance,
we integrate against a test function, use Fubini to integrate on rays, 
write $u$ as the integral of its radial derivative, and on each good ray 
do a soft integration by parts (using Fubini). 
Again, see \cite{D} for details. 

\ms
Let us record a last estimate where we mix the values in $B(0,r)$
and on $S_{r}$.

\begin{lem}\label{l4.2} 
Suppose $u\in W^{1,1}(B(0,r))$, and let $\overline u$ denote the
boundary values of $u$ on $S_{r}$, defined $\sigma$-almost everywhere on 
$S_{r}$ as in \eqref{e4.16}. Then $\overline u \in L^1(S_{r})$, and
\begin{equation}\label{e4.19}
|m^{\sigma}_{S_{r}} \overline u - m_{B(0,r)} u|
\leq C r \fint_{B(0,r)} |\nabla u|.
\end{equation}
\end{lem}

\begin{proof}
Change $u$ on a set of measure $0$ so that $u$ is absolutely continuous 
along almost all rays. Then, for almost every $y\in S_{r}$, we get that 
for every $\rho\in (1/2,1)$
\begin{equation}\label{e4.20}
|\overline u(y)-u(\rho y)| \leq r\int_{t=\rho}^1 |\nabla u|(ty) dt
\leq r\int_{1/2 < t < 1} |\nabla u(ty)| dt
\end{equation}
and hence, setting $m = m_{B(0,r)} u$,
\begin{equation}\label{e4.21}
|\overline u(y)-m| 
\leq |u(\rho y) - m| + r\int_{1/2 < t < 1} |\nabla u(ty)| dt.
\end{equation}
We average this over $\rho$ and get that
\begin{equation}\label{e4.22}
|\overline u(y)-m| 
\leq 2\int_{1/2 < t < 1}|u(t y) - m| dt + r\int_{1/2 < t < 1} |\nabla u(ty)| dt.
\end{equation}
Then we average on $y\in S_{r}$ and obtain
\begin{eqnarray}\label{e4.23}
\fint_{S_{r}} |\overline u(y)-m| d\sigma(y)
&=& 
\fint_{y\in S_{r}} \int_{1/2 < t < 1} 2|u(t y) - m|+r|\nabla u(ty)| 
\nonumber
\\
&\leq& C \fint_{z\in B(0,r)} |u(z)-m| + r |\nabla u(z)|,
\end{eqnarray}
where the constant $C$ comes from a Jacobian, but which we control
because we restricted to $t > 1/2$. We apply Poincar\'e's inequality \eqref{e4.2},
get that
\begin{equation}\label{e4.24}
\fint_{S_{r}} |\overline u(y)-m| d\sigma(y)
\leq C r \fint_{B(0,r)} |\nabla u|,
\end{equation}
and use the triangle inequality to conclude from there.
\qed
\end{proof}

\section{Minimizers are bounded}  \label{bounded} 

In this section we use a variant of the maximum principle to show that 
$\u$ is bounded when $(\u,\W) \in \F$ is a minimizer 
for the functional $J$ and the following assumptions hold:
\begin{equation}\label{e5.1}
|\Omega| < +\infty,
\end{equation}
and, for $1 \leq i \leq N$,
\begin{equation}\label{e5.2}
f_{i} \in L^p(\Omega) \ \text{ for some } p > n/2,
\end{equation}
\begin{equation}\label{e5.3}
f_{i}(x) \geq 0 \text{ almost everywhere,}
\end{equation}
and
\begin{equation}\label{e5.4}
g_{i}(x) \in L^p(\Omega) \text{ for some } p>n/2.
\end{equation}
When $n=1$, let us not allow exponents smaller than $1$.
That is, let us understand that \eqref{e5.2} and \eqref{e5.4}
mean that the $f_i$ and the $g_i$ lie in $L^1(\Omega)$ when
$n=1$.

In this section we shall not be able to escape from using Sobolev
exponents entirely, so we decided to choose our exponents
a little in the spirit of Remark \ref{r3.4}. The reader may assume
that the $f_i$ are bounded and (in dimensions $n \leq 3$)
the $g_i$ lie in $L^2$; this will simplify some estimates 
(as in Section \ref{existence}) but unfortunately not all.

We keep the same conditions on the $f_i$ as in 
Remark \ref{r3.4}, and this way we can compute the terms
$\int u_i^2 f_i$ because $u_i \in W^{1,2}(\R^n)$ and
$u_i = 0$ a.e. on $\R^n \sm \Omega$; see Remark \ref{r3.4}.
Our condition on the $g_i$ is stronger now because 
${n \over 2} \geq {2n \over n+2}$ when $n \geq 2$.
Thus we can also compute the $\int u_i g_i$. 

The new constraint that now $g_i \in L^p$ for some $p> n/2$
is about right: we expect Theorem~\ref{t5.1} to fail 
when $g_{i} \notin L^{n/2}$; 
see Remark \ref{r5.2} at the end of this section. 

Here we shall only assume that $F$ is bounded; then we may not have
an existence theorem, but this does not matter. We only need to assume
that $F$ is bounded because we want a bound on $E(\u)$ (see below);
otherwise we don't need information on $F$ because we shall not modify 
the $W_{i}$ in the proof.

\begin{thm}\label{t5.1} 
Assume that \eqref {e5.1}-\eqref {e5.4} hold, 
that the volume functional $F$ is bounded,
and that $(\u,\W)$ is a minimizer for $J$ in $\F$
(see Definition~\ref{d1.1} and \eqref {e1.3}-\eqref {e1.5}).
Then $\u$ is bounded, and we have bounds on the $||u_{i}||_{\infty}$
that depend only on $n$, $N$, $p$, a bound for $F$,
and the constants that arise in \eqref{e5.1} and\eqref{e5.4}.
\end{thm}

\begin{proof}
First we should observe that we have bounds on the energy
$E(\u)$ of \eqref{e3.6}, which depend only on $n$, $N$, $p$, 
a bound for $F$, and the constants that arise in 
\eqref{e5.1} and \eqref{e5.4}. This follows from 
\eqref{e3.8} (with $f_{i,-}=0$ and where we never use 
estimates on the size of $f_i$), and a trivial bound
for $J(\u,\W)$ that we obtain by testing the function $\u=0$.
We included the assumption that $F$ is bounded precisely for this;
otherwise, our estimates would also depend on $E(\u)$.

We shall first do the proof in dimensions larger than $2$, because 
of complications with the sign of the fundamental solution of $-\Delta$ 
in dimension 2; the proof will have to be modified when $n=2$; we 
will take care of that near \eqref{e5.21}. 

So let us assume that $n \geq 3$.
We intend to show that each $u_{i}$ is bounded by comparing 
it with a function $v$ that we shall construct by hand. 
Fix $i$, and choose $\rho \in L^p(\R^n)$, with $||\rho||_{p} \leq ||g_{i}||_{p}$
and such that $\rho(x) = 0$ for $x\in \R^n \sm \Omega$ but
$\rho(x) > |g_{i}(y)|/2$ almost everywhere on $\Omega$
(we lose a factor $2$ but we win a strict inequality). Then set
\begin{equation}\label{e5.5}
v(x) = (\rho \ast G)(x) = \int_{\R^n} G(x-y) \rho(y) dy,
\end{equation}
where $G$ is the fundamental solution of $-\Delta$. Here $n \geq 3$ so
$G(z) = c_{n} |z|^{2-n}$ for some positive constant $c_{n}$.
In dimension $2$, we would get a logarithm, and we would not like 
that as much because it takes negative values.

Let us first check that the integral converges, and even that for all $x\in \R^n$,
\begin{equation}\label{e5.6}
0 \leq v(x) = C |\Omega|^{2p-n\over np} ||g_{i}||_{p}.
\end{equation}
Let $q$ denote the conjugate exponent of $p$ and $r > 0$ be such that
$|\Omega| = |B(0,r)|$. Observe that since $G$ is radial and decaying,
\begin{equation}\label{e5.7}
\int_{\Omega} G(x-y)^q dy \leq \int_{B(0,r)} G(z)^q dz 
= C \int_{B(0,r)} |z|^{-q(n-2)} dz.
\end{equation}
Recall that $p > n/2$, so $p^{-1} < 2/n$, $q^{-1} > (n-2)/n$,
$q < n/(n-2)$, $q(n-2) < n$, and the integral converges. We also get that
$\int_{\Omega} G(x-y)^q dy \leq C r^{n-q(n-2)}$. Then by H\"older
\begin{equation}\label{e5.8}
||v||_{\infty} \leq ||G||_{L^q(B(0,r))} ||\rho||_{p}
\leq C r^{{n \over q}-(n-2)} ||\rho||_{p} 
= C |\Omega|^{{1\over q}-{n-2 \over n}} ||\rho||_{p},
\end{equation}
which implies \eqref{e5.6} because $||\rho||_{p} \leq ||g_{i}||_{p}$.
We may as well assume that $p < n$, because the result is easier
otherwise. We shall need to know that 
\begin{equation}\label{e5.9}
\Delta v = - \rho  \ \text{ and } \ \nabla^2 v \in L^p(\R^n)
\end{equation}
and 
\begin{equation}\label{e5.10}
\nabla v \in L^{r}(\R^n), \text{ with } r = {n \over n-p}.
\end{equation}
We start with \eqref{e5.9}; when $\varphi$ is smooth and compactly supported, 
$\Delta(G \ast \varphi) = -\varphi$ (by our choice of $G$), 
and $\nabla^2 (G \ast \varphi)$ is obtained from $-\varphi$ by applying 
second order Riesz transforms. These are known to be bounded on $L^p$ 
because $p > 1$ and we assumed that $p< n < +\infty$
(for this fact and the proof of \eqref{e5.10}, see \cite{S}), 
so we get that $||\nabla^2 (G \ast \varphi)||_{p} \leq C ||\varphi||_{p}$.
Then \eqref{e5.9} follows by writing $\rho$ as a limit in $L^p$ of
test functions $\varphi$. 

So $v$ lies in the Sobolev space
$W^{2,p}(\R^n)$, hence $\nabla v \in W^{1,p}(\R^n) \i L^r$,
where $1/r = 1/p - 1/n$, and \eqref{e5.10} holds.
See for instance Theorem 1 on page 119 of \cite{S}.

\ms
Return to the proof of Theorem \ref{t5.1}, still when $n \geq 3$. 
We shall show that $u_{i} \leq v$ almost everywhere on $\Omega$
(hence on $\R^n$, since $u_i = 0$ on $\R^n \sm \Omega$)
and the proof will also show that $u_{i} \geq -v$
(either change the sign of $u_{i}$ and $g_{i}$ in the functional, 
or modify slightly the proof).
Theorem \ref{t5.1} will follow because $v$ is bounded by \eqref{e5.6}.
Set 
\begin{equation}\label{e5.11}
w = \min(u_{i},v), h = u_{i}-w \geq 0, 
\end{equation}
and, for $0 < t < 1$,
\begin{equation}\label{e5.12}
u_{t,i} = u_{i} - t h = (1-t) u_{i} + t w.
\end{equation}
We want to see whether replacing $u_{i}$ with $u_{t,i}$
yields a better competitor. First observe that for almost every 
$x\in \R^n\sm W_{i}$, $u_{i}(x)=0$ and hence $w(x)=h(x)=u_{t,i}(x)=0$.
Next, $v \in W^{1,2}_{loc}(\R^n)$, by \eqref{e5.10} and because
${n \over n-p} \geq 2$ when $n/2 \leq p < n$.
Since both $u_{i}$ and $v$ lie in $W^{1,2}_{loc}(\R^n)$, so do 
$w$, $h$, and $u_{t,i}$. Since their derivative vanishes a.e. on 
$\R^n \sm W_{i}$, we get that
\begin{equation}\label{e5.13}
\text{$w$, $h$, and $u_{t,i}$ lie in $W^{1,2}(\R^n)$.}
\end{equation}
Notice also that if $u_{i} \geq 0$, then $w \geq 0$ too, and
(by \eqref{e5.12} and because $0 < t <1$), $u_{t,i} \geq 0$. 
Because of all this, if we replace $u_{i}$ with $u_{t,i}$,
leave the other $u_{j}$ as they are, and also keep the same sets $W_{j}$,
we get a new pair $(\u_{t},\W)$ that still lies in $\F$.
By minimality, $J(\u_{t},\W) \geq J(\u,\W)$.
We compute
\begin{eqnarray}\label{e5.14}
J(\u_{t},\W)-J(\u,\W) &=& E(\u_{t})-E(\u) + M(\u_{t})-M(\u)
\nonumber \\
&= &\int_{\Omega} |\nabla u_{t,i}|^2 - |\nabla u_{i}|^2 
+ [u_{t,i}^2 - u_{i}^2] f_{i} - [u_{t,i}-u_{i}] g_{i} \, ;
\end{eqnarray}
then we use \eqref{e5.12} and compute the derivative 
\begin{eqnarray}\label{e5.15}
{d \over dt}J(\u_{t},\W)_{|t=0} &=& 
-2 \int_{\Omega} \langle \nabla u_{i},\nabla h \rangle
-2 \int_{\Omega} h u_{i} f_{i}
+ \int_{\Omega} hg_{i}
\nonumber \\
&\leq & -2 \int_{\Omega} \langle \nabla u_{i},\nabla h \rangle
+ \int_{\Omega} hg_{i}
\end{eqnarray}
because $h$ and $f_{i}$ are nonnegative, and $u_{i}(x) > v(x) \geq 0$
when $h(x) \neq 0$. Notice that our assumption \eqref{e5.2} that $f_i$
lie in $L^p$ for some $p > n/2$ was useful to define
$\int [u_{t,i}^2 - u_{i}^2] f_{i}$ (see Remark \ref{r3.4});
now we used the main assumption that $f_{i} \geq 0$
and we can forget about \eqref{e5.2} (and in fact $f_i$ altogether).

The derivative is nonnegative (because 
$J(\u_{t},\W) \geq J(\u,\W)$), hence
\begin{equation}\label{e5.16}
2 \int_{\Omega} \langle \nabla u_{i},\nabla h \rangle
\leq \int_{\Omega} hg_{i}.
\end{equation}
We may now use the fact that $\Delta v = - \rho$ (by \eqref{e5.9}),
which means that
\begin{equation}\label{e5.17}
\int_{\R^n} \langle \nabla v,\nabla \varphi \rangle = 
- \int_{\R^n} \Delta v \varphi  = \int_{\R^n} \rho \varphi
\end{equation}
for every test function $\varphi$. Let us check that
this remains true with $\varphi = h$. Denote by $r'$
the conjugate exponent of $r = {n \over n-p}$ in \eqref{e5.10};
thus $r' = {n \over p} < 2$.
Recall from \eqref{e5.13} that $h \in W^{1,2}(\R^n)$;
thus $h \in W^{1,r'}(\R^n)$ (recall that $h$ and $\nabla h$ are supported
on $W_i \i \Omega$), and we can write $h$ as a limit in $W^{1,2} \cap W^{1,r'}$ 
of test functions $\varphi_k$.
Then the left-hand side of \eqref{e5.17} converges to 
$\int_{\R^n} \langle \nabla v,\nabla h \rangle$, because
$\nabla v \in L^r(\Omega)$ by \eqref{e5.10}.

For the left-hand side, we know that $\rho \in L^p$, so it is enough
to show that the $\varphi_k$ converge to $h$ in $L^{p'}$.
But $p' < {n \over n-2}$ because $p > {n \over 2}$, so the desired
convergence follows from the convergence of the $\varphi_k$ in 
$W^{1,2}$, because the Sobolev exponent for $W^{1,2}$
is ${2n \over n-2} > {n \over n-2}$, and by the proof of 
Lemma \ref{l3.2}. So \eqref{e5.17} holds with $\varphi = h$, i.e.,
\begin{equation}\label{e5.18}
\int_{\R^n} \langle \nabla v,\nabla h \rangle =  \int_{\R^n} \rho h.
\end{equation}

Set $Z = \big\{ x\in \R^n \, ; \, h(x) > 0 \big\}$. By \eqref{e5.11},
$w(x) \neq u_{i}(x)$ and hence $w(x) = v(x)$ on $Z$.
Since $\nabla h = 0$ almost everywhere on $\R^n \sm Z$, we 
get that 
\begin{eqnarray}\label{e5.19}
\int_{\R^n} \rho h &=& \int_{\R^n} \langle \nabla v,\nabla h\rangle
= \int_{Z} \langle \nabla v,\nabla h\rangle = \int_{Z} \langle \nabla w,\nabla h\rangle
\nonumber \\
&=& \int_{\R^n} \langle \nabla w,\nabla h\rangle
= \int_{\R^n} \langle \nabla (u_{i}-h),\nabla h\rangle
\leq \int_{\R^n} \langle \nabla u_{i},\nabla h\rangle
\end{eqnarray}
(by \eqref{e5.11} again, and because $\int |\nabla h|^2 \geq 0$).
By \eqref{e5.19} and \eqref{e5.16}
\begin{equation}\label{e5.20}
\int_{\R^n} 2\rho h \leq 2\int_{\R^n} \langle \nabla u_{i},\nabla h\rangle
= 2\int_{\Omega} \langle \nabla u_{i},\nabla h\rangle
\leq \int_{\Omega} hg_{i}.
\end{equation}
Recall that we chose $\rho$ such that $\rho = O$ on $\R^n\sm \Omega$
and $\rho > |g_{i}|/2$ everywhere on $\Omega$; 
since $h \geq 0$ by \eqref{e5.11}, we deduce from \eqref{e5.20} that $h(x) = 0$ 
almost everywhere on $\Omega$, which precisely means that 
$u_{i} \leq v$ almost everywhere on $\Omega$. 
A minor modification of the argument would show that
$-u_{i} \leq v$ too, and we have seen that this proves 
Theorem \ref{t5.1} when $n\geq 3$.

\ms
We now turn to $n=2$. A fundamental solution of $-\Delta$ is now
$\log(1/|x|)$, but since it becomes negative for $|x|$ large,
we shall need to localize the argument above. Let $i$ be given, and also
choose $r$ so that $|B(0,r)| = 2 |\Omega|$. We will find a constant $C_{0}$
(that depends only on $p > 1$, $|\Omega|$ and $||g_{i}||_{p}$) such that
$u_{i} \leq C_{0}$ almost everywhere on $B(0,r/2)$. The same proof would
also yield $-u_{i} \leq C_{0}$, and since we can choose the origin arbitrarily, 
this will give the desired $L^\infty$ bound. We still choose $\rho \in L^p$
such that $||\rho||_{p} \leq ||g_{i}||_{p}$, $\rho = 0$ on $\R^n \sm \Omega$,
and $2\rho > |g_{i}|$ on $\Omega$, and set 
\begin{equation}\label{e5.21}
v(x) = \int_{B(0,2r)} G(x-y) \rho(y) dy,
\text{ with } G(x) = c \log(10r/|x|),
\end{equation}
where $c$ is still chosen so that $\Delta G = -\delta_{0}$ (a Dirac mass);
we added a constant to the logarithm to make sure that 
$v \geq 0$ in $B(0,2r)$. Denote by $q$ the conjugate exponent of $p$
and observe that for $x\in B(0,2r)$,
\begin{equation}\label{e5.22}
\int_{B(0,2r)} |G(x-y)|^q dy 
\leq \int_{z\in B(0,4r)} |G(z)|^q dz
\leq \int_{B(0,4r)} \log^q(10r/|x|) \leq C,
\end{equation}
where $C$ depends on $p$ and $r$, so by H\"older
\begin{equation}\label{e5.23}
|v(x)| \leq C ||\rho||_{p} \leq C ||g_i||_{p}
\ \text{ for } x\in B(0,2r).
\end{equation}

Now we want to choose a radius $s \in (r,2r)$ where we will do some 
surgery. We modify $u_{i}$ on a set of measure zero so that
\eqref{e4.13} (the continuity of $u_{i}$ along almost all rays) holds.
We want the restriction of $u_i$ to $S_{s}$ to be in $W^{1,2}$;
this is true for almost every $s$, because $u_i\in W^{1,2}(\R^n)$
(compare with \eqref{e4.14}). We also require that
\begin{equation}\label{e5.24}
\int_{S_{s}} |\nabla_{t} u_i|^2 d\sigma \leq 10 r^{-1} \int_{\R^2} |\nabla u_i|^2,
\end{equation}
which by \eqref{e4.15} and Chebyshev is true except for a set of measure
at most $r/10$ of radii $s$. In addition, we chose $s$ such that
\begin{equation}\label{e5.25}
u_i(y) = 0 \text{ for $\sigma$-almost every } y\in S_{s}\sm \Omega
\end{equation}
(also true for almost every $s$, because $u_i = 0$ almost everywhere
on $\R^n \sm W_i$) and
\begin{equation}\label{e5.26}
\sigma(S_{s}\sm \Omega) \geq 10^{-1} \sigma(S_{s})
\end{equation}
which by Fubini holds for most $s \in (r,2r)$ because
\begin{eqnarray}\label{e5.27}
\int_{s=r}^{2r} \sigma(S_{s}\sm \Omega) ds
&=& \Big|[B(0,2r)\sm B(0,r)]\sm \Omega \Big|
\nonumber
\\
&\geq& |B(0,2r)\sm B(0,r)| - |\Omega|
\geq {1 \over 2} |B(0,2r)\sm B(0,r)|
\end{eqnarray}
by \eqref{e4.3} and because $|B(0,r)| = 2 |\Omega|$.
So we choose $s$ with all these properties.

Here things will be simpler because $n=2$; we could make the argument work
in higher dimensions (instead of adding a constant below, add the harmonic extension
of the restriction of $|u_i|$ to $S_{s}$), but let us not do that.
By \eqref{e5.24}, we have that for almost all choices of $x, y \in S_{s}$,
\begin{equation}\label{e5.28}
|u_i(x)-u_i(y)| \leq \int_{S_{s}} |\nabla_{t} u_i| d\sigma
\leq C s^{1/2} \Big\{\int_{S_{s}} |\nabla_{t} u_i|^2 d\sigma\Big\}^{1/2}
\leq C ||\nabla u_i||_{2}.
\end{equation}
We apply this with some $y\in S_{s}\sm \Omega$ and get that
\begin{equation}\label{e5.29}
|u_i(x)| < C_{1}
\text{ for $\sigma$-almost every } x\in S_{s},
\end{equation}
where $C_{1}$ depends on $||\nabla u_i||_{2}$. 

We may now define a competitor for $(\u,\W)$, as we did near
\eqref{e5.11}. First set 
\begin{equation}\label{e5.30}
\begin{array}{lll}
w(x) =& u_{i}(x) &\text{ for } x\in \R^2 \sm B(0,s)
\\
w(x) =& \min(u_{i}(x),v(x)+ C_{1}) &\text{ for } x\in B(0,s),
\end{array}
\end{equation}
where $v$ is as in \eqref{e5.21} and $C_{1}$ as in \eqref{e5.29}.
We want to show that 
\begin{equation}\label{e5.31}
w \in W^{1,2}(\R^2),
\end{equation}
and we shall first check that
\begin{equation}\label{e5.32}
v \in W^{1,2}(B(0,2r)) \ \text{ with  } \ 
\nabla v = \nabla G \ast (\1_{B(0,2r)}\rho).
\end{equation}
Recall from \eqref{e5.21} that $v = G \ast \psi$,
with $\psi = \1_{B(0,2r)}\rho \in L^p$.

Let $r$ be such that $1/r = 1/p - 1/2 < 1/2$;
since $|\nabla G|$ is a Riesz potential of order $1$,
the Hardy-Littlewood-Sobolev theorem of fractional integration 
on page 119 of \cite{S} says that for any $\psi \in L^p(\R^n)$,
$|\nabla G| \ast |\psi| \in L^r$, with
\begin{equation}\label{e5.33}
\big|\big| |\nabla G| \ast |\psi| \big|\big|_{r} \leq C ||\psi||_{p}.
\end{equation}
Returning to $\psi = \1_{B(0,2r)}\rho$, we see that
$\nabla G \ast \psi \in L^r$ and, since $r > 2$, 
$\nabla G \ast \psi \in L^2(B(0,2r))$ as well.

We still need to check that $\nabla G \ast \psi$ is the distribution
derivative of $v$ in $B(0,2r)$.
If $\psi$ were a test function, or just bounded with compact support,
we could brutally differentiate $v = G \ast \psi$ under the integral,
using the fact that $|\nabla G|$ is integrable near the origin,
and get that $\nabla v =\nabla G \ast \psi$ (a continuous derivative).
Unfortunately, $\psi$ is not bounded, but we can write it as the limit
in $L^p$ of a sequence of test functions $\psi_k$. Then set
$v_k = G \ast \psi_k$; we just observed that 
$\nabla v_k = \nabla G \ast \psi_k$,
and \eqref{e5.33} shows that $\nabla v_k$ tends to 
$\nabla G \ast \psi$ in $L^r$, hence in $L^2(B(0,2r))$.
But $v_k$ tends to $v$ in $L^2(B(0,2r))$, more brutally
because $|G|$ is integrable near the origin, and this implies that
$v \in W^{1,2}(B(0,2r))$, with $\nabla v = \nabla G \ast \psi$
(pair $v$ against a the gradient of test function).

So \eqref{e5.32} holds, and by \eqref{e5.30} this implies that
$w\in W^{1,2}(B(0,r))$, because the minimum of two functions of 
$W^{1,2}$ lies in $W^{1,2}$. 

We also have that  $w \in W^{1,2}(\R^2 \sm \overline B(0,r))$,
again by \eqref{e5.30} and  because $u_{i} \in W^{1,2}(\R^2)$.

We want to glue the two pieces, so we consider the radial 
limits of $w$ on $S_{s}$. From $\R^2 \sm B(0,s)$, 
this limit is equal to the restriction of $u_{i}$ to $S_{s}$
(because $u_{i}$ is continuous along almost all rays). 
From $B(0,s)$, let us modify $v$ so that it is also continuous along almost all
rays; then $v + C_{1}$ has limits on $S_{r}$ that are larger than
the values of $u_{i}$, by \eqref{e5.29} and because $v \geq 0$
on $B(0,2r)$. Then the radial limit of $w$ from $B(0,s)$ is the same
as for $u_{i}$, the gluing condition \eqref{e4.17} holds, and \eqref{e4.18}
says that $w\in W^{1,2}(\R^2)$, as needed for \eqref{e5.31}.

\ms
Obviously $w \leq u_{i}$ (just by \eqref{e5.30}), so $h = u_{i} - w$
is nonnegative, and null outside of $B_{s}$. 
Now we follow the same argument as when $n \geq 3$,
starting below \eqref{e5.12}. We still have \eqref{e5.16}
for the same reasons, but we need to compute 
its left-hand side $2\int_{\R^n} \langle \nabla u_{i},\nabla h \rangle$
differently. We want to replace \eqref{e5.17} with the fact that
\begin{equation}\label{e5.34}
\int_{\R^n} \langle \nabla v,\nabla \varphi \rangle 
= \int_{\R^n} \rho \1_{B(0,2r)} \varphi
\end{equation}
for every test function $\varphi$,
which we expect to hold because $v = G \ast \rho \1_{B(0,2r)}$
should yield $\Delta v = - \rho \1_{B(0,2r)}$ in the sense of distributions.
And indeed \eqref{e5.32} yields
\begin{eqnarray}\label{e5.35}
\int_{\R^n} \langle \nabla v,\nabla \varphi \rangle 
&=& \int_{\R^n} \langle \nabla G \ast \rho \1_{B(0,2r)},\nabla \varphi \rangle 
\nonumber\\
&=&\int_{x\in \R^n} \int_{y\in B(0,2r)}
\rho (y) \langle \nabla G(x-y) ,\nabla \varphi(x) \rangle dxdy,
\end{eqnarray}
where the double integral converges absolutely by \eqref{e5.33}.
We apply Fubini, use the fact that
\begin{equation}\label{e5.36}
\int \langle \nabla G(x-y) ,\nabla \varphi(x) \rangle dx
= \int \langle \nabla G(x) ,\nabla \varphi(x+y) \rangle
= - \langle \Delta G, \varphi(\cdot+y)\rangle
= \varphi(y)
\end{equation}
by definition of $G$ (and because $\varphi$ is a test function),
and get \eqref{e5.34}. Then we wish to replace $\varphi$ with 
$h$ in \eqref{e5.34}, i.e., get that
\begin{equation}\label{e5.37}
\int_{\R^n} \langle \nabla v,\nabla h \rangle =  \int_{B(0,2r)} \rho h.
\end{equation}
as in \eqref{e5.18}. Recall that $h = u_i - w$, so $h\in W^{1,2}(\R^n)$.
In addition, $h= 0$ on $\R^n \sm B^s$, so the Sobolev embedding
theorem says that in lie in $L^q$ for every $q < +\infty$. We choose 
for $q$ the dual exponent of $p$ (recall that $p > n/2 = 1$).
Then we write $h$ as a limit of test functions $\varphi_k$, so that
the $\nabla \varphi_k$ converge to $\nabla h$ in $L^2$
and the $\varphi_k$ converge to $h$ in $L^q$. The identity 
\eqref{e5.34} for $\varphi_k$ goes to the limit, because $\nabla v \in L^2$
by \eqref{e5.32} and $\rho \in L^p$ by definition. So \eqref{e5.37}
holds.

Set $Z = \big\{ x\in \R^n \, ; \, h(x) > 0 \big\}$ as before, and notice that
$Z \i B_s$ (modulo a negligible set). This time the definition of $h$ yields 
$w(x) = v(x) + C_{1}$ on $Z$, and
\begin{eqnarray}\label{e5.38}
\int_{\R^n} \rho h &=& \int_{\R^n} \langle \nabla v,\nabla h\rangle
= \int_{Z} \langle \nabla v,\nabla h\rangle 
= \int_{Z} \langle \nabla (w-C_{1}),\nabla h\rangle
\nonumber \\
&=& \int_{\R^n} \langle \nabla w,\nabla h\rangle
= \int_{\R^n} \langle \nabla (u_{i}-h),\nabla h\rangle
\leq \int_{\R^n} \langle \nabla u_{i},\nabla h\rangle
\end{eqnarray}
by \eqref{e5.37}, as in \eqref{e5.19}, and because
$\nabla C_{1} = 0$. So the conclusion
of \eqref{e5.19} still holds, and we may end the argument as before.
This completes our proof of Theorem \ref{t5.1} when $n=2$.

We are left with the easier case when $n=1$.
Recall that we have a bound on $E(\u)$ (see the very beginning
of the proof, and recall that we could take $p=r=1$ in Remark \ref{r3.4});
then $u_i$ is H\"older continuous, with bounds that depend only $E(\u)$,
and since $u_i = 0$ almost everywhere on $\R \sm \Omega$, we get
the desired $L^\infty$ bounds on the $u_i$ because $|\Omega| < +\infty$.
\qed
\end{proof}

\begin{rem}\label{r5.2} 
The constraint that $p > n/2$ in \eqref{e5.4} is not so far from optimal,
in the sense that the exponent $n/2$ cannot be made smaller.
\end{rem}

Let us just consider the case when $N=1$, $f_{1}=0$, $F = 0$, and $\Omega = B(0,1)$.
It is not too hard to check that the minimizer for $J$ is the pair $(u,\Omega)$,
where $u$ is the solution of $\Delta u = -g_{1}/2$, with the usual Dirichlet constraint
$u = 0$ on $\R^n \sm \Omega$ (see \eqref{e9.6} below for the equation). 
Locally $u$ behaves like $G \ast g_{1}$, which in general is not bounded (even locally) 
when $g_{1} \notin L^{n/2}$; for instance, take $g_{1}(x) = |x|^{-2}$ near the origin
and observe that $G \ast g_{1}(x)$ tends to $G \ast g_{1}(0)= +\infty$ 
when $x$ tends to $0$.

\section{Two favorite competitors}  \label{favorites} 

We shall soon start for good our study of the local regularity of minimizers
for the functional $J$, and in this section we present constructions of competitors
that we shall often use to obtain information on such minimizers. In particular,
the second one (harmonic competitors) will sometimes be a good replacement
for the main competitor that people use when $N \leq 2$.

We are given a pair $(\u,\W) \in \F$ (see Definition \ref{d1.1}), and a ball
$B$, and we want to define other pairs $(\u^\ast,\W^\ast) \in \F$ by modifying
$\u$ and $\W$ in $B$ (with $u=u^*$ and $W=W^*$ in $\Omega\sm B$), and then compare them with $(\u,\W) \in \F$ to get 
valuable information if $(\u,\W)$ is a minimizer.  

\ms
The first pair will be called the \underbar{cut-off competitor}. We may as well suppose
that $B = B(0,r)$ (by translation invariance of our problems), we give ourselves
a number $a \in (0,1)$ (often close to $1$), and we choose a smooth cut-off 
function $\varphi$ such that 
\begin{equation} \label{e6.1}
\begin{array}{cl}
\varphi(t) = 0 & \text{ for }  0 \leq t \leq a r
\\
0 \leq \varphi(t) \leq  1  & \text{ for }  a r \leq t \leq r
\\
\varphi(t) = 1 & \text{ for }  t \geq r 
\\ 
0 \leq \varphi'(t) \leq 2 (1-a)^{-1} r^{-1}  &\text{ everywhere.}
\end{array}
\end{equation}
Then we pick any collection $I$ of indices $1 \leq i \leq N$, and set
\begin{equation} \label{e6.2}
\begin{array}{rll}
u_{i}^\ast(x) &= \varphi(|x|) u_{i}(x) & \text{ when  }  i\in I
\\
u_{i}^\ast(x) &=  u(x) & \text{ when  }  i\notin I.
\end{array}
\end{equation}
Here the simplest is to keep $\W^\ast = \W$, because $u_{i}^\ast(x) = 0$ as soon as
$u_{i}(x) = 0$, but the fact that $u_{i}^\ast = 0$ in $B(0,ar)$ allows us to modify
some of the $W_{i} \cap B(0,ar)$ in some arguments. Anyway, it is easy to see that
$(\u^\ast,\W) \in \F$; the only thing left to check is that $u_{i} \in W^{1,2}(\R^n)$
for $i\in I$. But the definition of a distribution derivative yields
\begin{equation}\label{e6.3}
\nabla u_{i}^\ast(x) = \varphi(|x|) \nabla u_{i}(x) + u_{i}(x) \varphi'(|x|) {x \over |x|}
\end{equation}
which lies in $L^2$ because $\varphi'$ is bounded and $u\in L^2(B)$ 
(recall that $u\in W^{1,2}(B)$ and use Poincar\'e's inequality \eqref{e4.2}). 
We shall often choose a small number $\tau > 0$, and apply the fact that
$(A+B)^2 = A^2 + B^2 + 2AB \leq (1+\tau) A^2 + (1+\tau^{-1}) B^2$
for $A, B \geq 0$ to get that for $x\in B$,
\begin{equation}\label{e6.4}
|\nabla u_{i}^\ast(x)|^2 \leq  (1+\tau) |\nabla u_{i}(x)|^2 
 + 4(1-a)^{-2} (1+\tau^{-1}) r^{-2} |u_{i}(x)|^2 .
\end{equation}
This will typically be useful when we know that, by some application of Poincar\'e's
inequality, $\int_{B} |u_{i}|^2$ is small. Of course $\nabla u_{i}^\ast(x) = 0$
when $x\in B(0,ar)$ (by (\ref{e6.1}),  (\ref{e6.2}), and because $i\in I$).
We now integrate and get that for $i\in I$,
\begin{equation}\label{e6.5}
\int_{B(0,r)} |\nabla u_{i}^\ast|^2 
\leq  (1+\tau) \int_{B(0,r) \sm B(0,ar)} |\nabla u_{i}|^2 
 + 4(1-a)^{-2} (1+\tau^{-1}) r^{-2} \int_{B(0,r) \sm B(0,ar)} |u_{i}|^2.
\end{equation}
Let us also record trivial estimates for the difference in the terms of $M(\u)$; 
for $i\in I$,
\begin{eqnarray}\label{e6.6}
\Big|\int_{\Omega} (u_{i}^\ast)^2 f_{i} - \int_{\Omega} u_{i}^2 f_{i}\Big|
&= &
\Big|\int_{B(0,r)} (1-\varphi(|x|)^2) u_{i}(x)^2 f_{i}(x) \Big|
\nonumber
\\
&\leq& \int_{B(0,r)} |u_{i}^2 f_{i}|
\leq C r^{n-{n \over p}} ||u_{i}||_{\infty}^2 ||f_{i}||_{p}
\end{eqnarray}
and 
\begin{eqnarray}\label{e6.7}
\Big|\int_{\Omega} u_{i}^\ast g_{i} - \int_{\Omega} u_{i}g_{i}\Big|
&= &\Big|\int_{B(0,r)} (1-\varphi(|x|)) u_{i}(x) g_{i}(x) \Big|
\nonumber
\\
&\leq&\int_{B(0,r)} |u_{i} g_{i}|
\leq C r^{n-{n \over p}} ||u_{i}||_{\infty} ||g_{i}||_{p}.
\end{eqnarray}

\ms
The next competitors that we want to introduce are obtained
by extending the values of $\u$ on $S_{r} = \d B(0,r)$. 
Recall from the discussion near \eqref{e4.13} that modulo modifying
$\u$ on a set of measure zero, we can assume that $\u$
is continuous along almost every ray, and (as in \eqref{e4.14})
that the restriction of $u$ to $S_{r}$ is itself in $W^{1,2}(S_{r})$
for almost every $r$, with partial derivatives that correspond
to the restriction of the derivative $Du$ to $S_{r}$.

It will be easier to define our competitors for these radii $r$,
because estimates for the tangential gradient $\nabla_{t} u$ 
on the sphere will often be useful to control the extension. 

The simplest description is when $N=1$ (and $u=u_{1}$ is
allowed to be real-valued). We assume that
\begin{equation}\label{e6.8}
\text{the restriction of $u$ to $S_{r}$ lies in $W^{1,2}(S_{r})$}
\end{equation}
(this holds for almost all $r$, by the discussion above), and also that
\begin{equation} \label{e6.9}
B(0,r) \i \Omega \ \text{ (modulo a set of vanishing Lebesgue measure).}
\end{equation}
Denote by $\overline u$ the restriction of $u$ to $S_{r}$, and by
$u^\ast$ the harmonic extension of $\overline u$ to $B(0,r)$, 
obtained by convolution of $\overline u$ with the Poisson kernel. 
Also set $u^\ast = u(x)$ for $x\in \R^n \sm B(0,r)$.
In this simple case the harmonic competitor (for the pair
$(u_{1},W_{1}) \in \F$, in the ball $B(0,r)$) is just the pair 
$(u^\ast,W^\ast)$, where $W^\ast = W_{1} \cup (\Omega \cap B(0,r))$.
Let us check that $(u^\ast,W^\ast) \in \F$.
It is well known that when $\overline u \in W^{1,2}(S_{r})$,
$u^\ast \in W^{1,2}(B(0,r))$, and even 
\begin{equation}
\label{e6.10}
\int_{B(0,r)} |\nabla u^\ast|^2  = \inf\Big\{ |\nabla v|^2 \, ; \, 
v \in W^{1,2}(B(0,r)) \hbox{ and } v = \overline u \text{ on } S_{r} \big\},
\end{equation}
where by $v = \overline u$ on $S_{r}$ we mean that the radial limits
of $v$ on $S_{r}$ (which exist as in \eqref{e4.13} and \eqref{e4.16})
coincide with $\overline u$ almost everywhere on $S_{r}$.
See for instance \cite{D}, Chapter 15 (and use the maximum principle
for the uniqueness of the harmonic extension $u^\ast$).
Also, $u^\ast$ itself satisfies $u^\ast = \overline u$ on $S_{r}$
and it is the unique minimizer in \eqref{e6.10}.

The gluing condition \eqref{e4.17} on $S_{r}$ holds, because we kept
$u^\ast = u$ on $\R^n \sm B(0,r)$, so \eqref{e4.18} implies that
$u^\ast \in W^{1,2}(\R^n)$. Since we added the constraint 
\eqref{e6.9}, almost every point of $B(0,r)$ lies in $W^\ast$,
and hence $u^\ast = 0$ almost everywhere on $\R^n \sm W^\ast$
(because $u = 0$ almost everywhere on $\R^n \sm W_{1}$).
If in addition we required that $u_{1} \geq 0$, we still get $u^\ast \geq 0$.
So $(u^\ast,W^\ast) \in \F$, and it is often a very good competitor to use,
because of \eqref{e6.10} and the fact that $u^\ast$ is as smooth as possible
in $B(0,r)$.

When $N=2$, but we only consider nonnegative functions $u_{i}$, there
is a nice trick that allows us to define the harmonic competitor.
We start from $\u = (u_{1},u_{2})$, and set $u = u_{1}-u_{2}$, which is
now a real-valued function in $W^{1,2}(\R^n)$. For $r$ as above
(i.e. satisfying \eqref{e6.8} and \eqref{e6.9}), we define $\overline u$
and $u^\ast$ as above. Then we set $\u^\ast = (u^\ast_{1},u^\ast_{2})$, 
where $u^\ast_{1}$ is the positive part of $u^\ast$ and $u^\ast_{2}$ 
is its negative part. The sets $Z_{i} = \big\{ x\in B(0,r) \, ; \, u^\ast_{i}>0\big\}$
are disjoint, so we may set $W_{i}^\ast = [W_{i}\sm B(0,r)]
\cup [Z_{i} \cap \Omega]$, and get disjoint subsets of $\Omega$.
This gives a pair $(\u^\ast,\W^\ast)$, which lies in $\F$ as before
(notice in particular that $u_{i}^\ast = 0$ almost everywhere on $\R^n \sm W_{i}^\ast$,
again by \eqref{e6.9}). This trick has been used extensively in the literature,
often implicitly by setting the problem directly in terms of $u = u_{1}-u_{2}$.

Unfortunately, this trick is not available when $N \geq 3$
(or when $N=2$ and we use real-valued functions).
If we just extend the restrictions to $S_{r}$ of the $u_{i}$, 
we get functions with overlapping supports,
and we will not be able to find sets $W_{i}^\ast$ for which
$(\u^\ast,\W^\ast) \in \F$. We shall be able to circumvent 
this problem at a price; we shall decide that the main contribution to our
functional comes from one of the $u_{i}$ (say, $u_{1}$), 
and we'll make the other ones vanish brutally on a slightly smaller ball.
Of course this will only be useful in special situations, where for some reason
there is a dominant component $u_{1}$.
Let us do this; the competitor that we will define now will be still referred to as the
\underbar{harmonic competitor} of $(\u,\W)$ in $B(0,r)$.
Suppose as before that $B(0,r) \i \Omega$ (as in (\ref{e6.9})),
and that (\ref{e6.8}) holds. Pick $a \in (0,1)$ (rather close to $1$), and define
$\varphi$ as in (\ref{e6.1}). For $i \geq 2$, set
\begin{equation} \label{e6.11}
\begin{array}{rll}
u_{i}^\ast(x) &= \varphi(|x|) {\overline u}_{i}(rx/|x|) & \text{ when  }  x\in B(0,r)
\\
u_{i}^\ast(x) &=  u_{i}(x) & \text{ when  }  x\in \R^n \sm B(0,r).
\end{array}
\end{equation}
This is not exactly the same formula as in the first part of (\ref{e6.2}),
because here we only use the values of $u_{i}$ on $S_{r}$
to do the extension, and this is why we prefer to have \eqref{e6.8}
(and often some bounds on the norm in $W^{1,2}(S_{r})$).
Notice that 
\begin{equation} \label{e6.12}
u_{i}^\ast(x)=0  \text{ when $i \geq 2$ and }  x\in B(0,ar),
\end{equation}
which will allow $u_{1}$ to be nonzero on the whole $B(0,ar)$.

For $u_{1}$ we shall use a harmonic extension. Denote by
$\overline u_{1}$ the restriction of $u_{1}$ to $S_{r}$, 
and then by $v_{1}$ the harmonic extension of $\overline u_{1}$ to $B(0,r)$, 
obtained by convolution of $\overline u_{1}$ with the Poisson kernel. 
By (\ref{e6.8}), even though we do not know whether $\overline u_{1}$
is continuous, we still have that $v_{1} \in W^{1,2}(B(0,r))$, and even that
\begin{equation}\label{e6.13}
\int_{B(0,r)} |\nabla v_{1}|^2  = \inf\Big\{ \int_{B(0,r)} |\nabla v|^2 \, ; \, 
v \in W^{1,2}(B(0,r)) \hbox{ and } v = \overline u_{1} \text{ on } S_{r} \big\},
\end{equation}
with the same definition of the boundary condition $v = \overline u_{1}$ 
as in \eqref{e6.10}. We set
\begin{equation} \label{e6.14}
\begin{array}{rll}
u_{1}^\ast(x) &=  u_{1}(x) & \text{ when  }  x\in \R^n \sm B(0,r),
\\
u_{1}^\ast(x) &= {\overline{u}}_{1}(rx/|x|) & \text{ when  }  x\in B(0,r)\sm B(a,r)
\\
u_{1}^\ast(x) &=  v_{1}(a^{-1} x) & \text{ when  }  x\in  B(0,ar).
\end{array}
\end{equation}
Let us now define $\W^\ast$ so that, with this definition of the $u_{i}^\ast$,
\begin{equation} \label{e6.15}
(\u^\ast,\W^\ast) \in \F.
\end{equation}
First we need to know that $u_{i}^\ast \in W^{1,2}(\R^n)$. By construction,
it is easy to see that $u_{i}^\ast \in W^{1,2}(\O)$ for $\O = \R^n \sm \overline B(0,r)$,
$\O = B(0,r)\sm \overline B(0,ar)$, and $\O = B(0,ar)$, but we also need to
check the gluing condition \eqref{e4.17} on the spheres $S_{ar}$ and 
$S_{r}$, to make sure that $Du_{i}^\ast$ does not have any extra piece on them.
But this is the case, because  we chose $u_{i}^\ast$ with equal radial limits 
from both sides of these spheres. 

We shall also assume that for $1 \leq i \leq N$,
\begin{equation} \label{e6.16}
u_{i}(x) = 0 \text{ for $\sigma$-almost every } x \in S_{r} \sm W_{i}.
\end{equation}
This is true for almost every $r > 0$, because $u_{i}(x)=0$ a.e. on 
$\R^n \sm W_i$, so this extra assumption will not cost us anything. 
Now set
\begin{equation} \label{e6.17}
W_{1}^\ast = [W_{1} \sm B(0,r)] \cup \big\{x\in \Omega \cap B(0,r) \sm B(0,ar) ; 
x/|x| \in W_{1} \big\} \cup [\Omega \cap B(0,ar)]
\end{equation}
and, for $i \geq 2$,
\begin{equation} \label{e6.18}
W_{i}^\ast = [W_{i} \sm B(0,r)] \cup \big\{x\in \Omega \cap B(0,r) \sm B(0,ar) ; 
x/|x| \in W_{i} \big\}.
\end{equation}
It is easy to see that these sets are disjoint and contained in $\Omega$.
Let us also check that $u_{i}^\ast(x) = 0$ for almost every $x\in \R^n \sm W_{i}^\ast$.
If $x\in \R^n \sm B(0,r)$, this comes from the corresponding property for $u_{i}$
(we did not change anything there). When $x\in B(0,r) \sm B(0,ar)$, 
$u_{i}^\ast(x) = 0$ because $u_{i}(rx/|x|) = 0$, which itself holds almost everywhere
by \eqref{e6.16}. When $x\in B(0,ar)$, either $i \geq 2$ and 
$u_{i}^\ast(x) = 0$ by \eqref{e6.11} and because $\varphi(|x|) = 0$,
or else $i=1$, but then $x$ lies in the set $B(0,ar)\sm \Omega$, which has
zero measure by (\ref{e6.9}). As usual, if some of the $u_{i}$ are nonnegative,
so are the corresponding $u_{i}^\ast$. This proves \eqref{e6.15}.

Let us also give a first estimate for $\int_{B(0,r)} |\nabla \u^\ast|^2$. 
For $i \geq 2$ and $x\in B(0,r) \sm B(0,ar)$, the definition \eqref{e6.11} says that
$\nabla u_{i}^\ast$ has a tangential gradient $\nabla_{t} u_{i}^\ast$ such that
$|\nabla_{t} u_{i}^\ast(x)|^2 = \varphi(|x|)^2 \big(r/|x|\big)^2 |\nabla_{t} u_{i}(rx/|x|)|^2
\leq \big(r/|x|\big)^2 |\nabla_{t} u_{i}(rx/|x|)|^2$, and a radial gradient
$\nabla_{r}u_{i}^\ast$ such that $|\nabla_{r} u_{i}^\ast(x)|^2 = \varphi'(|x|)^2 u_{i}(rx/|x|)|^2
\leq 4(1-a)^{-2}r^{-2} u_{i}(rx/|x|)|^2$. Thus 
\begin{eqnarray} \label{e6.19}
\int_{B(0,r)} |\nabla u_{i}^\ast|^2 &=& \int_{ar}^r \int_{\d B(0,t)} |\nabla u_{i}^\ast|^2
\nonumber\\
&\leq& \int_{ar}^r \int_{\d B(0,t)} \Big[(r/t)^2 |\nabla_{t} u_{i}(rx/|x|)|^2
+ 4(1-a)^{-2}r^{-2} u_{i}(rx/|x|)|^2 \big]d\sigma(x) dt
\nonumber
\\
&=& \int_{ar}^r \int_{S_{r}} \Big[(r/t)^{2-n} |\nabla_{t} u_{i}(\xi)|^2
+ 4(1-a)^{-2}r^{-2} (r/t)^{-n} u_{i}(\xi)|^2 \big]d\sigma(\xi) dt
\\
&\leq& (1-a) r  a^{2-n} \int_{S_{r}} |\nabla_{t} u_{i}|^2
+ 4 (1-a)^{-1} r  a^{-n} \int_{S_{r}} |r^{-1} u_{i}|^2.
\nonumber
\end{eqnarray}
This will be acceptably small (with $a= 1/2$) when $\int_{S_{r}} |\nabla_{t} u_{i}|^2$
and $\int_{S_{r}} |r^{-1} u_{i}|^2$ are small (the second often following
from the first one and Poincar\'e), or even (with $a$ close to $1$) if 
$\int_{S_{r}} |\nabla_{t} u_{i}|^2$ is not too large but 
$\int_{S_{r}} |r^{-1} u_{i}|^2$ is very small (which typically follows from
Poincar\'e if $W_{i} \cap S_{r}$ is very small).

For $i=1$, the estimate for $\nabla u_{i}^\ast$ on $B(0,r) \sm B(0,ar)$
is simpler, because it only has a tangential gradient, and
$|\nabla u_{1}^\ast(x)|^2 = |\nabla_{t} u_{1}^\ast(x)|^2 = 
\big(r/|x|\big)^2 |\nabla_{t} u_{1}(rx/|x|)|^2$ (as above), so
\begin{equation} 
\label{e6.20}
\int_{B(0,r)\sm B(0,ar)} |\nabla u_{1}^\ast|^2
\leq (1-a) r  a^{2-n} \int_{S_{r}} |\nabla_{t} u_{1}|^2.
\end{equation}
The remaining part is
\begin{eqnarray} \label{e6.21}
\int_{B(0,ar)} |\nabla u_{1}^\ast|^2 
&=& a^{-2}\int_{B(0,ar)} |\nabla v_{1}(a^{-1}x)|^2
= a^{n-2}\int_{B(0,r)} |\nabla v_{1}|^2
\nonumber
\\
&=& a^{n-2} \inf\Big\{ \int_{B(0,r)} |\nabla v|^2 \, ; \,
v\in W^{1,2}(B(0,r)) \text{ and } v=u_{1} \text{ on } S_{r} \Big\},
\nonumber
\\
&=& a^{n-2} \inf\Big\{ \int_{B(0,r)} |\nabla v|^2 \, ; \,
v\in W^{1,2}(\R^n) \text{ and } v=u_{1} \text{ a.e. on } \R^n\sm B(0,r) \Big\},
\end{eqnarray}
by \eqref{e6.13}, and where we mentioned the last infimum because 
its definition no longer involves radial limits; but this is the same by the
gluing property \eqref{e4.18}.

Notice that $||u_{i}^\ast||_{\infty} \leq ||u_{i}||_{\infty}$ because the Poisson kernel
is nonnegative and sends the constant $1$ to $1$; 
then the proof of \eqref{e6.6} and \eqref{e6.7} also yields 
\begin{equation}
\label{e6.22}
\Big|\int_{\Omega} (u_{i}^\ast)^2 f_{i} - \int_{\Omega} u_{i}^2 f_{i}\Big|
\leq C r^{n-{n \over p}} ||u_{i}||_{\infty}^2 ||f_{i}||_{p}
\end{equation}
and 
\begin{equation}\label{e6.23}
\Big|\int_{\Omega} u_{i}^\ast g_{i} - \int_{\Omega} u_{i}g_{i}\Big|
\leq C r^{n-{n \over p}} ||u_{i}||_{\infty} ||g_{i}||_{p}.
\end{equation}

\begin{rem} \label{r6.1}
When the $u_{i}$ are nonnegative, we could use the same trick as 
when $N=2$ to define a variant of the harmonic competitor that we
just defined, but where we select two main functions, say $u_{1}$ and $u_{2}$, 
and get rid of the other ones by the same cut-off argument.
That is, we would use the definition \eqref{e6.11} for $i \geq 3$,
and we would define $u_{1}^\ast$ and $u_{2}^\ast$ as follows.
We would group $u_{1}$ and $u_{2}$ as the real-valued $u = u_{1}-u_{2}$,
denote by $\overline u$ the restriction of $u$ to $S_{r}$,
call $v$ the harmonic extension of $\overline u$ to $B(0,r)$,
define $u^\ast$ by a formula like \eqref{e6.14}, and then cut it into
its positive part $u_{1}^\ast$ and its negative part $u_{2}^\ast$;
we could still define $W_{1}^\ast$ and $W_{2}^\ast$ so that
$(\u^\ast,\W^\ast) \in \F$ (with the other $W_{i}^\ast$ defined
as in \eqref{e6.18}). The estimates \eqref{e6.20}-\eqref{e6.23}
would have analogues too.
But we do not seem to need this trick in the present paper.
\end{rem}

\section{H\"older-continuity of $u$ inside $\Omega$}  \label{holder} 

In this section we keep the same assumptions \eqref {e5.1}-\eqref {e5.4}
as in Section \ref{bounded}, also assume that the function $F$ in the
volume term is H\"older-continuous with an exponent 
$\beta > {n-2 \over n}$, and prove  that if $(\u,\W)$ is a minimizer
for our functional $J$, then $\u$ is H\"older-continuous on 
the interior of $\Omega$. 
We only see this as a first step towards interior regularity, which will allow us 
to be more relaxed about the definition of $\{ u_{i} > 0 \}$
(see Remark \ref{r7.2} below), 
and more importantly to use results of \cite{CJK} in later sections. 
But we intend to get more regularity later on (under stronger assumptions).
Also, we shall discuss the H\"older-continuity of $u$ near $\d \Omega$ in the
next section. 
Precisely, our H\"older condition on $F$ is that for some $\beta > {n-2 \over n}$,
\begin{equation}\label{e7.1}
\Big| F(W_{1},W_{2},\ldots,W_{N}) - F(W'_{1},W'_{2},\ldots,W'_{N}) \Big|
\leq C \sum_{i=1}^N |W_{i} \Delta W'_{i}|^\beta
\end{equation}
for some $C \geq 0$ and all choices of $N$-uples 
$(W_{i})$ and $(W'_{i}) \i \Omega^N$, of disjoint sets, and where $A \Delta B$ 
still denotes the symmetric difference $(A \sm B) \cup (B \sm A)$.

\begin{thm}\label{t7.1} 
Assume that \eqref {e5.1}-\eqref {e5.4} and \eqref {e7.1} hold. 
There is an exponent $\alpha > 0$, that depends only on $n$, $N$,
$\beta$ and $p$ (from \eqref {e5.2} and \eqref {e5.4}), 
and a constant $C_0 \geq 0$, that also depends on $|\Omega|$ 
and the bounds in \eqref{e5.2}, \eqref{e5.4}, and \eqref {e7.1},
such that if $(\u,\W)$ is a minimizer for $J$ in $\F$
(see Definition~\ref{d1.1} and \eqref {e1.3}-\eqref {e1.5}),
$x_0\in \Omega$ and $0 < r_0 \leq 1$ are such that
$B(x_{0},r_{0}) \i \Omega$, then 
(possibly after modifying $\u$ on a set of measure zero)
\begin{equation}\label{e7.2}
|\u(x)-\u(y)| \leq C |x-y|^\alpha
\ \text{ for } x,y \in B(x_{0},r_{0}/2),
\end{equation}
with $C = C_0+ C_0 r_0^{1-\alpha} \big\{\fint_{B(0,r_0)} |\nabla \u|^2\big\}^{1/2}$.
\end{thm}

\ms
Observe that we do not require $\Omega$ to be nice (or even open), but 
we compensate by requiring that $B(x_{0},r_{0}) \i \Omega$
(and in fact, we only need this modulo a set of vanishing Lebesgue measure,
since modifying $\Omega$ on a set of measure zero does not change the problem).
We required $r_0 \leq 1$ in order to obtain a constant $C$ in \eqref{e7.2} that
depends on $r_{0}$ as stated. Of course the fact that $C_0$ does not depend on 
the specific choice of data, or of the minimizer, but only on the various constants 
in our assumptions, is more important.
Finally, the fact that $\alpha$ depends on $\beta$ and $p$ hides an
important defect of Theorem~\ref{t7.1}: even with $p=+\infty$
and $\beta = 1$, our proof will only give a very small exponent $\alpha > 0$,
while much more regularity is expected. 

A different proof of Theorem \ref{t7.1} may be given 
in a forthcoming paper, which is more complicated but perhaps more natural, 
relies on a different monotonicity argument, 
and most importantly works in the context of ${\cal L} = (-\Delta)^2$.

\ms 
Our proof of Theorem \ref{t7.1} will follow the same rough outline as a proof
of monotonicity for the normalized energy that A. Bonnet gave in the context of
the Mumford-Shah functional \cite{Bo}. 
But let us first observe that when $n=1$, Theorem \ref{t7.1} holds with $\alpha = 1/2$,
just because $|\u(x)-\u(y)| = \big|\int_{[x,y]} \nabla \u \big| \leq
|x-y|^{1/2} \big\{\int_{[x,y]} |\nabla \u|^2 \big\}^{1/2}$ by H\"older's inequality.
So we may assume that $n \geq 2$.

We fix $x_{0}$ (without loss of generality
we shall immediately assume that $x_{0} = 0$), and we want to prove a 
differential inequality on the function $E$, where 
\begin{equation}\label{e7.3}
E(r) = \int_{B(x_{0},r)} |\nabla \u|^2
= \int_{B(0,r)} |\nabla \u|^2
= \int_{B(0,r)} \sum_{i=1}^N  |\nabla u_{i}|^2.
\end{equation}
By Fubini (and as in \eqref{e4.3}),
\begin{equation}\label{e7.4}
E(r) =  \int_{t=0}^r  \int_{y\in S_{r}} |\nabla \u|^2 d\sigma(y)dt,
\end{equation}
where we still use the notation $S_{r} = \d B(0,r)$.
Hence the derivative $E'(r)$ exists for almost every $r > 0$, 
\begin{equation}\label{e7.5}
E'(r) =  \int_{S_{r}} |\nabla \u|^2 d\sigma
\ \text{ for almost every } r > 0,
\end{equation}
and $E$ is the indefinite integral of $E'$.
Our main goal is to estimate $E(r)$ in terms of $E'(r)$, and
then integrate in $r$ to get a good upper bound for $E(r)$ for $r$
small; the H\"older estimate \eqref{e7.2} will then follow easily.

Notice that $E(r)$ is a nice quantity to work with, because the minimality
of $(\u,\W)$ gives an estimate on $E(r)$ each time we build a competitor
for $(\u,\W)$ in $B(0,r)$. So let us do this. We restrict to $r \leq r_{0}$
(and this way we get that $B(0,r) \i \Omega$, as in \eqref{e6.9}), 
and also assume that \eqref{e6.8}, \eqref{e6.16}, and the conclusion
of \eqref{e7.5} hold (they all hold for almost all $r$, and we'll only need
almost all $r$ to integrate the differential inequality). 

We shall distinguish between two cases. Let $\varepsilon > 0$ be small, to be chosen later; 
we start with the case when $W_{1} \cap \d B(0,r)$ is very large, more precisely
\begin{equation}\label{e7.6}
\sigma(S_{r} \sm W_{1}) \leq \varepsilon \sigma(S_{r}),
\end{equation}
and we use the harmonic competitor $(\u^\ast,\W^\ast)$ defined near \eqref{e6.11}.
Since $J(\u,\W) \leq J(\u^\ast,\W^\ast)$, we get that
\begin{equation}\label{e7.7}
E(r) = \int_{B(0,r)} |\nabla \u|^2 \leq \int_{B(0,r)} |\nabla \u^\ast|^2
+ |M(\u^\ast) - M(\u)| +  |F(\W^\ast) - F(\W)|
\end{equation}
where $M$ and $F$ are as in \eqref{e1.3}-\eqref{e1.5}. 
Obviously $W_{i}^\ast$ coincides with
$W_{i}$ on $\R^n \sm B(0,r)$, so $|W_{i}^\ast \Delta W_{i}| \leq C r^n$, 
and \eqref{e7.1} yields
\begin{equation}\label{e7.8}
|F(\W^\ast) - F(\W)| \leq C \sum_{i} |W_{i}^\ast \Delta W_{i}|^\beta \leq C r^{\beta n}.
\end{equation}
For the $M$-terms, we use \eqref{e6.22} and \eqref{e6.23} and get that 
\begin{eqnarray}\label{e7.9}
|M(\u^\ast) - M(\u)| &\leq& \sum_{i}
\Big|\int_{\Omega} (u_{i}^\ast)^2 f_{i} - \int_{\Omega} u_{i}^2 f_{i}\Big|
+ \Big|\int_{\Omega} u_{i}^\ast g_{i} - \int_{\Omega} u_{i}g_{i}\Big|
\nonumber \\
&\leq& C r^{n-{n \over p}} ||u_{i}||_{\infty}^2 ||f_{i}||_{p} 
+ C r^{n-{n \over p}} ||u_{i}||_{\infty} ||g_{i}||_{p}
\leq C r^{n-{n \over p}},
\end{eqnarray}
where we just used Theorem \ref{t5.1}. Recall that Theorem \ref{t5.1} also
says that the $||u_{i}||_{\infty}^2$ are bounded in terms of the various constants
in the assumptions of Theorem \ref{t7.1}; hence $C$ only depends on these 
constants. This remark will also apply to the other constants $C$ 
in the computations that follow.

We now use the energy estimates \eqref{e6.19}-\eqref{e6.20},
plus the first part of \eqref{e6.21}, and get that
\begin{eqnarray}\label{e7.10}
\int_{B(0,r)} |\nabla \u^\ast|^2 
\leq \sum_{i =1}^N (1-a) r  a^{2-n} \int_{S_{r}} |\nabla_{t} u_{i}|^2
&+& \sum_{i \geq 2} 4 (1-a)^{-1} r  a^{-n} \int_{S_{r}} |r^{-1} u_{i}|^2
\nonumber
\\
&+& a^{n-2} \int_{B(0,r)} |\nabla v_{1}|^2
\end{eqnarray}
where $v_{1}$ (defined below \eqref{e6.11}), is the harmonic extension of 
the restriction $\overline u_{1}$ of $u_{1}$ to $S_{r}$. Thus,
by \eqref{e7.7}-\eqref{e7.10} and because $a^{n-2}\leq 1$,
\begin{equation}\label{e7.11}
E(r) \leq \int_{B(0,r)} |\nabla v_{1}|^2 + A,
\end{equation}
with 
\begin{eqnarray}\label{e7.12}
A &\leq& C r^{\beta n} + C r^{n-{n \over p}} + 
\sum_{i =1}^N (1-a) r  a^{2-n} \int_{S_{r}} |\nabla_{t} u_{i}|^2
+ \sum_{i \geq 2} 4 (1-a)^{-1} r  a^{-n} \int_{S_{r}} |r^{-1} u_{i}|^2
\nonumber
\\
&\leq& C r^{\beta n} + C r^{n-{n \over p}} + C (1-a) r \int_{S_{r}} |\nabla_{t} \u|^2
+ C (1-a)^{-1} r^{-1}  \sum_{i \geq 2}\int_{S_{r}} |u_{i}|^2
\end{eqnarray}
because we shall take $a \geq 1/2$. In addition, for $i \geq 2$
\eqref{e6.8} says that (the restriction of) $\overline u_{i}$ lies
in $W^{1,2}(S_{r})$,  and \eqref{e6.16} says that $u_{i}(x) = 0$ 
almost everywhere on $S_{r} \sm W_{i}$. Set $E = S_{r} \sm W_{1}$;
if $x\in S_{r} \sm E$, then $x\in W_{1}$, hence (by disjointness)
$x \in S_{r} \sm W_{i}$, and almost always $u_{i}(x) = 0$.
This allows us to apply Lemma \ref{l4.1} (with $p=2$) and get that
\begin{equation}\label{e7.13}
\int_{S_{r}} |u_{i}|^2 = \int_{E} |u_{i}|^2 \leq C_{2} \sigma(E)^{2 \over n-1} 
\int_{E} |\nabla_{t} u_{i}|^2
\leq C_{2} (\varepsilon \sigma(S_{r}))^{2 \over n-1}
\int_{E} |\nabla_{t} u_{i}|^2
\end{equation}
by \eqref{e7.6} and \eqref{e4.7}. Recall also that we restrict to
$r$ such that the conclusion of \eqref{e7.5} holds, so
\begin{equation}\label{e7.14}
\sum_{i} \int_{S_{r}} |\nabla_{t} u_{i}|^2
\leq  \sum_{i} \int_{S_{r}} |\nabla u_{i}|^2
= \int_{S_{r}} |\nabla \u|^2 = E'(r);
\end{equation}
then \eqref{e7.12} yields
\begin{equation}\label{e7.15}
A \leq C r^{\beta n} + C r^{n-{n \over p}} + C (1-a) r E'(r) 
+ C (1-a)^{-1} \varepsilon^{2 \over n-1} r E'(r).
\end{equation}
Let $\tau > 0$ be small, to be chosen later.
We choose $a$ close to $1$  (depending on $\tau$), and then $\varepsilon$ very small
(depending also on $a$), so that \eqref{e7.15} yields 
\begin{equation}\label{e7.16}
A \leq C r^{\beta n} + C r^{n-{n \over p}} + \tau r E'(r),
\end{equation}
and now we concentrate on $\int_{B(0,r)} |\nabla v_{1}|^2$.
We shall just need an estimate on the norm of the Poisson extension,
from $W^{1,2}(S_{r})$ to $W^{1,2}(B_{r})$, but since 
$\int_{B(0,r)} |\nabla v_{1}|^2$ is minimal (by \eqref{e6.13}
or \eqref{e6.21}), it will be enough to control the energy of some
extension. Call $u$ the restriction of $u_{1}-m_{S_{r}}^\sigma u_{1}$ 
to $S_{r}$, and define $v$ on $B(0,r)$ by
\begin{equation}\label{e7.17}
v(ty) = tu(y) \text{ for $y\in S_{r}$ and } 0 \leq t \leq 1.
\end{equation}
[We will not really lose much, because if we were to find the optimal
extension, it would happen to have the largest extension norm on spherical
harmonics of degree $1$, which happen to have homogeneous extensions of degree $1$.]
Anyway, it is easy to see that $v\in W^{1,2}(B(0,r))$, with a gradient that we compute
now. In fact, since $v$ is homogeneous of degree $1$, its gradient is 
homogeneous of degree $0$. So we compute it at $y\in S_{r}$. Its radial part
is $r^{-1}u(y)$, and its tangential part is just $\nabla_{t}u(y)$. Then
$|\nabla v(y)|^2 =  r^{-2}u(y)^2 + |\nabla_{t}u(y)|^2$, and 
\begin{eqnarray}\label{e7.18}
\int_{B(0,r)} |\nabla v|^2 
&=& \int_{t\in (0,r)} \int_{S_{t}} |\nabla v|^2 d\sigma dt
= \int_{t\in (0,r)} (t/r)^{n-1} \int_{S_{r}} |\nabla v|^2 d\sigma dt
\nonumber
\\
&=& {r^n \over n r^{n-1}} \int_{S_{r}} |\nabla v|^2 d\sigma
= {r \over n} \int_{S_{r}} r^{-2}u^2 + |\nabla_{t} u|^2 
\end{eqnarray}
by homogeneity. We use Poincar\'e's inequality on the sphere 
(exceptionally, with the right constant!), which says that
\begin{equation}\label{e7.19}
\int_{S_{r}} r^{-2} |u|^2 \leq {1 \over n-1} \int_{S_{r}} |\nabla_{t} u|^2
\end{equation}
because $\int_{S_{r}} u d\sigma = 0$. See for instance 
Exercise 76.21 in \cite{D}. 
Then \eqref{e7.18} and \eqref{e7.19} yield
$\int_{B(0,r)} |\nabla v|^2 \leq {r \over n}{n \over n-1} 
\int_{S_{r}} |\nabla_{t}u|^2$ and, since 
$v + m_{S_{r}}^\sigma u_{1}$ has boundary values on $S_{r}$ 
equal to $u + m_{S_{r}}^\sigma u_{1} = u_{1}$, the minimizing
property of $v_{1}$ yields
\begin{eqnarray}\label{e7.20}
\int_{B(0,r)}  |\nabla v_{1}|^2 
&\leq& \int_{B(0,r)}  |\nabla (v + m_{S_{r}}^\sigma u_{1})|^2
= \int_{B(0,r)}  |\nabla v |^2 
\nonumber
\\
&\leq& {r \over n-1} \int_{S_{r}} |\nabla_{t}u(y)|^2
= {r \over n-1} \int_{S_{r}} |\nabla_{t}u_{1}(y)|^2
\leq {r \over n-1} \, E'(r)
\end{eqnarray}
by \eqref{e7.14}. We combine with \eqref{e7.11} and \eqref{e7.16}
and get that 
\begin{equation}\label{e7.21}
E(r) \leq \int_{B(0,r)} |\nabla v_{1}|^2 + A
\leq C r^{\beta n} + C r^{n-{n \over p}} + \Big(\tau + {1\over n-1}\Big) \, r E'(r).
\end{equation}
This is our first differential inequality, valid at almost every $r \leq r_{0}$
such that \eqref{e7.6} holds. If \eqref{e7.6}, with $i=1$ replaced by some
other index $i$, holds, we do the same argument with $u_{1}$ replaced by $u_{i}$,
and we still get the conclusion of \eqref{e7.21}.

\ms
When \eqref{e7.6} fails for all indices, i.e., if
$\sigma(S_{r} \sm W_{i}) \geq \varepsilon \sigma(S_{r})$
for all $i$, we use another competitor to get a similar
differential inequality. This time we pick a very small $\gamma > 0$,
to be chosen later, and we set
\begin{equation}\label{e7.22}
u^\ast_{i}(ty) = t^{\gamma} u_{i}(y)
\ \text{ for $y\in S_{r}$ and $0 \leq t < 1$.}
\end{equation}
On $\R^n \sm B(0,r)$, we keep $u_{i}^\ast = u_{i}$, as usual.
It is easy to see, using again the gluing condition \eqref{e4.17}
that $u_{i}^\ast \in W^{1,2}(\R^n \sm \{ 0 \})$. 
Its gradient is now homogeneous of degree $\gamma-1$,
and we compute it at $y\in S_{r}$. The tangential gradient
is still $\nabla_{t} u_i(y)$, and the radial derivative is
$\gamma r^{-1} u_{i}(y)$. Thus
$|\nabla u^\ast_{i}(y)|^2 = |\nabla_{t} u_i(y)| + \gamma^2 r^{-2} u_{i}(y)^2$.
The same computation as in \eqref{e7.18} yields
\begin{eqnarray}\label{e7.23}
\int_{B(0,r)} |\nabla u^\ast_{i}|^2 
&=& \int_{t\in (0,r)} \int_{S_{t}} |\nabla u^\ast_{i}|^2 d\sigma dt
= \int_{t\in (0,r)} (t/r)^{n-1} (t/r)^{2\gamma - 2}
\int_{S_{r}} |\nabla u^\ast_{i}|^2 d\sigma dt
\nonumber
\\
&=& {r^{n+2\gamma -2} \over (n+2\gamma-2) r^{n+2\gamma-3}} 
\int_{S_{r}} |\nabla u^\ast_{i}|^2 d\sigma
= {r \over (n+2\gamma-2)} \int_{S_{r}} \gamma^2 r^{-2}u_i^2 + |\nabla_{t} u_i|^2 
\end{eqnarray}
In particular the integral converges (recall that $n \geq 2$), 
and it is not hard to show that $u^\ast_{i} \in W^{1,2}$ near the origin. 
For instance, we can approximate $t^\gamma$ in \eqref{e7.22} 
with functions that vanish near $0$, and take a limit.

Since \eqref{e7.6} fails for $i$, we can apply Lemma \ref{l4.1} to $u_{i}$,
with $E = W_{i} \cap S_{r}$ (by \eqref{e6.16}). We get that
\begin{equation}\label{e7.24}
\int_{S_{r}} |u_{i}|^2 = \int_{E} |u_{i}|^2 
\leq C r^2 {\sigma(S_{r}) \over \sigma(S_{r} \sm W_{i})} \int_{E} |\nabla_{t} u_{i}|^2
\leq C r^2 \varepsilon^{-1} \int_{S_{r}} |\nabla_{t} u_{i}|^2
\end{equation}
by \eqref{e4.6}. Set
\begin{equation}\label{e7.25}
\lambda = {1 \over (n+2\gamma-2)} + {C\gamma^2 \over (n+2\gamma-2) \varepsilon},
\end{equation}
with the same constant $C$ as in \eqref{e7.24}. Then \eqref{e7.23} yields
\begin{equation}\label{e7.26}
\int_{B(0,r)} |\nabla u^\ast_{i}|^2 
\leq \lambda r \int_{S_{r}} |\nabla_{t} u_{i}|^2.
\end{equation}
We sum over $i$ and get that
\begin{equation}\label{e7.27}
\int_{B(0,r)} |\nabla \u^\ast|^2 
\leq \lambda r \int_{S_{r}} |\nabla_{t} \u|^2 \leq \lambda r E'(r)
\end{equation}
with $\u^\ast = (u_{1}^\ast, \ldots, u_{N}^\ast)$, and by \eqref{e7.14}.

We now complete the definition of our competitor by choosing
sets $W_{i}^\ast \i \Omega$, so that $(\u^\ast,\W^\ast) \in \F$.
We keep $W_{i}^\ast \sm B(0,r) = W_{i}\sm B(0,r)$, and we 
take for $W_{i}^\ast \cap B(0,r)$ the intersection of $\Omega$
with the cone over $W_{i} \cap S_{r}$. The $W_{i}^\ast$ are disjoint
because the $W_{i}$ are disjoint, and $u_{i}^\ast = 0$
almost everywhere on $\R^n \sm W_{i}$ because $\Omega$ is almost 
contained in $B(0,r)$, and $u_{i}(y) = 0$ almost everywhere on 
$S_{r} \sm W_{i}$ (by \eqref{e6.16}). As usual, $u_{i}^\ast \geq 0$
if $u_{i}\geq 0$, so $(\u^\ast,\W^\ast) \in \F$.

We complete the estimate with the $F$ and $M$-terms. As before,
$||u_{i}^\ast||_{\infty} \leq ||u_{i}||_{\infty}$, so
\begin{eqnarray}\label{e7.28}
|M(\u^\ast) - M(\u)| 
&\leq&  C r^{n-{n \over p}} ||u_{i}||_{\infty}^2 ||f_{i}||_{\infty} 
+ C r^{n-{n \over p}} ||u_{i}||_{\infty} ||g_{i}||_{p}
\leq C r^{n-{n \over p}},
\end{eqnarray}
as in \eqref{e7.9}; also, 
$|W_{i}^\ast \Delta W_{i}| \leq C r^n$, so \eqref{e7.1} yields
\begin{equation}\label{e7.29}
|F(\W^\ast) - F(\W)| \leq C \sum_{i} |W_{i}^\ast \Delta W_{i}|^\beta \leq C r^{\beta n}.
\end{equation}
as in \eqref{e7.8}. Since $J(\u,\W) \leq J(\u^\ast,\W^\ast)$, we get (as in
\eqref{e7.7}) that
\begin{eqnarray}\label{e7.30}
E(r) &=& \int_{B(0,r)} |\nabla \u|^2 \leq \int_{B(0,r)} |\nabla \u^\ast|^2
+ |M(\u^\ast) - M(\u)| +  |F(\W^\ast) - F(\W)|
\nonumber
\\
&\leq& \int_{B(0,r)} |\nabla \u^\ast|^2 + C r^{n-{n \over p}} + C r^{\beta n}
\leq \lambda r E'(r) + C r^{n-{n \over p}} + C r^{\beta n}
\end{eqnarray}
by \eqref{e7.27}. This is our alternative differential inequality, which holds
for almost every $r$ such that \eqref{e7.21} fails for each $i$. 

Let us now check that we can choose $\delta$ such that
\begin{equation}\label{e7.31}
n-2 < \delta < \min\Big(n\beta,n-{n \over p}, 
\Big(\tau + {1\over n-1}\Big)^{-1}, \lambda^{-1} \Big).
\end{equation}
We know that $n\beta > n-2$ (see the definition of $\beta$ near \eqref{e7.1}),
and $n-{n \over p} > n-2$ because $p> n/2$ (see \eqref{e5.4}).
Next choose $\tau$ so small that $\tau + {1\over n-1} < {1\over n-2}$ 
(no condition if $n=2$); then $n-2 < \Big(\tau + {1\over n-1}\Big)^{-1}$, 
and this too leaves some room for $\delta $. 
This choice of $\tau$ forces $a$ to be chosen close enough to $1$,
and $\varepsilon$ small enough, but this is all right. We still need to check that
$(n-2) \lambda < 1$; rewrite \eqref{e7.25} as 
\begin{equation}\label{e7.32}
(n-2) \lambda = {n-2 \over (n+2\gamma-2)}\Big(1 + {C\gamma^2 \over \varepsilon}\Big)
= \Big(1 + {2\gamma \over n-2} \Big)^{-1}\Big(1 + {C\gamma^2 \over \varepsilon}\Big);
\end{equation}
even if we chose $\varepsilon$  very small, this expression becomes smaller than $1$
when $\gamma$ is small enough. So we choose $\gamma$ small, depending 
on $\varepsilon$, and we can choose $\delta $ as in \eqref{e7.31}. 
Notice that $a$, then $\varepsilon$, $\gamma$, and finally $\delta$
depend on $N$, because (we think that) $C$ in \eqref{e7.15} depends on $N$.
With all this new notation, \eqref{e7.21} and \eqref{e7.30} yield
\begin{equation}\label{e7.33}
\delta  E(r) \leq   r E'(r) + C r^{n-{n \over p}} + C r^{\beta n}
\ \hbox{ for almost every } r \leq r_{0},
\end{equation}
where $C$ also depends on constants like $\delta$, but not
on $r$ or $r_0$. 

Now we want to integrate this differential inequality between
$r\in (0,r_{0})$ and $r_{0}$. Set $f(r) = r^{-\delta} E(r)$;
then $f$ is differentiable almost everywhere on $(0,r_{0})$,
and 
\begin{equation}\label{e7.34}
f'(r) = - \delta  r^{-\delta -1} E(r)+ r^{-\delta} E'(r) 
\geq - C r^{-\delta -1} [r^{n-{n \over p}} + r^{\beta n}]
\end{equation}
almost everywhere, by \eqref{e7.33}. 
We want to integrate this between $r \in (0,r_{0})$ and $r_{0}$
and get that
\begin{eqnarray}\label{e7.35}
f(r) &=&  f(r_{0}) - \int_{r}^{r_{0}} f'(t) dt
\leq f(r_{0}) + C \int_{r}^{r_{0}} t^{-\delta-1}[t^{n-{n \over p}} + t^{\beta n}]dt
\nonumber
\\
&\leq& f(r_{0}) + C r^{-\delta} [r^{n-{n \over p}} + r^{\beta n}]
\end{eqnarray}
(recall that by \eqref{e7.31}, the final exponents are negative).
So we have to justify the first equality. Write $g(r) = r^{-\delta}$
to simplify the algebra, and recall from \eqref{e7.4} and \eqref{e7.5} that
$E(t) = E(r) + \int_{r}^t E'(s) ds$ for $r < t < r_{0}$. Then 
set $I = [r,r_{0}]$ and compute
\begin{eqnarray}\label{e7.36}
\int_{r}^{r_{0}} f'(t) dt 
&=& \int_{I} E(t)g'(t)+E'(t)g(t) dt
\nonumber
\\
&=& \Big(\int_{I} E(r)g'(t) + \int_{I\times I} E'(s)g'(t) \1_{s < t}\Big)
+ \Big(\int_{I} E'(t)g(r) + \int_{I\times I} E'(t)g'(s) \1_{s < t} \Big)
\nonumber
\\
&=& E(r)[g(r_{0})-g(r)] + g(r)[E(r_{0})-E(r)] + \int_{I\times I} E'(t)g'(s) dsdt
\\
\nonumber
&=& E(r)[g(r_{0})-g(r)] + g(r)[E(r_{0})-E(r)] +[E(r_{0})-E(r)][g(r_{0})-g(r)]
\\
&=& E(r_{0}) g(r_{0}) - E(r)g(r) = f(r_{0}) - f(r)
\nonumber
\end{eqnarray}
where Fubini can be used because $E'$ and $g'$ are both nonnegative
(or integrable). So \eqref{e7.35} holds.  We multiply it by $r^\delta$ and get that
\begin{equation}\label{e7.37}
E(r) = r^\delta  f(r) 
\leq r^\delta f(r_{0}) + C [r^{n-{n \over p}} + r^{\beta n}]
= (r/r_0)^\delta E(r_{0}) + C [r^{n-{n \over p}} + r^{\beta n}].
\end{equation}
Since we shall use it a few times in the future, let us record what we just got:
under the general assumptions of Theorem \ref{t7.1}, we just proved that
\begin{equation}\label{e7.38}
\int_{B(x_0,r)} |\nabla \u|^2 
\leq (r/r_0)^\delta \int_{B(x_0,r_0)} |\nabla \u|^2 + C [r^{n-{n \over p}} + r^{\beta n}] 
\ \text{ for } 0 < r < r_0
\end{equation}
as soon as $B(x_0,r_0) \i \Omega$ (recall that we immediately assumed
that $x_0 = 0$, and see \eqref{e7.3} for the definition of $E(r)$).

Return to the proof; \eqref{e7.37} is the energy estimate that we wanted, 
but we also want its analogue for other centers.
And indeed we can do the proof of \eqref{e7.37}, but with any other origin
$x \in B(0,2r_{0}/3)$, and with radii $0 < r \leq r_{0}/3$, and we get that
\begin{equation}\label{e7.39}
\int_{B(x,r)} |\nabla \u|^2 
\leq (r/r_0)^\delta E(r_{0}) + C [r^{n-{n \over p}} + r^{\beta n}] 
\end{equation}
for $x \in B(0,2r_{0}/3)$ and $0 < r \leq r_{0}/3$.
Here $C$ depends on the usual constants, i.e., $n$, $N$, $p$, $|\Omega|$, 
the $||f_{i}||_{\infty}$ and $||g_{i}||_{p}$, and the two constants in \eqref{e7.1}.
Incidentally, these constants also control $E(r_{0})$, by (3.10)
(modified as in Remark \ref{r3.4} if $p < 2$) 
and because they easily control $J(\u,\W)$. But we may lose some
information if we use this remark.

Anyway, let us rewrite \eqref{e7.39} as
\begin{equation}\label{e7.40}
r \Big\{\fint_{B(x,r)} |\nabla \u|^2\Big\}^{1/2} \leq C \theta(r),
\ \text{ with }
\theta(r) = r^{\delta -n +2 \over 2} [r_0^{-\delta} E(r_0)]^{1/2}
+ r^{2p-n \over 2p} + r^{2+\beta n - n \over 2}
\end{equation}
and notice that by \eqref{e7.31}, the smallest exponent in \eqref{e7.40} is
\begin{equation}\label{e7.41}
\alpha = {\delta -n +2 \over 2} > 0.
\end{equation}
We shall now check that \eqref{e7.2} follows from \eqref{e7.40}, with
this exponent $\alpha$. This will be a standard consequence of the 
Poincar\'e inequalities. 

Fix $i \in [1,N]$, set $u = u_{i}$ (to save notation), and define
$u(x,r) = \fint_{B(x,r)} u$ for $x \in B(0,2r_{0}/3)$ and 
$0 < r \leq r_{0}/3$. Then 
\begin{eqnarray}\label{e7.42}
|u(x,r/2) - u(x,r)| &=& \Big|\fint_{B(x,r/2)} u-u(x,r)\Big|
\leq \fint_{B(x,r/2)} |u-u(x,r)| 
\nonumber
\\
&\leq& 2^n \fint_{B(x,r)} |u-u(x,r)|
\leq C r \fint_{B(x,r)} |\nabla u|
\nonumber
\\
&\leq& C r \Big\{\fint_{B(x,r)} |\nabla u|^2 \Big\}^{1/2}
\leq C \theta(r)
\end{eqnarray}
by Poincar\'e (see \eqref{e4.2}), H\"older, and \eqref{e7.40}.
It follows from iterations of \eqref{e7.42} that 
$\overline u(x) = \lim_{k \to +\infty} u(x,2^{-k})$
exists for all $x\in B(0,2r_{0}/3)$, and that
\begin{equation}\label{e7.43}
|\overline u(x)-u(x,2^{-k})| \leq C \theta(2^{-k})
\ \text{ when } 2^{-k} \leq r_{0}/3
\end{equation}
(also use the special form of $\theta$ to sum three geometric series).
Since $\overline  u(x) = u(x)$ for every point of Lebesgue
differentiability for $u$, we see that replacing $u$ with $\overline u$
on $B(0,2r_{0}/3)$ will only change its values on a set of measure $0$,
so it is now enough to prove that
\begin{equation}\label{e7.44}
|\overline u(x) - \overline u(y)| \leq C_1 |x-y|^\alpha
\end{equation}
for $x,y \in B(x_{0},r_{0}/2)$, and with the announced value of
\begin{equation}\label{e7.45}
C_1 = C_0+ C_0 r_0^{1-\alpha} \Big\{\fint_{B(0,r_0)} |\nabla \u|^2\Big\}^{1/2},
\end{equation}
where $C_0$ depends on the usual constants but not on $r_0 \leq 1$.
It is even enough to prove this when $|x-y| \leq r_{0}/10$
(just use a short chain of points, and maybe multiply $C$
by $5$). Choose $k$ such that
$2^{-k-2} \leq |x-y| \leq 2^{-k-1}$; then $2^{-k} \leq r_{0}/3$, and
\begin{eqnarray}\label{e7.46}
|\overline u(x) - \overline u(y)| 
&\leq& |\overline u(x) - u(x,2^{-k-1})| 
+ |u(x,2^{-k-1})-u(y,2^{-k})|
+ |u(y,2^{-k})- \overline u(y)|
\nonumber
\\
&\leq&  |u(x,2^{-k-1})-u(y,2^{-k})| + C \theta(2^{-k})
\nonumber
\\
&\leq& |u(x,2^{-k-1})-u(y,2^{-k})| + C \theta(|x-y|).
\end{eqnarray}
But, as in \eqref{e7.42}
\begin{eqnarray}\label{e7.47}
|u(x,2^{-k-1})-u(y,2^{-k})| &=& \Big|\fint_{B(x,2^{-k-1})} u-u(y,2^{-k})\Big|
\leq \fint_{B(y,2^{-k})} |u-u(y,2^{-k})|
\nonumber
\\
&\leq& 2^n \fint_{B(y,2^{-k})} |u-u(y,2^{-k})|
\leq C 2^{-k} \fint_{B(y,2^{-k})} |\nabla u|
\nonumber
\\
&\leq& C 2^{-k} \Big\{\fint_{B(y,2^{-k})} |\nabla u|^2 \Big\}^{1/2}
\leq  C \theta(2^{-k}) \leq C \theta(|x-y|)
\end{eqnarray}
because $B(x,2^{-k-1}) \i B(y,2^{-k})$. But
\begin{equation}\label{e7.48}
\theta(r) =  r^{\alpha} [r_0^{-\delta} E(r_0)]^{1/2}
+ r^{2p-n \over 2p} + r^{2+\beta n - n \over 2}
\leq r^{\alpha} [r_0^{-\delta} E(r_0)]^{1/2} + 2 r^\alpha
\end{equation}
by \eqref{e7.40}, \eqref{e7.41}, and because $r \leq r_0 \leq 1$ and
$\delta$ was the smallest exponent (by \eqref{e7.31}), and
\begin{equation}\label{e7.49}
[r_0^{-\delta} E(r_0)]^{1/2} = \Big[r_0^{-2\alpha - n +2} \int_{B(0,r_0)} |\nabla \u|^2\Big]^{1/2} 
= r_0^{1-\alpha} \Big[\fint_{B(0,r_0)} |\nabla \u|^2\Big]^{1/2},
\end{equation}
so \eqref{e7.46} and \eqref{e7.47} yield
\begin{equation}\label{e7.50}
|\overline u(x) - \overline u(y)| \leq C \theta(|x-y|)
\leq C |x-y|^\alpha 
+ C |x-y|^\alpha r_0^{1-\alpha} \Big[\fint_{B(0,r_0)} |\nabla \u|^2\Big]^{1/2},
\end{equation}
which proves \eqref{e7.44} with $C_1$ as in \eqref{e7.45}.
This completes our proof of Theorem~\ref{t7.1}.
\qed

\begin{rem}\label{r7.2} 
Once we have Theorem \ref{t7.1}, we can be a little more relaxed about the
definition of the $W_i$. Suppose $\Omega$ is open and that the assumptions
of Theorem \ref{t7.1}are satisfied. We know that there is a locally H\"older
continuous function $ \widetilde \u$ that coincides with $\u$ almost everywhere
on $\Omega$ (the local continuous functions provided by applications of the theorem
on small balls $B \i \Omega$ can easily be glued). Then we can work with the open sets
\begin{equation}\label{e7.51}
\Omega_i = \big\{ x\in \Omega \, ; \, \widetilde u_i(x) > 0 \big\}; 
\end{equation}
it is easy to see that the $\Omega_i$ are disjoint, and the constraints that
$W_i \i \Omega$ and  $u_i =0$ almost everywhere on $\R^n\sm W_i$  just mean that 
$\Omega_i \i W_i$, modulo a set of vanishing measure.
After the next section, and under additional regularity assumptions on 
$\Omega$, we will also know that (if we set $\u = 0$ on $\R^n \sm \Omega$)
$\widetilde \u$ is also continuous across $\d \Omega$, and we will feel free to replace
$\u$ with $\widetilde \u$ without saying.
\end{rem}

\section{H\"older-continuity of $u$ on the boundary}  \label{boundary} 

In this section we keep the assumptions of Section \ref{holder},
add a smoothness assumption on $\Omega$, and prove that
$\u$ is H\"older-continuous on the whole $\R^n$ when 
$(\u,\W)$ is a minimizer for $J$.
For the main statement, let us be brutal and just assume that
\begin{equation}\label{e8.1}
\Omega \text { is a bounded open set with $C^1$ boundary.}
\end{equation}
But we shall see that the result holds under somewhat weaker
assumptions; see Remark \ref{r8.3} at the end of the section.

\begin{thm}\label{t8.1} 
Assume that \eqref {e5.1}-\eqref {e5.4}, \eqref {e7.1}, and \eqref {e8.1} hold. 
There is an exponent $\alpha > 0$, that depends only on $n$, $N$,
$\beta$ and $p$ (from \eqref {e5.4}), 
such that if $(\u,\W)$ is a minimizer for $J$ in $\F$
(see Definition~\ref{d1.1} and \eqref {e1.3}-\eqref {e1.5})
there is a constant $C \geq 0$ such that 
(possibly after modifying $\u$ on a set of measure zero)
\begin{equation}\label{e8.2}
|\u(x)-\u(y)| \leq C |x-y|^\alpha
\ \text{ for $x,y \in \R^n$ such that } |x-y|\leq 1.
\end{equation}
\end{thm}

\ms
As for Theorem \ref{t7.1}, we even get that
$C = C_0+ C_0 r_0^{1-\alpha} \big\{\fint_{B(x,r_0)} |\nabla \u|^2\big\}^{1/2}$,
where $C_0$ depends only on $n$, $N$, $\beta$, $p$, and $|\Omega|$.

Again the main difficulty for the proof will be to find $\delta  > n-2$ such that
\begin{equation}\label{e8.3}
\int_{B(x,r)} |\nabla \u|^2 \leq C_{0} r^\delta 
\end{equation}
for $x \in \R^n$ and $r>0$; the conclusion, with
$\alpha = {1\over 2}(\delta -n +2)$ will follow by the
same argument as in Section \ref{holder}, near \eqref{e7.41}.

We would still be happy to prove that for $x\in \R^n$,
$E(r) = \int_{B(x,r)} |\nabla u|^2$ often satisfies a differential inequality,
and there will be one main new case, when $x \in \partial \Omega$.
We cannot repeat the argument of Section \ref{holder} as it is,
because we want to make sure that the function $\u^\ast$
that we build still vanishes on $\Omega$. There is a special case
where we can still use our second competitor, the one where
we used a homogeneous extension of $\u$,
and this is when $\Omega \cap B(x,r)$ is a cone centered at $x$.
In the next proposition, which is the main ingredient for
Theorem~\ref{t8.1}, we will assume that $\Omega$ looks like
a cone near the origin, and then we shall get some decay for $E(r)$.

Again let us center our balls at the origin. Let $r_0 \in (0,1]$ be given.
We shall assume that 
\begin{equation}\label{e8.4}
0 \in \d \Omega
\end{equation}
and that there is a (measurable) cone $\Gamma$, centered at the origin,
such that
\begin{equation}\label{e8.5}
\sigma(S_1 \sm \Gamma) \geq \varepsilon \sigma(S_1),
\end{equation}
as well as a mapping $\Phi : B(0,2r_0) \to \R^n$, 
which is $(1+\eta)$-bilipschitz in the sense that
\begin{equation}\label{e8.6}
(1+\eta)^{-1} |x-y| \leq |\Phi(x)-\Phi(y)| \leq (1+\eta) |x-y|
\ \text{ for } x, y \in B(0,2r_0),
\end{equation}
and for which
\begin{equation}\label{e8.7}
\Phi(0) = 0,
\end{equation}
\begin{equation}\label{e8.8}
\Phi(B(0,2r_0)\cap \Omega) \i \Gamma,
\end{equation}
\begin{equation}\label{e8.9}
\Phi(B(0,2r_0)\sm \Omega) \i \R^n \sm\Gamma,
\end{equation}
and
\begin{equation}\label{e8.10}
\Phi(B(0,2r_0)) \supset B(0,3r_0/2).
\end{equation}
Here $\eta > 0$ is a small constant, that will be chosen in terms of
$\varepsilon$, and then \eqref{e8.10} is quite probably a consequence
of the other, but we are too lazy to prove this.

This will be our main additional assumption of approximation by a good cone.
When $\Omega$ is a $C^1$ domain, as in the statement of Theorem \ref{t8.1},
and $0 \in \d \Omega$, this property holds for $r_0$ small enough, 
and we can even take for $\Gamma$ an open half space. 
In addition, by compactness of $\d \Omega$,
the same property holds with the origin replaced by any point
$x\in \d \Omega$, and for $0 < r_0 \leq R$, where $R$ does not depend
on $x$.  We now state the main decay estimate.

\begin{pro}\label{p8.2} 
For each $\varepsilon > 0$, we can find $\eta \in (0,1/2)$ and $\delta > n-2$,
that depend only on $n$, $N$, $\beta$, $p$,
and $\varepsilon$, such that if $(\u,\W)$ satisfies the assumptions
of Theorem~\ref{t8.1}, and in addition \eqref{e8.4}-\eqref{e8.10} hold
for some choice of $r_0 \leq 1$, $\Gamma$ and $\Phi$, then
\begin{equation}\label{e8.11}
\int_{B(0,r)} |\nabla \u|^2 \leq  2 (r/r_0)^\delta \int_{B(0,2r_0)} |\nabla \u|^2
+ C_1 r^{n-{n \over p}} + C_1 r^{\beta n}
\end{equation}
for $0 < r < r_0$. Here $C_1$ depends on the various constants in the
assumptions, but not on $r_0$. 
\end{pro}

\ms
This will be our analogue of \eqref{e7.38}. For the proof we intend to 
conjugate by $\Phi$ to simplify the geometry, and then copy the
the proof of \eqref{e7.30}-\eqref{e7.38} in Section \ref{holder}.

Observe that because of \eqref{e8.10} and \eqref{e8.6}, we can define an 
inverse mapping $\psi = \Phi^{-1} : B(0,3r_0/2) \to B(0,2r_0)$. 
This allows us to define $\v$ on $B(0,3r_0/2)$ by 
\begin{equation}\label{e8.12}
\v(y) = \u(\psi(y)) \ \text{ for } y\in B(0,3r_0/2).
\end{equation}
Notice that 
\begin{equation}\label{e8.13}
v \in W^{1,2}(B(0,3r_0/2)
\end{equation}
because $u\in W^{1,2}(\R^n)$ and $\Phi$ is bilipschitz.
See for instance \cite{Z}.
Because of the geometry, it will be preferable to work with the 
function $\v$, and prove appropriate differential inequalities on the energy
\begin{equation}\label{e8.14}
E(r) =  \int_{B(0,r)} |\nabla \v|^2
= \sum_i \int_{B(0,r)} |\nabla v_i|^2.
\end{equation}

So let $(\u,\W)$ and $r_0$ satisfy the assumptions of the 
proposition, and let $r \leq r_0$ be such that the restriction of $\v$
to $S_r$ lies in $W^{1,2}(S_r)$ (as in \eqref{e6.8}), with derivatives 
that can be computed from the restriction of $D\v$, and that
\begin{equation}\label{e8.15}
v_i(x) = 0 \ \text{ for $\sigma$-almost every }
x \in S_r \sm \Phi(W_i),
\end{equation}
as in \eqref{e6.16}. These properties hold for the same reason as before 
(and by \eqref{e8.13}). Let us also assume that
\begin{equation}\label{e8.16}
E'(r) = \int_{S_r} |\nabla \v|^2
\end{equation}
(which again holds a.e. as in \eqref{e7.5}). We may now define
$\v^\ast$ by
\begin{equation}\label{e8.17}
\v^{\ast}(z) = \v(z)
\ \text{ for  } z\in B(0,3r_0/2)\sm B(0,r),
\end{equation}
and 
\begin{equation}\label{e8.18}
\v^{\ast}(tz) = t^\gamma \v(z)
\ \text{ for  $z\in S_r$ and } 0 \leq t  < 1,
\end{equation}
where the small $\gamma$ will be chosen later, depending on $\varepsilon$.
And then we set
\begin{equation}\label{e8.19}
\u^{\ast}(x) =  \u(x)
\ \text{ for } x\in \R^n \sm B(0,5r/4),
\end{equation}
and 
\begin{equation}\label{e8.20}
\u^{\ast}(x) =  \v^\ast(\Phi(x))
\ \text{ for } x\in B(0,4r/3),
\end{equation}
which is defined because $\Phi(x) \in B(0,3r/2) \i B(0,3r_0/2)$
(by \eqref{e8.6} and \eqref{e8.7});
the two definitions coincide when $x\in B(0,4r/3) \sm B(0,5r/4)$,
because $\Phi(x) \in B(0,3r/2) \sm B(0,r)$, and by \eqref{e8.17}).

We deduce from \eqref{e8.13} that $\v^\ast \in W^{1,2}(B(0,3r_0/2))$
(as we did near \eqref{e7.22}-\eqref{e7.23}), and then 
$\u^\ast \in W^{1,2}(\R^n)$ (because we have a whole gluing region
$B(0,4r/3) \sm B(0,5r/4)$). As always, $u_i^\ast \geq 0$ everywhere when
$u_i \geq 0$ everywhere.  We now need to define sets $W_i^\ast$ such that 
\begin{equation}\label{e8.21}
(\u^\ast,\W^\ast) \in \F.
\end{equation}
We keep 
\begin{equation}\label{e8.22}
W_i^\ast \sm \psi(B(0,r)) = W_i \sm \psi(B(0,r))
\end{equation}
(where $\psi = \Phi^{-1}$ as before) and set
\begin{equation}
\label{e8.23}
W_i^\ast \cap \psi(B(0,r)) = \psi(H_i),
\ \text { with }
H_i = \big\{ ty \, ; \, y \in S_r \cap \Phi(W_i) \text{ and } 0 < t < 1 \big\}. 
\end{equation}
The $W_i^\ast$ are disjoint, because the $W_i$ are disjoint
and $\psi : B(0,r) \to \psi(B(0,r))$ is injective. Next let us check
that $W_i^\ast \i \Omega$. Pick $x\in W_i^\ast$. If $x\in 
W_i^\ast \sm \psi(B(0,r))$, then $x\in W_i \i \Omega$ by
\eqref{e8.22} and the definition of $\F$. Otherwise, 
$x \in \psi(H_i)$, so there exist $y \in S_r \cap \Phi(W_i)$
and $0 < t < 1$ such that $x= \psi(ty)$. But $y = \Phi(z)$ for some 
$z\in W_i$, $z\in B(0,3r/2)$ by \eqref{e8.6} and \eqref{e8.7},
$y\in \Gamma$ because $z\in W_i \i \Omega$ and by \eqref{e8.8},
$ty \in \Gamma \cap B(0,r)$ because $\Gamma$ is a cone, and finally
$x=\psi(ty) \in \Omega$ by \eqref{e8.9}. So $W_i^\ast \i \Omega$.

Finally, we claim that $u_i^\ast(x) = 0$ for almost every $x\in \R^n \sm W_i^\ast$.
Start when $x\in \psi(B(0,r))$. Write $x = \psi(z)$, with $z = \Phi(x) \in B(0,r)$.
Further write $z=ty$, with $y\in S_r$ and $t < 1$; then $y\in S_r \sm \Phi(W_i)$,
because otherwise $z\in H_i$ and $x\in W_i^\ast$. 
If $v_i(y) \neq 0$, then $y$ lies in the $\sigma$-negligible set 
from \eqref{e8.15} (we just saw that $y\in S_r \sm \Phi(W_i)$);
then $z=ty$ lies in a negligible set too, and so does $x = \psi(z)$.
Hence $v_i(y) = 0$ for almost every $x$, and so 
$v_i^\ast(z) = 0$ by \eqref{e8.18}
and $u_i^\ast(x) = v_i^\ast(\Phi(x))=v_i^\ast(z) =0$ by \eqref{e8.20}.

We are left with the case when $x\in \R^n \sm\psi(B(0,r))$, and then
$x\in \R^n \sm W_i$ by \eqref{e8.22}. But we claim that
\begin{equation}\label{e8.24}
\u^\ast(x) = \u(x) \ \text{ for } x\in \R^n \sm\psi(B(0,r)).
\end{equation}
Indeed, if $x\in B(0,5r/4)$,
$u_i^\ast(x) = v_i^\ast(\Phi(x)) = v_i(\Phi(x)) = u_i(x)$
by \eqref{e8.20}, \eqref{e8.17}, and \eqref{e8.12}. Otherwise,
$x\in \R^n \sm B(0,5r/4)$ and $u_i^\ast(x) = u_i(x)$ directly by 
\eqref{e8.19}. This completes our proof of \eqref{e8.24},
and we deduce from \eqref{e8.24} that $u_i^\ast(x) = u_i(x)= 0$
for almost every $x\in [\R^n \sm W_i^\ast] \cap [\R^n \sm\psi(B(0,r))]$.
In turn \eqref{e8.21} follows.

\ms
We now need to estimate various terms. Since all our functions are bounded
and $\u^\ast = \u$ on $\R^n \sm B(0,5r/4)$, we have the same estimates on
the $F$ and $M$ terms as in \eqref{e7.28} and \eqref{e7.29}. We still
have that $J(\u,\W) \leq J(\u^\ast,\W^\ast)$ and hence, as in
\eqref{e7.30}) 
\begin{eqnarray}\label{e8.25}
\int_{B(0,5r/4)} |\nabla \u|^2 &\leq& \int_{B(0,5r/4)} |\nabla \u^\ast|^2
+ |M(\u^\ast) - M(\u)| +  |F(\W^\ast) - F(\W)|
\nonumber
\\
&\leq& \int_{B(0,5r/4)} |\nabla \u^\ast|^2 + C r^{n-{n \over p}} + C r^{\beta n}.
\end{eqnarray}
By \eqref{e8.24}, the energy contributions of  $B(0,5r/4) \sm \psi(B(0,r))$
cancel and we get that
\begin{equation}\label{e8.26}
\int_{\psi(B(0,r))} |\nabla \u|^2 
\leq \int_{\psi(B(0,r))} |\nabla \u^\ast|^2 + C r^{n-{n \over p}} + C r^{\beta n}.
\end{equation}
We now change use our bilipschitz mapping to change variables. By \eqref{e8.20}
and \eqref{e8.6},
\begin{equation}\label{e8.27}
\int_{\psi(B(0,r))} |\nabla \u^\ast|^2 
= \int_{\psi(B(0,r))} |\nabla (\v^\ast \circ \Phi)|^2 
\leq (1+\eta)^{n+2} \int_{B(0,r)} |\nabla \v^\ast|^2.
\end{equation}
Then the proof of \eqref{e7.23} (or \eqref{e7.18}) yields
\begin{eqnarray}\label{e8.28}
\int_{B(0,r)} |\nabla v^\ast_{i}|^2 
&=&  {r \over (n+2\gamma-2)} \int_{S_{r}} \gamma^2 r^{-2}v_i^2 + |\nabla_{t} v_i|^2.
\end{eqnarray}
We claim that $v_i = 0$ almost everywhere on $S_r \sm \Gamma$.
Indeed, if $y\in S_r \sm \Gamma$ and $x= \psi(y)$, then
$x\in B(0,2r) \sm \Omega$ by \eqref{e8.8}, hence $x$ lies out of $W_i$
and $y = \Phi(x)$ lies out of $\Phi(W_i)$. Almost always, $v_i(y)=0$,
by \eqref{e8.15}. This,  \eqref{e8.5}, and our assumption that $\v \in W^{1,2}(S_r)$)
allows us to apply (4.6) and get 
\begin{equation}\label{e8.29}
\int_{S_{r}} |v_{i}|^2 = \int_{S_r \cap \Phi(W_i)} |v_{i}|^2 
\leq C r^2 \varepsilon^{-1} \int_{S_{r}} |\nabla_{t} v_{i}|^2
\end{equation}
as in \eqref{e7.24}. We return to \eqref{e8.28},
sum over $i$ (the pieces are still orthogonal because
of disjoint supports), use \eqref{e8.29},
and obtain as in \eqref{e7.27}, and with the same $\lambda$ 
as in \eqref{e7.25}, that
\begin{equation}\label{e8.30}
\int_{B(0,r)} |\nabla \v^\ast|^2
= \sum_i \int_{B(0,r)} |\nabla v^\ast_i|^2
\leq \lambda r \sum_i  \int_{S_{r}} |\nabla_{t} v_i|^2
= \lambda r  \int_{S_{r}} |\nabla_{t} \v|^2.
\end{equation}
We complete the estimate with a change of variable in the other direction:
\begin{eqnarray}\label{e8.31}
E(r) &=& \int_{B(0,r)} |\nabla \v|^2 = \int_{B(0,r)} |\nabla (\u \circ \psi)|^2
= (1+\eta)^{n+2} \int_{\psi(B(0,r))} |\nabla \u|^2
\nonumber
\\
&\leq& (1+\eta)^{n+2} \int_{\psi(B(0,r))} |\nabla \u^\ast|^2 + C r^{n-{n \over p}} + C r^{\beta n}
\nonumber
\\
&\leq& (1+\eta)^{2n+4} \int_{B(0,r)} |\nabla \v^\ast|^2 + C r^{n-{n \over p}} + C r^{\beta n}
\\
&\leq& (1+\eta)^{2n+4} \lambda r  \int_{S_{r}} |\nabla_{t} \v|^2
+ C r^{n-{n \over p}} + C r^{\beta n}
\nonumber
\\
&\leq& (1+\eta)^{2n+4} \lambda r  E'(r)
+ C r^{n-{n \over p}} + C r^{\beta n}
\nonumber
\end{eqnarray}
by \eqref{e8.12}, \eqref{e8.6}, \eqref{e8.26}, by
\eqref{e8.27}, \eqref{e8.30}, and \eqref{e8.16}.
This is our analogue of \eqref{e7.30}, with the only difference that
we have the extra term $(1+\eta)^{2n+4}$. Also, we do not need to care
about \eqref{e7.21} (there is no first case). Anyway, $(n-2)\lambda < 1$
if $\gamma$ is small enough, depending on $\varepsilon$
(see \eqref{e7.32}), so we can choose $\eta$ so small that
$(n-2)\lambda (1+\eta)^{2n+4}< 1$, and then choose $\delta$
such that 
\begin{equation}\label{e8.32}
n-2 < \delta < \min\Big(n\beta, n-{n\over p}, (1+\eta)^{-2n-4}\lambda^{-1}\Big)
\end{equation}
(see \eqref{e7.31}), and \eqref{e8.31}) becomes
\begin{equation}\label{e8.33}
\delta E(r) \leq r E'(r) +  C r^{n-{n \over p}} + C r^{\beta n}.
\end{equation}
This is the same as \eqref{e7.33}. We integrate this as we did before and get 
the analogue of \eqref{e7.38}:
\begin{equation}\label{e8.34}
\int_{B(x,r)} |\nabla \v|^2 
\leq (r/r_0)^\delta E(r_{0}) + C [r^{n-{n \over p}} + r^{\beta n}].
\end{equation}
We complete this by a last change of variable: we are interested in
\begin{eqnarray}\label{e8.35}
\int_{x\in B(x,r)} |\nabla \u(x)|^2 
&\leq& (1+\eta)^{n+2}\int_{y\in \Phi(B(x,r))} |\nabla \v(y)|^2 
\nonumber
\\
&\leq& (1+\eta)^{n+2}\int_{y\in B(0,(1+\eta r))} |\nabla \v(y)|^2 
\\
&\leq& (1+\eta)^{n+3} (r/r_0)^\delta E(r_{0}) + 2C [r^{n-{n \over p}} + r^{\beta n}]
\nonumber
\end{eqnarray}
by \eqref{e8.12}, \eqref{e8.6}, \eqref{e8.7}, if $(1+\eta)r \leq r_0$,
and by \eqref{e8.34}. Since 
$E(r_{0}) \leq (1+\eta)^{n+2} \int_{B(0,2r_0)} |\nabla \u|^2$
by the usual change of variable, \eqref{e8.35} implies \eqref{e8.11}
when $(1+\eta)r \leq r_0$. The other case is trivial (recall that $\delta$
is very small). This completes our proof of Proposition \ref{p8.2}.
\qed

\ms
We are now ready to prove Theorem \ref{t8.1}. Let $\Omega$ satisfy 
\eqref{e8.1}. That is, $\d \Omega$ is a compact $C^1$ embedded submanifold
if codimension $1$, and we can assume that $\Omega$ is locally on one side of 
$\d \Omega$; otherwise, remove the set of points of $\d \Omega$ which
have $\Omega$ on both sides (this set is open and closed in $\d \Omega$), 
without changing the problem.
By compactness, we can find $r_0 \in (0,1]$ such that, for each $x \in \d \Omega$,
the set $\Omega_x = \Omega-x$ satisfies the geometrical assumptions 
\eqref{e8.4}-\eqref{e8.10} of Proposition \ref{p8.2}. We can even take for
$\Gamma$ a half space (and so $\varepsilon = 1/2$).

Now let $(\u,\W)$ be a minimizer, as in the statement of Theorem \ref{t8.1},
and let $x\in \d \Omega$ be given. We can apply Proposition \ref{p8.2}, with
the value of $r_0$ that we just found, to a translation by $-x$ of
$\Omega$, $(\u,\W)$, and the data $f_i$ and $g_i$. We get that
\eqref{e8.11} holds, so
\begin{equation}\label{e8.36}
\int_{B(x,r)} |\nabla \u|^2 
\leq 2 (r/r_0)^\delta \int_{B(x,2r_0)} |\nabla \u|^2
+ C_1 r^{n-{n \over p}} + C_1 r^{\beta n}
\end{equation}
for $0 < r \leq r_0$. The constant $C_1$ depends only on the various constants that show
up in the assumptions. 

This shall take care of balls $B(x,r)$ centered on $\d \Omega$. Now consider
$x\in \R^n \sm \d\Omega$, set $d(x) = \dist(x,\d \Omega)$, and 
choose $y \in \d \Omega$ such that $|x-y| = d(x)$.
 
Let $r > 0$ be given, with $r \leq r_0/3$. We shall need to discuss cases.
If $r \geq d(x)/2$, just observe that
\begin{eqnarray}\label{e8.37}
\int_{B(x,r)} |\nabla \u|^2 &\leq& \int_{B(y,r+d(x))} |\nabla \u|^2
\leq \int_{B(y,3r)} |\nabla \u|^2 
\nonumber
\\
&\leq& 
2 (3r/r_0)^\delta \int_{B(x,2r_0)} |\nabla \u|^2
+ 3^n C_1 r^{n-{n \over p}} + 3^n C_1 r^{\beta n}
\nonumber
\\
&\leq& 
C (r/r_0)^\delta \int_{B(x,2r_0)} |\nabla \u|^2
+ C [r^{n-{n \over p}} +  r^{\beta n}]
\end{eqnarray}
by \eqref{e8.36}, and where we do not need to keep track of the 
dependence in $C_1$.
So we may assume that $r\leq d(x)/2$, and 
the only interesting case is when $x\in \Omega$,
because otherwise $B(x,r) \i \R^n \sm \Omega$ and 
$\int_{B(x,r)} |\nabla \u|^2 = 0$.

Let us first assume that $d(x) \leq r_0/3$. The proof of \eqref{e8.37}, 
with $r=d(x)$, yields
\begin{equation}
\label{e8.38}
\int_{B(x,d(x))} |\nabla \u|^2 \leq   C (d(x)/r_0)^\delta \int_{B(x,2r_0)} |\nabla \u|^2
+ C d(x)^{n-{n \over p}} + C d(x)^{\beta n}.
\end{equation}
Let us use the proof of Theorem \ref{t7.1}, applied as usual after translating 
everything by $-x$, and with $r_0 = d(x)$, so that $B(x,r_0) = B(x,d(x))  \i \Omega$.
In fact, we are only interested by \eqref{e7.38}, which implies that
\begin{eqnarray}
\label{e8.39}
\int_{B(x,r)} |\nabla \u|^2 
&\leq&   C (r/d(x))^\delta \int_{B(x,d(x))} |\nabla \u|^2 + C [r^{n-{n \over p}} +  r^{\beta n}]
\nonumber
\\
&\leq& C (r/r_0)^\delta \int_{B(x,2r_0)} |\nabla \u|^2
+ C (r/d(x))^\delta [d(x)^{n-{n \over p}} + d(x)^{\beta n}]
+ C [r^{n-{n \over p}} +  r^{\beta n}]
\nonumber
\\
&\leq& C (r/r_0)^\delta \int_{B(x,2r_0)} |\nabla \u|^2
+ C r^\delta
\end{eqnarray}
by \eqref{e8.38} and because $\delta \leq \min(n-{n \over p},\beta n)$
and $d(x) \leq r_0/3 \leq 1/3$.

In the last case when $d(x) > r_0/3$, we also use the proof of Theorem \ref{t7.1},
but with the radius $r_0/3$ (which is all right because $B(x,r_0/3) \i \Omega$), 
and deduce directly from \eqref{e7.38} that
\begin{eqnarray}\label{e8.40}
\int_{B(x,r)} |\nabla \u|^2 
&\leq&   C (3r/r_0)^\delta \int_{B(x,2r_0/3)} |\nabla \u|^2 
+ C [r^{n-{n \over p}} +  r^{\beta n}]
\nonumber
\\
&\leq& C (r/r_0)^\delta \int_{B(x,2r_0)} |\nabla \u|^2 
+ C [r^{n-{n \over p}} +  r^{\beta n}]
\end{eqnarray}
because $r \leq r_0/3$ by assumption.

Thus in all the cases we get the same conclusion as in \eqref{e8.39}, 
or better, which now holds for all balls of radius $r \leq r_0/3$.
This is not exactly as good as in \eqref{e7.38}, because the error term
$C r^\delta$ is a little larger. In fact if we want to really get \eqref{e8.2}
later, let us observe that we can get that
\begin{equation}\label{e8.41}
\int_{B(x,r)} |\nabla \u|^2 
\leq  C (r/r_0)^\delta \int_{B(x,2r_0)} |\nabla \u|^2
+ C r^{\delta'}
\end{equation}
for some $\delta' > \delta$ (we just use a constant a little
larger than $\delta$ in the estimates above that lead to \eqref{e8.39}).
The conclusion of Theorem \ref{t8.1}, namely \eqref{e8.2}, 
now follows by the same proof as for \eqref{e7.2}; see \eqref{e7.41}-\eqref{e7.50}.
\qed

\begin{rem}\label{r8.3} 
The regularity assumption \eqref{e8.1} that we put in
Theorem \ref{t8.1} is really far from optimal.
First, we just use the boundedness of $\Omega$ to have some uniformity in 
the approximation by cones. But more importantly, 
we do not need to have a good approximation by cones at all scales $r \leq r_0$. 
It would be more than enough, for instance,
if for some choice of $R > 0$ and $C \geq 0$ and all $x \in \d \Omega$, the set
$\Omega-x$ had the approximation condition \eqref{e8.4}-\eqref{e8.10} for all $r_0 \leq R$, 
except perhaps for $r_0$ in an exceptional set $Z(x)$ such that $\int_{Z(x)} dr/r \leq C$.
But many other conditions would probably work as well.
\end{rem}

\section{The monotonicity formula}  \label{mono} 

The H\"older-continuity of $\u$ (Theorems \ref {t7.1} and \ref{t8.1})
will allow us to apply a near monotonicity result of \cite{CJK} that will be very useful,
in particular for proving that $\u$ is Lipschitz inside $\Omega$ and controlling 
blow-up limits.

We shall need slightly stronger assumptions on the data. We still assume that
$|\Omega| < +\infty$ (as in \eqref{e3.1}), that $F$ is H\"older-continuous 
with exponent $\beta > {n-2\over n}$, as in \eqref{e7.1}, and that
\begin{equation}\label{e9.1}
f_i \in L^p(\Omega), \, \text{ for some $p > {n \over 2} \ $
and $f_i(x) \geq 0$ almost everywhere on } \Omega,
\end{equation}
(see \eqref{e5.2}-\eqref{e5.3}), but this time we shall also require 
\begin{equation}\label{e9.2}
g_i \in L^\infty(\Omega)
\end{equation}
(and not just $L^p$ for some $p> n/2$); 
this is probably not optimal, 
but we should probably at least require $p > n$. 
See Remark \ref{r9.2} 
When we consider balls that meet $\d \Omega$, we shall also assume 
$\Omega$ to be a bounded open set with a $C^1$ boundary, as in \eqref{e8.1}. 

A consequence of these assumptions is that we can change $\u$ on a set
of zero measure to make it H\"older-continuous. We shall always assume that 
this modification has been done, which will allow us to talk about the open sets
$\Omega_i = \big\{ x\in \Omega \, ; \, u_i(x) > 0 \big\}$. See Remark \ref{r7.2}.

Fix $x_0 \in \R^n$,  two indices $i_1, i_2 \in [1,N]$
and two signs $\varepsilon_1, \varepsilon_2 \in \{ -1, +1 \}$. Then
define functions $v_1$ and $v_2$ by
\begin{equation}\label{e9.3}
v_j(x) = [\varepsilon_j u_{i_j}(x)]_+ = \max(0,\varepsilon_j u_{i_j}(x)) \in [0,+\infty)
\end{equation}
for $j=1,2$ and $x\in \R^n$. We always take different pairs $(i_j,\varepsilon_j)$,
so typical choices of the two $v_j$ would be $v_1 = (u_1)_+$ and $v_2 = (u_1)_-$,
or $v_1 = (u_1)_+$ and $v_2 = (u_2)_+$. Our complicated notation is designed to accommodate
both cases. Finally set
\begin{equation}\label{e9.4}
\Phi_j(r) = {1 \over r^2} \int_{B(x_0,r)} {|\nabla v_j|^2 \over |x-x_0|^{n-2}} \, dx
\ \text { for $j=1,2\ $ and } \,  \Phi(r) = \Phi_1(r)\Phi_2(r)
\end{equation}
for $r > 0$. This is the function which will be nearly monotone. As we shall see later,
the integrals often converge because of \eqref{e7.38} or \eqref{e8.11}.

\begin{thm}\label{t9.1} 
Assume that \eqref{e3.1}, \eqref{e7.1}, \eqref{e9.1}, and \eqref{e9.2} hold.
Let $(\u,\W)$ be a minimizer of the functional $J$, and let $x_0$ and $r_0 > 0$
be given. If $B(x_0,r_0)$ is not contained in $\Omega$, also assume $\eqref{e8.1}$.
Then for all choices of $(i_1,\varepsilon_1) \neq (i_2,\varepsilon_2)$ as above,
and $0 < r  \leq r_0$,
\begin{equation}\label{e9.5}
\Phi(r) \leq C \Big(r_0^2 ||g_{i_1}||^2_\infty
+r_0^2 ||g_{i_2}||^2_\infty+ \Phi_1(r_0) + \Phi_2(r_0) \Big)^2,
\end{equation}
with a constant $C$ that depends only on $n$.
\end{thm}

\begin{proof}
Our proof will mostly consist in checking that $v_1$ and $v_2$ satisfy the assumptions
of Theorem 1.3 in \cite{CJK}, which is perfectly fit for our situation. This result is in the
same spirit as in the initial monotonicity formula in \cite{ACF}. 
It looks a little less nice because \eqref{e9.5} is less precise than saying that
$\Phi$ is nondecreasing, but this allows more general situations (as here), and will
give almost as good consequences.

The first assumption of \cite{CJK}, that the $v_j$ be continuous, 
is a consequence of Theorems \ref {t7.1} or \ref{t8.1}, and this is why 
we include \eqref{e3.1}, $L^p$ bounds on the $f_i$,  
the H\"older assumption \eqref{e7.1}, and sometimes \eqref{e8.1},
which will not show up in the estimates. They also satisfy the exclusion relation
$v_1v_2=0$, just because $(\u,\W) \in \F$ (and $\u$ is continuous).

Next we want to show that for each $i$, $u_i$ satisfies the equation
\begin{equation}\label{e9.6}
\Delta u_i = f_i u_{i} -  {1 \over 2} g_{i} 
\end{equation}
in the open set $\Omega_i = \big\{ x\in \R^n \, ; u_i(x) \neq 0 \big\} \i \Omega$. 
Here we restrict to $\Omega_i$ because in other places we may not
modify $u_i$
freely as we do in the proof below.
Otherwise we proceed in the most usual way. For each test
function $\varphi$ with compact support in $\Omega_i$, 
we observe that if we replace $u_i$ with $u_i + t \varphi$, $t\in \R$ small,
and otherwise change nothing, we get a new competitor $(\u_t,\W)$.
Thus $J(\u_t,\W) \geq J(\u,\W)$ for $t$ small. But $J(\u_t,\W)$ has a derivative
at $t=0$, which is
\begin{equation}\label{e9.7}
{\d J(\u_t,\W) \over \d t}(0) = 2 \int \langle \nabla u_i, \nabla\varphi\rangle
+ 2 \int f_i u_i \varphi - \int  g_i \varphi
\end{equation}
(see \eqref{e1.5} and recall that only $u_i$ changes). 
This derivative vanishes, so by definition of the distribution $\Delta u_i$, 
\begin{equation}\label{e9.8}
0 = 2 \int \langle \nabla u_i, \nabla\varphi\rangle + 2 \int f_i u_i \varphi - \int  g_i \varphi
= \langle - 2\Delta u_i + 2 f_i u_i - g_i,\varphi \rangle.
\end{equation}
This holds for every test function $\varphi$, and this gives \eqref{e9.6}.
As an immediate consequence of \eqref{e9.6} and the definitions,
\begin{equation}\label{e9.9}
\Delta v_j = \varepsilon_j \Delta u_{i_j} = 
\varepsilon_j f_i u_{i_j} - {1 \over 2}\varepsilon_j g_{i_j} 
\geq - {1 \over 2} ||g_{i_j}||_\infty
\end{equation}
in the sense of distributions, in the open set
$\Omega(j) = \big\{ x\in \R^n \, ; \, v_j(x) > 0 \big\} \i \Omega$.

We want to take advantage of the normalization in \cite{CJK}, 
so we don't apply the result directly to the $v_j$, but to 
$w_j(x) = \lambda_j v_j(x_0 + r_0 x)$, which are defined on the unit ball
and such that $\Delta w_j(x) = \lambda_j r_0^2 \Delta v_j(x_0 + r_0 x)
\geq - 1$ if $\lambda_j r_0^2 ||g_{i_j}||_\infty \leq 2$. 
A brutal, but acceptable choice will be to take
\begin{equation}\label{e9.10}
\lambda_1 = \lambda_2 = r_0^{-2} (\tau+||g_{i_1}||_\infty+||g_{i_2}||_\infty)^{-1},
\end{equation} 
with a very small $\tau > 0$ that will tend to $0$ soon.

Notice that the $w_j$ satisfy all the assumptions of Theorem 1.3 in
\cite{CJK}, in particular because Remark 1.4 in \cite{CJK} says that
since $w_j$ is nonnegative and continuous and $\Delta w_j \geq -1$ on $\{ w_j > 0 \}$,
we get that $\Delta w_j \geq -1$ (as a distribution and on the whole $\R^n$).
Set
\begin{equation}\label{e9.11}
\wt\Phi_j(\rho) = {1 \over \rho^2} \int_{B(x_0,\rho)} {|\nabla w_j|^2 \over |x|^{n-2}} \, dx
\ \text { and } \,  \wt \Phi(\rho) = \Phi_1(\rho)\Phi_2(\rho)
\end{equation}
for $0 < \rho \leq 1$; then by \cite{CJK}
\begin{equation}\label{e9.12}
\wt\Phi(\rho) \leq C \big(1+ \wt\Phi_1(1) + \wt\Phi_2(1)\big)^2.
\end{equation}
Also, a change of variable yields
$\wt\Phi_j(\rho) = \lambda_j^2 r_0^2 \Phi_j(r_0 \rho)$ for
$0 < r \leq 1$, and now
\begin{eqnarray}\label{e9.13}
\Phi(r) &=& r_0^{-4} \lambda_1^{-2}\lambda_2^{-2}\wt\Phi(r/r_0) 
\leq C r_0^{-4} \lambda_1^{-2}\lambda_2^{-2}
\big(1+\wt\Phi_1(1) +  \wt\Phi_2(1)\big)^2
\nonumber \\
&\leq&  C r_0^{-4} \lambda_1^{-2}\lambda_2^{-2}
\big(1+\lambda_1^2 r_0^2\Phi_1(r_0) +  \lambda_2^2 r_0^2\Phi_2(r_0)\big)^2
\\
&=& \big(\lambda_1^{-2} r_0^{-2} + \Phi_1(r_0) +  \Phi_2(r_0)\big)^2
\nonumber 
\end{eqnarray}
because $\lambda_1 = \lambda_2$ by \eqref{e9.10}.
Since $\lambda_1^{-2} r_0^{-2}
= r_0^2 (\tau+||g_{i_1}||_\infty+||g_{i_2}||_\infty)^2$
by \eqref{e9.10}, we get that
\begin{equation}\label{e9.14}
\Phi(r) \leq  C \big(r_0^2(\tau+||g_{i_1}||_\infty+||g_{i_2}||_\infty)^2
+ \Phi_1(r_0) + \Phi_2(r_0) \big)^2.
\end{equation}
We now let $\tau$ tend to $0$ and get \eqref{e9.5}; Theorem \ref{t9.1} follows.
\qed
\end{proof}

\begin{rem}\label{r9.2} 
Our previous assumption than $g_i \in L^p$ for some $p > n/2$
is no longer enough. In the simple case when $N=1$ and $f_1=0$,
we get a solution $u$ that satisfies $\Delta u = - {1 \over 2} g_1$
locally (see \eqref{e9.6}), and that looks like $G \ast g_{1}$,
where $G$ is the fundamental solution of $-\Delta$ (as in Section \ref{bounded}).
Then $\nabla u$ looks like $\nabla G \ast g_{1}$ (a Riesz transform
of order $1$). If we want to make sure that $u$ behaves like a Lipschitz function 
(this is what is suggested by the normalization in \eqref{e9.5}), we should 
probably require that $\Delta u = - {1 \over 2} g_1 \in L^p$, where
$p > n$ is larger than the Sobolev exponent. 
\end{rem}

\section{Interior Lipschitz bounds for $u$}  \label{lip} 

In this section we make our assumptions just a bit stronger than before
(we do not want error terms much larger than $r^n$), 
and show that $u$ is locally Lipschitz 
inside $\Omega$ when $(\u,\W)$ is a minimizer for $J$. 
We shall take care of the Lipschitz
regularity near $\d \Omega$ in the next section.

We now assume that for $1 \leq i \leq n$,
\begin{equation}\label{e10.1}
f_i \geq 0 \text{ a.e. on } \Omega , \ f_i \in L^\infty(\Omega), 
\text{ and } g_i \in L^\infty(\Omega),
\end{equation}
and we also require $F$ to be a Lipschitz function of $\W$,
i.e., that
\begin{equation}\label{e10.2}
\Big| F(W_{1},W_{2},\ldots,W_{N}) - F(W'_{1},W'_{2},\ldots,W'_{N}) \Big|
\leq C \sum_{i=1}^N |W_{i} \Delta W'_{i}|
\end{equation}
for some $C \geq 0$ and all choices of  $W_{i}, W'_{i} \i \Omega$, 
$1 \leq i \leq N$. As usual $\Delta$ denotes a symmetric difference.

\begin{thm}\label{t10.1} 
Assume that $|\Omega| < +\infty$ (as in \eqref{e3.1}), 
and that \eqref{e10.1} and \eqref{e10.2} hold.
Let $(\u,\W)$ be a minimizer of the functional $J$, and let $x_0$ and $r_0 \in (0,1]$
be such that $B(x_0,2r_0) \i \Omega$. Then 
\begin{equation}\label{e10.3}
|\u(x)-\u(y)| \leq C_2 \Big(1+ \fint_{B(x_0,2r_0)} |\nabla \u|^2 \Big)^{1/2}
 |x-y|
\ \text{ for } x, y \in B(x_0,r_0),  
\end{equation}
with a constant $C_2$ that depends only on $n$, $N$, $|\Omega|$, and 
the  constants in \eqref{e10.1} and \eqref{e10.1}.
\end{thm}  

\ms
The main ingredient for the proof of Theorem \ref{t10.1} is Theorem \ref{t9.1}, 
which we shall use to say that when $\int_{B(x,r)} |\nabla \u|^2$ is very large,
then one of the $\int_{B(x,r)} |\nabla u_i|^2$ is much larger than the other ones,
which will allow us to use the harmonic competitor described in Section \ref{favorites}.

Before we start with the proof itself, let us describe a small decoupling trick that 
will allow us to simplify our notation. 

\begin{lem}\label{t10.2}
It is enough to prove Theorem \ref{t10.1} when, in the definition of $\F$,
all the functions $u_i$ are required to be nonnegative.
\end{lem}

\begin{proof}
Let $J$ be our initial  functional (for which we want to prove Theorem \ref{t10.1});
we want to construct new functional $\wt J$, defined on a new set $\wt\F$ of
competitors, so that the minimization of $J$ on $\F$ is equivalent to the 
minimization of $\wt J$ on $\wt\F$. Let $I$ denote the set of indices $i$ for which 
$u_i$ is not required to be nonnegative (in $\F$). For each $i\in I$, decouple
$i$ as two indices $i_+$ and $i_-$; for $i\in [1,N] \sm I$, just keep the same 
index $i$. This gives a new set of indices, which we call $I'$.

Define $\wt\F$ as in Section \ref{intro}, but with the new set $I'$ of indices,
and the constraint that all $u_i$, $i\in I'$, are nonnegative. For the $M$-term
of the functional, keep the $f_i$ and $g_i$, $i\in  [1,N] \sm I$, as they were,
and for $i\in I$, set $f_{i,+} = f_{i,-} = f_i$ and $g_{i,+} = g_{i,-} = g_i$.
Also define $\wt F$ by setting $W_i = W_{i,+} \cup W_{i,-}$ for each $i\in I$,
and then substituting $W_i$ in the definition of $F$. 
That is, if $\wt\W$ is indexed by $I'$, we define $\W$ indexed by $[1,N]$ 
by the rule above, and set $\wt F(\wt\W) = F(\W)$.

All this gives a new functional $\wt J$ defined on $\wt\F$. If
$(\u,\W) \in \F$, we define a pair $(\wt\u,\wt\W) \in \wt\F$
in the natural way: we keep $u_i$ as it is when $i\in [1,N] \sm I$,
and when $i\in I$ we set $u_{i,\pm} = \max(0,\pm u_i)$
(the positive and negative part). We keep $W_i$ when $i\in [1,N] \sm I$,
and otherwise we set $W_{i,+} = \big\{ x\in W_i \, ; \, u_i(x) \geq 0 \big\}$
and $W_{i,-} = \big\{ x\in W_i \, ; \, u_i(x) < 0 \big\}$.
Of course we could have sent part of the set $\big\{ x\in W_i \, ; \, u_i(x) = 0 \big\}$
in $W_{i,-}$, but this will not matter. It is easy to see that this gives a pair
$(\wt\u,\wt\W) \in \wt\F$, and that $J(\wt\u,\wt\W) = J(\u,\W)$.

Conversely, given $(\wt\u,\wt\W) \in \wt\F$, we construct a pair
$(\u,\W)$ by setting $W_i = W_{i,+} \cup W_{i,-}$ and
$u_i = u_{i,+}-u_{i,-}$ when $i\in I$, and changing nothing otherwise.
It is easy to see that $(\u,\W) \in \F$ and $J(\wt\u,\wt\W) = J(\u,\W)$.

Now, if we prove Theorem \ref{t10.1} for $\wt J$ in $\wt\F$, it immediately follows for
$J$ on $\F$, as needed for the lemma.
\qed
\end{proof}

\ms
Of course Lemma \ref{t10.2} does not change the nature of our problem, it will just
allow us to simplify our notation. Notice however that things would not have been so
easy if $F$ was required to be a strictly convex function of the volume in each variable, 
since the new function $\wt F$ is not. 

Return to the proof of Theorem \ref{t10.1}. Now assume that all the $u_i$
are required to be nonnegative. We shall try to control quantities like 
\begin{equation}\label{e10.4}
E(r) = \int_{B(0,r)} |\nabla \u|^2 = \sum_i \int_{B(x_1,r)} |\nabla u_i|^2
\end{equation}
and for the interior regularity, we shall concentrate on the case when
$0 < r \leq r_0$ for some $r_0$ such that $B(0,r_0) \i \Omega$.
Let us also define, for $1 \leq i \leq N$ and $0 < r \leq r_0$,
\begin{equation}\label{e10.5}
E_i(r) = \int_{B(0,r)} |\nabla u_i|^2.
\end{equation}
Let us now record what we get when we apply Theorem \ref{t9.1}.

\begin{lem}\label{t10.3} 
Suppose $B(0,r_0) \i \Omega$ and $r_0 \leq 1$. Then
\begin{equation}\label{e10.6}
r^{-2n} E_i(r)  E_j(r) \leq C_3 \Big(1+ \fint_{B(0,r_0)}  |\nabla \u|^2 \Big)^2 
\end{equation}
for $0 < \rho \leq r_0$ and $1 \leq i \neq j \leq N$.
The constant $C_3$ depends only on the usual constants, i.e., 
$n$, $N$, $|\Omega|$, the $||f_{i}||_\infty$ (we would even get away
with bounds on $||f_{i}||_p$ for some $p>n/2$), the $||g_{i}||_\infty$, 
and the Lipschitz constant in \eqref{e10.2}.
\end{lem}

\begin{proof}
Set
\begin{equation} \label{e10.7}
\Phi_i(r) = r^{-2} \int_{B(0,r)} {|\nabla u_i|^2 \over |x|^{n-2}}
\end{equation}
for $1 \leq i \leq N$ and $0 < \rho \leq r_0$, observe that 
\begin{equation} \label{e10.8}
r^{-n} E_i(r)  \leq \Phi_i(r)
\end{equation}
because $|x| \leq r$ in the integral, and that $\Phi_i(r)$ is the same number 
that we called $\Phi_j(r)$ in \eqref{e9.4}, if we take $x_0 = 0$ there and 
$(i_j,\varepsilon_j) = (i,+1)$. Thus Theorem \ref{t9.1} says that
\begin{equation}
\label{e10.9}
r^{-2n} E_i(r)  E_j(r) \leq \Phi_i(r)\Phi_j(r) 
\leq C \Big(r_0^2 ||g_{i_1}||^2_\infty+ r_0^2 ||g_{i_2}||^2_\infty
+ \Phi_i(r_0) + \Phi_j(r_0) \Big)^2
\end{equation}
for $i \neq j$ and $0 < r \leq r_0$. We also need bounds on the right-hand side
of \eqref{e10.9}, and indeed
\begin{eqnarray}
\label{e10.10}
\Phi_i(r_0) &=& {1 \over r_0^2} \int_{B(0,r_0)} {|\nabla u_i|^2 \over |x|^{n-2}} \, dx
\leq \sum_{k \geq 0} {1 \over r_0^2}\int_{B(0,2^{-k}r_0) \sm B(0,2^{-k-1}r_0)} 
(2^{-k-1}r_0)^{2-n} |\nabla u_i|^2 
\nonumber
\\
& \leq & C r_0^{-n} \sum_{k \geq 0} 2^{k(n-2)}\int_{B(0,2^{-k} r_0)} |\nabla \u|^2
\nonumber
\\
& \leq & C r_0^{-n} \sum_{k \geq 0} 2^{k(n-2)} 
\Big\{2^{-k \delta} \int_{B(0,r_0)}  |\nabla \u|^2
+ [(2^{-k}r_0)^{n-{n \over p}} + (2^{-k}r_0)^{\beta n}] \Big\}
\\
& = & C r_0^{-n} \sum_{k \geq 0} 2^{k(n-2)} 
\Big\{2^{-k \delta} \int_{B(0,r_0)}  |\nabla \u|^2 + (2^{-k}r_0)^{n}  \Big\}
\nonumber
\end{eqnarray}
by \eqref{e10.7} and \eqref{e7.38}, and because with our new assumptions
\eqref{e10.1} and \eqref{e10.2}, we now have $p=+\infty$ and $\beta = 1$.
We do not really need this additional information here, but it simplifies the formulas.
Recall from \eqref{e7.31} that $\delta > n-2$. Thus the sum over $k$ converges 
geometrically, and
\begin{equation}
\label{e10.11}
\Phi_i(r_0) \leq C r_0^{-n} \int_{B(0,r_0)}  |\nabla \u|^2 + C.
\end{equation}
Then \eqref{e10.6} follows from \eqref{e10.9} and \eqref{e10.11}.
\qed
\end{proof}

Notice that we can get upper bounds for $\int_{\R^n}  |\nabla \u|^2$
in terms of the usual constants, by \eqref{e3.10}, but when we cannot find
a large ball $B(0,r_0) \i \Omega$, we may need to content ourselves 
with a small $r_0$, and get a large lower bound in \eqref{e10.6}.

We now state the main decay estimate in the proof of Theorem \ref{t10.1},
which concerns the case when $\u(0) = 0$ (the notion makes sense because
$\u$ is H\"older continuous inside $\Omega$).

\begin{lem}\label{t10.4} 
We can find $\tau \in (0,10^{-1})$, that depends on $n$ and $N$, 
and $C_4$, that also depends on $|\Omega|$ and the constants in
\eqref{e10.1} and \eqref{e10.2},  such that
if $\u(0) = 0$, $0 < r_0 \leq 1$ and $B(0,r_0) \i \Omega$,
\begin{equation}\label{e10.12}
\fint_{B(0,\tau \rho)} |\nabla \u|^2 
\leq  C(\tau,r_0) + {1 \over 10}\, \fint_{B(0,\rho)} |\nabla \u|^2  
\end{equation} 
for $0 < \rho \leq r_0$, with 
$C(\tau,r_0) = C_4 \big(1+ \fint_{B(0,r_0)}  |\nabla \u|^2 \big)$.
\end{lem}

\ms
We see this as decay because in the most unpleasant situation when 
$\rho^{-n} \int_{B(0,\rho)} |\nabla \u|^2$ is very large 
\eqref{e10.12} will say that $\fint_{B(0,\tau \rho)} |\nabla \u|^2 \leq
{1\over 2} \fint_{B(0,\rho)} |\nabla \u|^2$.
We keep for later the case of balls that are not centered on the set $\{\u = 0\}$.

\begin{proof}
Let $\rho \leq r_0$ be given. Let $M \geq 0$ be a very large number, 
to be chosen later. Let us first treat the easy case when 
\begin{equation}\label{e10.13}
\fint_{B(0,\rho)} |\nabla \u|^2 \leq M.
\end{equation}
In this case, we just need to say that
$\fint_{B(0,\tau\rho)} |\nabla \u|^2 \leq \tau^{-n}\fint_{B(0,\rho)} |\nabla \u|^2 
\leq \tau^{-n} M$, and so \eqref{e10.12} holds if we choose
\begin{equation}\label{e10.14}
C(\tau,r_0) \geq  \tau^{-n} M.
\end{equation}
So we may now assume that \eqref{e10.13} fails.
Select $i$ so that $E_{i}(\rho)$ is largest; without loss of generality, 
we may assume that $i= 1$. 
Since $\int_{B(0,\rho)} |\nabla \u|^2 = E(\rho) = \sum_{j=1}^{N} E_{j}(\rho)$, we deduce from
the failure of \eqref{e10.13} that for the largest term
\begin{equation}\label{e10.15}
\rho^{-n} E_{1}(\rho) 
\geq N^{-1} \rho^{-n} \int_{B(0,\rho)} |\nabla \u|^2 \geq C^{-1}M
\end{equation}
and hence, by \eqref{e10.6}, that
\begin{equation}\label{e10.16}
\rho^{-n} E_{i}(\rho) \leq \Lambda
\ \text{ for } i > 1,
\end{equation}
with
\begin{equation}\label{e10.17}
\Lambda = C M^{-1} C_3 \Big(1+ \fint_{B(0,r_0)}  |\nabla \u|^2 \Big)^2.
\end{equation}
Next we want to choose a radius $r \in (\rho/2,\rho)$,
with a few good properties that will help us 
define and use the harmonic competitor of Section \ref{favorites}.
First we want the restriction of each $u_i$ to $S_r$ to lie in $W^{1,2}(S_r)$,
with tangential derivatives that can be computed from the restriction of 
$\nabla \u$ to $S_r$. This is easy to arrange, because it is true for almost 
every $r \in (\rho/2,\rho)$ (see the discussion near \eqref{e4.14}). Next,
\begin{equation}\label{e10.18}
u_i(x) = 0 \ \text{ for $1 \leq i \leq N$ and $\sigma$-almost every } x\in S_r \sm W_i
\end{equation}
(as in \eqref{e6.16}), which is also true for almost all $r$. Finally,
we choose $r$ so that 
\begin{equation}\label{e10.19}
\int_{S_r} |\nabla_t u_i|^2 \leq 2N\rho^{-1} \int_{B(0,\rho)} |\nabla u_i|^2
\leq 2N\rho^{-1} E(\rho)
\end{equation}
for $1 \leq i \leq N$. The first inequality is easy to arrange by Chebyshev
(use \eqref{e4.15} with $p=2$), and the second one is trivial
(see \eqref{e10.4}). When $i>1$, we deduce from \eqref{e10.16}
and the first part of \eqref{e10.19} that
\begin{equation}\label{e10.20}
\int_{S_r} |\nabla_t u_i|^2 \leq 2N\rho^{-1} E_i(\rho) 
\leq 2 N \rho^{n-1} \Lambda,
\end{equation}
which is a much better estimate if we choose $M$ large enough.

\ms
Let $(\u^\ast,\W^\ast)$ denote the harmonic competitor that was defined
near \eqref{e6.11}. The construction has a parameter $a\in (0,1)$, which will be 
chosen soon, close to $1$.
All the prerequisites that were mentioned before \eqref{e6.11}
are satisfied, and in particular \eqref{e6.9} holds because $B(0,r) \i B(0,r_0) \i \Omega$.
Since $(\u^\ast,\W^\ast) \in \F$ (by construction, see \eqref{e6.15}) and 
$(\u,\W)$ is a minimizer, we get that 
\begin{equation}\label{e10.21}
J(\u,\W) \leq J(\u^\ast,\W^\ast).
\end{equation}
The $M$-term of the functional is estimated as usual: by
Theorem \ref{t5.1}, $||u_i||_\infty \leq C$, with a constant $C$ that depends only 
on the usual constants, and by construction (and the maximum principle, or rather 
the corresponding properties of the Poisson kernel, for the harmonic
extension), $||u_i^\ast||_\infty \leq ||u_i||_\infty \leq C$. Then
\begin{equation}\label{e10.22}
|M(\u) - M(\u^\ast)| 
\leq C (||u_i||^2_\infty ||f_i||_\infty + ||u_i||_\infty||g_i||_\infty)\, |B(0,r)|
\leq C r^n,
\end{equation}
because $u_i^\ast = u_i$ outside of $B(x,r)$, and 
where again $C$ depends on the usual constants 
(see the definition \eqref{e1.4}).
Similarly, the $W_i^\ast$ only differ from the $W_i$ in the ball $B(0,r)$, 
so \eqref{e10.2} yields
\begin{equation}\label{e10.23}
|F(\W) - F(\W^\ast)| \leq C r^n
\end{equation}
and \eqref{e10.21} implies that
\begin{equation}\label{e10.24}
\int_{B(0,r)} |\nabla \u|^2 \leq \int_{B(0,r)} |\nabla \u^\ast|^2 + C r^n
\end{equation}
(compare with the definitions \eqref{e1.5} and \eqref{e1.3}).
Recall from \eqref{e6.19}-\eqref{e6.21} that
\begin{eqnarray} \label{e10.25}
\int_{B(0,r)} |\nabla \u^\ast|^2 &=& \sum_{i \geq 2}
\int_{B(0,r)} |\nabla u_{i}^\ast|^2 
+ \int_{B(0,r)\sm B(0,ar)} |\nabla u_{1}^\ast|^2
+ \int_{B(0,ar)} |\nabla u_{1}^\ast|^2 
\nonumber\\
&\leq& (1-a) r  a^{2-n} \sum_{i=1}^N \int_{S_{r}} |\nabla_{t} u_{i}|^2
+ 4 (1-a)^{-1} r  a^{-n} \sum_{i=2}^N\int_{S_{r}} |r^{-1} u_{i}|^2
\nonumber
\\
&\, & \hskip6cm 
+ a^{n-2}\int_{B(0,r)} |\nabla v_{1}|^2
\end{eqnarray}
where $v_1$ still denotes the harmonic extension (to $B(0,r)$) of the restriction
of $u_1$ to $S_r$; see above \eqref{e6.13}. 
For the first term, we just use \eqref{e10.19} and get that
\begin{equation}\label{e10.26}
(1-a) r  a^{2-n} \sum_{i=1}^N \int_{S_{r}} |\nabla_{t} u_{i}|^2
\leq 2N (1-a) a^{2-n} r \rho^{-1} E(\rho) \leq C (1-a) E(\rho)
\end{equation}
because $r \leq \rho$ and we will choose $a > 1/2$.
In fact, we will take $a$ very close to $1$ to make this term small. 
We shall only be able to continue our estimates with the present competitor
when
\begin{equation}\label{e10.27}
\sigma(S_r \sm W_i) \geq \varepsilon \sigma(S_r)
\ \text{ for } 2 \leq i \leq N,
\end{equation}
where the small $\varepsilon > 0$ will be chosen later. 
In the remaining case when \eqref{e10.27} fails, we shall use a different 
harmonic competitor; this will be somewhat easier, but will be done
later, near \eqref{e10.56}.

Our assumption \eqref{e10.27} allows us to use the Poincar\'e estimate \eqref{e4.6},
with $p=2$ and $E = S_r \cap W_i$ (recall \eqref{e10.18}). We get that 
\begin{equation}\label{e10.28}
\int_{S_{r}} |r^{-1} u_{i}|^2 
= r^{-2} \int_{S_{r}\cap W_i} |u_{i}|^2  
\leq C \varepsilon^{-1} \int_{S_{r}} |\nabla_t u_{i}|^2.
\end{equation}
For our second term, we use \eqref{e10.28} and \eqref{e10.20}, and get that
\begin{eqnarray}\label{e10.29}
4 (1-a)^{-1} r  a^{-n} \sum_{i=2}^N\int_{S_{r}} |r^{-1} u_{i}|^2
&\leq& C (1-a)^{-1} r \varepsilon^{-1} \sum_{i=2}^N\int_{S_{r}} |\nabla_t u_{i}|^2
\nonumber
\\
&\leq& C (1-a)^{-1}\varepsilon^{-1} \rho^n \Lambda.
\end{eqnarray} 
For this term, the two large constants $(1-a)^{-1}$
and $\varepsilon^{-1}$ will be neutralized by taking $M$ large enough
and hence $\Lambda$ very small.
We drop $a^{n-2}$ in the last term of \eqref{e10.25}, and get that 
\begin{equation}\label{e10.30}
\int_{B(0,r)} |\nabla \u^\ast|^2 \leq \int_{B(0,r)} |\nabla v_{1}|^2 + C \alpha_1 
\end{equation}
with 
\begin{equation}\label{e10.31}
\alpha_1 = (1-a) E(\rho) + (1-a)^{-1}\varepsilon^{-1} \rho^n \Lambda.
\end{equation}

Now recall from \eqref{e6.13} that $v_1$ is the minimizer of 
$\int_{B(0,r)} |\nabla v|^2$, among functions $v\in W^{1,2}(B(0,r))$
that coincide with $u_1$ on $S_r$. Of course we can take $v=u_1$
in the definition, and even $v_t = v_1 + t (u_1-v_1)$. 
Notice that $\int_{B(0,r)} |\nabla v_t|^2$ is a quadratic function of $t$,
whose derivative at $t=0$ vanishes by minimality. We compute this derivative and get
that $\int_{B(0,t)} \langle \nabla v_1, \nabla(u_1-v_1) \rangle = 0$.
Hence, by Pythagorus,
\begin{equation}\label{e10.32}
\int_{B(0,r)} |\nabla u_1|^2 = \int_{B(0,r)} |\nabla v_{1}|^2
+ \int_{B(0,r)} |\nabla (u_1-v_{1})|^2.
\end{equation}
We may now deduce from \eqref{e10.24} and \eqref{e10.30} that
\begin{eqnarray}\label{e10.33}
\int_{B(0,r)} |\nabla (u_1-v_{1})|^2  + \int_{B(0,r)} |\nabla v_{1}|^2 
&=& \int_{B(0,r)} |\nabla u_1|^2
\leq \int_{B(0,r)} |\nabla \u|^2
\nonumber
\\
&\leq& \int_{B(0,r)} |\nabla \u^\ast|^2 + C r^n
\\
&\leq& \int_{B(0,r)} |\nabla v_{1}|^2
+ C \alpha_1 + C \rho^n.
\nonumber
\end{eqnarray}
We subtract $\int_{B(0,r)} |\nabla v_{1}|^2$ from both sides and get that
\begin{equation}\label{e10.34}
\int_{B(0,r)} |\nabla (u_1-v_{1})|^2 
\leq C \alpha_1 + C \rho^n.
\end{equation}
Next we need to find ways to say that $\nabla v_1$ is small near the origin.
We claim that 
\begin{equation}\label{e10.35}
r^{-2}\int_{B(0,r)} |u_1-v_1|^2  \leq C \int_{B(0,r)} |\nabla (u_1-v_1)|^2
\leq C \alpha_1 + C \rho^n.
\end{equation}
The second part follows from\eqref{e10.34}.
The most direct way to see the first part is to notice that if we extend  $u_1-v_1$ 
by setting $(u_1-v_1)(x) = 0$ for $x\in \R^n \sm B(0,r)$, we get 
a function of $W^{1,2}(\R^n)$ (see \eqref{e4.18}) . Then 
by Poincar\'e's inequality (for instance, apply \eqref{e4.2} to $B(x,2r)$ and
use it to control $m_{B(x,2r)}(u_1-v_1)$), we get the claim. Or we 
can use Lemma \ref{l4.2} to control $m_{B(0,r} (u_1-v_1)$
and the apply Poincar\'e's inequality \eqref{e4.2} on $B(0,r)$.

Next we want to use the (estimates that lead to the) H\"older-continuity of $\u$ 
to control $\u(x) - \u(0)$ near the origin.
Recall from \eqref{e7.38} that for $0 < s \leq r$,
\begin{eqnarray}\label{e10.36}
\int_{B(0,s)} |\nabla u_1|^2 
&\leq& \int_{B(0,s)} |\nabla \u|^2
\leq (s/r)^\delta \int_{B(0,r)} |\nabla \u|^2 + C [s^{n-{n \over p}} + s^{\beta n}]
\\
\nonumber
&\leq& (s/r)^\delta E(\rho) + C [s^{n-{n \over p}} + s^{\beta n}]
= (s/r)^\delta E(\rho) + C s^n,
\end{eqnarray}
where $C$ depend on the usual constants, and the last equality comes
from the fact that $p=+\infty$ and $\beta = 1$ with our new assumptions 
\eqref{e10.1} and \eqref{e10.2}. Set $m(s) = \fint_{B(0,s)} u_1$ for $s < r$; 
the same proof as in \eqref{e7.42} yields
\begin{eqnarray}\label{e10.37}
|m(s/2) - m(s)| &=& \Big|\fint_{B(0,s/2)} u_1-m(s) \Big|
\leq \fint_{B(0,s/2)} |u_1-m(s)| 
\nonumber
\\
&\leq& 2^n \fint_{B(0,s)} |u_1-m(s)|
\leq C s \fint_{B(0,s)} |\nabla u_1|
\nonumber
\\
&\leq& C s \Big\{\fint_{B(0,s)} |\nabla u_1|^2 \Big\}^{1/2}
\leq C \Big\{s^{2-n} \int_{B(0,s)} |\nabla u_1|^2 \Big\}^{1/2}
\\
&\leq& C  \Big\{ s^{2-n} (s/r)^{\delta} E(\rho)  + s^{2} \Big\}^{1/2}.
\nonumber
\end{eqnarray}
By \eqref{e7.31}, the exponent $\delta+2-n$ is positive, so 
$\sum_{k \geq 0} |m(2^{-k-1}s) - m(2^{-k}s)| < +\infty$
and
\begin{equation}\label{e10.38}
m(s) \leq \sum_{k \geq 0} |m(2^{-k-1}s) - m(2^{-k}s)|
\leq C \Big\{s^{2-n} (s/r)^{\delta} E(\rho) + s^2 \Big\}^{1/2}
\end{equation}
because $u_1$ is continuous and $\u(0)=0$, by \eqref{e10.37}, and after summing a
geometric series whose main term is for $k=0$.
Let $\tau\in (0,1)$ be small, to be chosen soon, and apply this with 
$s = \tau\rho$. This yields
\begin{equation}\label{e10.39}
m(\tau\rho)^2 \leq C \tau^{\delta+2-n} \rho^{2-n} E(\rho) + C \tau^2\rho^2.
\end{equation}
We are interested in 
\begin{equation}\label{e10.40}
v_1(0) = \fint_{B(0,\tau\rho)} v_1 = \fint_{S_r} u_1,
\end{equation}
where both identities hold because $v_1$ is the Poisson integral
of the restriction of $u_1$ to $S_r$. Notice that
\begin{equation}\label{e10.41}
v_1(0) = \fint_{B(0,\tau\rho)} v_1 \leq\fint_{B(0,\tau\rho)} u_1 + \fint_{B(0,\tau\rho)} |u_1-v_1|
= m(\tau\rho) + \fint_{B(0,\tau\rho)} |u_1-v_1|
\end{equation}
and that
\begin{eqnarray}\label{e10.42}
\fint_{B(0,\tau\rho)} |u_1-v_1| &\leq& \Big\{\fint_{B(0,\tau\rho)} (u_1-v_1)^2 \Big\}^{1/2}
\leq C \Big\{(\tau\rho)^{-n}\int_{B(0,\tau\rho)} (u_1-v_1)^2 \Big\}^{1/2}
\nonumber
\\
&\leq& C \Big\{(\tau\rho)^{-n}\int_{B(0,r)} (u_1-v_1)^2 \Big\}^{1/2}
\leq C\Big\{(\tau\rho)^{-n} r^2 [\alpha_1 + \rho^n] \Big\}^{1/2}
\end{eqnarray}
by \eqref{e10.35}. Hence
\begin{eqnarray}\label{e10.43}
v_1(0)^2 &\leq& 2 m(\tau\rho)^2 + 2 \Big\{\fint_{B(0,\tau\rho)} |u_1-v_1| \Big\}^2
\nonumber
\\
&\leq& C \tau^{\delta+2-n} \rho^{2-n} E(\rho) + C \tau^2\rho^2
+ C(\tau\rho)^{-n} r^2 [\alpha_1 + \rho^n] 
\nonumber
\\
&\leq& C \tau^{\delta+2-n} \rho^{2-n} E(\rho)
+ C\tau^{-n}\rho^{2-n}  \alpha_1 + C \tau^{-n} \rho^2
\end{eqnarray}
by \eqref{e10.41}, \eqref{e10.39}, and \eqref{e10.42}.

Recall that $v_1$ is the Poisson integral of the restriction 
of $u_1$ to $S_r$, and its derivative is simply obtained by
differentiating under the integral sign. In addition, the derivative
of the Poisson kernel, say, from the unit sphere to the unit ball,
is uniformly bounded in $B(0,1/2)$. This and the obvious invariance
under dilations yield
\begin{equation}\label{e10.44}
||\nabla v_1||_{L^\infty(B(0,r/2)} \leq C r^{-1} \fint_{S_r} |v_1|
= C r^{-1} \fint_{S_r} |u_1|.
\end{equation}
But $u_1 \geq 0$ (this was our initial reduction, and if we did not do it we would
just have been working with $u_{1,+}$), so 
$\fint_{S_r} |u_1| = \fint_{S_r} u_1 = v_1(0)$,
by \eqref{e10.40}. Therefore
\begin{eqnarray}\label{e10.45}
\fint_{B(0,\tau\rho)} |\nabla v_1|^2 
&\leq&  ||\nabla v_1||_{L^\infty(B(0,r/2)}^2
\leq C \rho^{-2} v_1(0)^2
\nonumber
\\
&\leq& 
C \tau^{\delta+2-n} \rho^{-n} E(\rho) + C\tau^{-n}\rho^{-n}  \alpha_1 + C \tau^{-n} 
\end{eqnarray}
by \eqref{e10.43}. We add this to \eqref{e10.34} and get that
\begin{eqnarray}\label{e10.46}
\int_{B(0,\tau\rho)} |\nabla u_1|^2 
&\leq& 2 \int_{B(0,\tau\rho)} |\nabla v_1|^2 + 2 \int_{B(0,\tau\rho)} |\nabla (u_1-v_1)|^2
\nonumber
\\
&\leq& C \tau^n \rho^n \fint_{B(0,\tau\rho)} |\nabla v_1|^2
+ 2 \int_{B(0,r)} |\nabla (u_1-v_1)|^2
\nonumber
\\
&\leq& 
[C \tau^{\delta+2} E(\rho) + C \alpha_1 + C \rho^n]
+ [C \alpha_1 + C \rho^n]
\\
&\leq& C \tau^{\delta+2} E(\rho) + C \alpha_1 + C  \rho^{n}.
\nonumber
\end{eqnarray}
We also need to estimate the contribution of the other indices $i > 1$,
but fortunately
\begin{eqnarray}\label{e10.47}
\sum_{i > 1}\int_{B(0,\tau\rho)} |\nabla u_i|^2 
&\leq & \sum_{i > 1}\int_{B(0,r)} |\nabla u_i|^2 
= \int_{B(0,r)} |\nabla \u|^2 - \int_{B(0,r)} |\nabla u_1|^2
\nonumber
\\
&\leq& 
\int_{B(0,r)} |\nabla \u^\ast|^2 + C r^n - \int_{B(0,r)} |\nabla u_1|^2
\nonumber
\\
&\leq& \int_{B(0,r)} |\nabla v_1|^2 + C \alpha_1 + C r^n - \int_{B(0,r)} |\nabla u_1|^2
\leq  C \alpha_1 + C r^n
\end{eqnarray}
by \eqref{e10.24}, \eqref{e10.30}, and the minimization property
of $v_1$ (see \eqref{e6.13} or \eqref{e10.32}). Altogether
\begin{equation}\label{e10.48}
\int_{B(0,\tau\rho)} |\nabla \u|^2 
\leq  C \tau^{\delta+2} E(\rho) + C \alpha_1 + C  \rho^{n}
\end{equation}
by \eqref{e10.46} and \eqref{e10.47}. We divide this by $(\tau \rho)^{n}$
and get that 
\begin{eqnarray}\label{e10.49}
\fint_{B(0,\tau\rho)} |\nabla \u|^2 
&\leq&  C \tau^{\delta+2-n} \rho^{-n} E(\rho) + C\tau^{-n}\alpha_1\rho^{-n} + C  \tau^{-n}
\nonumber
\\
&=& C \tau^{\delta+2-n} \fint_{B(0,\rho)} |\nabla \u|^2
+ C\tau^{-n}\alpha_1\rho^{-n} + C  \tau^{-n}.
\end{eqnarray}
We need to compare this with the desired conclusion \eqref{e10.12}.
To take care of the first term, we just choose $\tau$ so small that
\begin{equation}\label{e10.50}
 C \tau^{\delta+2-n} \leq 30^{-1}
\end{equation}
(recall from \eqref{e7.31} that $\delta > n-2$). The second term is 
\begin{eqnarray}\label{e10.51}
C\tau^{-n}\alpha_1\rho^{-n} &= &
C\tau^{-n}\rho^{-n} [(1-a) E(\rho) + (1-a)^{-1}\varepsilon^{-1} \rho^n \Lambda]
\nonumber
\\
&=& C(1-a)\tau^{-n}\rho^{-n} E(\rho) + C \tau^{-n} (1-a)^{-1}\varepsilon^{-1} \Lambda
\nonumber
\\
&\leq& C(1-a)\tau^{-n}\fint_{B(0,\rho)} |\nabla \u|^2 
+ C \tau^{-n} (1-a)^{-1}\varepsilon^{-1}
M^{-1} C_3 \Big(1+ \fint_{B(0,r_0)}  |\nabla \u|^2 \Big)^2
\end{eqnarray}
by \eqref{e10.31} and \eqref{e10.17}. We shall choose $a$ so close to $1$
(depending on $\tau$, but not on $M$) that
\begin{equation} \label{e10.52}
C(1-a)\tau^{-n}\leq 30^{-1}.
\end{equation}
We do not choose $a$ yet, because there will be a third case where
a similar condition on $a$ will arise (in \eqref{e10.67}).

For the second piece of \eqref{e10.51}, recall that
$M \leq \fint_{B(0,\rho)} |\nabla \u|^2$ because we are in the interesting
case when \eqref{e10.13} fails; hence the second piece is at most
\begin{equation} \label{e10.53}
C \tau^{-n} (1-a)^{-1}\varepsilon^{-1}
M^{-2} C_3 \Big(1+ \fint_{B(0,r_0)}  |\nabla \u|^2 \Big)^2 \fint_{B(0,\rho)} |\nabla \u|^2.
\end{equation}
We promise that we shall choose $\varepsilon$ soon, and that it will
not depend on $M$. Then we shall choose $M$ so large, depending on
$\tau$, $a$, and $\varepsilon$, that
\begin{equation} \label{e10.54}
C \tau^{-n} (1-a)^{-1}\varepsilon^{-1}
M^{-2} C_3 \Big(1+ \fint_{B(0,r_0)}  |\nabla \u|^2 \Big)^2 \leq 30^{-1}.
\end{equation}
With all these choices, the two first terms of \eqref{e10.49} are dominated
by $10^{-1} \fint_{B(0,\rho)} |\nabla \u|^2$. Then, in the present case,
\eqref{e10.12} holds as soon as we choose 
\begin{equation} \label{e10.55}
C(\tau,r_0) \geq C \tau^{-n}. 
\end{equation}

\ms
We are not finished yet, because we still have to deal with the case when
\eqref{e10.27} fails, i.e., when we can find $i \geq 2$ such that
\begin{equation}\label{e10.56}
\sigma(S_r \sm W_i) < \varepsilon \sigma(S_r).
\end{equation}
In this case, we shall also use the harmonic competitor defined in Section \ref{favorites},
but with $u_i$ (rather that $u_1$) as our preferred variable, and with the same 
parameter $a$ as above (just to simplify the discussion). 
Denote by $(\u^\ast,\W^\ast)$ this new competitor. 
We still have \eqref{e10.21}-\eqref{e10.24} for the same reasons 
(all our functions are bounded and we change nothing outside
of $B(0,r)$), but we estimate $\int_{B(0,r)} |\nabla \u^\ast|^2$ differently.
For $j \neq i$, including $j=1$, we use \eqref{e6.19}, which says that
\begin{equation}\label{e10.57}
\int_{B(0,r)} |\nabla u_{j}^\ast|^2
\leq (1-a) r  a^{2-n} \int_{S_{r}} |\nabla_{t} u_{j}|^2
+ 4 (1-a)^{-1} r  a^{-n} \int_{S_{r}} |r^{-1} u_{j}|^2.
\end{equation}
By \eqref{e10.18}, $u_j$ is essentially supported in the small
set $W_j \cap S_r \i S_r \sm W_i$, and \eqref{e4.7} yields
\begin{equation}\label{e10.58}
\int_{S_{r}} |r^{-1} u_{j}|^2 
\leq C r^{-2} \sigma(S_r \sm W_i)^{2\over n-1} \int_{S_{r}} |\nabla_t u_{j}|^2 
\leq C \varepsilon^{2\over n-1} \int_{S_{r}} |\nabla_t u_{j}|^2.
\end{equation}
Hence 
\begin{eqnarray}\label{e10.59}
\int_{B(0,r)} |\nabla u_{j}^\ast|^2
&\leq& [(1-a) a^{2-n} + C (1-a)^{-1} a^{-n}\varepsilon^{2\over n-1}] \, r
\int_{S_{r}} |\nabla_{t} u_{j}|^2
\nonumber\\
&\leq& C [(1-a)+ (1-a)^{-1} \varepsilon^{2\over n-1}] \, E(\rho)
\end{eqnarray}
because we shall choose $a$ close to $1$, and by \eqref{e10.19}.
Next we use \eqref{e6.20} to say that
\begin{equation}\label{e10.60}
\int_{B(0,r)\sm B(0,ar)} |\nabla u_{i}^\ast|^2
\leq (1-a) r  a^{2-n} \int_{S_{r}} |\nabla_{t} u_{i}|^2
\leq C (1-a) E(\rho)
\end{equation}
by \eqref{e10.19} again. Finally, by \eqref{e6.21},
\begin{equation}\label{e10.61}
\int_{B(0,ar)} |\nabla u_{i}^\ast|^2
= a^{n-2} \inf\Big\{ \int_{B(0,r)} |\nabla v|^2 \, ; \,
v\in W^{1,2}(B(0,r)) \text{ and } v=u_{i} \text{ on } S_{r} \Big\}.
\end{equation}
A simple choice of $v$ is given by the following extension.
Set $m = \fint_{S_r} u_i$, and then
\begin{equation}\label{e10.62}
v(ty) = m + t (u_i(y)-m)
\ \text{ for $y\in S_r$ and } 0 \leq t \leq 1.
\end{equation}
This is the same extension that we used near \eqref{e7.17}, and the 
same computations as above yield
\begin{eqnarray}\label{e10.63}
\int_{B(0,r)} |\nabla v|^2 &=& \int_{B(0,r)} |\nabla (v-m)|^2 
\leq C r \int_{S_r} \big[r^{-2} (u_i-m)^2 +  |\nabla u_i|^2 \big]
\nonumber
\\
&\leq & C r \int_{S_r} |\nabla u_i|^2 \leq C \rho^n \Lambda
\end{eqnarray}
(see \eqref{e7.18} and then use \eqref{e4.2} and \eqref{e10.20}).
We combine this with \eqref{e10.59}, \eqref{e10.60}, and \eqref{e10.61},
and get that
\begin{eqnarray}\label{e10.64}
\int_{B(0,r)} |\nabla \u^\ast|^2
&\leq& C \big[(1-a)+ (1-a)^{-1} \varepsilon^{2\over n-1}\big] E(\rho)
+ \int_{B(0,r)} |\nabla u_i^\ast|^2
\nonumber
\\
&\leq & C \big[(1-a)+ (1-a)^{-1} \varepsilon^{2\over n-1}\big] E(\rho)
+ C \rho^n \Lambda =: \alpha_2,
\end{eqnarray}
where the last equality is a definition of $\alpha_2$.
Then, by \eqref{e10.24},
\begin{equation}\label{e10.65}
\int_{B(0,r)} |\nabla \u|^2 \leq \int_{B(0,r)} |\nabla \u^\ast|^2 + C r^n
\leq \alpha_2 + C r^n
\end{equation}
We keep the same $\tau$ that we chose in \eqref{e10.50}, and notice that
\begin{eqnarray}\label{e10.66}
\fint_{B(0,\tau\rho)} |\nabla \u|^2 
&\leq& C \tau^{-n} \rho^{-n} \int_{B(0,r)} |\nabla \u|^2 
\leq C \tau^{-n} \rho^{-n} \alpha_2 + C \tau^{-n}
\nonumber
\\
&= & C \big[(1-a)+ (1-a)^{-1} \varepsilon^{2\over n-1}\big] \tau^{-n} \rho^{-n} E(\rho)
+ C \tau^{-n} \Lambda + C \tau^{-n}
\nonumber
\\
&\leq & C \big[(1-a)+ (1-a)^{-1} \varepsilon^{2\over n-1}\big] \tau^{-n} \fint_{B(0,\rho)} |\nabla \u|^2 
+ C \tau^{-n} \Lambda + C \tau^{-n}.
\end{eqnarray}
We may now chose $a$, satisfying the old constraint \eqref{e10.52} and in addition
\begin{equation}\label{e10.67}
C (1-a) \tau^{-n} \leq 30^{-1}
\end{equation}
(with $C$ as in \eqref{e10.66}). Then we choose $\varepsilon$, depending
on $\tau$ and $a$, so that
\begin{equation}\label{e10.68}
C (1-a)^{-1} \varepsilon^{2\over n-1} \tau^{-n} \leq 30^{-1}.
\end{equation}
Next
\begin{eqnarray}\label{e10.69}
C \tau^{-n} \Lambda &=& 
C \tau^{-n} M^{-1} C_3 \Big(1+ \fint_{B(0,r_0)}  |\nabla \u|^2 \Big)^2
\nonumber
\\
&\leq& C \tau^{-n} M^{-2} C_3 \Big(1+ \fint_{B(0,r_0)}  |\nabla \u|^2 \Big)^2 
\fint_{B(0,\rho)} |\nabla \u|^2 
\end{eqnarray}
by \eqref{e10.17} and because \eqref{e10.13} fails, and now we choose $M$
so large, depending on $\tau$, $a$, and $\varepsilon$, that
\begin{equation}\label{e10.70}
C \tau^{-n} M^{-2} C_3 \Big(1+ \fint_{B(0,r_0)}  |\nabla \u|^2 \Big)^2 < 30^{-1},
\end{equation}
in addition to the earlier similar condition \eqref{e10.54}.
With all these choices, \eqref{e10.66} implies that
$\fint_{B(0,\tau\rho)} |\nabla \u|^2 \leq 10^{-1}\fint_{B(0,\rho)} |\nabla \u|^2 
+ C \tau^{-n}$, and \eqref{e10.12} holds if we choose
$C(\tau,r_0) \geq C \tau^{-n}$. We had a similar constraint in \eqref{e10.55},
but our main constraint on $C(\tau,r_0)$ comes from \eqref{e10.14}, which
demands that $C(\tau,r_0) \geq C \tau^{-n} M$. In view of our constraints
\eqref{e10.54} and \eqref{e10.70} on $M$, we see that we can choose $C(\tau,r_0)$
such that
\begin{equation} \label{e10.71}
C(\tau,r_0) \leq C \tau^{-3n/2} (1-a)^{-1/2}\varepsilon^{-1/2} 
\Big(1+ \fint_{B(0,r_0)}  |\nabla \u|^2 \Big)
\leq C \big(1+ \fint_{B(0,r_0)}  |\nabla \u|^2 \big).
\end{equation}
This completes our proof of \eqref{e10.12}, with the announced bound 
on $C(\tau,r_0)$; Lemma \ref{t10.4} follows.
\qed
\end{proof}

\begin{cor}\label{t10.5} 
Let $B(0,r_0)$ satisfy the assumptions of Lemma \ref{t10.4}; then
\begin{equation} \label{e10.72}
\fint_{B(0,r)} |\nabla \u|^2 
\leq C\big(1+ \fint_{B(0,r_0)}  |\nabla \u|^2 \big)
\text{ for $0 < r \leq r_0$,}
\end{equation}
with a constant $C$ that depends on the usual constants.
\end{cor}

\begin{proof}
Define a sequence $\{ \omega_k \}$ by
\begin{equation} \label{e10.73}
\omega_k = (\tau^k r_0)^{-n} \int_{B(0,\tau^k r_0)} |\nabla \u|^2
\ \text{ for } k \geq 0,
\end{equation}
and let $C(\tau,r_0)$ be as in Lemma \ref{t10.4}.
Let us show by induction that $\omega_k \leq 2 C(\tau,r_0) + \omega_0$.
This is obviously true for $k=0$. Now suppose that it is true for $k$; 
Lemma \ref{t10.4} says that
\begin{equation} \label{e10.74}
\omega_{k+1} \leq 10^{-1 }\omega_k + C(\tau,r_0)
\leq 10^{-1 }(2 C(\tau,r_0) + \omega_0) + C(\tau,r_0) < 2 C(\tau,r_0) + \omega_0,
\end{equation}
which proves our claim. That is, 
\begin{equation} \label{e10.75}
\omega_{k} \leq 2C(\tau,r_0) + \omega_0  
\leq C\big(1+ \fint_{B(0,r_0)}  |\nabla \u|^2 \big)
\end{equation}
for all $k$. The corollary follows easily, by comparing any $r\leq r_0$ with a slightly
larger $\tau^k r_0$.
\qed
\end{proof}

\ms
We shall now consider balls that do not meet the set $\{ \u = 0 \}$.
They will be easier to deal with, because we can use the equation \eqref{e9.6}.

\begin{lem}\label{t10.6} 
Suppose that $0 < r_0 \leq 1$, $B(0,r_0) \i \Omega$ 
and $\u(x) \neq 0$ for $x\in B(0,r_0)$. Then 
\begin{equation}\label{e10.76}
|\nabla \u(x)| \leq C r_0 + 2^n\fint_{B(0,r_0)} |\nabla \u| 
\ \text{ for } x\in B(0,r_0/2),
\end{equation}
with a constant $C$ that depends only on the usual constants.
\end{lem} 

\begin{proof}
Since $\u(0) \neq 0$, we can find $i$ such that $u_i(0) \neq 0$.
Then we claim that $u_i > 0$ on $B(0,r_0)$. Indeed, the
set $V = \big\{ x\in B(0,r_0) \, ; \, u_i(x) > 0 \big\}$
contains $0$ (recall that $u_i \geq 0$ since Lemma \ref{t10.2}),
and is open in $B(0,r_0)$. But $u_j(z) = 0$ for $j \neq i$ and $z\in V$,
because $u_i u_j = 0$ everywhere (by \eqref{e1.2} because $\u$
is continuous), so $u_j(z) = 0$ when $j \neq i$ and $z$ lies in the closure
of $V$ in $B(0,r_0)$, and since $\u(z) \neq 0$, this forces $u_i(z) \neq 0$.
So $V$ is closed too, $V = B(0,r_0)$, and this proves our claim.

So we can use \eqref{e9.6}, which says that 
\begin{equation}\label{e10.77}
\Delta u_i = f_i u_{i} -  {1 \over 2} \, g_{i} 
\ \text{ in } B(0,r_0)
\end{equation}
Set $h = 1_{B(0,r_0)} [f_i u_{i} -  {1 \over 2} g_{i}]$; by Theorem \ref{t5.1},
$h\in L^\infty(\R^n)$, with bounds that depend only on the usual constants.
Set $w = G \ast h$, where $G$ is the fundamental solution of $-\Delta$;
from \eqref{e10.77} we deduce that
$u_i - w$ is harmonic in $B(0,r_0)$. At the same time,
$\nabla w = (\nabla G) \ast h$, so $\nabla w \in L^\infty$, and even
\begin{equation}\label{e10.78}
|\nabla w(x)| \leq \int_{B(0,r_0)} |\nabla G(x-y)| |h(y)| dy
\leq C ||h||_\infty \int_{B(0,r_0)} |x-y|^{1-n} dy
\leq C ||h||_\infty r_0
\end{equation}
for $x\in \R^n$. If $x\in B(0,r_0/2)$,
\begin{eqnarray}
\label{e10.79}
|\nabla u_i(x)| &\leq& |\nabla (u_i-w)(x)| + |\nabla w(x)|
\leq |\nabla (u_i-w)(x)| + C ||h||_\infty r_0
\nonumber
\\
&\leq& \fint_{B(x,r_0/2)} |\nabla (u_i-w)| + C r_0
\leq 2^n \fint_{B(x,r_0)} |\nabla (u_i-w)| + C r_0
\nonumber
\\
&\leq& 2^n \fint_{B(x,r_0)} |\nabla u_i| + 2^n ||\nabla w||_\infty + C r_0
\leq2^n \fint_{B(x,r_0)} |\nabla u_i| + C r_0
\end{eqnarray}
because $u_i - w$ is harmonic in $B(x,r_0/2) \i B(0,r_0)$;
\eqref{e10.76} and the lemma follow because all the other $u_j$
vanish on $B(0,r_0)$.
\qed
\end{proof}

\ms
After all these lemmas, we are now ready to prove Theorem \ref{t10.1}.
Let $B(x_0,r_0)$ be as in the theorem, and let $x$ be any point of $B(x_0,r_0)$.
We distinguish between cases, depending on the size of 
$d(x) = \dist(x, Z)$, where $Z = \big\{ z\in \R^n \, ; \, \u(z) = 0 \big\}$.

If $d(x) \geq r_0/4$, we notice that $\u(x) \neq 0$ on
$B(x,r_0/4)$ and that $B(x,r_0/4) \i B(x_0,2r_0) \i \Omega$, 
apply Lemma \ref{t10.6} to a translation by $-x$
of the minimizer $(\u,\W)$, and to the radius $r_0/4$,
and get that for $y$ near $x$,
\begin{equation}
\label{e10.80}
|\nabla \u(y)| \leq C r_0 + 2^n\fint_{B(x,r_0/4)} |\nabla \u| 
\leq C + C \fint_{B(x_0,2r_0)} |\nabla \u|.
\end{equation}

If $0 < d(x) < r_0/4$, choose $z\in Z$ such that $|z-x| = d(x)$,
notice that $B(z,r_0/2) \i B(0,2r_0) \i \Omega$,
and apply Corollary \ref{t10.5} to the translation of $(\u,\W)$ by $-z$
and with the radius $r_0/2$; then 
\begin{equation}
\label{e10.81}
\fint_{B(z,2d(x))} |\nabla \u|^2 \leq C\big(1+ \fint_{B(z,r_0/2)}  |\nabla \u|^2 \big)
\leq C + C \fint_{B(x_0,2r_0)}  |\nabla \u|^2.
\end{equation}
We now use the fact that $\u \neq 0$ on $B(x,d(x)) \i B(z,2d(x))\i \Omega$, apply
Lemma \ref{t10.6} to the translation by $-x$ of $(\u,\W)$ and with the radius 
$d(x)$, and get that for $y$ near $x$, 
\begin{eqnarray}
\label{e10.82}
|\nabla \u(y)| &\leq& C d(x) + 2^n\fint_{B(x,d(x))} |\nabla \u| 
\leq C + C \fint_{B(z,2d(x))} |\nabla \u|
\nonumber
\\
&\leq& C + C \Big\{\fint_{B(z,2d(x))} |\nabla \u|^2 \Big\}^{1/2}
\leq C + C \Big\{\fint_{B(x_0,2r_0)} |\nabla \u|^2 \Big\}^{1/2}
\end{eqnarray}
by \eqref{e10.81}. In both cases, by the proof of Lemma \ref{t10.6},
$u$ is even $C^1$ near $x$, so we don't need to worry about the
definition of $\nabla \u(x)$.

We are left with the case when $d(x) = 0$. We may observe that 
$\u(x) = 0$ on the corresponding set, and $\nabla \u = 0$ almost everywhere
on that set. But we can also restrict our attention to the case when  $x$ is a Lebesgue 
density point for $\nabla \u$, and use Corollary \ref{t10.5}
(applied to the translation of $(\u,\W)$ by $-x$ and with the radius $r_0$)
to get that $|\nabla \u(x)|^2 = \lim_{r \to 0} \fint_{B(x,r)} |\nabla \u|^2
\leq C + C \fint_{B(x,r_0)} |\nabla \u|^2 \leq C + C \fint_{B(x_0,2r_0)} |\nabla \u|^2$. 

Anyway, we get that $|\nabla \u(x)| 
\leq C + C \big(\fint_{B(x_0,2r_0)} |\nabla \u|^2\big)^{1/2}$
almost everywhere on $B(x_0,r_0)$. The Lipschitz bound in \eqref{e10.2}, and then
Theorem \ref{t10.1}, follow.
\qed

\section{Global Lipschitz bounds for $u$ when 
$\Omega$ is smooth}  \label{global} 

In this section we prove that if $\Omega$ is a bounded open set with $C^{1+\alpha}$ 
boundary and the assumptions of Section \ref{lip} are satisfied, then $\u$ is Lipschitz.

\begin{thm}\label{t11.1} 
Assume that for some $\alpha > 0$, $\Omega$ is a bounded open set with a 
$C^{1+\alpha}$ boundary, 
and that \eqref{e10.1} and \eqref{e10.2} hold. Then for each minimizer
$(\u,\W)$ of the functional $J$, $\u$ is $C$-Lipschitz,
with a constant $C$ that depends only on $n$, $N$, $|\Omega|$, $\alpha$,
the  constants in \eqref{e10.1} and \eqref{e10.2}, and the 
$C^{1+\alpha}$ constants for $\d\Omega$.
\end{thm}  

\ms
In fact, the theorem also holds if we replace our $C^{1+\alpha}$  assumption
with  the weaker assumption that $\d \Omega$ is a Lyapunov-Dini surface, 
i.e., that it is $C^1$ with a Dini condition on the modulus of continuity on the unit normal. 
As we shall see in the proof, this is just because we want to know to know that if $f$ is a smooth
function on $\d\Omega$, its harmonic extension to $\Omega$ is Lipschitz.

\begin{rem}\label{r11.2}
Theorem \ref{t11.1} may fail if $\Omega$ is merely $C^1$. 
\end{rem}

\begin{proof}
To see this, consider the simple case when $n=2$, 
$\Omega$ is a simply connected domain in the plane, 
$N=1$, and $f_1 = 0$. Also take $F(\W) = 0$. 
Let $(\u,\W)$ be a minimizer for $J$; since $N=1$, we 
may write $(\u,\W) = (u,W)$, with $u=u_1$ and $W = W_1$.
Since making $W$ larger does not make $F(W)$ larger,
we may assume that $W=\Omega$. Let us not put any constraint
on the sign of $u$; then by the proof of \eqref{e9.6}, we get that
$\Delta u = -g/2$ on $\Omega$. Now we do not expect solutions
of this equation, with Dirichlet boundary values on $\d \Omega$,
to be Lipschitz if $\Omega$ is merely $C^1$. 

In fact, suppose that $0 \in \d\Omega$; even if
$\Delta u = 0$ in $\Omega \cap B(0,r)$ for some $r> 0$,
we do not expect $u$ to be Lipschitz near $0$, so the regularity
of $g$ is not an issue. To turn these considerations into a counterexample,
we shall use conformal mappings. 

Let $\Omega$ be a $C^1$, simply connected domain in
$\R^2 \simeq \Bbb C$, with $0 \in \d\Omega$, 
and let $\Psi$ be a conformal mapping from $\Omega$ to the 
unit ball. A result of Caratheodory says that $\Psi$ has a continuous
extension to $\overline \Omega$; for this and 
other information on conformal mappings that we shall use, 
we refer to \cite{P}. 
It is also known that we can choose $\Omega$ so that
$\Psi$ is not Lipschitz in $\Omega \cap B(0,r)$, where
$B(0,r)$ is a small ball centered at $0$. 
We can further arrange that $\Psi$ is Lipschitz on
a neighborhood $\Omega \sm B(0,r)$, simply by taking $\d \Omega$
smooth enough on $\R^2 \sm B(0, r/2)$ (the regularity
of $\Psi$ is a local notion). Denote by $\Gamma_1$ a small
arc of $\d \Omega$ that contains $\d\Omega \cap B(x,2r)$
and by $\Gamma_2$ the rest of $\d\Omega$. We can make
sure that $\Gamma_2$ is not empty, and we have a neighborhood 
$V_2$ of $\Gamma_2$ such that $\Psi$ is Lipschitz on
$V_2 \cap \overline\Omega$.

Now compose $\Psi$ with a conformal mapping 
$\Phi$, from the unit disk to the upper half disk 
$\big\{ z\in B(0,1) \, ; \, Im(z) > 0 \big\}$, which we choose
so that  $\Phi \circ \Psi(0) = 0$ and 
$\Phi \circ \Psi(\Gamma_1) \i [-1/2,1/2]$.
This last condition is easy to obtain, by composing $\Psi$
first with a M\"obius transform that sends $\Psi(\Gamma_1)$
to a small enough arc of circle, before we apply a standard 
conformal mapping to the half disk.

Let $v$ denote the imaginary part of $\Phi \circ \Psi$; this is a 
harmonic function on $\Omega$, with vanishing boundary values
near $0$, and yet it is not Lipschitz near $\Gamma_1$
because its gradient is not bounded (if it were, by Cauchy-Riemann's
equation the complex derivative  of $\Phi \circ \Psi$ would be too, 
which is false by construction). To make $v$ into an acceptable
solution $u$, we multiply it by $h = \varphi \circ \Phi \circ \Psi$,
where $\varphi$ is a smooth radial function such that 
$\varphi(z) = 1$ for $|z| \leq 1/2$ and $\varphi(z) = 0$ for $|z| \leq 2/3$.
Now $u = v h$ is continuous on $\overline \Omega$ and vanishes
on $\d \Omega$. It also lies in $W^{1,2}(\Omega)$, because
$v\in W^{1,2}(\Omega)$ (for these counterexamples, $\psi$ is barely
not Lipschitz), and if we extend $u$ by setting $u=0$ on $\R^2 \sm \Omega$,
we get that $u \in W^{1,2}(\R^2)$ by the proof of \eqref{e4.18}.
Also, $g = -2 \Delta u$ is a bounded function: it vanishes near
$\Gamma_1$ because $h = 1$ there, and otherwise we compute
$\Delta(vh)$ by the chain rule, and use the fact that $\Delta v=0$,
$\nabla v$ is bounded, and $h$ is smooth.

Finally let $\wt u$ denote (the first component of) a minimizer for $J$,
with the data $g$; we know that $\Delta \wt u = -g/2 = \Delta u$
on $\Omega$ (the reader may check that $\Delta u = -g/2$ also as
a distribution), then $\wt u - u$ is harmonic in $\Omega$, continuous
on $\overline\Omega$ (by Theorem \ref{t8.1} to be lazy),
and null on the boundary. By the maximum principle, $\wt u =u$
and $\wt u$ is not Lipschitz.

We are even so lucky that $u \geq 0$, so things do not get better
if we restrict to nonnegative functions. We expect that this lack of regularity
is the general rule for $C^1$ domains, even though our suggested example
was fairly special.
\qed
\end{proof}

\ms
The main new  ingredient for the proof of Theorem \ref{t11.1} will be the following simple
estimate, obtained by the maximum principle.

\begin{lem}\label{t11.3} 
Let $(\u,\W)$ be as in the theorem;
then there is a constant $C \geq 0$ such that 
\begin{equation}\label{e11.1}
|\u(x)| \leq C \dist(x,\d \Omega) \ \text{ for } x\in \Omega.
\end{equation} 
The constant $C$ depends only on the $C^{1+\alpha}$ constants for $\Omega$,
its diameter, and the $||g_i||_\infty$.
\end{lem}

\begin{proof}
Assume, for the sake of normalization, that $0 \in \Omega$.
Set $h_1(x) = - C_1 |x|^2$, where 
$C_1 = {1 \over 2n} \max_{1 \leq i \leq N} ||g_i||_\infty$, 
and then let $h$ be the harmonic extension of the restriction of $h_1$ to $\Omega$.
That is, $h$ is continuous on $\overline \Omega$, $h=h_1$ on $\d \Omega$,
and $h$ is harmonic in $\Omega$. The existence of $h$ is classical, and
since $h_1$ is smooth and $\d \Omega$ is $C^{1+\alpha}$ for some
$\alpha > 0$, we can apply Theorem~2.4 on page 23 of \cite{Wi} to get that
$f_2$ is Lipschitz on $\overline \Omega$, with estimates that depend 
only on the $C^{1+\alpha}$ constants for $\Omega$, its diameter,
and $C_1$. The same theorem also applies, and gives the same result, when
$\d \Omega$ is a Lyapunov-Dini surface, as defined on page 18 of \cite{Wi}.
Now set $w = h_1 - h$; by construction, $w = 0$ on
$\d\Omega$ and  $\Delta w = - 2n C_1$ on $\Omega$. By the maximum
principle, $w \geq 0$ on $\Omega$.

Let us compare this with what happens for $u_{i,+} = \max(0,u_i)$, $1 \leq i \leq N$.
We know from \eqref{e9.6} that $\Delta u_i = f_i u_i - {1\over2} g_i$
on the open set $\Omega_{i,+} = \big\{ x\in \Omega \, ; \, u_i(x) > 0 \big\}$; 
Since $w - u_{i,+} = w \geq 0$ on $\d \Omega_{i,+}$ and
$\Delta(w-u_{i,+}) = \Delta w - \Delta u_{i,+} \leq 0$ on $\Omega_{i,+}$,
we get that $u_{i,+} \leq w$ on $\Omega_{i,+}$, and hence trivially on $\Omega$.
That is, $u_{i,+}(x) \leq w(x) \leq C \dist(x,\d\Omega)$ for $x\in \Omega$,
just because $w$ is Lipschitz and vanishes on $\Omega$. Of course the other
inequality $u_{i,-}(x) \leq w(x) \leq C \dist(x,\d\Omega)$ would be proved the same way,
and the lemma follows.
\qed
\end{proof}

Next we use Lemma \ref{t11.3} to prove a decay estimate for balls centered on $\d \Omega$.
Still let $(\u,\W)$ be a minimizer, as in Theorem \ref{t11.1}.

\begin{lem}\label{t11.4} 
We can find $C \geq 0$, that depend on the same constants as in the
statement of the theorem, such that if $0 \in \d \Omega$, then
\begin{equation}\label{e11.2}
\fint_{B(0,r)} |\nabla \u|^2 \leq C
\ \text{ for } 0 \leq r \leq 1.
\end{equation} 
\end{lem}

\begin{proof}
It is time for us to start using our first favorite competitor, 
the cut-off competitor of Section \ref{favorites}. 
Let $r\in (0,1]$ be given, and denote by $(\u^\ast,\W^\ast)$ the
competitor given by \eqref{e6.1}-\eqref{e6.2}, 
with the simplest choice $a=1/2$ and $I = [1,N]$ 
(we cut of all the $u_i$). Let $\tau > 0$ be small, to be chosen soon, and recall from \eqref{e6.5}
that for $1 \leq i \leq N$,
\begin{eqnarray}\label{e11.3}
\int_{B(0,r)} |\nabla u_i^\ast|^2
&\leq& (1+\tau) \int_{B(0,r)\sm B(0,r/2)} |\nabla \u_i|^2 
+ 4(1-a)^{-2} (1+\tau^{-1}) r^{-2} \int_{B(0,r) \sm B(0,r/2)} |u_{i}|^2
\nonumber
\\
&\leq& (1+\tau) \int_{B(0,r)\sm B(0,r/2)} |\nabla \u_i|^2 
+ C \tau^{-1} r^{-2} \int_{B(0,r)} |u_{i}|^2
\nonumber
\\
&\leq& (1+\tau) \int_{B(0,r)\sm B(0,r/2)} |\nabla \u_i|^2 
+ C \tau^{-1} r^{n} 
\end{eqnarray} 
because $0 \in \d\Omega$ and 
Lemma \ref{t11.3} says that $|\u_i| \leq C r$ on $B(0,r)$. 
We also deduce from \eqref{e6.6} and \eqref{e6.7}
(with $p=+\infty$) that
\begin{equation} \label{e11.4}
\Big|\int_{\Omega} (u_{i}^\ast)^2 f_{i} - \int_{\Omega} u_{i}^2 f_{i}\Big|
\leq C r^n ||u_{i}||_{\infty}^2 ||f_{i}||_{\infty} \leq C r^n
\end{equation}
and 
\begin{equation} \label{e11.5}
\Big|\int_{\Omega} u_{i}^\ast g_{i} - \int_{\Omega} u_{i}g_{i}\Big|
\leq C r^{n} ||u_{i}||_{\infty} ||g_{i}||_{\infty} \leq C r^n.
\end{equation}
Recall also that we did not change the sets $W_i$, so the volume term will not
interfere here. We sum over $i$ and use the fact that $(\u,\W)$ minimizes $J$
to get that
\begin{eqnarray}\label{e11.6}
0 &\leq& J(\u^\ast,\W^\ast) - J(\u,\W) 
= \int_{B(0,r)} |\nabla \u^\ast|^2 - \int_{B(0,r)} |\nabla \u|^2  + M(\u^\ast) - M(\u) 
\nonumber
\\
&=& \int_{B(0,r)} |\nabla \u^\ast|^2 - \int_{B(0,r)} |\nabla \u|^2 
+ \sum_i \int_{\Omega} 
[(u_{i}^\ast)^2 f_{i} + u_{i}^\ast g_{i} - u_{i}^2 f_{i} - u_{i}^\ast g_{i}]
\nonumber
\\
&\leq& \int_{B(0,r)} |\nabla \u^\ast|^2 - \int_{B(0,r)} |\nabla \u|^2  + C r^n
\\
&\leq& \tau \int_{B(0,r)\sm B(0,r/2)} |\nabla \u|^2
- \int_{B(0,r/2)} |\nabla \u|^2 
+ C \tau^{-1} r^{n} 
\nonumber
\end{eqnarray} 
by \eqref{e1.5}, \eqref{e11.4} and \eqref{e11.5}, and then \eqref{e11.3}.
We move $\int_{B(0,r/2)} |\nabla \u|^2$ to the other side and get that
\begin{equation} \label{e11.7}
\int_{B(0,r/2)} |\nabla \u|^2 
\leq \tau \int_{B(0,r)\sm B(0,r/2)} |\nabla \u|^2 + C \tau^{-1} r^n
\leq \tau \int_{B(0,r)} |\nabla \u|^2 + C \tau^{-1} r^n.
\end{equation}
Then we choose $\tau = 2^{-n-1}$ and obtain
\begin{equation} \label{e11.8}
\fint_{B(0,r/2)} |\nabla \u|^2 \leq {1 \over 2} \fint_{B(0,r)} |\nabla \u|^2 + C_0
\end{equation}
for some $C_0$.

Now this holds for $0 < r \leq 1$. When $r=1$, we start with
$\fint_{B(0,1)} |\nabla \u|^2 \leq C \fint_{\Omega} |\nabla \u|^2 \leq C_1$.
Then we easily prove by induction that 
$\fint_{B(0,2^{-k})} |\nabla \u|^2 \leq C_1+2C_0$
for $k \geq 0$; \eqref{e11.2} and Lemma~\ref{t11.4} follow easily.
\qed
\end{proof}

\ms
We may now complete the proof of Theorem \ref{t11.1} a little 
as we did with the H\"older estimate near \eqref{e8.36}. 
We want to prove that
\begin{equation} \label{e11.9}
\fint_{B(x,r)} |\nabla \u|^2 \leq C
\end{equation}
for $x\in \R^n$ and $0 < r \leq 1$.
By invariance under translations, it follows from Lemma \ref{t11.4}
that \eqref{e11.9} holds for $x\in \d\Omega$ and $0 \leq r \leq 1$.
It also holds when $r \geq 10^{-3}$, just because 
\begin{equation} \label{e11.10}
\fint_{B(x,r)} |\nabla \u|^2 
\leq C \int_{B(x,r)} |\nabla \u|^2 \leq C \int_{\Omega} |\nabla \u|^2 \leq C,
\end{equation}
as before. So we can suppose that $r \leq 10^{-3}$ and $x\in \R^n \sm \d \Omega$.

Let $x\in \R^n \sm \d \Omega$ be given, set $d(x) = \dist(x,\d\Omega)$
and choose $z\in \d\Omega$ such that $|z-x| = d(x)$. 
Also set $d'(x) = \min(d(x),1)$. 

If $r \geq 10^{-2} d(x)$, notice that $B(x,r) \i B(z,d(x)+r)$
and $d(x)+r \leq 101 r \leq 1$, so 
\begin{equation} \label{e11.11}
\fint_{B(x,r)} |\nabla \u|^2 \leq C \fint_{B(z,d(x)+r)} |\nabla \u|^2 \leq C
\end{equation}
because the radii are comparable and we already know \eqref{e11.9} for $B(z,d(x)+r)$
So we can assume that $r \leq 10^{-2} d(x)$. Let us check that
\begin{equation} \label{e11.12}
\fint_{B(x,d'(x))} |\nabla \u|^2 \leq C.
\end{equation}
When $d'(x) \geq 10^{-2}$, this follows from \eqref{e11.10}.
Otherwise, the proof of \eqref{e11.11} yields
$\fint_{B(x,d'(x))} |\nabla \u|^2 \leq C \fint_{B(z,2d'(x))} |\nabla \u|^2 \leq C$,
as needed for \eqref{e11.12}.

Recall that we are left with the case when $r \leq 10^{-2} d(x)$
and $r \leq 10^{-3}$; thus $r \leq 10^{-2} d'(x)$
(because $d'(x) \neq d(x)$ implies that $d(x) \geq 1$).
We distinguish new cases, depending on the value of  $d_1(x) = \dist(x, Z)$, 
where $Z$ still denotes the zero set of $\u$. Notice that $0 \leq d_1(x) \leq d(x)$, 
choose $z_1 \in Z$ such that $|z_1-x| = d_1(x)$, and also set 
$d'_1(x) = \min(d_1(x),1) \leq d'(x)$.

If $r \geq 10^{-1} d_1(x)$, notice that $B(x,r) \i B(z_1,d_1(x)+r)$
and $d_1(x)+r \leq 11 r \leq d'(x)/4$ (because $r \leq 10^{-2} d'(x)$).
In addition, $B(z_1,d'(x)/4) \i B(x,d'(x)) \i \Omega$ because 
$|z_1-x| = d_1(x) \leq 10r \leq d'(x)/10$, so
\begin{eqnarray} \label{e11.13}
\fint_{B(x,r)} |\nabla \u|^2 
&\leq& 11^n \fint_{B(z_1,d_1(x)+r)} |\nabla \u|^2 
\leq C \big(1+ \fint_{B(z_1,d'(x)/4)} |\nabla \u|^2 \big) 
\nonumber\\
&\leq& C \big(1+ \fint_{B(x,d'(x))} |\nabla \u|^2 \big) \leq C
\end{eqnarray}
by Corollary \ref{t10.5} and \eqref{e11.12}. 
Next we check that
\begin{equation} \label{e11.14}
\fint_{B(x,d'_1(x))} |\nabla \u|^2 \leq C.
\end{equation}
If $d_1(x) \leq 10^{-2} d'(x)$, we can apply \eqref{e11.13} 
with $r=d_1(x)=d'_1(x)$ (precisely because then $r \leq 10^{-2} d'(x)$), 
and we get that 
$\fint_{B(x,d'_1(x))} |\nabla \u|^2 = \fint_{B(x,d_1(x))} |\nabla \u|^2 \leq C$. 
If $d_1(x) \geq 10^{-2} d'(x)$, then $d'_1(x) \geq 10^{-2} d'(x)$ 
(because $d'(x) \leq 1$ anyway) and we immediately 
get that $\fint_{B(x,d'_1(x))} |\nabla \u|^2 \leq C\fint_{B(x,d'(x))} |\nabla \u|^2 
\leq C$, because $d'_1(x) \leq d'(x)$ and by \eqref{e11.12}. So 
\eqref{e11.14} holds in both cases. 

In our remaining case when $r \leq 10^{-1} d_1(x)$, we notice that
$\u \neq 0$ in $B(x,d'_1(x))$, so we can apply Lemma \ref{t10.6} and get that
\begin{eqnarray} \label{e11.15}
\fint_{B(x,r)} |\nabla \u|^2 &\leq& ||\nabla \u||_{L^\infty(B(x,r))}^2
\leq C + \fint_{B(x,d'_1(x))} |\nabla \u|^2 \leq C
\end{eqnarray}
by \eqref{e11.14}. This completes our list of cases, and we get \eqref{e11.9}.

Once we have \eqref{e11.9}, we also get that $|\nabla \u| \leq C$ almost everywhere
(for instance at Lebesgue density points for $\nabla u$), and then $\u$ is Lipschitz, 
as desired. This completes our proof of Theorem~\ref{t11.1}.
\qed

\begin{rem} \label{r11.5}
Our $C^{1+\alpha}$, or Dini condition, is just here to get Lipschitz bounds on
harmonic functions on $\Omega$ that are smooth on $\d \Omega$;
we clearly want to forbid corners pointing inside (for instance, if 
$\Omega$ is the union of two half spaces through the origin), 
but a $C^{1+\alpha}$ domain with a few corners pointing outside would be all right.
\end{rem}

\section{A sufficient condition for $|\u|$ to be positive}  \label{positive} 

We shall now start adding assumptions to our regularity assumptions, that will
yield some form of non-degenerescence results for minimizers of the functional $J$.
Thus special properties (mostly partial monotonicity properties) of our volume functional $F$ 
will start playing a role, in this section and the next ones.

In this section we try to get minimizers $(\u,\W)$ for which the 
$W_i$ almost cover the set $\Omega$, and for which $u_i > 0$
almost everywhere on $W_i$. 
The first property will be easy to get, if $F$ is decreasing (or not increasing)
in some directions. The second one will be more interesting, and we will give
two conditions that imply that $|\u(x)|> 0$ almost everywhere on $\Omega$
when $(\u,\W)$ is a minimizer, both involving the positivity of some $g_i$
and the fact that $F$ is non increasing in some directions.
See Propositions \ref{t12.3} and \ref{t12.4}.

\ms
We shall start with a (trivial) sufficient condition for the existence of a 
minimizer $(\u,\W)$ such that the $W_i$ almost cover $\Omega$. 

Some notation will be useful.
Denote by ${\cal W}(\Omega)$ the class of $N$-uples 
$\W = (W_1, \ldots, W_N)$ such that the $W_i$ are disjoint Borel subsets
of $\Omega$; thus our functional $F$ is defined on ${\cal W}(\Omega)$.
Then let $\W = (W_1, \ldots, W_N)$ and $\W' = (W_1', \ldots, W_N')$ be given;
we say that $\W \preceq \W'$ when $W_i \i W'_i$ for $1 \leq i \leq N$.

Also, we say that $\W \in {\cal W}(\Omega)$ fills $\Omega$ when
$\big| \Omega \sm \bigcup_{i=1}^N W_i \big| = 0$. The following
lemma will not be a surprise.

\begin{lem}\label{t12.1} 
Assume that for each $\W \in {\cal W}(\Omega)$, we can find
$\W' \in {\cal W}(\Omega)$ such that $\W \preceq \W'$,
$\W'$ fills $\Omega$, and $F(\W') \leq F(\W)$.
Then for every minimizer $(\u,\W)$ for $J$ in $\F$ 
(see Section \ref{intro} for the definitions),
we can find $\W' \in {\cal W}(\Omega)$ such that $\W \preceq \W'$,
$\W'$ fills $\Omega$, and $(\u,\W')$ is a minimizer for $J$ in $\F$.
\end{lem}

Notice that the sufficient condition is satisfied if $F$ is a non increasing
function of some variable $W_i$ (when the other variables are fixed,
and subject to the constraint $\W \in {\cal W}(\Omega)$).
When $F$ is given by \eqref{e1.7}, it is satisfied when for each $x\in \Omega$,
we can find $i$ such that $q_i(x) \leq 0$.
The lemma is obvious, because if $\W'$ is the $N$-uple associated
to $\W$ by the sufficient condition, it is clear that $(\u,\W') \in \F$
and $J(\u,\W') = J(\u,\W) + F(\W') - F(\W) \leq J(\u,\W)$.

If we want to show that $\W$ fills $\Omega$ for every minimizer
$(\u,\W)$, it is reasonable to require some strict monotonicity.

\begin{lem}\label{t12.2} 
Assume that for each $\W \in {\cal W}(\Omega)$ that does not
fill $\Omega$, we can find $\W' \in {\cal W}(\Omega)$ such that 
$\W \preceq \W'$, $\W'$ fills $\Omega$, and $F(\W') < F(\W)$.
Then $\W$ fills $\Omega$ whenever $(\u,\W)$ is a minimizer
for $J$ in $\F$.
\end{lem}

Again the sufficient condition is satisfied as soon as $F$ is a 
decreasing function of some variable $W_i$, i.e., if
$F(\W') < F(\W)$ whenever $\W' \in {\cal W}(\Omega)$,
$W_i \i W'_i$, $|W'_i \sm W_i| > 0$, and $W'_j = W_j$
for $j \neq i$. This lemma also is obvious: if $(\u,\W)$ is a minimizer
and $\W'$ does not fill $\Omega$, the hypothesis gives $\W' \in {\cal W}(\Omega)$
such that $(\u,\W') \in \F$ (as before)
and $J(\u,\W') = J(\u,\W) + F(\W') - F(\W) < J(\u,\W)$.

\ms
Notice that both lemmas are atypical in the world of 
Alt, Caffarelli, and Friedman free boundaries, because when $F$
is given by  \eqref{e1.7}, they tend to require $q_i(x)\geq 0$, or even
$q_i(x) >0$, for some $i$ (that may depend on $x$, but even so). 

In the context of eigenfunctions, it can make sense to assume that $F$
is a non increasing function of the volumes, so as to get a 
partition of $\Omega$ by the $W_i$. We still can add a convexity
assumption on $F$ to try to get regularity properties on the $W_i$,
except those which have the minimal volume. See Section \ref{good}.

\ms
Next we want to state sufficient conditions for $\u$ to be
nonzero almost everywhere on $\Omega$. Of course, if we want this
to happen, we need $\W$ to fill $\Omega$, and also each
$u_i$ to be nonzero almost everywhere on the corresponding $W_i$.
If our usual assumptions are satisfied, $\u$ is continuous,
and since $u_i=0$ almost everywhere on $\R^n \sm W_i$,
we also get that $u_i = 0$ on $\R^n \sm W_i$, and even on its closure.
So the best that we can do is to make sure that
$u_i(x) \neq 0$ a.e. on the interior of $W_i$, and then we will
also need to show that the interiors of the $W_i$ fill $\Omega$.
A simple way to make sure that 
\begin{equation} 
\label{e12.1}
u_i(x) > 0 \ \text{ on the interior of } W_i
\end{equation}
(of course, if the definition of $\F$ allows positive functions $u_i$)
is to require that $g_i(x) > 0$ almost everywhere on $W_i$.
Let us not check this for the moment because we shall prove a more
general result later; the general idea is that we first check that
$u_i \geq 0$ on $W_i$, then use the fact that $J(\u,\W)$
does not decrease when we replace $u_i$ by $u_i+t\varphi$,
where $\varphi$ is a bump function supported on a small
ball contained in the interior of $W_i$. But the argument seems
to require a bootstrap, that we shall do later.

Of course we could also assume that $g_i < 0$ a.e. on $W_i$,
to get that $u_i < 0$, but if $g_i = 0$ on $W_i$,  
we shall get $u_i=0$. We may also get $u_i = 0$ on
large pieces of $W_i$ if we allow $g$ to change signs and vanish
on a small disk, or if our definition of $\F$ only allows nonnegative
functions $u_i$, and we take $g$ very negative somewhere.
[Recall that we can always pick $u_i$ in advance, and then choose
$g_i = -2 \Delta u_i$ and $f_i = 0$ to make a counterexample.]
So the assumption that $g_i > 0$ on $W_i$ is reasonable.
But remember that we shall also need to make sure that the interiors
of the $W_i$ almost cover $\Omega$.

We are ready for a first statement. For this one, we require $F$
to depend only on the volumes of the $W_i$. That is, we suppose
that there is a function $\wt F : [0,|\Omega|]^N \to \R$ such that
\begin{equation} 
\label{e12.2}
F(W_1, \dots , W_N) = \wt F(|W_1|, \dots , |W_N|)
\ \text{ for }(W_1, \dots , W_N) \in {\cal W}(\Omega).
\end{equation}

\begin{pro}\label{t12.3} 
Assume that $\Omega$ is open, that the $f_i$ and $g_i$ are bounded,
that $F$ is given by \eqref{e12.2},
that $(\u,\W)$ is a minimizer for $J$ in $\F$, 
and that $\W$ fills $\Omega$.
Further assume that 
\begin{equation} 
\label{e12.3}
g_i(x) > 0  \ \text{ for $1 \leq i \leq N$ and almost every $x\in \Omega$,}
\end{equation}
and that the definition of $\F$ allows all the functions $u_i$ to take
positive values. Then $\u(x) \neq 0$ almost everywhere on $\Omega$.
\end{pro}

Recall that we say that $\W$ fills $\Omega$ when
$\big| \Omega \sm \bigcup_{i=1}^N W_i \big| = 0$. 
Of course this is needed if we want to have $\u(x) \neq 0$ a.e. on 
$\Omega$ (because $u_i = 0$ a.e. on $\R^n \sm W_i$.
If the hypothesis of Lemma~\ref{t12.1} holds, we do not need
to assume that $\W$ fills $\Omega$, because Lemma \ref{t12.1}
provides $\W'$ such that $(\u,\W')$ is a minimizer and 
$\W'$ fills $\Omega$; we apply Proposition \ref{t12.3}
to $(\u,\W')$ and get the desired conclusion. Of course we then get
that $\W$ fills $\Omega$ anyway, as a consequence of the proposition.

We can see the proposition as a very weak regularity result on the free
boundary set $\Omega \setminus \big\{ \u(x) \neq 0 \big\}$, since it
says that this set is Lebesgue negligible.

We do not need to require any regularity on $\wt F$, because
our proof will only use competitors for which the volumes of the
$W_i$ do not change. 

The statement allows some $u_i$ to be valued in
$\R$, while other ones are valued in $\R_+$. We could
also work with the assumption that $\varepsilon_j g_j > 0$,
with a sign that depends on $j$, provided that we exclude again
the ridiculous case when we require that $\varepsilon_j u_i \leq 0$
(because, if $\varepsilon_j g_j \geq 0$, we can be sure that $u_i=0$
in that case).

We are happy with \eqref{e12.2} for our initial setup with eigenfunctions,
because we intended to use such an $F$ anyway.
When $F$ is given by \eqref{e1.7}, \eqref{e12.2} requires all the $q_i$
to be constant, and this may be a little too much to ask. So we state a
second result, which will be proved together with 
Proposition \ref{t12.2}, and which is a little more flexible in this respect.

This time we select one index $i$ (and for convenience we will pick $i=1$)
and require something like a negative half derivative of $F$ in the direction
of that variable. More precisely, we shall assume that there exist $\varepsilon > 0$,
that may even depend on the minimizer $(\u,\W)$, such that
$F(\W') \leq F(W)$ for every choice of $\W' = (W_1, \ldots, W_N)$ 
such that
\begin{equation} 
\label{e12.4}
W_i' \i W_i \text{ and } | W_i \sm W'_i| \leq \varepsilon
\text{ for } 2 \leq i \leq N, 
\end{equation}
and
\begin{equation}
 \label{e12.5}
W'_1 = W_1 \bigcup \Big(\bigcup_{i \geq 2} (W_i \sm W'_i) \Big).
\end{equation}
That is, we transfer small pieces of the $W_i$, $i\geq 2$, into $W_1$;
notice that this gives $\W' \in {\cal W}(\Omega)$.

When $F(\W) = \sum_{i} \int_{W_i} q_i$ as in \eqref{e1.7}, this
property holds for all $\varepsilon$ as soon as $q_1 \leq \min(q_2, \ldots q_N)$ 
everywhere. When $F$ is given by \eqref{e12.2}, it holds as soon as
\begin{equation} \label{e12.6}
\wt F(V_1 + t_2 + \ldots t_N, V_2-t_2, \ldots, V_N-t_N)
\leq \wt F(V_1, \ldots, V_N)
\end{equation}
for some $\varepsilon > 0$ and all choices of $t_2, \ldots t_N \in [0,\varepsilon]$ 
such that $0 \leq t_i \leq V_i$ for $2 \leq i \leq N$. 

We add another requirement, that the reader probably implicitly assumed already,
that $F$ is insensitive to zero sets, in the sense that
$F(\W) = F(\W')$ when the $W_i$ coincide with the $W'_i$ almost everywhere,
i.e. when $|W_i \Delta W'_i| = 0$. Notice that this is contained in our continuity
assumptions. 

\begin{pro}
\label{t12.4} 
Assume that $\Omega$ is open, that the $f_i$ and $g_i$ are bounded,
that $(\u,\W)$ is a minimizer for $J$ in $\F$, and that $\W$ fills $\Omega$.
Further assume that $F$ is insensitive to zero sets and
we can find $\varepsilon > 0$ such that
$F(\W') \leq F(W)$ whenever $\W'$ satisfies \eqref{e12.4} and \eqref{e12.5},
and that
\begin{equation}
 \label{e12.7}
g_1(x) > 0  \ \text{ for almost every $x\in \Omega$,}
\end{equation}
and that the definition of $\F$ allows all the functions $u_i$ to take
positive values. Then $\u(x) \neq 0$ almost everywhere on $\Omega$.
\end{pro}

The same sort of comments as for Proposition \ref{t12.3} apply
here. In particular, if the hypothesis of Lemma \ref{t12.1} holds,
we don't need to assume that $\W$ fills $\Omega$.
The advantage of picking $i=1$ first is that we just need to check
that $g_1 > 0$.

We shall prove Propositions \ref{t12.3} and \ref{t12.4} at the same time,
and the idea will be to add a small bump function to one of the $u_i$
near a density point of $\{ \u = 0 \}$. But some surgery will be needed, 
so we shall first prove a lemma that applies to any minimizer (regardless
of our assumptions on $F$), and says that the energy of $\u$ decays
rather fast near such a point.

\begin{lem}
\label{t12.5} 
Assume that the $f_i$ and $g_i$ are bounded,
and that $(\u,\W)$ is a minimizer for $J$ in $\F$.
Then let $x_0\in \Omega$ be a Lebesgue density point of the
set $Z = \big\{x\in \R^n \, ; \, \u(x) = 0 \big\}$. Then
\begin{equation}
 \label{e12.8}
\lim_{r \to 0} r^{-n-2} \int_{B(x_0,r)} |\nabla u|^2 = 0.
\end{equation}
\end{lem}

\begin{proof}
Let $(\u,\W)$ and $x_0$ be as in the statement. Without loss
of generality, we assume that $x_0 = 0$.
The general idea is that by Poincar\'e, $\u$ should stay very small near 
the origin, and even smaller if $\fint_{B(x_0,r)} |\nabla u|^2$ is small;
then the $M$-term of the functional should only play a small role, and
in turn there is no reason for the energy to be large to compensate.

In practice we shall repeatedly test $(\u,\W)$ against the cut-off 
competitor of Section \ref{favorites}, and use this
to shows that $\fint_{B(x_0,r)} |\nabla u|^2$ decays rapidly.
We shall use the quantities
\begin{equation} 
\label{e12.9}
\theta_i(r) = r^{-n} \big| \big\{ x\in B(0,r) \, ; \, u_i(x) \neq 0 \big\} \big|
\end{equation}
and
\begin{equation} 
\label{e12.10}
\theta(r) = \sum_i \theta_i(r)
= r^{-n} \big| \big\{ x\in B(0,r) \, ; \, \u(x) \neq 0 \big\} \big|;
\end{equation}
notice that our assumption that $0$ is a Lebesgue density point
of $Z$ exactly means that $\lim_{r \to 0} \theta(r) = 0$.
We shall restrict our attention to radii $r$ so small that
\begin{equation} 
\label{e12.11}
\theta(r) \leq \eta^n,
\end{equation}
where the small number $\eta$ will be chosen soon.

Fix such an $r$, and apply the analogue of \eqref{e4.7} (with $p=2$) 
to the ball $B(0,r)$ (we observed before that the proof of \eqref{e4.7}
that we gave on spheres also work on balls). We get that
\begin{eqnarray} 
\label{e12.12}
\int_{B(0,r)} |u_i|^2 
&\leq& C \big| \big\{ x\in B(0,r) \, ; \, u_i(x) \neq 0 \big\} \big|^{2 \over n}
\int_{B(0,r)} |\nabla u_i|^2
\nonumber\\
&=& C r^2 \theta_i(r)^{2 \over n} \int_{B(0,r)} |\nabla u_i|^2
\leq C r^2 \eta^2 \int_{B(0,r)} |\nabla u_i|^2. 
\end{eqnarray}

Now consider the cut-off competitor 
$(\u^\ast,\W)$ described at the beginning of Section \ref{favorites}. 
We take $I = [1,N]$ (i.e, multiply all the $u_i$ by $\varphi(|x|)$, as in 
\eqref{e6.2}) and $a = 1/2$. 
Notice that we do not even need $B(0,r)$ to be contained in $\Omega$
for this one, because $u_i^\ast = 0$ whenever $u_i=0$.
We now estimate 
the terms that we get from \eqref{e6.5}-\eqref{e6.7}.
Let $\tau > 0$ be small, to be chosen soon; then 
\eqref{e6.5}, with $p=+\infty$, yields
\begin{eqnarray}
\label{e12.13}
\int_{B(0,r)} |\nabla u_{i}^\ast|^2 
&\leq&  (1+\tau) \int_{B(0,r) \sm B(0,ar)} |\nabla u_{i}|^2 
 + 4(1-a)^{-2} (1+\tau^{-1}) r^{-2} \int_{B(0,r) \sm B(0,ar)} |u_{i}|^2
 \nonumber
 \\
 &\leq& (1+\tau) \int_{B(0,r) \sm B(0,r/2)} |\nabla u_{i}|^2 
 +  C \tau^{-1} \eta^2 \int_{B(0,r)} |\nabla u_{i}|^2
  \nonumber
 \\
 &\leq& \int_{B(0,r) \sm B(0,r/2)} |\nabla u_{i}|^2
 + \Big(\tau + C \tau^{-1} \eta^2\Big) \int_{B(0,r)} |\nabla u_{i}|^2
\end{eqnarray}
or equivalently
\begin{equation} 
\label{e12.14}
\int_{B(0,r)} |\nabla u_{i}^\ast|^2 - \int_{B(0,r)} |\nabla u_{i}|^2 
\leq - \int_{B(0,r/2)} |\nabla u_{i}|^2 
+ \Big(\tau + C \tau^{-1} \eta^2 \Big) \int_{B(0,r)} |\nabla u_{i}|^2.
\end{equation}
For the two $M$-terms, \eqref{e6.6} yields
\begin{eqnarray} 
\label{e12.15}
\Big|\int_{\Omega} (u_{i}^\ast)^2 f_{i} - \int_{\Omega} u_{i}^2 f_{i}\Big|
&\leq& \int_{B} |u_{i}^2 f_{i}|
\leq ||f_{i}||_{\infty} \int_{B(0,r)} |u_{i}|^2
\nonumber\\
&\leq& C r^2 \eta^2 \int_{B(0,r)} |\nabla u_i|^2
\end{eqnarray}
by \eqref{e12.12}, and \eqref{e6.7} gives
\begin{eqnarray}
\label{e12.16}
\Big|\int_{\Omega} u_{i}^\ast g_{i} - \int_{\Omega} u_{i}g_{i}\Big|
&\leq& \int_{B(0,r)} |u_{i} g_{i}|
\leq ||g_{i}||_{\infty} \int_{B(0,r)} |u_{i}|
\leq C r^{n/2} \Big\{\int_{B(0,r)} |u_{i}|^2\Big\}^{1/2}
\nonumber
 \\
 &\leq& C  r \eta \, r^{n/2} \Big\{\int_{B(0,r)} |\nabla u_{i}|^2\Big\}^{1/2}.
\end{eqnarray}
Set $E(r) = \int_{B(0,r)} |\nabla \u|^2$ as usual; we sum \eqref{e12.15} and
\eqref{e12.16} over $i$ and get that
\begin{equation} 
\label{e12.17}
|M(\u^\ast) - M(\u)| 
\leq C r^2 \eta^2 E(r) + C  r  \eta \,  r^{n/2} E(r)^{1/2}
=: \pi(r),
\end{equation}
where the last part is the definition of $\pi(r)$.
There is no difference in the volume terms, because we did not
change $\W$, so the fact that $(\u,\W)$ minimizes $J$ yields
\begin{eqnarray}
\label{e12.18}
0 &\leq& J(\u^\ast,\W^\ast) - J(\u,\W) 
= \int_{B(0,r)} |\nabla \u^\ast|^2 - \int_{B(0,r)} |\nabla \u|^2  + M(\u^\ast) - M(\u) 
\nonumber
\\
&\leq& \int_{B(0,r)} |\nabla \u^\ast|^2 - \int_{B(0,r)} |\nabla \u|^2  
+ \pi(r)
\\
&\leq& - \int_{B(0,r/2)} |\nabla \u|^2
+ \Big(\tau + C \tau^{-1} \eta^2 \Big) \int_{B(0,r)} |\nabla \u|^2
+ \pi(r)
\nonumber
\end{eqnarray} 
by \eqref{e12.17} and \eqref{e12.14} (summed over $i$). That is,
\begin{eqnarray} 
\label{e12.19}
E(r/2) &\leq& \Big(\tau + C \tau^{-1} \eta^2 \Big) E(r) + \pi(r)
\nonumber
\\
&\leq& \Big(\tau + C \tau^{-1} \eta^2 \Big) E(r)
+ C  r  \eta \,  r^{n/2} E(r)^{1/2}
\end{eqnarray}
if $r \leq 1$, say, so that the first term of \eqref{e12.17} is smaller.

We now choose $\eta$ so small, depending on $\tau$ that will be chosen soon,
that \eqref{e12.19} implies that
\begin{equation} 
\label{e12.20}
E(r/2) \leq 2\tau E(r) + \tau r^{1+{n \over 2}} E(r)^{1/2}.
\end{equation}
Let us rewrite this in terms of $e(r) = r^{-n-2} E(r)$; we get that
\begin{eqnarray} 
\label{e12.21}
e(r/2) &=& 2^{n+2} r^{-n-2} E(r/2)
\nonumber
\\
&=& 2^{n+2} r^{-n-2} \big[2\tau E(r) + \tau r^{1+{n \over 2}} E(r)^{1/2} \big]
\nonumber
\\
&=& 2^{n+2} r^{-n-2} \big[2\tau r^{n+2}e(r) + \tau r^{1+{n \over 2}} 
r^{{n\over 2}+1}e(r)^{1/2} \big]
\\
&=& 2^{n+3} \tau e(r) + 2^{n+2} \tau e(r)^{1/2}
\nonumber
\end{eqnarray}
Thus, if we set $e_k = e(2^{-k}r)$, we get the induction relation 
$e_{k+1} \leq 2^{n+3} \tau e_k + 2^{n+2} \tau e_k^{1/2}$.
If $\tau$ is so small enough, then $2^{n+3} \tau < 1/4$, and
$e_{k+1} \leq e_k/4 + 2^{n+2} \tau e_k^{1/2}$.
If $e_k \geq 2^{2n+8} \tau^2$, then
$2^{n+2} \tau e_k^{1/2} \leq e_k/4$ and so
$e_{k+1} \leq e_k/2$. Otherwise,
$e_{k+1} \leq e_k/4 + 2^{n+2} \tau e_k^{1/2}
\leq 2^{2n+6} \tau^2 + 2^{2n+6}\tau \leq 2^{2n+7}\tau$.
It is then easy to see that $e_k \leq 2^{2n+7}\tau$ for $k$ large.
In the present situation, we can choose $\tau$ as small as we want, 
and then we get that
\begin{equation} 
\label{e12.22}
\lim_{r \to 0} r^{-n-2} E(r) = \lim_{r \to 0} e(r) = 0;
\end{equation}
this completes our proof of Lemma \ref{t12.5}.
\qed
\end{proof}

\ms
We are now ready to prove our two propositions.
We are given a minimizer $(\u,\W)$ of $J$, and we want
to prove that $\u \neq 0$  almost everywhere on $\Omega$,
so we proceed by contradiction, and assume that 
$\big\{ x\in \Omega \, ; \, \u(x) = 0 \big\}$ 
has positive measure. Then we can find a 
Lebesgue density point $x_0$ in that set, and we can even
choose it so that for $1 \leq i \leq N$,
\begin{equation} 
\label{e12.23}
\lim_{r \to 0} \fint_{B(x_0,r)} |g_i(x) - g_i(x_0)| dx = 0,
\end{equation}
because this Lebesgue density property for $g_i$ holds
for almost every $x_0$ (see \cite{M}),  
and
\begin{equation} 
\label{e12.24}
g_i(x_0) > 0 
\end{equation}
for all $i$ if we prove proposition \ref{t12.3}, or for $i=1$ only
if we prove proposition \ref{t12.4}. Without loss of generality,
we may assume that $x_0=0$ (just to save notation).

Now we want to try a slightly different competitor, with the promised
bump function. Let $ r> 0$ be given, and denote by $\u^\ast$ the cut-off 
competitor that we used in Lemma \ref{t12.5}. 
We want to use the fact that $\u^\ast = 0$ in $B(0,r/2)$
to modify $\u^\ast$ again in $B(0,r/2)$ (and in particular add a small
bump function to some $u_i$).

Let $\psi$ be a smooth, nonnegative bump function, with compact support
in $B(0,1/3)$ and $\int \psi = 1$. For the sake of Proposition \ref{t12.3}, 
we choose the support of $\psi$ a little smaller, so that 
\begin{equation} 
\label{e12.25}
\big|\big\{ x\in B(0,1/2) \, ; \, \psi(x) \neq 0 \big\}\big| \leq 
{1 \over N} |B(0,1/2)|.
\end{equation}
We first define a new function $\u^\sharp$. We first select an index $i$.
In the case of Proposition \ref{t12.3}, choose 
$i$ such that 
\begin{equation} 
\label{e12.26}
|W_i \cap B(0,r/2)| \geq {1 \over N} |B(0,r/2)|.
\end{equation}
For Proposition \ref{t12.4}, choose $i=1$.
Then define $\u^\sharp$ by
\begin{equation}
\label{e12.27}
\begin{array}{rll}
u^\sharp_i(x) =& u^\ast_i(x) + c r^2 \psi(x/r)  &\text{ for } x\in \R^n
\\
u^\sharp_j(x) =& u^\ast_j(x)  &\text{ for $x\in \R^n$ and } j \neq i,
\end{array}
\end{equation}
where the small constant $c$ will be chosen soon. 

We also need to define sets $W_j^\sharp$. For both propositions, we keep 
\begin{equation} 
\label{e12.28}
W_j^\sharp \sm B(0,r/2) = W_j^\ast \sm B(0,r/2) = W_j \sm B(0,r/2)
\end{equation}
for all $j$. 
For Proposition \ref{t12.3}, we want to keep the same volumes,
so we choose  the $W_j^\sharp \cap B(0,r/2)$ so that they are
disjoint, that $|W_j^\sharp \cap B(0,r/2)| = |W_j\cap B(0,r/2)|$ for
all $j$, and that $W_i^\sharp \cap B(0,r/2)$ contains 
$\big\{ x\in B(0,r/2) \, ; \, \psi(x/r) \neq 0 \big\}$. This is possible,
precisely by \eqref{e12.26} and because we chose the support of $\psi$ 
small enough in \eqref{e12.25}.

For Proposition \ref{t12.4}, the most efficient is to take 
$W_1^\sharp \cap B(0,r/2) = B(0,r/2)$, and 
$W_j^\sharp \cap B(0,r/2) = \emptyset$ for $j > 1$. 

It is easy to see that the $W_i^\sharp$ are disjoint,
that $u^\sharp \in W^{1,2}(\R^n)$, and that
$u^\sharp=0$ almost everywhere on $\R^n \sm W_i$.
If $r$ is small enough, $B(0,r) \i \Omega$, the $W_i^\sharp$
are contained in $\Omega$, and $(\u^\sharp,\W^\sharp) \in \F$;
this was the reason why we required $\Omega$ to be open.

Now we estimate the functional, starting with the volume term.
For Proposition \ref{t12.3}, we did not change the volumes $|W_i|$,
so by \eqref{e12.2} $F(\W^\sharp) = F(\W)$.
For Proposition \ref{t12.4}, we took some pieces of the $W_j$,
$j \geq 2$ and threw them inside $W_1$. 
In fact, we also threw the set $B(0,r/2) \setminus \cup_{i=1}^N W_i$
in $W_1$, but that since $\W$ fills $\Omega$, this set has vanishing measure.
Thus $\W^\sharp$ satisfies \eqref{e12.4} and \eqref{e12.5}, modulo
this set of measure zero that does not matter because we assumed that
$F$ is insensitive to zero sets, and if $r$ is so small
that $|B(0,r/2)| \leq \varepsilon$. Thus we get that
\begin{equation} 
\label{e12.29}
F(\W^\sharp) \leq F(\W)
\end{equation}
in this case too.

We also modified the $M$-term a little, because we added $c r^2 \psi(x/r)$
to $u_i$. But
\begin{eqnarray}
\label{e12.30}
M(\u^\sharp) - M(\u^\ast)  &=& \int_{B(0,r/2)} [ (u_i^\sharp)^2 f_i - u_i^\sharp g_i ]
= c^2 r^4 \int_{B(0,r/2)} \psi(x/r)^2 f_i - c r^2 \int_{B(0,r/2)} \psi(x/r) g_i
\nonumber
\\
&\leq& C c^2 r^{4+n} - c r^2 \int_{B(0,r/2)} \psi(x/r) g_i
\nonumber
\\
&\leq& C c^2 r^{4+n} - c r^2 g_i(0) \int_{B(0,r/2)} \psi(x/r)
+ c r^{2}  \int_{B(0,r/2)} ||\psi||_\infty |g_i(x)-g_i(0)|
\\
&=& Cc^2 r^{4+n} - c r^{2} g_i(0)
+ Cc r^{2+n} ||\psi||_\infty \fint_{B(0,r/2)} |g_i(x)-g_i(0)|
\nonumber
\end{eqnarray}
by \eqref{e1.4}, because $\u^\ast = 0$ in $B(0,r/2)$, and by \eqref{e12.27}.
For the energy term, 
\begin{equation}
\label{e12.31}
\int |\nabla \u^\sharp|^2 - \int |\nabla \u^\ast|^2
= c^2  r^4 \int_{B(0,r/2)} |\nabla \psi(\cdot/r)|^2
= c^2  r^{2+n} \int_{B(0,1/2)} |\nabla \psi|^2.
\end{equation}
Hence, by \eqref{e1.5} and because $F(\W^\sharp) \leq F(\W)$,
\begin{eqnarray} 
\label{e12.32}
J(\u^\sharp,\W^\sharp) - J(\u^\ast,\W^\ast) 
&\leq& M(\u^\sharp) - M(\u^\ast) + \int |\nabla \u^\sharp|^2 - \int |\nabla \u^\ast|^2
\nonumber
\\
&\leq&  - c r^{2+n} g_i(0) + \pi_1(r),
\end{eqnarray}
where 
\begin{equation} 
\label{e12.33}
\pi_1(r) = Cc^2 r^{4+n}
+ Cc r^{2+n} ||\psi||_\infty \fint_{B(0,r/2)} |g_i(x)-g_i(0)|
+ Cc^2  r^{2+n} \int_{B(0,1/2)} |\nabla \psi|^2.
\end{equation}

Recall from \eqref{e12.18} and \eqref{e12.17} that
if $r$ is so small that \eqref{e12.11} holds,
\begin{eqnarray} 
\label{e12.34}
J(\u^\ast,\W^\ast) - J(\u,\W)  
&\leq& - \int_{B(0,r/2)} |\nabla \u|^2
+ \Big(\tau + C \tau^{-1} \eta^2\Big) \int_{B(0,r)} |\nabla \u|^2
+ \pi(r)
\nonumber
\\
&\leq & \Big(\tau + C \tau^{-1} \eta^2 \Big) \int_{B(0,r)} |\nabla \u|^2
+ \pi(r)
\nonumber
\\
& = & \Big(\tau + C \tau^{-1} \eta^2\Big) E(r) + \pi(r)
\\
& \leq & \Big(\tau + C \tau^{-1} \eta^2\Big) E(r) + 
 C  r \eta \, r^{n/2} E(r)^{1/2} 
\nonumber
\end{eqnarray} 
(we also use the fact that $r\leq 1$ to control the first half of $\pi(r)$).
We no longer need to optimize too much, so let us choose $\tau = 1$ 
and even $\eta = 1$, and deduce from this that
\begin{equation} 
\label{e12.35}
J(\u^\ast,\W^\ast) - J(\u,\W) \leq C E(r) + C r r^{n/2} E(r)^{1/2}
= : \pi_2(r),
\end{equation}
and where the last part is a definition of $\pi_2(r)$. 
But $(\u,\W)$ is a minimizer for $J$, so 
$J(\u,\W) \leq J(\u^\sharp,\W^\sharp)$, which means that
\begin{equation} 
\label{e12.36}
c r^{2+n} g_i(0) \leq \pi_1(r) + \pi_2(r),
\end{equation}
by \eqref{e12.32} and \eqref{e12.35}.
Let us now check that we can choose $c > 0$ so small that \eqref{e12.36} fails
for $r$ small; this contradiction will prove that our initial assumption that 
$|\big\{ x\in \Omega \, ; \, \u(x) = 0 \big\}| > 0$
was false, and the proposition will follow.

The first term of $\pi_1(r)$ is $C c^2 r^{4+n} = o(r^{2+n})$. The second term is
$C c r^{2+n} \fint_{B(0,r/2)} |g_i(x)-g_i(0)| = o(r^{2+n})$, by \eqref{e12.23}.
The last term, $c^2  r^{2+n} \int_{B(0,1/2)} |\nabla \psi|^2$, is smaller than
$c r^{2+n} g_i(0) /2$ if $c$ is small enough. It does not matter that $c$ depends on
$g_i(0)$ or our choice of $\psi$. For $\pi_2(r)$, we observe that 
since $x_0 = 0$ lies in the open set $\Omega$ and is a density point for 
$\big\{ x\in \Omega \, ; \, \u(x) = 0 \big\}$,
we can apply Lemma \ref{t12.5}; we get that
$E(r) = o(r^{n+2})$; then $\pi_2(r) = o(r^{n+2})$ too,
which is much smaller than $c r^{2+n} g_i(0) /2$. 
So \eqref{e12.36} fails for $r$ small, and Propositions~\ref{t12.3} and \ref{t12.4}
follow.
\qed

\ms
Notice that we have no margin in our last estimates, i.e., the decay exponent in 
Lemma~\ref{t12.5} is just enough for our purposes. This probably means that since
the amount of $J$ that we can save by adding a small jump function is of higher
order, our boundedness assumption on the $f_i$ and $g_i$ are about right.

\section{Sufficient conditions for minimizers to be nontrivial}  
\label{number-p} 

In this section, we check that if the volume functional is defined by
\begin{equation} \label{ep1}
F(\W) = \sum_{i=1}^N a|W_i| + b|W_i|^{1+\alpha}
\end{equation}
for some $\alpha > 2/n$ and suitable constants
$a$ and $b$, 
depending on $\Omega$, the $f_i$, and the $g_i$, then the minimizers of 
our functional $J$ are not trivial, in the sense that at least one function
$u_i$ is nonzero, and $|W_i| < |\Omega|$ for all $i$.

For this we need some assumptions on $\Omega$, the $f_i$, and the $g_i$,
which we shall not try to optimize. Let us assume that \eqref{e3.1} and \eqref{e3.2}
hold. That is, $|\Omega| < +\infty$, the $f_i$ are bounded and nonnegative, and the
$g_i$ lie in $L^2$. With these assumptions and $F$ as in \eqref{ep1}, 
Theorem~\ref{t3.1} says that $J$ admits minimizers.

Our first result says that if $b$ is large enough and $(\u,\W)$ is a minimizer for $J$,
the $W_i$ cannot be too large.

\ms
\begin{lem} \label{tp1}
Suppose $\Omega$, the $f_i$, and the $g_i$ satisfy \eqref{e3.1} and \eqref{e3.2}.
For each choice of $\alpha \geq 0$ and $\eta > 0$, we can find $b_0 > 0$, that depends
only on $n$, $\alpha$, $|\Omega|$, the $||f_i||_\infty$, and the $||g_i||_2$, so that
if $F$ is given by \eqref{ep1} for some $a \geq 0$ and $b \geq b_0$,
 and $(\u,\W)$ is a minimizer for $J$, then
\begin{equation} \label{ep2}
|W_i| \leq \eta \ \text{ for } 1 \leq i \leq N.
\end{equation}
\end{lem}

\ms\begin{proof}
Let us first estimate the first two terms of $J(\u,\W)$. 
Let $(\u_0,\W_0)$ denote the trivial competitor for which $u_i = 0$ and 
$W_i = \emptyset$ for $1 \leq i \leq N$.
Then
\begin{equation} \label{ep3}
E(\u) + M(\u) \leq E(\u) + M(\u) + F(\W) = J(\u,\W) \leq J(\u_0,\W_0) = 0
\end{equation}
by \eqref{e1.3}-\eqref{e1.5} and because $F(\W) \geq 0$. Since $f_i \geq 0$
by \eqref{e3.2}, \eqref{e1.4} yields
\begin{equation} \label{ep4}
E(\u) \leq - M(\u) \leq \sum_{i=1}^N \int u_i(x) g(x) dx
\leq \sum_{i=1}^N ||g_i||_2 ||u_i||_2.
\end{equation}
By Lemma \ref{l3.2}, 
$||u_i||^2_2 \leq C |\Omega|^{2/n} \int |\nabla u_i|^2 \leq C |\Omega|^{2/n} E(\u)$, 
hence
\begin{equation} \label{ep5}
E(\u) \leq C |\Omega|^{1/n} E(\u)^{1/2} \sum_{i=1}^N ||g_i||_2,
\end{equation}
and so $E(\u) \leq C'$, where $C'$ depends on the data as above. 
We return to \eqref{ep3} and notice that for $1 \leq i \leq N$,
\begin{equation} \label{ep6}
b|W_i|^{1+\alpha} \leq F(\W) \leq - M(\u) \leq \sum_{i=1}^N ||g_i||_2 ||u_i||_2
\leq C |\Omega|^{1/n} E(\u)^{1/2} \sum_{i=1}^N ||g_i||_2
\leq C'',
\end{equation}
where $C''$ depends only on $n$, $|\Omega|$, the $||f_i||_\infty$, and the $||g_i||_2$.
The conclusion \eqref{ep2} follows easily (for $b \geq b_0$ and if $b_0$ is large enough).
\qed
\end{proof}

\ms
The next result says that if at least one of the $g_i$ is nontrivial and $a$ in \eqref{ep1}
is small enough, the minimizers for $J$ are nontrivial.

\ms
\begin{lem} \label{tp2}
Suppose $\Omega$, the $f_i$, and the $g_i$ satisfy \eqref{e3.1} and \eqref{e3.2}.
Suppose in addition that $\Omega$ is open (and non empty) and that $g_i \neq 0$ 
for some $i$. Then for each choice of parameters $\alpha > 2/n$ and $b > 0$, 
we can find $a_0 > 0$ such that if $F$ is given by \eqref{ep1} for some 
$a \in [0,a_0]$ and $(\u,\W)$ is a minimizer for $J$, then $\u \neq 0$.
\end{lem}

\ms
\begin{proof}
Here $a_0$ will also depend on $g_i$ in a more complicated way that its norm,
through the choice of a small ball where an average of $g_i$ is not too small.

For the proof we may assume that $i=1$. Let $x_0$ be a point
of $\Omega$ such that $g_1(x_0) \neq 0$ and $x_0$ is a Lebesgue point for $g_1$,
in the sense that
\begin{equation} \label{ep7}
\lim_{r \to 0} {1 \over r^n} \int_{B(x_0,r)} |g_1(x)-g_1(x_0)| dx = 0;
\end{equation}
such a point $x_0$ exists because \eqref{ep7} holds almost everywhere and
$g_1(x_0) \neq 0$ on a set of positive measure. As in Proposition \ref{t12.3},
we required $\Omega$ to be open so that small bump functions near $x_0$
yield competitors for $J$.

Let $\varphi$ be a smooth radial bump function, supported in $B(0,1)$ and
such that $\int \varphi = 1$, and for $r > 0$, set
\begin{equation} \label{ep8}
\varphi_r(x) = \beta r^2 \varphi((x-x_0)/r),
\end{equation}
where the small constant $\beta$
will be chosen near the end of the proof. If $r < \dist(x_0,\R^n \sm \Omega)$,
we can use $\varphi_r$ to define an admissible pair $(\u,\W) \in \F(\Omega)$.
That is, we take $u_1 = \varphi_r$, $W_1 = B(x_0,r)$, and for $i > 1$
we take $u_i = 0$ and $W_i = \emptyset$; it is then clear that $(\u,\W)$
satisfies the requirements of Definition \ref{d1.1}.

We just need to prove that $J(\u,\W) < 0$ (if $a_0$, $\beta$, and $r$
are chosen correctly), because then $\u=0$ cannot yield a minimizer, no matter
which choice of $\W$ we associate to it. So let us evaluate all the terms in $J(\u,\W)$.
We start with the energy
\begin{equation} \label{ep9}
E(\u) = \int |\nabla \varphi_t|^2 = \beta^2 r^2 ||\nabla \varphi||_2^2
\leq C \beta^2 r^{n+2},
\end{equation}
where we do not need to worry about the dependence of $C$ on $\varphi$.
Let us also record that
\begin{equation} \label{ep10}
F(\W) = a |B(x_0,r)| + b |B(x_0,r)|^{1+\alpha} \leq C a r^n + C b r^{n(1+\alpha)}
\end{equation}
The first part of $M(\u)$ is
\begin{equation} \label{ep11}
\sum_i \int u_i^2 f_i = \int \varphi_r f_1 \leq ||f_1||_\infty ||\varphi_r||_2^2
= ||f_1||_\infty \beta^2 r^4 r^n ||\varphi||_2^2
\leq C \beta^2 r^{n+4}.
\end{equation}
The remaining part of $M(\u)$ is
\begin{eqnarray} \label{ep12}
- \sum_i \int u_i g_i &=& - \int \varphi_r g_1 
\leq - g_1(x_0) \int \varphi_r +  \int \varphi_r(x) |g(x)-g(x_0)| dx 
\nonumber\\
&\leq& - \beta r^{n+2} g_1(x_0) 
+ \beta r^2 ||\varphi||_\infty  \int_{B(x_0,r)} |g(x)-g(x_0)| dx.
\end{eqnarray}
We shall take $\beta$ small, with the same sign as $g_1(x_0)$,
and we want the negative term $-\beta r^{n+2} g_1(x_0)$ in \eqref{ep12} to 
dominate all the other ones.
Let us now choose our parameters $r$, $\beta$, and $a_0$ so that this is the case.
For the second term $\beta r^2 ||\varphi||_\infty  \int_{B(x_0,r)} |g(x)-g(x_0)| dx$
of \eqref{ep12}, this happens as soon as $r$ is small enough, because of \eqref{ep7}.
For $C \beta^2 r^{n+4}$ in \eqref{ep11}, this is also true as soon as $r$ is small enough
(because we'll take $|\beta| \leq 1$). The term $C b r^{n(1+\alpha)}$ in \eqref{ep10}
can be treated the same way, because $\alpha > 2/n$. 

At this stage we choose $r$ so small that $B(x_0,r) \i \Omega$ and
the terms mentionned above are smaller than $\beta r^{n+2} g_1(x_0)/10$.
Then we choose $\beta$ small, so that $E(\u) < \beta r^{n+2} g_1(x_0)/10$,
and $a_0$ so small (hence depending on our choice of $r$) that
$C a_0 r^n < \beta r^{n+2} g_1(x_0)/10$. Now \eqref{ep9}-\eqref{ep12}
imply that $J(\u,\W) < 0$ as soon as $0 \leq a \leq a_0$, and the lemma follows.
\qed
\end{proof}

\section{A bound on the number of components}  
\label{number} 

This short section answers a natural question on the implementation
of our functional: even if we choose to allow a very large number $N$ of 
regions, will the functional naturally limit the number of indices $i$ such
that $u_i \neq 0$ somewhere?

We shall check that under reasonable assumptions on  $F$, 
there is a lower bound on the volume of $W_i$ when
$u_i(x) \neq 0$ somewhere. If $|\Omega|$ is assumed to be 
bounded, this will give the desired bound on the number of nontrivial
components $W_i$.

Let us state our main assumption for the index $i = 1$. 
We assume that there exist an exponent $1 \leq p < {n+2 \over n}$
and a constant $\lambda \geq 0$ such that
there exist disjoint subsets $A_1, A_2, \ldots, A_N$ of $W_1$ 
such that
\begin{equation} \label{en1}
 F(A_1, W_2 \cup A_2, \ldots, W_N \cup A_N)
\leq F(W_1, \ldots, W_N) - \lambda  |W_1|^{p}.
\end{equation}
Thus we have the right to take away a part of $W_1$, 
dispatch some of it among the other components, and this will make
the volume form somewhat smaller. A simple special case of this is when
\begin{equation} \label{en2}
 F(\emptyset, W_2, \ldots, W_N)
\leq F(W_1, \ldots, W_N) - \lambda |W_1|^{p},
\end{equation}
where $\lambda |W_1|^{p}$ is a minimum price that we had to pay for
the volume $|W_1|$. Even more specifically, \eqref{en1} holds if
\begin{equation} \label{en3}
F(W_1, \ldots, W_N) = \sum_{i=1}^N  a_i |W_i| + b_i |W_i|^2,
\end{equation}
with $a_i  \geq \lambda > 0$ and $b_i \geq 0$.

If we do not assume anything, i.e., if volume is too cheap, the functional may decide 
to have a tiny components $W_1$ (even if this is not very useful), win something
on the $M$-part of the functional, pay less in the energy term if $u_1$ is small enough
(the homogeneity of $\int |\nabla u_1|^2$ is higher than for $\int u_1 g_1$), 
and essentially not pay for it in the volume term.

We shall only need \eqref{en1} when $\W$ comes from the minimizer $(\u,\W)$
(and then we will get bounds on $|W_1|$ that depend on $p$, $\lambda$, and $n$),
but in general we do not expect to know $(\u,\W)$ in advance,
so we may need to require \eqref{en1} for all $\W \in {\cal W}(\Omega)$.

Our condition \eqref{en1} is a simple form of the main nondegeneracy assumption
that will be introduced in Section \ref{good}.

\begin{pro} \label{tn1}
Let $(\u,\W)$ be a minimizer for $J$, and suppose 
that $|W_1| > 0$ and that \eqref{en1} holds for some 
choice of $p \in [1, {n+2 \over n})$ and $\lambda > 0$. 
Also suppose that $f_1$ and $g_1$ are bounded. Then 
\begin{equation} \label{en4}
|W_1| \geq C^{-1}( || g_1||_\infty^{2n \over n+2-np}+ ||f_1||_\infty^{n/2})^{-1} ,
\end{equation}
with a constant $C$ that depends only on $n$, $p$, and $\lambda$.
\end{pro}

\begin{proof}  
The reader should not worry about the case when $f_1 = g_1= 0$;
the proof below will just show that it does not happen. 
We construct a simple competitor that will be compared to $(\u,\W)$.
Set $\u^\ast = (0, u_2, \ldots, u_N)$ and 
$\W^\ast = (A_1, W_2 \cup A_2, \ldots, W_N \cup A_N)$, where the
$A_i$ are the same as in \eqref{en1}. It is easy to see that
$(\u^\ast,\W^\ast)$ is an acceptable pair, i.e., lies in the class $\F(\Omega)$
of Definition \ref{d1.1};
simply notice that the $A_i$ are contained in $\Omega$ because they are contained 
in $W_1$. Thus $J(\u,\W) \leq J(\u^\ast,\W^\ast)$, and when we remove the
identical parts of $E(\u) + M(\u)$ and $E(\u^\ast) + M(\u^\ast)$, we are left with
\begin{equation} \label{en5}
\int_{W_1} |\nabla u_1|^2 + u_1^2 f_1 - u_1 g_1 + F(\W)
\leq F(\W^\ast) \leq F(\W) - \lambda |W_1|^{p},
\end{equation}
by \eqref{e1.4} and \eqref{en1}. Set $E = \int_{W_1} |\nabla u_1|^2$
and $M = \big|\int_{W_1} u_1^2 f_1 - u_1 g_1\big|$, and then apply 
the Poincar\'e inequality from Lemma \ref{l3.2}; we get that 
\begin{eqnarray} \label{en6}
M \leq ||f_1||_\infty \int_{W_1} u_1^2 + ||g_1||_\infty \int_{W_1} |u_1|
&\leq& C ||f_1||_\infty |W_1|^{2/n} E + C ||g_1||_\infty |W_1|^{1/2} ||u_1||_2
\nonumber\\
&\leq& C ||f_1||_\infty |W_1|^{2/n} E + C ||g_1||_\infty |W_1|^{n+2 \over 2n} E^{1/2}
\end{eqnarray}
and \eqref{en5} yields
\begin{equation} \label{en7}
E + \lambda |W_1|^{p} \leq  M \leq  C ||f_1||_\infty |W_1|^{2/n} E 
+ C ||g_1||_\infty |W_1|^{n+2 \over 2n} E^{1/2}.
\end{equation}
If $C ||f_1||_\infty |W_1|^{2/n} \geq 1/2$, we are happy because \eqref{en4} holds.
Otherwise we simplify \eqref{en7} and get that 
\begin{equation} \label{en8}
E + 2\lambda |W_1|^{p} \leq 2C ||g_1||_\infty |W_1|^{n+2 \over 2n} E^{1/2}.
\end{equation}
Set $\alpha = C ||g_1||_\infty |W_1|^{n+2 \over 2n}$; then
$E - 2 \alpha E^{1/2} = (E^{1/2} - \alpha)^2 - \alpha^2 \geq -\alpha^2$,
so \eqref{en8} implies that $2\lambda |W_1|^p \leq 2 \alpha E^{1/2} - E \leq \alpha^2$, 
and hence $2 \lambda \leq |W_1|^{-p} \alpha^2 \leq C^2 ||g_1||_\infty^2 
|W_1|^{{n+2 \over n}-p}$. Then \eqref{en4} holds, and the proposition follows.
\qed
\end{proof}

\ms
In fact the proof of Proposition \ref{tn1} shows that the domain $W_1$
is not too thin, in the sense that its Poincar\'e constant is fairly large. That
is, denote by $V(W_1)$ the smallest constant such that
\begin{equation} \label{en9}
\int_{W_1} |u|^2 \leq V(W_1)^{2/n} \int_{W_1} |\nabla u|^2
\end{equation}
for every function $u\in W^{1,2}(\R^n)$ such that $u(x) = 0$ almost
every $x\in \R^n \sm W_1$. We make $V(W_1)$ scale like a volume to simplify
the computations below. 
Then we have the same bounds as above for $V(W_1)$.

\begin{cor} \label{tn2}
Let $(\u,\W)$ and $W_1$ be as in Proposition \ref{tn1}; then
\begin{equation} \label{en10}
V(W_1) \geq C^{-1}( || g_1||_\infty^{2n \over n+2-np}+ ||f_1||_\infty^{n/2})^{-1} ,
\end{equation}
with a constant $C$ that depends only on $n$, $p$, and $\lambda$.
\end{cor}

\ms
\begin{proof}  
Observe that
\begin{equation} \label{en11}
V(W_1) \leq C |W_1|,
\end{equation}
for instance by Lemma \ref{l3.2}; thus \eqref{en10} is better than \eqref{en4}.
But let us follow the argument above. The proof of \eqref{en6} yields
\begin{equation} \label{en12}
M \leq  ||f_1||_\infty V(W_1)^{2/n} E 
+ C ||g_1||_\infty |W_1|^{1/2} V(W_1)^{1/n} E^{1/2}.
\end{equation}
Then \eqref{en7} becomes
\begin{equation} \label{en13}
E + \lambda |W_1|^{p} \leq  M \leq  C ||f_1||_\infty  V(W_1)^{2/n} E 
+ C ||g_1||_\infty |W_1|^{1/2} V(W_1)^{1/n} E^{1/2}
\end{equation}
If $C ||f_1||_\infty V(W_1)^{2/n} \geq 1/2$, we are happy 
as before, and otherwise we are left with
\begin{equation} \label{en14}
E + 2\lambda |W_1|^{p} \leq   2C ||g_1||_\infty |W_1|^{1/2} V(W_1)^{1/n} E^{1/2}.
\end{equation}
Then we set $\alpha = C ||g_1||_\infty |W_1|^{1/2} V(W_1)^{1/n}$ and the computation
above yields
\begin{equation} \label{en15}
2\lambda |W_1|^{p} \leq 2\alpha E^{1/2} - E \leq \alpha^2 
= C^2 ||g_1||_\infty^2 |W_1| \, V(W_1)^{2/n}.
\end{equation}
We assumed that $p \geq 1$, because it did not disturb and now we can say that
\begin{equation} \label{en16}
C^2 ||g_1||_\infty^2 V(W_1)^{2/n} \geq 2\lambda |W_1|^{p-1}
\geq C^{-1} \lambda  V(W_1)^{p-1}
\end{equation}
and conclude as before. 
\qed
\end{proof}

For the initial goal of the section, it is important to notice that
our constant $C$ does not depend on $N$. So, if $|\Omega| < +\infty$
the $f_i$ and the $g_i$ are bounded, and the analogue of $\eqref{en1}$
holds for all $i$ (all this with constants that do not depend on $N$), we 
get a bound on the number of indices $i$ such that $|W_i| > 0$.

\begin{rem} \label{tn3}
In Section \ref{good} we will get more precise lower bounds 
on the measure of $W_i \cap B$ when $B$ is s small ball centered
on the boundary of $\Omega_1 = \big\{ u_1(x) > 0 \big\}$,
but the more complicated proofs will make the constants depend on $N$.
Thus it seems that we cannot use them to get easily the result of this section. 
On the contrary, because of the result of this section, we can apply the
results of Section \ref{good} to an equivalent minimizer with $N$ bounded,
and get that the constants of Section \ref{good} do not depend on $N$.
\end{rem}

\section{The main non degeneracy condition; good domains}  
\label{good} 

In this section we return to one of the main schemes of the study
of free boundaries and give a a sufficient condition for the positive part
of the function $u_i$ associated to a minimizer $(\u,\W)$
to behave like the distance to the boundary of 
$\Omega_i = \big\{ x\in \Omega \, ; \, u_i(x) > 0 \big\}$.

This is a condition on $F$, which essentially says that we can remove any small part
$A$ of $W_i$, distribute some of it among the other $W_j$, and win in $F$ an amount 
which is proportional to $|A|$, and then all sorts of useful non degeneracy properties
for $u_i$ and $\Omega_i$ follow. 
See Theorems~\ref{t13.1}, \ref{t13.2},  \ref{t13.3}, and  \ref{t13.4} below.

But let us first describe this condition. Without loss of generality,
we shall restrict our attention to $i=1$.
We say that $i=1$ is a good index, or (with a small abuse of notation)
that $W_1$ is a \underbar{good domain}, when there exist 
$\lambda > 0$ and $\varepsilon > 0$ such that, 
for each measurable set $A \i W_1$, with $0 < |A| \leq \varepsilon$,
we can find disjoint subsets $A_j \i A$, $2 \leq j \leq N$, such that 
\begin{equation} \label{e13.1}
F(W_1\sm A, W_2 \cup A_2, \ldots, W_N \cup A_N) \leq F(\W) - \lambda |A|.
\end{equation}
Again, this means that if we have a small set $A \i W_1$, and we can somehow 
dispense with it (typically, because it is not very useful for making $E(\u)+M(\u)$
small, we can then give some of it to the other regions $W_j$, $j \geq 2$, 
throw out the rest, and we will win something substantial in the $F$-term of the functional. 
We will see that for good regions of a minimizer, due to the fact that the whole set
$W_1$ is really needed, we have some additional regularity
properties of $u$ near the free boundary $\d\big(\big\{ u_1 \neq 0 \big\}\big)$).

Here are some simple sufficient conditions for $W_1$ to be a good region.
First assume that $F(\W)$ is a function of $\V = (|W_1|, \dots , |W_N|)$, i.e., that
$F(W_1, \dots , W_N) = \wt F(|W_1|, \dots , |W_N|)$
for some function $\wt F : [0,|\Omega|]^N \to \R$, as in \eqref{e12.3}.
If $\wt F$ is differentiable at the point $\V$ (coming from the minimizer
$(\u,\W)$); then \eqref{e13.1}
holds (for some choice of $\lambda$ and $\varepsilon$) as soon as 
\begin{equation} \label{e13.2}
{\d \wt F\over \d V_1}(\V) > \inf\Big(0, {\d \wt F\over \d V_2}(\V), \ldots,
{\d \wt F\over \d V_N}(\V)\Big).
\end{equation}
Even more specifically, if $F((W_1, \dots , W_N) = \sum_{1 \leq i \leq N} F_i(|W_i|)$,
and each $F_i$ is differentiable at $|W_i|$, then \eqref{e13.1} and \eqref{e13.2} hold
as soon as 
\begin{equation} \label{e13.3}
F'_1(|W_1|) > 0 \ \text{ or } \ 
F'_1(|W_1|) > F'_j(|W_j|) \ \text{ for some } j > 1.
\end{equation}
In the more familiar context of Alt, Caffarelli, and Friedman where 
$F(\W) = \sum_i \int_{W_i} q_i$ as in \eqref{e1.7}, \eqref{e13.1} holds as soon as
\begin{equation} \label{e13.4}
q_1(x) \geq \lambda + \min(0, q_2(x), \ldots, q_N(x))
\ \text{ almost everywhere on } \Omega
\end{equation}
(and often the $q_i$ are nonnegative and the condition becomes $q_1 \geq \lambda$).
Notice that if $q_1 \geq \lambda$ everywhere, we do not need to compare $q_1$
with the other $q_i$.

We state two similar nondegeneracy results now, which we distinguish because they have slightly
different assumptions, then prove them, and then state and prove other ones.

\begin{thm}\label{t13.1} 
Let $(\u,\W)$ is a minimizer in $\F$ of the functional $J$, suppose that 
the data $f_1$ and $g_1$ are bounded, and that
\eqref{e13.1} holds for some choice of $\lambda>0$ and $\varepsilon>0$.
Also let $x\in \R^n$ and $r > 0$ be such that $\u$ is continuous on $B(x,r)$ and 
\begin{equation} \label{e13.5}
r \leq \min(1,\varepsilon^{1/n})  \text{ and $B(x,r/2)$ meets the boundary of  }
\Omega_1 = \big\{ x\in \Omega \, ; \, u_1(x) > 0 \big\}.
\end{equation}
Then
\begin{equation} \label{e13.6}
\fint_{B(x,r)} |u_{1,+}|^2 \geq c_1 r^2,
\end{equation}
where we set $u_{1,+} = \max(0,u_1)$, and
\begin{equation} \label{e13.7}
|\Omega_1 \cap B(x,r)| \geq c_2 r^n,   
\end{equation}
with constants $c_1 > 0$ and $c_2 > 0$ that depend only on
$n$, $\lambda$, $||f_1||_\infty$, $||g_1||_\infty$, and an upper bound for
$\fint_{B(x,r)} |\nabla u_{1,+}|^2$.
\end{thm}  

\ms
The rest of the paper is full of sufficient conditions for $\u$
to be continuous (and anyway we just ask this so that we can talk about
the open set $\Omega_1$), and ways to estimate $\fint_{B(x,r)} |\nabla u_1|^2$;
we state the theorem like this to stress the small amount of information
that we will use.

\begin{thm}\label{t13.2} 
Let $(\u,\W)$ is a minimizer in $\F$ of the functional $J$, suppose that 
$f_1$ and $g_1$ are bounded and that \eqref{e13.1} holds for some choice 
of $\lambda>0$ and $\varepsilon>0$.
Let $x\in \R^n$ and $r > 0$ be such that $\u$ is continuous on $B(x,r)$,
that \eqref{e13.5} holds, and that for some $C_0 > 0$, either
\begin{equation} \label{e13.8}
\text{$u_1$ is $C_0$-Lipschitz on $B(0,r)$,}
\end{equation}
or 
\begin{equation}\label{e13.9}
|B(0,r) \sm \Omega_1| \geq C_0^{-1} r^n,
\end{equation}
or both. Then 
\begin{equation} \label{e13.10}
\fint_{B(x,r)} |\nabla u_{1,+}|^2 \geq c_3.
\end{equation}
The constant $c_3$ depends only on $n$, $\lambda$, $||f_1||_\infty$, $||g_1||_\infty$, 
and $C_0$.
\end{thm}

\ms
These two results are standard in the context of Alt, Caffarelli, and Friedman; 
our proof is slightly different because we try to rely more on the cut-off competitors, 
but it is not surprising. Even the assumption \eqref{e13.8} will not cost us much in practice,
because we are ready to assume that $\Omega$ is bounded and $C^{1+\alpha}$ and apply
Theorem \ref{t11.1}, and also because the conclusions of Theorem \ref{t13.2} 
are easier to use when we know that $\u$ is Lipschitz. In addition, \eqref{e13.9}
would automatically hold if $x\in \d\Omega$ and $r$ is small, under very weak assumptions 
on $\Omega$. Other than that, we shall not use the geometry of $\Omega$ in this section.

We shall prove the two theorems at the same time.
The main ingredient will be a comparison with the cut-off competitor, which we shall
use repeatedly with sometimes slightly different estimates.

We shall start the proof with any ball $B(y,r)$, which may be different from the ball of 
the theorem, and we shall only assume that
\begin{equation} \label{e13.11}
r \leq 1 \ \text{ and } |B(y,r/2)| \leq \varepsilon,
\end{equation}
where $\varepsilon$ still comes from \eqref{e13.1}.

For simplicity, we shall assume that $y=0$. We choose the following
the cut-off competitor: we select the first index $i=1$, take $a=1/2$, 
define $\varphi$ as in \eqref{e6.1}, and denote by $\u^\ast$ the function 
defined by $u_1^\ast(x) = \varphi(|x|) u_1(x)$ for $x\in \Omega_1$,
$u_1^\ast(x) = u_1(x)$ for $x\in \R^n \sm \Omega_1$,
and $u_j^\ast = u_j$ for $j \geq 2$. Another equivalent definition of $u_1^\ast$ is
$u_1^\ast(x) = \min(u_1(x),\varphi(|x|) u_1(x))$ (because $0 \leq \varphi \leq 1$),
from which the fact that $u_1^\ast \in W^{1,2}(\R^n)$ is more obvious.

To  define $\W^\ast$, we use the fact that $u_1 = 0$ on $\Omega_1 \cap B(0,r/2)$ 
to change the $W_i$ to our advantage. Since \eqref{e13.11} holds, we can set
$A = \Omega_1 \cap W_1 \cap B(0,r/2)$ and choose $A_2, \ldots A_N$ 
so that \eqref{e13.1} holds (we intersect with $W_1$ because formally
$\Omega_1$ is only almost everywhere contained in $W_1$).
Then we set $\W^\ast = (W_1\sm A, W_2 \cup A_2, \ldots, W_N \cup A_N)$,
and \eqref{e13.1} says that
\begin{equation} 
\label{e13.12}
F(\W^\ast) \leq F(\W) - \lambda |A| = F(\W) - \lambda |\Omega_1 \cap B(0,r/2)|.
\end{equation}
As in Section \ref{favorites}, it is easy to see that $(\u^\ast,\W^\ast) \in \F$,
the integrals $\int |\nabla u_j|^2$, $j \geq 2$, stay the same, and for $j=1$,
$\int_{\R^n \sm \Omega_1} |\nabla u_1|^2$ stays the same, while
the analogue of \eqref{e6.5} for $u_{1,+}$, with $a=1/2$ says that
\begin{eqnarray}\label{e13.13}
\int_{\Omega_1 \cap B(0,r)} |\nabla u_{1}^\ast|^2 
&=& \int_{B(0,r)} |\nabla u_{1,+}^\ast|^2
\nonumber\\
&\leq& (1+\tau) \int_{B(0,r) \sm B(0,r/2)} |\nabla u_{1,+}|^2 
 + 16 (1+\tau^{-1})\, r^{-2} \int_{B(0,r) \sm B(0,r/2)} |u_{1,+}|^2,
\end{eqnarray}
where we can choose $\tau > 0$ as we like. That is,
\begin{eqnarray}\label{e13.14}
\int_{B(0,r)} |\nabla u_{1}^\ast|^2 - \int_{B(0,r)} |\nabla u_{1}|^2
&\leq&  - \int_{B(0,r/2)} |\nabla u_{1,+}|^2 
+\tau \int_{B(0,r)\sm B(0,r/2)} |\nabla u_{1,+}|^2 
\nonumber\\
&\,& \hskip 2.5cm
+ 16 (1+\tau^{-1}) \, r^{-2} \int_{B(0,r) \sm B(0,r/2)} |u_{1,+}|^2.
\end{eqnarray}
Then we estimate the $M$-terms. Notice that only $u_{1,+}$ changes, so
\begin{eqnarray}\label{e13.15}
|M(\u^\ast) - M(\u)| &=&  \Big|\int [(u^\ast_1)^2 f_1 - u_1^2 f_1 
- u^\ast_1 g_1 + u_1 g_1] \Big|
\nonumber\\
&\leq& ||f_1||_\infty \int_{B(0,r)} u_{1,+}^2 
+ ||g_1||_\infty \int_{B(0,r)} u_{1,+}
\nonumber\\
&\leq& C \int_{B(0,r)} u_{1,+}^2 
+ C r^{n/2} \Big\{ \int_{B(0,r)} u_{1,+}^2 \Big\}^{1/2}
\end{eqnarray}
because $0 \leq |u_1^\ast| \leq |u_1|$ everywhere and by Cauchy-Schwarz.
Since 
\begin{eqnarray}\label{e13.16}
0 &\leq& J(\u^\ast,\W^\ast) - J(\u,\W)
\nonumber\\
&=& \int_{B(0,r)} |\nabla u_{1}^\ast|^2 - \int_{B(0,r)} |\nabla u_{1}|^2
+ M(\u^\ast) - M(\u) - F(\W^\ast) - F(\W)
\end{eqnarray}
by minimality, we deduce from \eqref{e13.12}, \eqref{e13.14}, and \eqref{e13.15} that
\begin{equation}\label{e13.17}
\int_{B(0,r/2)} |\nabla u_{1,+}|^2 + \lambda |\Omega_1 \cap B(0,r/2)| \leq \alpha(r),
\end{equation}
where 
\begin{eqnarray}\label{e13.18}
\alpha(r)
&=&  \tau \int_{B(0,r) \sm B(0,r/2)} |\nabla u_{1,+}|^2 
+ 16 (1+\tau^{-1}) r^{-2} \int_{B(0,r) \sm B(0,r/2)} u_{1,+}^2
+ |M(\u^\ast) - M(\u)|
\nonumber\\
&\leq& \tau \int_{B(0,r)} |\nabla u_{1,+}|^2 
+ [16 (1+\tau^{-1})  + Cr^2 ] \int_{B(0,r)} r^{-2} u_{1,+}^2
+ C r^{n/2} \Big\{ \int_{B(0,r)} u_{1,+}^2 \Big\}^{1/2}
\nonumber\\
&\leq& \tau \int_{B(0,r)} |\nabla u_{1,+}|^2 
+ C \tau^{-1} \int_{B(0,r)} r^{-2} u_{1,+}^2
+ C r r^{n/2} \Big\{ \int_{B(0,r)} r^{-2} u_{1,+}^2 \Big\}^{1/2}
\end{eqnarray}
because we shall take $\tau \leq 1/2$, and since $r \leq 1$.
It will be easier to use this when \eqref{e13.9} holds, because then 
the analogue for the ball $B(0,r)$ of the Poincar\'e estimate \eqref{e4.6} yields
\begin{equation}\label{e13.19}
\int_{B(0,r)} r^{-2} u_{1,+}^2 \leq C C_0 \int_{B(0,r)} |\nabla u_{1,+}|^2
\end{equation}
and \eqref{e13.18} implies that
\begin{equation}\label{e13.20}
\alpha(r) \leq (\tau + C C_0 \tau^{-1}) \int_{B(0,r)} |\nabla u_{1,+}|^2
+ C C_0^{1/2} r^{{n \over 2}+1} \Big\{ \int_{B(0,r)} |\nabla u_{1,+}|^2 \Big\}^{1/2}
\end{equation}
when \eqref{e13.9} holds. Similarly, set 
\begin{equation}\label{e13.21}
\theta(r) = r^{-n} |\Omega_1 \cap B(0,r)| = 
r^{-n}\big|\big\{ x\in B(0,r) \, ; \, u_1(x) > 0 \big\}\big|
\end{equation}
almost as in \eqref{e12.9}; when $\theta(r)$ is small, we can rather use 
the analogue of \eqref{e4.7} for the ball $B(0,r)$, which yields
\begin{equation}\label{e13.22}
\int_{B(0,r)} r^{-2} u_{1,+}^2 \leq C \theta(r)^{2/n} \int_{B(0,r)} |\nabla u_{1,+}|^2.
\end{equation}
In this case, it is also our interest to revise our application of Cauchy-Schwarz
in \eqref{e13.15}, because we can say that
\begin{eqnarray}\label{e13.23}
||g_1||_\infty \int_{B(0,r)} u_{1,+}
&\leq& C |\Omega_1 \cap B(0,r)|^{1/2} \Big\{ \int_{B(0,r)} u_{1,+}^2 \Big\}^{1/2}
\nonumber\\
&\leq& C \theta(r)^{{1 \over 2} + {1 \over 2n}} \,  r^{{n \over 2}+1}
\Big\{ \int_{B(0,r)} |\nabla u_{1,+}|^2 \Big\}^{1/2}.
\end{eqnarray}
Thus we can win an extra $\theta(r)^{1/2}$ in our estimate, and deduce from
the proof of \eqref{e13.20} that
\begin{equation}\label{e13.24}
\alpha(r) \leq \big(\tau + C \tau^{-1}\theta(r)^{2/n}\big) \int_{B(0,r)} |\nabla u_{1,+}|^2
+ C \theta(r)^{{1 \over 2} + {1 \over 2n}} \, r^{{n \over 2}+1} 
\Big\{ \int_{B(0,r)} |\nabla u_{1,+}|^2 \Big\}^{1/2}.
\end{equation}
To be honest, we don't really need this extra power if we are willing
to take $r$ small, but we shall use it to show that we don't need the extra power 
of $r$ that we also get in the last term of \eqref{e13.24}, and which we could use 
to make things small. Set 
\begin{equation}\label{e13.25}
e(r) = r^{-n} \int_{B(0,r)} |\nabla u_{1,+}|^2
\end{equation}
to simplify our discussion, and let us show that there exists a constant
$c_4 > 0$ (that depends only on $n$, $||f_1||_\infty$, and $||g_1||_\infty$)
such that
\begin{equation}\label{e13.26}
e(r/2)+\theta(r/2) \leq {1\over 2} \big(e(r) + \theta(r) \big)
\ \text { if $\theta(r) \leq c_4$.}
\end{equation}
Indeed, we can use \eqref{e13.24}, and then
\begin{eqnarray}\label{e13.27}
e(r/2)+\theta(r/2) 
&=& 2^n r^{-n} \Big\{ \int_{B(0,r/2)} |\nabla u_{1,+}|^2 + |B(0,r) \cap \Omega_1| \Big\}
\leq 2^n r^{-n} \lambda^{-1} \alpha(r)
\nonumber\\
&\leq&  C \lambda^{-1} (\tau + C \tau^{-1}\theta(r)^{2/n}) e(r)
+ C \lambda^{-1} \theta(r)^{{1 \over 2} + {1 \over 2n}} r e(r)^{1/2}
\end{eqnarray}
by \eqref{e13.17}, because we can safely assume that $\lambda \leq 1$, and 
by \eqref{e13.24}. We drop the extra power of $r$, 
use the assumption that $\theta(r) \leq c_4$, and get that
\begin{equation}\label{e13.28}
e(r/2)+\theta(r/2) 
\leq C \lambda^{-1} \tau e(r)
+ C \lambda^{-1} \tau^{-1} c_4 ^{2/n}  e(r)
+ C \lambda^{-1} c_4 ^{1/2n} (\theta(r) e(r))^{1/2}.
\end{equation}
We choose $\tau$ so small (depending on $\lambda$) 
that $C \lambda^{-1} \tau \leq 1/8$, and choose $c_4$,
depending on $\tau$ and $\lambda$, so small that \eqref{e13.28} yields
\begin{equation} \label{e13.29}
e(r/2)+\theta(r/2) \leq {1 \over 4} e(r) + {1 \over 4}(\theta(r) e(r))^{1/2}
\leq {1 \over 4} e(r) + {1 \over 4}(e(r) + \theta(r)),
\end{equation}
(because $ab \leq a^2+b^2$), as needed for \eqref{e13.26}. Observe now that
\begin{equation} \label{e13.30}
u_{1,+}(0)= 0 \ \text { if $\theta(r) \leq c_4$ for $r$ small enough,}
\end{equation}
because iterations of \eqref{e13.26} imply that 
$\lim_{k \to +\infty}\theta(2^{-k} r) = 0$, which is impossible
if $u_1(0) > 0$ (recall that $u$ is continuous).

\ms
We may now return to the proof of our two theorems. Let $B(x,r)$ be as in 
any of the two statements, and first assume that 
$r \leq \min(1,\varepsilon^{1/n})$ but \eqref{e13.6} fails.
We want to show that $u_1(y) \leq 0$ for every $y\in B(x,r/2)$,
and as before we may assume for simplicity that $y=0$.
We want to apply the estimates above to the ball
$B(y,r_1)$, with $r_1 = (2 \sqrt n)^{-1} r$. 
We choose this strange formula for $r_1$ just to make sure
that $|B(y,r_1/2)| \leq \varepsilon$ since $r^n \leq \varepsilon$
(as in our assumptions). This way \eqref{e13.11} holds for $r_1$, we can
use the estimates above, and we get that
\begin{equation} \label{e13.31}
e(r_1/2)+\theta(r_1/2) \leq 2^n r_1^{-n} \lambda^{-1} \alpha(r_1)
\end{equation}
by the first part of \eqref{e13.27}. Then, by \eqref{e13.18},
\begin{eqnarray} 
\label{e13.32}
r_1^{-n} \alpha(r_1) &\leq& C \tau \fint_{B(0,r_1)} |\nabla u_{1,+}|^2 
+ C \tau^{-1} \fint_{B(0,r_1)} r^{-2} u_{1,+}^2
+ C r \Big\{ \fint_{B(0,r_1)} r^{-2} u_{1,+}^2 \Big\}^{1/2}
\nonumber\\
&\leq&
C \tau \fint_{B(x,r)} |\nabla u_{1,+}|^2 
+ C \tau^{-1} \fint_{B(x,r)} r^{-2} u_{1,+}^2
+ C r \Big\{ \fint_{B(x,r)} r^{-2} u_{1,+}^2 \Big\}^{1/2}
\end{eqnarray}
Set $\Lambda = \fint_{B(x,r)} |\nabla u_{1,+}|^2$ (recall that we are allowed
to choose $c_1$ depending on $\Lambda$), and use the fact that
\eqref{e13.6} fails. This gives
\begin{equation} 
\label{e13.33}
r_1^{-n} \alpha(r_1) \leq C \tau  \Lambda + C \tau^{-1} c_1 + C r c_1^{1/2}
\end{equation}
and, by \eqref{e13.31}, 
$e(r_1/2)+\theta(r_1/2) \leq C \lambda^{-1} [\tau \Lambda + \tau^{-1} c_1 + r c_1^{1/2}]$.
If we choose $\tau > 0$ small enough, then $c_1$ even smaller, this implies
that $e(r_1/2)+\theta(r_1/2) \leq c_4$.
Then we can apply \eqref{e13.26} to the radius $r_1/2$, and get
that $e(r_1/4)+\theta(r_1/4) \leq c_4/2$. We iterate the argument
and find that $e(2^{-k} r_1)+\theta(2^{-k}r_1) \leq 2^{-k+1}c_4$
for $k \geq 1$, hence $\lim_{\rho \to 0} \theta(\rho) = 0$.
And now \eqref{e13.30} says that $u_1(y) \leq 0$
(recall that we assumed that $y=0$ for simplicity).

Thus we proved that if $r \leq \min(1,\varepsilon^{1/n})$ but \eqref{e13.6} fails, 
$u_1 \leq 0$ on $B(x,r/2)$. This is impossible under our assumption \eqref{e13.5},
so \eqref{e13.6} holds.

\ms 
We now prove \eqref{e13.7} similarly. Suppose that $r \leq \min(1,\varepsilon^{1/n})$
but \eqref{e13.7} fails, pick any $y\in B(x,r/2)$, and assume by translation invariance
that $y = 0$. Let $k \geq 0$ be an integer, that will be chosen soon in terms
of $\Lambda = \fint_{B(x,r)} |\nabla u_{1,+}|^2$. 
Set $r_1 = (2 \sqrt n)^{-1} r$ as above, and observe that for $0 \leq j \leq k$,
\begin{equation} 
\label{e13.34}
\theta(2^{-j} r_1) = 2^{nj} r_1^{-n} |\Omega_1 \cap B(0,2^{-j}r_1)| 
\leq C 2^{nj} r^{-n} |\Omega_1 \cap B(x,r)| \leq C 2^{nk} c_2 < c_4
\end{equation}
by \eqref{e13.21}, because \eqref{e13.7} fails, and if $c_2$ is chosen small
enough (depending on $k$).

Thus we can apply \eqref{e13.26} to the radii $2^{-j}r_1$,
 $0 \leq j \leq k$, and we get that 
\begin{equation} 
\label{e13.35}
\theta(2^{-k-1} r_1) + e(2^{-k-1} r_1)
\leq 2^{-k-1} (\theta(r_1) + e(r_1))
\leq 2^{-k-1} (c_4 + C \Lambda) < c_4,
\end{equation}
if $k$ is chosen large enough. Starting from that point, we can 
use \eqref{e13.26} to prove by induction that
$\theta(2^{-j} r_1) + e(2^{-j} r_1) \leq 2^{-j} (c_4 + C \Lambda) 
< c_4$ for $j \geq k+1$, as we did below \eqref{e13.33}.
Again \eqref{e13.30} says that $u_1(y) \leq 0$, and we can conclude
as above. This concludes our proof of Theorem \ref{t13.1}.
\qed

\ms
Now we switch to Theorem \ref{t13.2}, and start under the assumption \eqref{e13.9}.
Suppose that $r \leq \min(1,\varepsilon^{1/n})$ and \eqref{e13.9} holds, but
that \eqref{e13.10} fails, and pick any $y\in B(x,r/2)$. As usual, assume that $y=0$
for simplicity. Choose $r_1$ as above; 
because of \eqref{e13.9}, we can use \eqref{e13.19}, and hence
\begin{eqnarray} \label{e13.36}
\lambda \theta(r_1/2)
&=& 2^n\lambda r_1^{-n} |\Omega_1 \cap B(0,r_1/2)| 
\leq 2^n r_1^{-n} \alpha(r_1)
\nonumber\\
&\leq& C\tau \fint_{B(0,r_1)} |\nabla u_{1,+}|^2 
+ C \tau^{-1} \fint_{B(0,r_1)} r^{-2} u_{1,+}^2
+ C r  \Big\{ \fint_{B(0,r_1)} r^{-2} u_{1,+}^2 \Big\}^{1/2}
\nonumber\\
&\leq& C\tau \fint_{B(x,r)} |\nabla u_{1,+}|^2 
+ C \tau^{-1} \fint_{B(x,r)} r^{-2} u_{1,+}^2
+ C r  \Big\{ \fint_{B(x,r)} r^{-2} u_{1,+}^2 \Big\}^{1/2}
\\
&\leq& C [\tau + \tau^{-1} C_0] \fint_{B(x,r)} |\nabla u_{1,+}|^2 
+ C r C_0^{1/2} \Big\{ \fint_{B(x,r)} |\nabla u_{1,+}|^2 \Big\}^{1/2}
\nonumber\\
&\leq& C [\tau + \tau^{-1} C_0] c_3  + C r C_0^{1/2} c_3^{1/2}
\nonumber
\end{eqnarray}
by \eqref{e13.17}, \eqref{e13.18}, \eqref{e13.19}, and because \eqref{e13.10} fails.
In addition, $e(r_1/2) \leq C \fint_{B(x,r)} |\nabla u_{1,+}|^2 \leq C c_3 < c_4/2$
by the definition \eqref{e13.25} and because \eqref{e13.10} fails.

We take $\tau = 1/2$, then choose $c_3$ very small and get that
$\theta(r_1/2) + e(r_1/2) \leq c_4$. Then we can apply \eqref{e13.26}
iteratively, as we did twice before, get that $u_1(y) \leq 0$,
and conclude as before. Again we managed not to use the extra $r$ in the last term.

\ms
We are left with the case when we only assume that $u_1$ is
$C_0$-Lipschitz, as in \eqref{e13.8}. 
Let $B(x,r)$ satisfy the assumptions of the theorem, 
and use \eqref{e13.5} to pick $y\in B(x,r/2)$, with $u_1(y) = 0$. 
Then let $\eta > 0$ be small,
to be chosen soon, and notice that $|u_1(z)| \leq C_0 \eta r$
for $z\in B(y,\eta r)$. Set $m = \fint_{B(x,r)} u_{1,+}$; then by Poincar\'e
(see \eqref{e4.2}),
\begin{eqnarray} \label{e13.37}
m &=& \fint_{B(y,\eta r)} \big[u_{1,+}(z) - (u_{1,+}(z)-m)\big] 
\leq \fint_{B(y,\eta r)} \big[ |u_{1,+}(z)|  + |u_{1,+}(z)-m| \big]
\nonumber\\
&\leq& C_0 \eta r + \fint_{B(y,\eta r)} |u_{1,+}(z)-m|
\leq C_0 \eta r + \eta^{-n} \fint_{B(x,r)} |u_{1,+}(z)-m|
\nonumber\\
&\leq&  C_0 \eta r + C\eta^{-n} r \fint_{B(x,r)} |\nabla u_{1,+}|
\leq C_0 \eta r + C \eta^{-n} r \Big\{\fint_{B(x,r)} |\nabla u_{1,+}|^2 \Big\}^{1/2}
\\
&\leq& C_0 \eta r + C \eta^{-n} c_3^{1/2} r
\nonumber
\end{eqnarray}
 if \eqref{e13.10} fails. Then, by Poincar\'e again, 
 \begin{eqnarray} \label{e13.38}
\fint _{B(x,r)} |u_{1,+}|^2 &\leq & 2 m^2 + 2 \fint _{B(x,r)} |u_{1,+}-m|^2 
\leq 2 m^2 + C r^2 \fint _{B(x,r)} |\nabla u_{1,+}|^2 
 \nonumber\\
&\leq& 2 m^2 + C r^2 c_3  \leq C C_0^2 \eta^2 r^2 + C \eta^{-2n} c_3 r^2.
\end{eqnarray}
We shall choose $c_3 \leq 1$ soon; then let $c_1$ be as in 
Theorem \ref{t13.1}, with the bound $\fint _{B(x,r)} |\nabla u_{1,+}|^2 \leq 1$.
If we choose $\eta$, and then $c_3$ small enough (depending on this $c_1$),
we deduce from \eqref{e13.38} that $\fint _{B(x,r)} |u_{1,+}|^2 \leq c_1 r^2$,
i.e., that \eqref{e13.6} fails. This is impossible, by Theorem \ref{t13.1}, 
and this contradiction completes our proof of Theorem \ref{t13.2}.
\qed

\ms
For the next nondegeneracy result, we assume that \eqref{e13.1} holds
and $u_1$ is continuous, and compare $u_1(x)$ with the distance
\begin{equation} \label{e13.39}
\delta(x) =  \dist(x, \R^n \sm \Omega_1),
\text{ where again } \Omega_1 = \big\{ x\in \Omega \, ; \, u_1(x) > 0 \big\}.
\end{equation}

\begin{thm}\label{t13.3} 
Let $(\u,\W)$ is a minimizer in $\F$ of the functional $J$, suppose that 
$u_1$ is continuous, that $f_1$ and $g_1$ are bounded, 
and that \eqref{e13.1} holds for some choice of 
$\lambda>0$ and $\varepsilon>0$. For each $C_0 \geq 1$, there 
is a constant $c_5 > 0$, that depends only on 
$n$, $\lambda$, $||f_1||_\infty$, $||g_1||_\infty$, and $C_0$,
such that 
\begin{equation} \label{e13.40}
u_1(x) \geq c_5 \min(\delta(x), \varepsilon^{1/n}, 1)
\end{equation}
as soon as $x\in \Omega_1$ and $u_1$ is 
$C_0$-Lipschitz on $B(x,\delta(x)/2)$.
\end{thm}  

\ms
Our Lipschitz assumption looks complicated because it depends on $\delta(x)$;
of course the simplest is to assume that $\u$ is globally Lipschitz,
for instance because of Theorem \ref{t11.1}, but we should also be able to manage 
without smoothness assumption on $\Omega$. Indeed, $u_1 > 0$
on $B = B(x,\delta(x))$, so $B$ is almost everywhere contained
in $\Omega$ (because $u_1(x) = 0$ on $\R^n \sm \Omega$), which means that 
$\Omega$ is equivalent to a set that contains $B$. Then we may use 
Lemma \ref{t10.6} to get local Lipschitz bounds on $B(x,\delta(x)/2)$,
which we can then use to get \eqref{e13.40}.

Notice also that if $u_1$ is $C$-Lipschitz in $B(x,\delta(x))$, we immediately
get the opposite inequality $|u_1(x)| \leq C \delta(x)$.

\begin{proof}
The general idea will be that if a positive harmonic function is very small near
the center of the ball, it must also be small on average on the whole ball; 
if $u_1(x)$ is too small, we shall try to approximate $u_1$ with such a harmonic function, 
show that the average of $u_1^2$ on a ball is small, and use a cut-off competitor to conclude.

Let $x\in \Omega_1$ be as in the statement; as usual, we can assume that $x=0$
to simplify the notation. Let $r > 0$ be small, to be chosen later. For the moment,
let us just assume that
\begin{equation} \label{e13.41}
r < \delta(x)/2.
\end{equation}
This way, we know that $u_1 > 0$ on $B(x,r)$, hence $B(x,r) \i \Omega$
(modulo replacing $\Omega$ with an equivalent domain, as above)
and $u_1$ is $C_0$-Lipschitz near $\overline B(x,r)$. 
Notice that the restriction of $u_1$ to $S_r = \d B(0,r)$ is Lipschitz,
so we can define its harmonic extension $u^\ast_1$ to $B(0,r)$
more easily than in Section \ref{favorites}.
Also set $u^\ast_1(x) = u_1(x)$ for $x\in \R^n \sm B(0,r)$,
take $u_j^\ast = 0$ on $B(0,r)$ for $j \geq 2$, keep $\u$ as it was on
$\R^n \sm B(0,r)$, and set $\W^\ast = \W$. Thus $(\u^\ast,\W^\ast)$
is a simpler variant of the harmonic competitor of Section \ref{favorites},
we get that $(\u^\ast,\W^\ast) \in \F$ at once 
(because $u^\ast_1 \in W^{1,2}(\R^n)$),
and we shall now see what the comparison yields. For the energy part, the usual
computation yields
$\int_{B(0,r)} |\nabla u_1^\ast|^2 \leq \int_{B(0,r)} |\nabla u_1|^2$
(because $u_1^\ast$ minimizes the energy with the given boundary data; 
see \eqref{e6.13}), and even
\begin{equation} 
\label{e13.42}
\int_{B(0,r)} |\nabla u_1|^2 - \int_{B(0,r)} |\nabla u_1^\ast|^2
= \int_{B(0,r)} |\nabla (u_1-u_1^\ast)|^2
\end{equation}
because $u_1^\ast + t (u_1-u_1^\ast)$, $t\in \R$, is a competitor
for $u_1^\ast$, and by Pythagorus (see the proof of \eqref{e10.32}). 
There is no difference in the $F$-terms, and 
\begin{eqnarray} 
\label{e13.43}
|M(\u^\ast) - M(\u)| &=&  \Big|\int_{B(0,r)} [(u^\ast_1)^2 f_1 - u_1^2 f_1 
- u^\ast_1 g_1 + u_1 g_1] \Big|
\nonumber\\
&\leq& \Big(||(u_1^\ast + u_1) f_1||_{L^\infty(B(0,r))} + ||g_1||_\infty \Big)
\int_{B(0,r)} |u^\ast_1- u_1|.
\end{eqnarray}
If $u_1(0) \geq 1$, we are happy because \eqref{e13.40} holds with $c_5=1$,
so we may assume that $u_1(0) \leq 1$; then 
$|u_1(x)| \leq 1 + C_0 r$ on $B(0,r)$, the same thing holds
for $u_1^\ast$ by the maximum principle, and hence
\begin{equation}\label{e13.44}
||(u_1^\ast + u_1) f_1||_{L^\infty(B(0,r))}
\leq 2(1+ C_0 r)||f_1||_\infty \leq C
\end{equation}
if we forget to write the dependence on $C_0$ and demand that $r \leq 1$.
Thus 
\begin{equation}\label{e13.45}
|M(\u^\ast) - M(\u)| \leq C \int_{B(0,r)} |u^\ast_1- u_1|
\leq C r^{n/2} \Big\{\int_{B(0,r)} |u^\ast_1- u_1|^2\Big\}^{1/2}.
\end{equation}
Next we claim that since $u_1^\ast = u_1$ on $S_r$,
\begin{equation}\label{e13.46}
\int_{B(0,r)} |u^\ast_1- u_1|^2 \leq C r^2 \int_{B(0,r)} |\nabla (u^\ast_1- u_1)|^2
\end{equation}
because $u^\ast_1- u_1 \in W^{1,2}(\R^n)$ (see near \eqref{e4.18})
and by Lemma \ref{l3.2}, or, if the reader prefers, because Lemma \ref{l4.2} says that
$\big|\fint_{B(0,r)} (u^\ast_1- u_1) \big| \leq C r \fint_{B(0,r)} |\nabla (u^\ast_1- u_1)|$,
and then by the usual Poincar\'e inequality \eqref{e4.2}. In addition,
\begin{eqnarray} \label{e13.47}
\int_{B(0,r)} |\nabla (u_1-u_1^\ast)|^2
&=& \int_{B(0,r)} |\nabla u_1|^2 - \int_{B(0,r)} |\nabla u_1^\ast|^2
\nonumber\\
&=& J(\u,\W) - M(\u) - J(\u^\ast,\W^\ast) + M(\u^\ast)
\nonumber\\
&\leq&  |M(\u^\ast) - M(\u)| 
\leq C r^{n/2} \Big\{\int_{B(0,r)} |u^\ast_1- u_1|^2\Big\}^{1/2}
\\
&\leq& C r^{{n\over 2}+1} \Big\{\int_{B(0,r)} |\nabla (u^\ast_1- u_1)|^2\Big\}^{1/2}
\nonumber
\end{eqnarray}
by \eqref{e13.42}, because $F(\W^\ast) = F(\W)$ and $(\u,\W)$ is a minimizer,
and by \eqref{e13.45} and \eqref{e13.46}. We simplify and get that
\begin{equation}\label{e13.48}
\int_{B(0,r)} |\nabla (u_1-u_1^\ast)|^2 \leq C r^{n+2}.
\end{equation}
Let $\eta > 0$ be small, to be chosen soon. Then 
\begin{equation} \label{e13.49}
u(z) \leq u(0) + C_0 \eta r \ \text{ for } z\in B(0,\eta r),
\end{equation}
hence
\begin{eqnarray} 
\label{e13.50}
|u_1^\ast(0)| &=& \Big|\fint_{B(0,\eta r)} u_1^\ast \Big|
\leq u(0) + C_0 \eta r + \fint_{B(0,\eta r)} |u_1^\ast - u_1| 
\nonumber\\
&\leq& u(0) + C_0 \eta r + \eta^{-n} \fint_{B(0, r)} |u_1^\ast - u_1| 
\nonumber\\
&\leq& u(0) + C_0 \eta r + \eta^{-n} \Big\{\fint_{B(0,r)} |u^\ast_1- u_1|^2\Big\}^{1/2}
\\
&\leq&  u(0) + C_0 \eta r + C \eta^{-n} r \Big\{\fint_{B(0,r)} |\nabla(u^\ast_1- u_1)|^2\Big\}^{1/2}
\leq u(0) + C_0 \eta r + C \eta^{-n} r^2
\nonumber
\end{eqnarray}
because $u_1^\ast$ is harmonic, and by \eqref{e13.46} and \eqref{e13.48}.
In addition, $u_1^\ast = u_1$ on $S_r$ and by \eqref{e13.41} $u_1 > 0$ on 
$\overline B(x,r) \i \Omega_1$, so $u_1^\ast > 0$ on $S_r$ and on
$B(0,r)$ (by the maximum principle). So
\begin{equation} \label{e13.51}
\fint_{B(0,r)} |u_1^\ast| = \fint_{B(0,r)} u_1^\ast  = u_1^\ast(0) 
\end{equation}
and, by simple estimates on harmonic functions,
\begin{equation} \label{e13.52}
\fint_{B(0,r/2)} |\nabla u_1^\ast|^2 
\leq C r^{-2} \Big\{\fint_{B(0,r)} |u_1^\ast|\Big\}^2
= C  r^{-2} u_1^\ast(0)^2.
\end{equation}
Then 
\begin{equation} \label{e13.53}
\fint_{B(0,r/2)} |\nabla u_1|^2 \leq 2 \fint_{B(0,r/2)} |\nabla u_1^\ast|^2 +
2 \fint_{B(0,r/2)} |\nabla (u_1 - u_1^\ast)|^2 
\leq C r^{-2} u_1^\ast(0)^2 + C r^{2}
\end{equation}
by \eqref{e13.48}. Set $m = \fint_{B(0,r/2)} u_1$; By Poincar\'e,
\begin{equation} \label{e13.54}
\fint_{B(0,r/2)} |u_1-m|^2 \leq C r^2 \fint_{B(0,r/2)} |\nabla u_1|^2 
\leq C u_1^\ast(0)^2 + C r^{4}.
\end{equation}
Let $\eta_1 > 0$ be another small number, to be chosen later,
notice that $u_1(z) \leq u_1(0) + C_0 \eta_1 r$ for $z\in B(0,\eta_1 r)$,
and use this and Poincar\'e to estimate $m$: 
\begin{eqnarray} \label{e13.55}
m^2 &=& \fint_{B(0,\eta_1 r)} m^2 
\leq 2 \fint_{B(0,\eta_1 r)} u_1^2 + 2 \fint_{B(0,\eta_1 r)} |u_1-m|^2
\nonumber\\
&\leq& 2 (u_1(0) + C_0 \eta_1 r)^2 + 2^{1-n} \eta_1^{-n} \fint_{B(0,r/2)} |u_1-m|^2
\nonumber\\
&\leq& 2 (u_1(0) + C_0 \eta_1 r)^2 + C \eta_1^{-n} u_1^\ast(0)^2 + C \eta_1^{-n} r^{4}
\end{eqnarray}
by \eqref{e13.54}. Finally,
\begin{eqnarray} \label{e13.56}
\fint_{B(0,r/2)} u_1^2 &\leq& 2 m^2 + 2 \fint_{B(0,r/2)} |u_1-m|^2
\leq 2 m^2 + C u_1^\ast(0)^2 + C r^{4}
\nonumber\\
&\leq& 
4 (u_1(0) + C_0 \eta_1 r)^2 + C \eta_1^{-n} u_1^\ast(0)^2 + C \eta_1^{-n} r^{4}
\nonumber\\
&\leq& C u_1(0)^2 + C \eta_1^2 r^2 + C \eta_1^{-n} \big(u_1(0) + \eta r + \eta^{-n} r^2\big)^2
+ C \eta_1^{-n} r^{4} 
\\
&\leq& C \eta_1^{-n} u_1(0)^2 + C \eta_1^2 r^2 + C \eta_1^{-n} \eta^2 r^2
+\eta_1^{-n} \eta^{-2n} r^4 =: r^2\alpha(r),
\nonumber
\end{eqnarray}
by \eqref{e13.54}, \eqref{e13.55}, and \eqref{e13.50} and where the last
identity is a definition of $\alpha(r)$.

\ms
We now have enough information to compare $(\u,\W)$ with the cut-off
competitor associated to $B(0,r/2)$ and the constant $a=1/2$. That is, 
we want to replace $u_1$ with a cut-off function $u_1^\ast$, defined
by $u_1^\ast(x) = u_1(x) \varphi(|x|)$ as in Section \ref{favorites} or in the 
beginning of this section. We do not need to touch the $u_j$, $j \geq 2$,
because they all vanish on $B(0,r/4)$. But we want to take 
advantage of \eqref{e13.1} to modify $\W$. Let us assume, in addition 
to \eqref{e13.41} and the fact that $r \leq 1$, that $r$ is so small that 
\begin{equation} \label{e13.57}
|B(0,r/4)| \leq \varepsilon
\end{equation}
Observe that $B(0,r) \i \Omega_1 \i W_1$ (modulo a null set), 
because $u_1 > 0$ there; we apply \eqref{e13.1} to 
$A = B(0,r/4) \cap W_1$, and get disjoint sets $A_j \i A$, $j \geq 2$.
Then set $\W^\ast = (W_1 \sm A, W_2 \cup A_1, \ldots, W_N \cup A_N)$;
it is easy to see that $(\u^\ast,\W^\ast) \in \F$, and \eqref{e13.1} says
that
\begin{equation} 
\label{e13.58}
F(\W^\ast) \leq F(\W) - \lambda |A| = F(\W) - \lambda |B(0,r/4)|.
\end{equation}
From \eqref{e6.5} with $\tau = 1$ we deduce that
\begin{eqnarray} 
\label{e13.59}
\int_{B(0,r/2)} |\nabla u_1^\ast|^2 
&\leq&  2 \int_{B(0,r/2)} |\nabla u_1|^2  + C r^{-2} \int_{B(0,r/2)} |u_1|^2
 \nonumber\\
&\leq& 
C r^{n-2} |u_1^\ast(0)|^2 + C r^{n+2} + C r^{n}\alpha(r)
=: r^n \beta(r)
\end{eqnarray}
by \eqref{e13.53} and \eqref{e13.56}, and where the last identity is a definition of $\beta(r)$. 
Concerning  $M(\u^\ast)-M(\u)$, even though we are now looking at a different
competitor, we can still use the estimates \eqref{e13.43}-\eqref{e13.45}.
We get that
\begin{eqnarray}
\label{e13.60}
|M(\u^\ast) - M(\u)| &\leq& C \int_{B(0,r/2)} |u^\ast_1- u_1|
\leq C r^{n/2} \Big\{\int_{B(0,r/2)} |u^\ast_1- u_1|^2\Big\}^{1/2}
 \nonumber\\
&\leq&
C r^{{n\over 2}+1} \Big\{\int_{B(0,r/2)} |\nabla (u^\ast_1- u_1)|^2\Big\}^{1/2}
\nonumber\\
&\leq& C r^{{n\over 2}+1} \Big\{\int_{B(0,r/2)} |\nabla u^\ast_1|^2+ |\nabla u_1)|^2\Big\}^{1/2}
\leq C r^{n+1} \beta(r)^{1/2}
\end{eqnarray}
by \eqref{e13.45}, Cauchy-Schwarz, Poincar\'e and the fact that $u_1^\ast = u_1$
on $\d B(0,r/2)$ (as in \eqref{e13.46}-\eqref{e13.47}), and \eqref{e13.59}.
Now
\begin{eqnarray}\label{e13.61}
\lambda |B(0,r/4)| &\leq & F(\W)- F(\W^\ast)  
\nonumber\\
&=& J(\u,\W)-J(\u^\ast,\W^\ast) + [M(\u^\ast) - M(\u)] 
+ \int_{B(0,r/2)} |\nabla u_1^\ast|^2 - |\nabla u_1|^2
 \nonumber\\
&\leq& C r^{n+1} \beta(r)^{1/2} +r^n \beta(r) 
\end{eqnarray}
by \eqref{e13.58}, \eqref{e1.5}, because $(\u,\W)$ is a minimizer,
and by \eqref{e13.59} and \eqref{e13.60}.
The game now consists in checking out the right-hand side of \eqref{e13.61},
and showing that all the terms that we get, except those coming from $u_1(0)$,
are much smaller than $\lambda r^n$. 
For this we shall put the additional
constraint that $r \leq r_0$, where $r_0$ depends on $n$, $\lambda$,
$||f_1||_\infty$, $||g_1||_\infty$, and $C_0$. 

We first look at $\alpha(r)$ in \eqref{e13.56}. If we chose $\eta_1$ small
enough, then $\eta$ small, and then $r_0$ small enough,
we get that $\alpha(r) \leq C \eta_1^{-n} r^{-2} u(0)^2 + o(1)$,
where $o(1)$ is as small as we want. The rest of 
$\beta(r)$ is 
\begin{equation} \label{e13.62}
\beta_1(r) = r^{-2} u_1^\ast(0)^2 + C r^{2}
\leq C r^{-2} u(0)^2 + C \eta^2 + C \eta^{-2n} r^2 + C r^{2}
\end{equation}
(see \eqref{e13.59} and \eqref{e13.50}), which is of the same type.
Then \eqref{e13.61} says that

\begin{eqnarray} \label{e13.63}
\lambda &\leq& C r^{-n} \lambda |B(0,r/4)|
\leq C r \beta(r)^{1/2} + C \beta(r)
 \nonumber\\
&\leq& C r \eta_1^{-n/2} r^{-1} u(0) +  C \eta_1^{-n} r^{-2} u(0)^2 + o(1)
\end{eqnarray}
which, if we choose our constants so that $o(1) \leq \lambda/2$,
yields $r^{-1} u(0) \geq c$ for some $c > 0$.
Our constraints on $r$ (namely \eqref{e13.41}, \eqref{e13.57}, 
$r \leq 1$, and $r \leq r_0$)
allow us to chose $r \geq C^{-1} \min(\delta(0), \varepsilon^{1/n},1)$,
and we obtain \eqref{e13.40} (recall that we took $x=0$). Theorem \ref{t13.3}
follows.
\qed
\end{proof}

\ms
The last result of this section concerns the size of the complement of a good region
where $u_1 \geq 0$.

\begin{thm}\label{t13.4} 
Let $(\u,\W)$ be a minimizer in $\F$ of the functional $J$, suppose that 
the $f_i$ and the $g_i$ in the data are bounded, that $F$ is Lipschitz
(i.e., \eqref{e10.2} holds), and that 
\eqref{e13.1} holds for some choice of $\lambda>0$ and $\varepsilon>0$.
For each $C_0 \geq 1$, we can find $r_0 > 0$ and $c_6 > 0$,
such that if $x\in \R^n$ and $0 \leq r \leq r_0$ are such that $B(x,r) \i \Omega$,
\begin{equation} 
\label{e13.64}
\text{ $\u$ is $C_0$-Lipschitz on $B(x,r)$,} 
\end{equation}
and $B(x,r/2)$ meets the boundary of  
$\Omega_1 = \big\{ x\in \Omega \, ; \, u_1(x) > 0 \big\}$,
then 
\begin{equation} 
\label{e13.65}
\big|\big\{ x\in B(x,r) \, ; \, u_1(x) \leq 0 \big\}\big| \geq c_6 r^n.   
\end{equation}
The constant $c_6 > 0$ depends only on $n$, 
the $||f_i||_\infty$, the $||g_i||_\infty$, the Lipschitz constant for $F$ (in \eqref{e10.2})
$\lambda$, and $C_0$, and $r_0$ depends on these constants, plus $\varepsilon$.
\end{thm}

The main case is probably when $u_1 \geq 0$. Then \eqref{e13.65} implies
that $|B(x,r) \sm W_1| \geq c_6 r^n$, because the remaining set
$\big\{ x\in W_1 \, ; \, u_1(x) \leq 0 \big\} =
\big\{ x\in W_1 \, ; \, u_1(x) = 0 \big\}$ is negligible, because otherwise
we could just use \eqref{e13.1} to send some of it to the other $W_i$
and some of it to the trash, and make a profit.

We decided also to include the case when
$u_1$ is real-valued, and then \eqref{e13.65} also counts the part of $W_1$
where $u_1 < 0$. 

We did not mention the case when $B(x,r)$ is centered near $\d \Omega$,
because even though the result is still true in that case (if $\d \Omega$ 
is reasonably smooth), this is for the stupid reason that $B(x,r) \sm \Omega$
is already large enough. So we shall restrict to $B(x,r) \i \Omega$. In many cases, 
we will be able to deduce \eqref{e13.64} from Theorem \ref{t10.1}, 
applied to $B(x,2r)$ if it is contained in $\Omega$; but then $C_0$ will depend 
on $\int_{B(x,2r)} |\nabla \u|^2$.

\begin{proof}
The proof will be based on the following (admittedly vague) idea.
We want to show that $u_1$ is close to a harmonic function $v_1$ in
some ball centered on $\d \Omega_1$;
if \eqref{e13.65} fails, $u_1$ and $v_1$ should almost be positive. 
Theorem \ref{t13.1} says that $u_1$, hence also $v_1$, is reasonably large
on average, so the mean value property for $v_1$ says that it should also be 
large near the center of the ball. But the closeby $u_1$ vanishes on that
center (a contradiction). The harmonic function $v_1$ will be the same one as in the
harmonic competitor.

Let the pair $(x,r)$ be as in the statement, and suppose in addition
that \eqref{e13.65} fails. By assumption, we can find
$y\in B(x,r/2) \cap \d \Omega_1$, and without lost of generality, we may assume that
$y=0$. We first want to select a radius $\rho \in (r/4,r/2)$, with the following properties.
Since we want to use the harmonic competitor (with main function $u_1$),
we first demand that the restriction of $\u$ to $S_\rho$ lie in $W^{1,2}(S_\rho)$,
with a derivative that can be computed from the restriction of $\nabla \u$
to $S_\rho$, and with the usual estimates
\begin{equation} 
\label{e13.66}
\int_{S_\rho} |\nabla_t u_i|^2 \leq 20N r^{-1} \int_{B(x,r)}  |\nabla u_i|^2 
\end{equation}
for $1 \leq i \leq N$ that we can easily obtain by Chebyshev.
In the present case, we know that $\u$ is Lipschitz, so we don't need to be 
as prudent as usual about restrictions. We also require that
\begin{equation} 
\label{e13.67}
u_i(z) = 0 \ \text{ $\sigma$-almost everywhere on } S_\rho\sm W_i
\end{equation}
(as in \eqref{e6.16}), which is true for almost every $\rho$, and that
\begin{equation} \label{e13.68}
\sigma\big(\big\{ z\in S_\rho \, ; \, u_1(z) \leq  0 \big\}\big) \leq 20 c_6 r^{n-1}.
\end{equation}
We can also get this last condition, by Chebyshev and because we assume \eqref{e13.65}
to fail. 

Define the harmonic competitor $(\u^\ast,\W^\ast)$ as we did near \eqref{e6.11}
(but with the radius $\rho$); notice in particular that the requirement \eqref{e6.9} 
holds because $B(0,\rho) \i B(x,r) \i \Omega$. 
This time, we shall need to take $a \in [1/2,1)$ close to $1$, to be chosen later.
Notice that $\u_1^\ast$ is given in terms of a harmonic function $v_1$
(see \eqref{e6.14} and above), and a good part of the estimates that follow is aimed
at showing that $v_1-u_1$ is quite small.

Let us first estimate $F(\W^\ast)-F(\W)$. Recall the definition
\eqref{e6.17}-\eqref{e6.18} of $\W^\ast$. For $i \geq 2$,
$|W_i \sm W_i^\ast| \leq |W_i| \leq c_6 r^n$ because $u_1 = 0$ on $W_i$.
Also, $|W_i^\ast \sm W_i| \leq |W_i^\ast| \leq B(0,\rho) \sm B(0,a\rho) \leq C(1-a)r^n$.
Next, $|W_1 \sm W_1^\ast| \leq |B(0,\rho) \sm W_1^\ast| \leq |B(0,\rho) \sm B(0,a\rho)| 
\leq C(1-a)r^n$. Finally, $|W_1^\ast \sm W_1| \leq |B(0,\rho) \sm W_1|
\leq |B(0,\rho) \sm \Omega_1| \leq c_6 r^n$. Thus by \eqref{e10.2}
\begin{equation} \label{e13.69}
|F(\W^\ast)-F(\W)| \leq C\sum_{i=1}^N  |W_i \Delta W_i^\ast|
\leq C (c_6 + (1-a)) r^n.
\end{equation}
Notice that $u_1(y) = 0$ because $y \in \d \Omega_1$, and
all the other $u_i$ also vanish somewhere in $B(y,\rho)$ (in fact, anywhere on 
$B(y,\rho) \cap \Omega_1$), because \eqref{e13.65} fails. Since $\u$ is 
assumed to be $C_0$-Lipschitz on $B(x,r)$, 
we get that $|\u| \leq C r$ on $B(0,\rho)$, and (by the maximum principle) 
the same thing holds for $\u^\ast$. [We shall not keep track of the dependence
of our constants on $C_0$, or the other constants mentioned in the statement.]
Then
\begin{equation} \label{e13.70}
|M(\u^\ast)-M(\u)| \leq \sum_{i=1}^N  \int_{B(0,\rho)} 
\big[||f_i||_\infty |u^2-(u^\ast)^2| + ||g_i||_\infty |u-u^\ast|\big] \leq C r^{n+1}.
\end{equation}
For the energy part, we start with $i \geq 2$. By \eqref{e6.19},
\begin{equation} \label{e13.71}
\int_{B(0,\rho)} |\nabla u_{i}^\ast|^2 
\leq (1-a) r  a^{2-n} \int_{S_{\rho}} |\nabla_{t} u_{i}|^2
+ 4 (1-a)^{-1} r  a^{-n} \int_{S_{\rho}} |\rho^{-1} u_{i}|^2.
\end{equation}
We continue as usual, but at this point we could also use the fact that $\u$
is Lipschitz on $\overline B(0,\rho)$ to get the same conclusion. 
We want to use \eqref{e4.7} with $E = S_\rho \sm \Omega_1$,
so we first observe that if $z\in S_\rho \sm E = S_\rho \cap \Omega_1$,
then $z \in W_1$ almost surely (by \eqref{e13.67} for $i=1$), 
so $z\in S_\rho \sm W_i$, and $\sigma$-almost surely $u_i(z) = 0$
(by \eqref{e13.67} for $i$). So $u_i(z)=0$ almost everywhere on
$S_\rho \sm E$. Also, $u_1(z) \leq 0$ on $E$, so
$\sigma(E) \leq 20 c_6 r^{n-1}$ by \eqref{e13.68}, and
\begin{equation} \label{e13.72}
\int_{S_{\rho}} |\rho^{-1} u_{i}|^2 \leq C \rho^{-2}
\sigma(E)^{2 \over n-1}\int_{S_{\rho}} |\nabla_t u_{i}|^2
\leq C c_6^{2 \over n-1}\int_{S_{\rho}} |\nabla_t u_{i}|^2
\end{equation}
by \eqref{e4.7} and 
\begin{eqnarray} \label{e13.73}
\int_{B(0,\rho)} |\nabla u_{i}^\ast|^2 
&\leq& C [(1-a) + (1-a)^{-1} c_6^{2 \over n-1}] \, r  
\int_{S_{\rho}} |\nabla_{t} u_{i}|^2
\nonumber\\
&\leq&  C [(1-a) + (1-a)^{-1} c_6^{2 \over n-1}] 
\int_{B(x,r)} |\nabla u_{i}|^2
\nonumber\\
&\leq& C [(1-a) + (1-a)^{-1} c_6^{2 \over n-1}] r^n
\end{eqnarray}
by \eqref{e13.71} and \eqref{e13.66}, and brutally because
$\u$ is Lipschitz in $\overline B(0,\rho)$.

We have the estimate \eqref{e6.20} for the exterior part
of $\int |\nabla u_1^\ast|^2$, namely,
\begin{eqnarray} \label{e13.74}
\int_{B(0,\rho)\sm B(0,a\rho)} |\nabla u_{1}^\ast|^2
&\leq& (1-a) \rho  a^{2-n} \int_{S_{\rho}} |\nabla_{t} u_{1}|^2
\nonumber\\
&\leq& C (1-a) \int_{B(x,r)} |\nabla u_1|^2 \leq C(1-a) r^n
\end{eqnarray}
(by \eqref{e13.66} again).
The last term is 
\begin{eqnarray}  \label{e13.75}
\int_{B(0,a\rho)} |\nabla u_{1}^\ast|^2 
&=& a^{n-2}\int_{B(0,\rho)} |\nabla v_{1}|^2
\nonumber
\\
&=& a^{n-2} \inf\Big\{ \int_{B(0,\rho)} |\nabla v|^2 \, ; \,
v\in W^{1,2}(B(0,\rho)) \text{ and } v=u_{1} \text{ on } S_{\rho} \Big\},
\nonumber
\\
&\leq& a^{n-2} \int_{B(0,\rho)} |\nabla u_{1}|^2 
\end{eqnarray}
by \eqref{e6.21}. 
Set $\Delta = \int_{B(0,\rho)} |\nabla u_{1}|^2  - \int_{B(0,\rho)} |\nabla v_{1}|^2$; 
then by the minimizing property in \eqref{e13.75},
the fact that $v_1 + t(u_1-v_1)$ is a competitor for $v_1$ for all $t$, 
and the usual Pythagorus argument (see for instance the proof of \eqref{e10.32}),
\begin{equation} \label{e13.76}
\Delta = \int_{B(0,\rho)} |\nabla (u_{1}-v_1)|^2.
\end{equation}
But also, by the first line of \eqref{e13.75} and because $a \leq 1$,
\begin{eqnarray}  
\label{e13.77}
a^{n-2}\Delta
&=& a^{n-2}\int_{B(0,\rho)} |\nabla u_{1}|^2 - \int_{B(0,a\rho)} |\nabla u_{1}^\ast|^2 
\nonumber \\
&\leq & \int_{B(0,\rho)} |\nabla u_{1}|^2 - \int_{B(0,a\rho)} |\nabla u_{1}^\ast|^2
\\
&=& \int_{B(0,\rho)} |\nabla u_{1}|^2 - \int_{B(0,\rho)} |\nabla u_{1}^\ast|^2
+\int_{B(0,\rho)\sm B(0,a\rho)} |\nabla u_{1}^\ast|^2
\nonumber
\end{eqnarray}
We now add the contribution of the other components. Notice that
$\int_{B(0,\rho)} |\nabla u_{1}|^2 \leq \int_{B(0,\rho)} |\nabla \u|^2$,
and $\int_{B(0,\rho)} |\nabla u_{1}^\ast|^2 =
\int_{B(0,\rho)} |\nabla \u^\ast|^2 
- \sum_{i \geq 2} \int_{B(0,\rho)} |\nabla u_{i}^\ast|^2$. 
We replace and get that
\begin{equation}  
\label{e13.78}
a^{n-2}\Delta
\leq \int_{B(0,\rho)} |\nabla \u|^2 - \int_{B(0,\rho)} |\nabla \u^\ast|^2
+ \sum_{i \geq 2} \int_{B(0,\rho)} |\nabla u_{i}^\ast|^2
+\int_{B(0,\rho)\sm B(0,a\rho)} |\nabla u_{1}^\ast|^2 
\end{equation}
Let us use the minimality of $(\u,\W)$. Since $\u = \u^\ast$ on $\R^n \sm B(0,\rho)$, 
\begin{eqnarray}  
\label{e13.79}
\int_{B(0,\rho)} |\nabla \u|^2 - \int_{B(0,\rho)} |\nabla \u^\ast|^2
&=& J(\u) - M(\u) - J(\u^\ast) + M(\u^\ast)
\nonumber\\
&\leq& M(\u^\ast) - M(\u) \leq C r^{n+1}
\end{eqnarray}
by \eqref{e1.5}, the minimality of $(\u,\W)$, and \eqref{e13.70}.
Then 
\begin{eqnarray}  
\label{e13.80}
\Delta &\leq& 2 a^{n-2}\Delta
\leq C r^{n+1} + \sum_{i \geq 2} \int_{B(0,\rho)} |\nabla u_{i}^\ast|^2
+\int_{B(0,\rho)\sm B(0,a\rho)} |\nabla u_{1}^\ast|^2 
\nonumber\\
&\leq& C r^{n+1} 
+ C [(1-a) + (1-a)^{-1} c_6^{2 \over n-1}] r^n
+ C(1-a) r^n
\\
&\leq& C [(1-a) + (1-a)^{-1} c_6^{2 \over n-1} + r] r^n
\nonumber
\end{eqnarray}
because we shall take $a$ close to $1$, and by \eqref{e13.78}, \eqref{e13.79}
\eqref{e13.73} and \eqref{e13.74}.

We can now use this to show that $u_1$ is 
close to the harmonic function $v_1$.
Notice that $v_1-u_1$ lies in $W^{1,2}(\R^n)$ and vanishes on 
$\R^n \sm B(0,\rho)$ so Lemma \ref{l3.2}
(or if the reader prefers, Lemma \ref{l4.2} and the standard Poincar\'e inequality), 
implies that
\begin{equation} \label{e13.81}
\int_{B(0,\rho)} |u_{1}-v_1|^2 \leq  C \rho^2 \int_{B(0,\rho)} |\nabla (u_{1}-v_1)|^2
\leq C r^2 \Delta,
\end{equation}
by \eqref{e13.76}.

Next we want to apply Theorem \ref{t13.1} to $B(0,\rho)$;
the assumption \eqref{e13.5} is satisfied, because 
$0=y\in \d \Omega_1$ and if $r_0 \leq \inf(1,\varepsilon^{1/n})$
(recall that $2\rho \leq r \leq r_0$).  So \eqref{e13.6} says that
$\fint_{B(0,\rho)} u_{1,+}^2 \geq c_1 \rho^2 \geq c_1 r^2/16$,
where $c_1$ also depends on $C_0$ through an upper bound
for $\fint_{B(0,\rho)} |\nabla u_{1,+}|^2$.
Since in addition $|u_1(z)| \leq C r$ for $z\in B(0,\rho)$, because $u_1(y) = 0$
(see below \eqref{e13.69}), we get that
\begin{equation} \label{e13.82}
{c_1 r^2\over 16} \leq \fint_{B(0,\rho)} u_{1,+}^2 \leq C r \fint_{B(0,\rho)} u_{1,+}
\end{equation}
If \eqref{e13.65} fails, we also get that
\begin{equation} \label{e13.83}
\fint_{B(0,\rho)} u_{1,-} \leq C  r \rho^{-n} 
\big|\big\{ x\in B(x,r) \, ; \, u_1(x) \leq 0 \big\}\big| 
\leq C c_6 r
\end{equation}
because $|u_1(z)| \leq C r$ for $z\in B(0,\rho)$.
If $c_6$ is small enough compared to $c_1$, we deduce from \eqref{e13.82}
and \eqref{e13.83} that
\begin{equation} \label{e13.84}
\fint_{B(0,\rho)} u_1 \geq C^{-1} c_1 r.
\end{equation}
Next set $m = \fint_{B(0,\rho} v_1$; we deduce from \eqref{e13.81} that
\begin{eqnarray} \label{e13.85}
\Big| m- \fint_{B(0,\rho)} u_1\Big| &=& \Big| \fint_{B(0,\rho)} (v_1 -u_1)\Big|
\leq \fint_{B(0,\rho)} |v_1 -u_1|
\nonumber
\\
&\leq& \Big\{\fint_{B(0,\rho)} |v_1 -u_1|^2 \Big\}^{1/2}
\leq C r (r^{-n}\Delta)^{1/2}.
\end{eqnarray}
Let $\eta > 0$ be small, to be chosen soon, and observe that 
$|u_1(z)| \leq C \eta \rho$ for $z\in B(0,\eta \rho)$
(because $u_1(0) = 0$ and $\u$ is $C_0$-Lipschitz in $B(0,\rho)$).
But $\fint_{B(0,\eta\rho)} v_1 = m$ because $v_1$ is harmonic;
then
\begin{eqnarray} \label{e13.86}
|m| &=& \Big| \fint_{B(0,\eta\rho)} v_1\Big|
\leq C \eta \rho + \Big| \fint_{B(0,\eta\rho)} (v_1-u_1)\Big|
\leq C \eta r +\fint_{B(0,\eta\rho)} |v_1 -u_1|
\nonumber
\\
&\leq& C \eta r + \eta^{-n} \fint_{B(0,\rho)} |v_1 -u_1|
\leq C \eta r + C r \eta^{-n} (r^{-n}\Delta)^{1/2}
\end{eqnarray}
by the end of \eqref{e13.85}. Now \eqref{e13.84}-\eqref{e13.86} yield
\begin{eqnarray} \label{e13.87}
C^{-1} c_1 &\leq& r^{-1} \fint_{B(0,\rho)} u_1 
\leq r^{-1} \Big| m- \fint_{B(0,\rho)} u_1\Big| + r^{-1} |m| 
\nonumber
\\
&\leq& C (r^{-n}\Delta)^{1/2} + C \eta  + C \eta^{-n} (r^{-n}\Delta)^{1/2}
\nonumber
\\
&\leq& C \eta + C  \eta^{-n} [(1-a) + (1-a)^{-1} c_6^{2 \over n-1} + r]^{1/2}
\end{eqnarray}
by \eqref{e13.80}. If we choose $\eta$, then $1-a$, then $c_6$
and $r_0$ small enough (also recall that we need to take 
$r_0 \leq \inf(1,\varepsilon^{1/n})$ to apply Theorem \ref{t13.1}),
we obtain the desired contradiction, which proves \eqref{e13.65}.
Theorem \ref{t13.4} follows.
\qed
\end{proof}

\section{The boundary of a good region is rectifiable}    \label{recti} 

Our assumptions for this section will be roughly the same as for Section \ref{good}.
We shall assume that 
\begin{equation} \label{e17.1}
\text{ the $f_i$ and the $g_i$, $1 \leq i \leq N$, are bounded} 
\end{equation}
and restrict our attention on an open ball $B_0$ such that
\begin{equation} \label{e17.2}
u_1 \text{ is $C_0$-Lipschitz on } B_0
\end{equation}
for some $C_0 \geq 0$. As usual, previous sections give sufficient conditions 
for this to happen.
We shall also assume the Lipschitz condition \eqref{e10.2}, and 
that $W_1$ is a good region for $F$, in the sense that we can find $\varepsilon > 0$
and $\lambda > 0$ such that the nondegeneracy condition \eqref{e13.1} holds.
And we start to study the regularity in $B_0$ of the boundary of the open set
\begin{equation} \label{e17.3}
\Omega_1 = \big\{ x\in \Omega \, ; \, u_1(x) > 0 \big\}.
\end{equation}
Maybe we should point out that $\Omega_1$ is the good free boundary to
study, as opposed to the larger set $\big\{ \u = 0 \}$. Suppose that
$u_1 \geq 0$ to simplify the discussion.
On the other side of  $\Omega_1$, there may be other components,
possibly not all good, and/or the black zone $\Omega \sm \cup_{i} W_i$,
and these may be much less regular. Even when $n=2$, $N=2$, the $u_i$
are required to be positive, 
and $(\u,\W)$ is a local minimizer of the most standard Alt-Caffarelli
functional with $q_1 q_2 = 1$ (or take $q_2 = 2$ if you are afraid of 
a potential degeneracy), it can happen that $\Omega_1$
and $\Omega_2$ are smooth regions, with a black zone 
$\Omega \sm (\Omega_1 \cup \Omega_2)$ in the middle, with lots of
thin parts,  cusps, and islands.

In this section we want to show that $\d \Omega_1$ is a locally Ahlfors regular
and (even uniformly) rectifiable set. We shall also get a reproducing formula  
for $\Delta u_{1,+}$, which will be used later, 
once we have a better description of the blow-up limits of $\u$. 
See Proposition~\ref{t17.2}.

For all this, we shall mostly follow the initial
arguments of \cite{AC}; 
this will be not be too hard to do because, as soon as we have the 
nondegeneracy results of Section \ref{good}, the other components $u_i$
$i \geq 2$, only play a small role in the estimates. 
Set 
\begin{equation} \label{e17.4}
v = u_{1,+} = \max(0,u_1) = u_1 \1_{\Omega_1}.
\end{equation}
The hero of this section is  $\mu$, the restriction of the distribution $\Delta v$ to
$\d \Omega_1$ (see a correct definition below). We shall prove that $\mu$ is 
in fact a (positive!) locally Ahlfors-regular measure, and this will help us prove that
$\Omega_1$ is locally a Caccioppoli set (a set with finite perimeter), 
with a reduced boundary almost equal to $\d \Omega_1$. 
The local Ahlfors-regularity and rectifiability of $\d \Omega$,
and the representation formulas of Proposition \ref{t17.2} will then easily follow.

Recall from \eqref{e9.6} that in $\Omega_1$, $v = u_1$ satisfies the equation
$\Delta v = f_1 v - {1\over 2} g_1$; the official definition of our hero $\mu$ is
the distribution 
\begin{equation} \label{e17.5}
\mu = \Delta v - [f_1 v - {1\over 2} g_1 ]\, \1_{\Omega_1}.
\end{equation}

\begin{pro}\label{t17.1} 
Let $(\u,\W)$, $W_1$ and $B_0$ satisfy the assumptions above
(up to \eqref{e17.3}), and let $\mu$ be the distribution defined by \eqref{e17.5}.
Then $\mu$ is, in $B_0$, a locally Ahlfors-regular positive measure whose 
support is $\d \Omega_1$. More precisely, there are constants 
$C_1 \geq 1$ and $r_0 \leq 1$ such that 
\begin{equation} \label{e17.6}
C_1^{-1} r^{n-1} \leq \mu(B(x,r)) \leq C_1 r^{n-1}
\end{equation} 
for $x\in \d\Omega_1$ and $0 < r \leq r_0$ such that $B(x,2r) \i B_0$.
The constant $C_1$ depends on $n$, the $L^\infty$ bounds in \eqref{e17.1},
the Lipschitz bound in \eqref{e17.2}, the Lipschitz constant in \eqref{e10.2}, and 
the constant $\lambda$ in that shows up in \eqref{e13.1}.
The radius $r_0$ depends on these constants, plus the $\varepsilon > 0$ 
from \eqref{e13.1}.
\end{pro}

\begin{proof}
We start with the positivity.
By \eqref{e9.6}, 
$\Delta v = f_1 v - {1\over 2} g_1 \geq - C$ in $\Omega_1\cap B_0$,
where $C$ is a large constant that we don't even want to compute.
Then Remark 1.4 in \cite{CJK} says that $\Delta v \geq -C$.
That is, $\Delta v + C$ is a positive distribution, and this implies that it is 
also a positive measure. Then $\Delta v$ and $\mu$ are measures too. 
We want to check that these three measures have the same restriction to 
$\d \Omega_1 \cap B_0$. Let us first check that
\begin{equation} \label{e17.7}
|\d \Omega_1 \cap B_0| = 0.
\end{equation}
Suppose not. Then we can find a point $x_0 \in \d\Omega_1 \cap B_0$,
which is a Lebesgue density point of $\d \Omega_1 \cap B_0$.
We apply Theorem \ref{t13.1} to this point, with a very small radius $r$,
and find out that $|\Omega_1 \cap B(x,r)| \geq c_2 r^n$
by \eqref{e13.7}. When $r$ is small, this is not compatible with
the definition of a Lebesgue point for $\d\Omega_1$; hence  
\eqref{e17.7} holds.

It easily follows from \eqref{e17.7} that the restrictions of our 
three measures to $\Omega_1 \cap B_0$ are the same; this shows that
$\mu$ is positive, like $\Delta v + C$, and we also get from this that
in $B_0$,
\begin{equation} \label{e17.8}
\text{$\mu$ is the restriction of $\Delta v$ to $\d \Omega_1$,}
\end{equation}
just because the definition \eqref{e17.5} makes it vanish on $\d \Omega_1$,
and the three measure vanish on the open set $\R^n \sm \overline \Omega_1$,
where $v=0$.

Next we want to check that $\mu$ is locally Ahlfors-regular of dimension $n-1$.
We start with the upper estimate. Let $B(x,r)$ be as in the statement, and
let $\varphi$ be a nonnegative smooth function
such that $\varphi(y) = 1$ for $y\in B(x,r)$, $|\nabla \varphi| \leq C r^{-1}$
everywhere, and $\varphi$ is compactly supported in $B(x,2r)$. Then
\begin{eqnarray} \label{e17.9}
\mu(B(x,r)) &\leq& \int \varphi d\mu = \langle \Delta v, \varphi \rangle
- \int_{\Omega_1} [f_1 v - {1\over 2} g_1 ] \varphi
\nonumber\\
&\leq& \langle \Delta v, \varphi \rangle + C ||\varphi||_\infty |B(x,2r)|
\leq |\langle \nabla v,\nabla \varphi\rangle| + Cr^n
\nonumber\\
&\leq& C ||\nabla v||_\infty ||\nabla \varphi||_{L^\infty(B_0)} r^n + C r^n \leq C r^{n-1}
\end{eqnarray}
because $\mu$ is positive, by definition of $\mu$ and of a distribution derivative
(and because we know that $v$ is locally Lipschitz),
and because $|\nabla \varphi| \leq C r^{-1}$, $|\nabla v| \leq C_0$, 
and $r \leq r_0 \leq 1$. This was shockingly easy, but positivity helped
a great deal.

For the lower bound, we shall need to choose our test function more precisely
and use results of Section \ref{good}.
Let us first choose an intermediate radius  $\rho \in (0,r)$.
If $r$ is small enough (depending on $\varepsilon$ in the condition \eqref{e13.1}),
Theorem \ref{t13.1} says that $\fint_{B(x,r)} u_{1,+}^2 \geq c_1 r^2$,
where $u_{1,+} = \max(0,u_1) = v$ and for some $c_1 > 0$ that depends only 
on the constants cited in the statement of Lemma \ref{t17.1}. 
In particular the needed bound on $\fint_{B(x,r)} |\nabla u_1|^2$
follows at once from our Lipschitz bound \eqref{e17.2}. 
By \eqref{e4.3} we can choose $\rho \in (0,r)$ such that
\begin{equation} \label{e17.10}
\int_{\d B(x,\rho)} v^2 \geq r^{-1} \int_{B(x,r)} v^2 \geq C^{-1} r^{n+1}.
\end{equation}
But $x\in \d\Omega_1$, so $u_1(x) = 0$, $v(x)=0$, and so
$v \leq C r$ on $B(x,r)$, by \eqref{e17.2}. 
So $\fint_{\d B(x,\rho)} v^2 \leq C r \fint_{\d B(x,\rho)} v$
and \eqref{e17.10} yields
\begin{equation} \label{e17.11}
\fint_{\d B(x,\rho)} v \geq  C^{-1} r.
\end{equation}
Notice that this forces $\rho \geq C^{-1} r$, because 
$\fint_{\d B(x,\rho)} v \leq C_0 \rho$ (again because $v(x)=0$ and by \eqref{e17.2}).

Let $\tau > 0$ be small (so that $\tau r < \rho$ in particular), 
and define a test function $\varphi$ by 
\begin{equation} \label{e17.12}
\varphi(y) = f(|y-x|),
\end{equation}
where $f : [0,+\infty) \to [0,1]$ is defined as follows. We set
$f(t) = 1$ for $0 \leq t \leq \tau r$, $f(t) =  0$ for $t \geq \rho$, and,
in the remaining region where $\tau r \leq t \leq \rho$, 
\begin{equation} \label{e17.13}
f(t) = {\tau^{n-2}\over 1-\tau^{n-2}}\,{\rho^{n-2} - t^{n-2} \over t^{n-2}}
\ \text{ when } n \geq 2
\end{equation}
and 
\begin{equation} \label{e17.14}
f(t) = (\log(1/\tau))^{-1} \log(\rho/t)
\ \text{ when } n = 2.
\end{equation}
[Here is a typical place when we don't really want to consider $n=1$.]
The point of this choice is that $f$ is continuous and 
$f'(t) = - a_n(\tau) \rho^{n-2} t^{1-n}$ for $\tau r < t < r$, 
where $a_n(\tau)$ is a positive constant that we don't need to compute.
We would like to use the fact that
\begin{eqnarray} \label{e17.15}
\langle \Delta v, \varphi \rangle &=& - \langle \nabla v, \nabla \varphi \rangle
= - \int_{B(x, \rho)\sm B(x,\tau r)} f'(|y-x|) {\d v \over \d r }(y) dy
\nonumber\\
&=& a_n(\tau) \rho^{n-2} 
\int_{B(x, \rho)\sm B(x,\tau r)} |y-x|^{1-n}{\d v \over \d r }(y) dy,
\end{eqnarray}
but $\varphi$ is not a real test function, so we have to do something
about it. There is no problem with the other identities, because
$v$ is Lipschitz; in particular the second line just comes from our 
specific choice of $\varphi$. Rather than doing an approximating now, 
let us keep it for later. In the mean time, let us still compute 
right-hand side of \eqref{e17.15}. Set $g(t) =  \fint_{\d B_(x,t)} v$ 
for $\tau r \leq t \leq r$, and observe that $g$ is Lipschitz, with
\begin{equation} \label{e17.16}
g'(t) = {\d \over \d t}\Big(\fint_{S_1} v(x+t\theta) d\sigma(\theta)\Big)
= \fint_{S_1} {\d v \over \d r}(x+t\theta) d\sigma(\theta)
= \sigma(S_1)^{-1} t^{1-n} \int_{\d B(x,t)} {\d v \over \d r},
\end{equation}
so that
\begin{eqnarray} \label{e17.17}
g(\rho) - g(\tau r) &=& \int_{t = \tau r}^\rho g'(t) dt
= \sigma(S_1)^{-1} \int_{t = \tau r}^\rho t^{1-n} \int_{\d B(x,t)} {\d v \over \d r}
\nonumber\\
&=&  \sigma(S_1)^{-1} \int_{B(x, \rho)\sm B(x,\tau r)} |y-x|^{1-n}{\d v \over \d r}
= - \sigma(S_1)^{-1} a_n(\tau)^{-1} \rho^{2-n}
\langle \nabla v, \nabla \varphi \rangle
\end{eqnarray}
by the correct part of \eqref{e17.15}.
Recall that $g(\rho) \geq C^{-1} r$ by \eqref{e17.11}, and 
that, since $v$ is $C_0$-Lipschitz, $g(\tau r) \leq C_0 \tau r$.
We choose $\tau$ so small that this implies that
$g(\rho) -g(\tau r) \geq (2C)^{-1} r$. Then \eqref{e17.17} yields
\begin{equation} \label{e17.18}
- \langle \nabla v, \nabla \varphi \rangle 
= \sigma(S_1) a_n(\tau) \rho^{n-2} [g(\rho) - g(\tau r)]
\geq C(\tau)^{-1} r^{n-1},
\end{equation}
because $\rho \geq r/C$.

Now we approximate $\varphi$.
For $\eta > 0$ small, pick a function $\varphi_\eta$, such that
$\varphi_\eta (y) = f_\eta(|y-x|)$, where $f_\eta$ is smooth, coincides
with $f$ except on the interval $(\tau r, \tau+\eta)$, 
and yet $|f'_\eta| \leq C(\tau,r)$ for some constant $C(\tau,r)$ that does not 
depend on $\eta$. Then 
\begin{eqnarray} \label{e17.19}
\big| \langle \nabla v, \nabla \varphi \rangle
-  \langle \nabla v, \nabla \varphi_\eta \rangle \big| 
&\leq& \int_{B(x,\tau r+\eta)\sm B(x,\tau)} |\nabla v| 
\,|\nabla \varphi-\nabla \varphi_\eta|
\nonumber\\
&\leq& C ||\nabla \varphi-\nabla \varphi_\eta||_\infty 
|B(x,\tau r+\eta)\sm B(x,\tau)| \leq C'(\tau,r) \eta 
\end{eqnarray}
because $\nabla \varphi_t = \nabla \varphi$ most of the time.
So, if $\eta$ is small enough (depending on $\tau$ and $r$), \eqref{e17.18} yields
\begin{equation} \label{e17.20}
- \langle \nabla v, \nabla \varphi_\eta \rangle 
\geq (2C(\tau))^{-1} r^{n-1}.
\end{equation}
But now $\varphi_\eta$ is smooth, so 
$\langle \Delta v, \varphi_\eta \rangle 
= - \langle \nabla v, \nabla \varphi_\eta \rangle$
by definitions, and 
\begin{eqnarray} \label{e17.21}
\mu(B(x,r)) &\geq& \int \varphi_\eta d\mu 
= \langle \Delta v, \varphi_\eta \rangle -
\int [f_1 v - {1\over 2} g_1 ]\, \1_{\Omega_1} \varphi
\nonumber\\
&\geq& (2C(\tau))^{-1} r^{n-1} - C r^n \geq (4C(\tau))^{-1} r^{n-1} 
\end{eqnarray}
because $\mu$ is a positive measure and (we can very easily arrange that)
$0 \leq \varphi_\eta \leq \1_{B(x,\rho)} \leq \1_{B(x,r)}$, then by \eqref{e17.5}, \eqref{e17.18}, and if $r_0$  is small enough. 
This completes our proof of \eqref{e17.6} and Proposition \ref{t17.1}.
\qed
\end{proof}

\ms
Proposition \ref{t17.1} has a few consequences that we record now.
Let $B_1 \i B_0$ be a strictly smaller open ball. Since \eqref{e17.6} holds for 
every ball of small enough radius centered on $\d \Omega_1 \cap B_1$,
an easy covering argument shows that on $B_1$, $\mu$ is equivalent to the
restriction of $\H^{n-1}$ to $\d \Omega_1 \cap B_1$, in the sense that
\begin{equation} \label{e17.22}
C^{-1} \H^{n-1}(E\cap \d \Omega_1) \leq \mu(E) \leq C\H^{n-1}(E\cap \d \Omega_1)
\ \text{ for every Borel set $E \i B_1$.}
\end{equation}
See for instance Lemma 18.11 on page 109 of \cite{D}. In particular,
$\H^{n-1}(\d \Omega_1 \cap B_1) < +\infty$ (by \eqref{e17.6}), 
and it is classical that in such a case,
$\Omega_1$ is a set of finite perimeter in $B_1$, with
$Per(\Omega_1; B_1) \leq \H^{n-1}(\d \Omega_1 \cap B_1) < +\infty$.
See for instance \cite{Gi}. 
This means that the restriction to $B_1$ of the vector-valued distribution
$\nabla \1_{\Omega_1}$ is in fact a (vector valued) measure, whose total variation,
often denoted by $Per(\Omega_1; B_1)$, is finite. More precisely, we may write
\begin{equation} \label{e17.23}
\nabla \1_{\Omega_1} = n \nu
\end{equation}
in $B_1$, where $\nu$ is the (total) variation of $\nabla \1_{\Omega_1}$
(a positive measure), and $n$ is a measurable function, with values in 
the set of unit vectors in $\R^n$ (the inwards unit vector, which is defined
$\nu$-almost everywhere). It is also known (see \cite{Gi})
that inside $B_1$, $\nu$ is in fact the restriction of $\H^{n-1}$ to 
the reduced boundary of $\Omega_1$, which is usually denoted
by $\d^\ast \Omega_1$, and is in general a Borel set
that is strictly contained in  $\d \Omega_1$.
For instance, corners of an otherwise smooth domain would lie in
$\d \Omega_1 \sm \d^\ast \Omega_1$.
Let us check that in fact
\begin{equation} \label{e17.24}
\H^{n-1}(B_0 \cap \d \Omega_1 \sm \d^\ast \Omega_1) = 0
\ \text{ and }\  \nu =  \H^{n-1}_{\vert \d \Omega_1}.
\end{equation}
The second affirmation will follow at once from the first one (since $\nu$
is the restriction of $\H^{n-1}$ to $\d^\ast \Omega_1$), and for the first one
let us first prove that for $B_1 \i B_0$ as above
there exist constants $C > 0$ and $r_1 > 0$ such that
\begin{equation} \label{e17.25}
\nu(B(x,r)) \geq C^{-1} r^{n-1}
\end{equation}
for $x\in \d\Omega_1 \cap B_1$ and $0 < r \leq r_1$.
Indeed, for such $B(x,r)$ (and if $r_1$ is small enough, in particular so
that $B(x,2r) \i B_0$), \eqref{e13.7} says that
$|\Omega_1 \cap B(x,r)| \geq c_2 r^n$, and \eqref{e13.65} says that
$|B(x,r) \sm \Omega_1| = 
\big|\big\{ x\in B(x,r) \, ; \, u_1(x) \leq 0 \big\}\big| \geq c_6 r^n$.
By the isoperimetric inequality in $B(x,r)$, we deduce from this that
$\nu(B(x,r)) = Per(\Omega_1; B(x,r)) \geq C^{-1} r^{n-1}$, as announced.
Since $\nu(B(x,r)) = H^{n-1}(\d^\ast\Omega_1 \cap B(x,r))
\leq \H^{n-1}(\d\Omega_1 \cap B(x,r)) \leq C r^{n-1}$
by \eqref{e17.6}, we see that $\nu$ also is locally Ahlfors-regular, 
and hence equivalent to $\H^{n-1}_{\vert \d \Omega_1}$. In particular, 
$\H^{n-1}(\d \Omega_1 \sm \d^\ast \Omega_1)
\leq C \nu(\d \Omega_1 \sm \d^\ast \Omega_1) = 0$
(because $\nu$ is supported on the Borel set $\d^\ast \Omega_1$),
which proves \eqref{e17.24}.

Notice that since the reduced boundary of a set with locally finite perimeter
is always rectifiable, \eqref{e17.24} implies that 
\begin{equation} \label{e17.26}
\d \Omega_1 \cap B_0 \text{ is rectifiable.} 
\end{equation}

In fact, we can even prove that $\d \Omega_1$ is uniformly rectifiable
in $B_1$, with big pieces of Lipschitz graphs. 
See Proposition \ref{t17.3} below. 
But let us continue with the same program as in \cite{AC} 
and discuss a representation formula for $\mu$. We know from 
Proposition \ref{t17.1} and its consequence \eqref{e17.22}
that $\mu$ is absolutely continuous with respect to $\H^{n-1}_{\vert \d \Omega_1}$,
which is therefore locally finite.
A differentiation result for measures in $\R^n$ allow us to write
\begin{equation} \label{e17.27}
\mu = h \H^{n-1}_{\vert \d \Omega_1} = h \nu
\end{equation}
on $B_0$, where the density $h$ can be computed by differentiation, i.e.,
\begin{equation} \label{e17.28}
h(x) = \lim_{r \to 0} {\mu(B(x,r)) \over \H^{n-1}(\d \Omega_1 \cap B(x,r))}
= \lim_{r \to 0} {\mu(B(x,r)) \over \nu(B(x,r))}
\end{equation}
for $\H^{n-1}$-almost every $x\in \d \Omega_1 \cap B_0$.
See for instance \cite{M} for this and the next discussion. 
Now $\d \Omega_1 \cap B_0$ is rectifiable, so
$\lim_{r \to 0} r^{1-n} \H^{n-1}(\d \Omega_1 \cap B(x,r)) = \omega_{n-1}$
for $\H^{n-1}$-almost every $x\in \d \Omega_1 \cap B_0$,
where $\omega_{n-1} = \H^{n-1}(\R^{n-1} \cap B(0,1))$ denotes the 
$\H^{n-1}$-measure of the unit ball in $\R^{n-1}$. We could also say that
$\lim_{r \to 0} r^{1-n}\nu(B(x,r)) = \omega_{n-1}$ everywhere on $\d^\ast \Omega_1$,
by definition of the reduced boundary, and get the same conclusion. Anyway, we 
deduce from \eqref{e17.28} that for $\H^{n-1}$-almost every 
$x\in \d \Omega_1 \cap B_0$,
\begin{equation} \label{e17.29}
h(x) = {1 \over \omega_{n-1}} \lim_{r \to 0} \, r^{1-n}\mu(B(x,r)).
\end{equation}
Also notice that in \eqref{e17.5}, the contribution of 
$[f_1 v - {1\over 2} g_1 ]\, \1_{\Omega_1}$ to small balls $B(x,r)$
is negligible compared to $r^{n-1}$, so \eqref{e17.29} also says that
\begin{equation} \label{e17.30}
h(x) = {1 \over \omega_{n-1}} \lim_{r \to 0} \, r^{1-n}
\langle \Delta u_{1,+}, \1_{B(x,r)}\rangle,
\end{equation}
where the right-hand side makes sense because $\Delta u_{1,+}$
is a locally finite measure.
Let us summarize this and a little bit of the previous results.

\begin{pro}\label{t17.2}
Let $(\u,\W)$ be a minimizer for $J$, assume that the $f_i$ and $g_i$
are bounded (as in $\eqref{e17.1}$), that $F$ is Lipschitz (i.e., \eqref{e10.2} holds), 
and that $W_1$ is a good region (i.e. \eqref{e13.1} holds).
Also assume that $u_1$ is Lipschitz on some open ball $B_0$.
Then, inside $B_0$, the Laplacian of $\u_{1,+} = \max(u_1,0)$ can be decomposed as
$\Delta u_{1,+} = \mu + [f_1 v - {1\over 2} g_1 ]\, \1_{\Omega_1}$,
where $\mu$ is a locally Ahlfors-regular measure supported on 
$\d \Omega_1$, and in addition $\mu = h \H^{n-1}_{| \d \Omega_1}$,
with a density $h$ that can be computed by \eqref{e17.28}, \eqref{e17.29},
or \eqref{e17.30}.
\end{pro}

The reader may wonder why we find it so interesting to have an expression
of $\Delta u_{1,+}$ in terms of its integrals on small balls. In some cases, 
for instance when we have a good control on the blow-up limits of $(\u,\W)$,
it will be possible to estimate $\langle \Delta u_{1,+}, \1_{B(x,r)}\rangle$
and compute it. See Section \ref{vari}.

Notice that in very smooth cases, $\Delta u_{1,+}$ is expected to be the jump
of the normal derivative of $u_{1,+}$ along $\d \Omega_1$, which plays some
role in the first variation of our functional $J$.

\begin{pro}\label{t17.3}
Let $(\u,\W)$, $W_1$ and $B_0$ satisfy the assumptions of Proposition \ref{t17.2}.
Then $\d\Omega_1$ is locally uniformly rectifiable in $B_0$, with big pieces of
Lipschitz graphs. This means that there
are constants $C_2 \geq 1$ and $r_1 \leq 1$ such that,
for each $x\in \d\Omega_1$ and $0 < r \leq r_1$ such that $B(x,2r) \i B_0$,
we can find a $C_2$-Lipschitz graph $\Gamma$ such that
\begin{equation} \label{e17.31}
\H^{n-1}(\Gamma \cap \d\Omega_1\cap B(x,r)) \geq C_2^{-1} r^{n-1}.
\end{equation} 
The constant $C_2$ depends on $n$, the $L^\infty$ bounds in \eqref{e17.1},
the Lipschitz bound in \eqref{e17.2}, the Lipschitz constant in \eqref{e10.2}, and 
the constant $\lambda$ in that shows up in \eqref{e13.1}.
The radius $r_1$ depends on these constants, plus the $\varepsilon > 0$ 
from \eqref{e13.1}.
\end{pro}

Of course Proposition \ref{t17.3} is a natural complement to Proposition \ref{t17.2};
we would usually not talk about uniform rectifiability unless $\d\Omega_1$ is 
locally Ahlfors-regular. A $C_2$-Lipschitz graph is just the graph, in some set
of orthonormal coordinates of $\R^n$, of a real valued Lipschitz function 
with a Lipschitz norm at most $C_2$. We refer to \cite{DS} 
for more information on uniform rectifiability.

Of course Proposition \ref{t17.3} is stronger that \eqref{e17.26}, even though this
is not completely obvious from the definition. The proof below does not use 
regularity properties of sets of finite perimeters, so it could seen as an alternative
route to \eqref{e17.26}, but in fact it relies on heavy machinery too. We shall
deduce Proposition \ref{t17.3} from the following lemma, which says that
$\d\Omega_1$ satisfies the so called Condition $B$ locally in $B_0$.

\begin{lem}\label{t17.4}
Keep the assumptions of Propositions \ref{t17.2} and \ref{t17.3}.
Then there are positive constants $c_7, c_8$, and $r_2 \leq 1$ such that,
for each $x\in \d\Omega_1$ and $0 < r \leq r_2$ such that $B(x,2r) \i B_0$,
we can points $y_+$ and $y_-$ in $B(x,r)$ such that
$B(y_+,c_7r) \i B(x,r) \cap \Omega_1$ and $B(y_-,c_8 r), \i B(x,r) \sm \Omega_1$.
The constants $c_7$ and $c_8$ depend on 
the same parameters as announced for $C_2$ above, and $r_2$ may also
depend on $\varepsilon$.
\end{lem}

\ms
Proposition \ref{t17.3} is a direct consequence of Lemma \ref{t17.4},
because any Ahlfors-regular set $E$ that satisfies Condition $B$ contains
big pieces of Lipschitz graphs. The condition was introduced by 
Semmes \cite{Se}, 
who proved that it implies the $L^2$-boundedness of
many singular integral operators on $E$. The existence of 
big pieces of Lipschitz graphs were proved slightly later; 
see \cite{Daa} or \cite{DJ}. 
Here there is a minor additional detail, which is that the references
above concern unbounded Ahlfors-regular sets, and we want a local version
of these results. This is not a serious issue; one check that 
both proofs go through (and we recommend the shorter second one).

\ms
So we just need to prove Lemma \ref{t17.4}. 
Let $B(x,r)$ be as in the statement; if $r_1$ is small enough, we 
can apply Theorem \ref{t13.1} to $B(x,r/2)$, and \eqref{e13.6} says that
$\fint_{B(x,r/2)} |u_{1,+}|^2 \geq c_1 r^2$.
So we can find $y_+ \in B(x,r/2)$ such that $u_{1,+}(y_+)^2 \geq c_1 r^2/4$.
But $u_{1,+}$ is $C_0$-lipschitz on $B(x,r)$, so
$u_{1,+}(z) > 0$ if $z\in B(x,r)$ is such that $|z-y_+| < C_0^{-1}\sqrt{c_1} r/2$; 
that is, $B(y_+,c_7 r) \i \Omega_1$ as soon as 
$c_7 < {1 \over 2} \min(\sqrt{c_1} C_0^{-1},1)$.

We do not have such a good estimate on the other component
$\R^n \sm \overline\Omega_1$, but at least Proposition \ref{t13.4}
says that for $B(x,r)$ as above (hence with $r$ small enough),
\begin{equation} 
\label{e17.32}
\big|\big\{ z\in B(x,r) \, ; \, u_1(z) \leq 0 \big\}\big| \geq c_6 r^n.   
\end{equation}
It turns out that this is enough: the existence of the points $y_-$
(again for $r$ small enough, and with a constant $c_8$ which is much
smaller than $c_7$) follows from \eqref{e17.32} and the existence of 
the points $y_{+}$ by a fairly simple porosity  argument. 
Here is the idea. Suppose that for some pair $(x,r)$ we cannot find $y_-$;
we want to show that the set $H = B(x,r/2)\sm \Omega_1$ is porous, at least at 
the scales $\rho \in [c_8 r, r/2]$. Indeed, the fact that we cannot find $y_-$ 
implies that $\dist(y,\d\Omega_1) \leq c_8 r$ for each $y\in H$,
so we we can choose $z=z(y)\in \d\Omega_1$ such that $|z-y| \leq c_8 r$, 
and for each radius $\rho \in [c_8 r, r/2]$, we apply the known existence of 
points $y_+$ to the pairs $(y,\rho)$, and get a point $y_+(y,\rho)$
such that $B(y_+(y,\rho),c_7 \rho) \i B(z,\rho) \cap \Omega_1
\i B(y,2\rho) \sm H$. This is what we mean by porous at 
the scales $\rho \in [c_8 r, r/2]$. It is known that porous sets
(at all scales) have vanishing measure, and are even of Hausdorff dimension
smaller than $n$; the proof of this fact also shows that 
$|H| \leq \eta(c_7,c_8) r^n$, with a function $\eta$ which for each fixed
value of $c_7$ (the porosity constant) tends to $0$ when $c_8$ tends to $0$;
then we take $c_8$ small and contradict \eqref{e17.32}.
We shall not include the details of the argument here, because they are the same
as in Proposition 10.3 of \cite{DT}, for instance, 
which even concerns a similar situation, and uses fairly close notation.
\qed

\section{Limits of minimizers}    \label{limits} 

The main point of this section is not to give a general theory of limits,
but mostly to allow a description of the blow-up limits 
of a given minimizer at a point, and give a little more information on the convergence 
to  the blow-up limits. It will be convenient to work with the following notion
of local minimizers.

We are given an open set $\O$, where we will work (and an open ball of $\R^n$
is our main example), and a measurable set $\Omega \i \O$, and we define
$\F = \F(\O,\Omega)$ to be the set of pairs $(\u,\W)$ such that 
$\u = (u_1, \ldots u_N)$ is a $N$-uple of functions $u_i \in W^{1,2}_{loc}(\O)$, and 
$\W = (W_1, \ldots W_N)$ is a $N$-uple of disjoint measurable subsets of
$\Omega$ such that $u_i = 0$ almost everywhere on $\O \sm W_i$.
When we say that $u_i \in W^{1,2}_{loc}(\O)$, we just mean that
$u_i \in W^{1,2}(B)$ for every ball $B$ such that $\overline B \i \O$.

We are also given a function $F$, defined
on the set ${\cal W}(\Omega)$ of $N$-uples of disjoint measurable 
subsets of $\Omega$, and measurable functions $f_i$ and $g_i$, $1 \leq i \leq N$, 
defined on $\O$ (but only the values on $\Omega$ matter). We shall assume that
these functions are bounded; lesser assumptions, in particular on the $g_i$
(and then on the $g_{i,k}$ below) would be enough, but usually we shall need 
these strong ones to check the other assumptions of the Theorem \ref{t14.1} anyway.

We say that a pair $(\u^\ast,\W^\ast)$ is \underbar{a competitor for $(\u,\W) \in \F$
in $\O$, relative to $\Omega$}, when $(\u^\ast,\W^\ast) \in \F$ there is a compact set 
$K \i \O$ such that $(\u^\ast,\W^\ast)$ coincides with $(\u,\W)$ in $\O \sm K$.
Then we say that $(\u,\W) \in \F$ is a \underbar{local minimizer for 
$J$ in $\O$, relative to $\Omega$}, when 
\begin{equation} \label{e14.1}
\int_K |\nabla \u|^2 + \sum_i \int_K [u_i^2f_i - u_i g_i] + F(\W)
\leq \int_K |\nabla \u^\ast|^2 + \sum_i \int_K [(u_i^\ast)^2f_i - u_i^\ast g_i] + F(\W^\ast)
\end{equation}
whenever $(\u^\ast,\W^\ast)$ is a competitor for $(\u,\W) \in \F$ in $\O$
and $K \i \O$ is a compact set such that $(\u^\ast,\W^\ast)$ 
coincides with $(\u,\W)$ in $\O \sm K$. We would have liked to say that
$J(\u,\W) \leq J(\u^\ast,\W^\ast)$, but these numbers may be infinite,
so \eqref{e14.1} is a good substitute for this.

Local minimizers can be thought of as a substitute of minimizers under a Dirichlet
constraint at the boundary of $\O$ (in a strong way, since competitors have to coincide
with $(\u,\W)$ in a neighborhood of $\d\O$), where the Dirichlet data is not given
in advance, and just comes from the pair $(\u,\W)$ itself.

Our definition is not perfect for large sets $\O$, because we decided to use 
functionals $F$ that may not be local, so we are essentially forced to 
assume that $F$ is defined globally on ${\cal W}(\Omega)$, and finite. 
For functions $F$ defined by \eqref{e1.7}, we should restrict $F$ to $K$
and rewrite \eqref{e14.1} accordingly.

Even when $\O$ bounded, the definition allows the possibility that
$\int_{\Omega} |\nabla \u|^2 = +\infty$, and then we should use \eqref{e14.1}.
But most often $J(\u,\W) < +\infty$ and we can use the simpler form 
$J(\u,\W) \leq J(\u^\ast,\W^\ast)$.

\ms
Notice that the regularity results that we proved so far also hold, locally inside $\O$, 
for local minimizers for $J$ in $\O$, and with essentially the same proof.
Let us say more specifically what we mean by this in the case of Theorem \ref{t11.1}.
We claim that if $\Omega$ is smooth, $F$ is Lipschitz, the $f_i$ and the $g_i$
satisfy \eqref{e10.1} and \eqref{e10.2}, and $(\u,\W)$ is a local minimizer for
$J$ in $\O$, then for each ball $B(x,r)$ such that $B(x,3r) \i \O$, 
the restriction of $\u$ to $B(x,r)$ is Lipschitz, with bounds that depend only on $n$, $N$, 
the regularity constants for $\Omega$, $L^\infty$ bounds for the $f_i$ and the $g_i$, $r$, 
and initial bounds $\int_{B(x,2r)} |\nabla \u|^2$ and $||\u||_{L^\infty(B(x,2r))}$.
The proof just consists in following the proof of Theorem \ref{t11.1}.
Probably, we could even dispense with this last bound on $||\u||_{L^\infty(B(x,r))}$,
but this is not the point of the remark.

\ms
Let us now set the notation for the next result. We are given an open set $\O$
(for instance, an open ball),  a sequence $\{ \Omega_k \}_{k \geq 0}$
of measurable subsets of $\O$, sequences $\{ f_{i,k} \}$ and $\{ g_{i,k} \}$ of 
bounded functions on $\O$, with $f_{i,k} \geq 0$ on $\O$,
and even a sequence of functions $F_k$ defined on the corresponding
${\cal W}(\Omega_k)$. 
We shall assume that there is a constant $C_0$ such that 
\begin{equation} \label{e14.2}
||f_{i,k}||_\infty \leq C_0 \ \text{ and } \ 
||g_{i,k}||_\infty \leq C_0
\text{ for every $k$,} 
\end{equation}
and that there exist weak limits $f_i$ and $g_i$ on $\O$, by which we mean that
\begin{equation} \label{e14.3}
\lim_{k \to +\infty}\int f_{i,k} \varphi = \int f_{i} \varphi
\ \text{ and } \ 
\lim_{k \to +\infty}\int g_{i,k} \varphi = \int g_{i} \varphi 
\end{equation}
for every continuous function $\varphi$ with compact support in $\O$.

We shall also make our life simpler, as in Section \ref{existence}, and assume that
$F_k$ has a simple form for which it will be easy to take limits. We shall start with
the case when $F_k$ is coming from a function of the volumes, i.e., when
\begin{equation} \label{e14.4}
F_k(W_1, \ldots, W_N) = \wt F_k(|W_1|, \ldots, |W_N|)
\end{equation}
for some function $F_k : [0,|\O|]^N \to \R$. We shall assume that 
$|\O| < +\infty$, and that 
\begin{equation} \label{e14.5}
\text{each $\wt F_k : [0,|\O|]^n \to \R$ is continuous, and the $\wt F_k$ 
converge uniformly to a limit $\wt F$.}
\end{equation}
We required $\wt F_k$ to be defined on $[0,|\O|]^N$ when $[0,|\Omega_k|]^N$
would have been enough, so that the $\wt F_k$ have a common domain
of definition. It costs us very little, because functions $\wt F_k$ on $[0,|\Omega_k|]^N$
would be easy to extend. 

These functions $F_k$ fit with the original motivation of the paper,
but we shall also give a statement for functions $F_k$ defined by \eqref{e1.7},
as in the standard setting of Alt, Caffarelli, and Friedman, are possible.
See Corollary \ref{t14.5} at the end of the section.

For the domains $\Omega_k$, we assume the existence of a measurable 
set $\Omega$ such that
\begin{equation} \label{e14.6}
 \lim_{k \to +\infty} \1_{\Omega_k} = \1_{\Omega}
\ \text{ in } L^1(\O),
\end{equation}
and the following weak regularity property of $\Omega$, which will be used
to approximate Sobolev functions by compactly supported ones. We suppose that
for each compact set $K \i \O$, there exist $r_K > 0$ and $c_K > 0$ such that 
\begin{equation} \label{e14.7}
|B(x,r) \sm \Omega| \geq c_K r^n
\ \text{ for $x\in K \cap \d \Omega$ and } 0 < r \leq r_K.
\end{equation}
In addition, \eqref{e14.6} is a little too weak to prevent something that we don't want:
the $\Omega_k$ may have islands of $\O \sm \Omega_k$ inside them,
with very small masses so that \eqref{e14.6} does not see it, but which become
dense in $\Omega$. If this happens, it could be that the $\u_k$ converge to $0$ 
because they need to vanish on $\O \sm \Omega_k$, but $\Omega$ is a nice ball 
for which some bump function will do better. 
So we also assume that for each compact set $K \i \O$, 
\begin{equation} \label{e14.8}
\lim_{k \to+\infty} \delta(K,k) = 0,
\ \text{ where }
\delta(K,k) = \sup\big\{ \dist(x,\O \sm \Omega) \, ; \, x\in K \sm \Omega_k\big\}.
\end{equation}
Since the numbers $\delta(K,k)$ are sensitive to adding small useless pieces to
$\d\Omega_k$, we should probably replace $\O \sm \Omega$ by the smaller
set $Z(k) = \big\{x\in \O \, ; \, |\O \cap B(x,r)\sm \Omega| > 0 \text{ for } r > 0 \big\}$
before we check \eqref{e14.8}. It is easy to see that $|(\O \sm \Omega_k)\sm Z(k)|=0$,
so $\Omega_k$ and $\O \sm Z(k)$ are equivalent for our functional. We do not
need to take this precaution for $\Omega$, because it is already included in 
\eqref{e14.7}. Finally, if the $\Omega_k$ satisfy \eqref{e14.7} uniformly,
then \eqref{e14.8} follows from \eqref{e14.6}; see the second part of the
proof of Lemma \ref{t16.1} below.

These three assumptions sound weak, but remember that we shall need to assume
more regularity on the $\Omega_k$ if we want to make sure that the $\u_k$
are uniformly Lipschitz, as in \eqref{e14.10}.

In addition to the data, we are also given a sequence of pairs 
$(\u_k,\W_k) \in \F(\O,\Omega_k)$, and we assume that
\begin{equation} \label{e14.9}
(\u_k,\W_k) \text{ is a local minimizer for $J_k$ in $\O$, relative to $\Omega_k$,}
\end{equation}
where $J_k$  is the analogue of $J$, but defined with the 
data $\Omega_k$, $f_{i,k}$, $g_{i,k}$, and $F_k$. 
We assume that for each compact ball $B \i \O$, there is a constant 
$C(B)$ such that
\begin{equation} \label{e14.10}
\u_k \text{ is $C(B)$-Lipschitz in $B$,}
\end{equation}
and also that there is a function $\u \in W^{1,2}(\O)$  
such that
\begin{equation} \label{e14.11}
\u(x) = \lim_{k \to +\infty} \u_k(x)
\ \text{ for $x\in \O$.}
\end{equation}
In practice, we shall obtain \eqref{e14.10} by an application of Theorem \ref{t11.1},
which means that we will have stronger assumptions on the $\Omega_k$,
and also explains why we don't try to give weaker assumptions on the $g_i$,
for instance. The existence of a subsequence for which the weak limits in
\eqref{e14.3} and the limit in \eqref{e14.11} will then be rather easy to get.
But again we do not try to give optimal assumptions here.

\begin{thm}\label{t14.1} 
Assume all the conditions above. Then we can find $\W \in {\cal W}(\Omega_k)$ 
such that $(\u,\W) \in \F(\O,\Omega)$ and 
$(\u,\W)$ is a local minimizer for $J$ in $\O$, relative to $\Omega$. In addition, 
\begin{equation} \label{e14.12}
\lim_{k \to +\infty} \u_k = \u \text{ in $W^{1,2}(B)$ for every ball such that 
$\overline B \i \O$.}
\end{equation}
\end{thm}  

\begin{proof}
The first thing that we need to do is define sets $W_i$, 
so that $(\u,\W) \in \F(\O,\Omega)$,
and we shall proceed as in Section \ref{existence}. 
Because of \eqref{e14.6}, we can find a subsequence $\{ k_j \}$ for which 
\begin{equation} \label{e14.13}
 \lim_{j \to +\infty} \1_{\Omega_{k_j}}(x) = \1_{\Omega}(x)
\end{equation}
for almost every $x\in \O$.

Denote by $Z$ the bad set of points $x\in \O$
such that \eqref{e14.13} fails, or there exists $k \geq 0$
such that $x\in W_{i,k}$ but $u_{i,k}(x) \neq 0$; then
$|Z| = 0$ by definitions. Next set
\begin{equation} \label{e14.14}
W'_{i,k} = \big\{ x\in \Omega_k \sm Z \, ; \, u_{i,k}(x) \neq 0 \big\}
\end{equation}
for each $k \geq 0$, and 
\begin{equation} \label{e14.15}
W'_{i} = \big\{ x\in \Omega \sm Z \, ; \, u_{i}(x) \neq 0 \big\}.
\end{equation}
Also set $\W' = (W'_1, \ldots, W'_N)$; 
we want to show that $(\u,\W') \in \F(\O,\Omega)$, and then we will
obtain the desired $\W$ by adding an extra piece to some $W'_i$, when needed.
Let us first check that
\begin{equation} \label{e14.16}
\text{the $W'_{i}$, $1 \leq i \leq N$, are disjoint}.
\end{equation}
If $x\in W'_{i,k}$, then $x\in W_{i,k}$ (by definition of $Z$); hence
the $W'_{i,k}$ are disjoint. Next, if $x\in W'_i$, then $x\in \Omega_{k_j}$
for $j$ large (because \eqref{e14.13} holds),
and $u_{i,k_j}(x) \neq 0$ (by \eqref{e14.11}), so $x\in W'_{i,k_j}$
for $j$ large. Hence \eqref{e14.16} holds.

In addition, we just proved that
$\1_{W'_{i}} \leq \liminf_{j \to +\infty} \1_{W'_{i,k_j}}$ everywhere, and
by Fatou
\begin{equation} \label{e14.17}
|W'_{i}| = \int \1_{W'_{i}} \leq \liminf_{j \to +\infty} \int\1_{W'_{i,k_j}}
= \liminf_{j \to +\infty} |W'_{i,k_j}| \leq \liminf_{j \to +\infty} |W_{i,k_j}|.
\end{equation}
We also need to know that
$u_i(x) = 0$ almost everywhere on $\O \sm W'_{i}$, and indeed if
$u_{i}(x) \neq 0$ but $x\notin W'_i \cup Z$, then for $j$ large,
$u_{i,k_j}(x) \neq 0$, hence $x\in W_{k_j}$ (because $x\notin Z$), 
so $x\in \Omega_{k_j}$, and then (by \eqref{e14.13}) $x\in \Omega$, 
which contradicts the definition \eqref{e14.15}. 
Notice also that $\u \in W^{1,2}_{loc}(\O)$, because $\u$ is locally Lipschitz
by \eqref{e14.13}; thus we proved that $(\u,\W')\in \F(\O,\Omega)$.

Observe that although our precise definition of $\W'$ depends on the subsequence
$\{ k_j \}$, this dependence is only through the set $Z$; since $|Z|=0$, different
subsequences would yield slightly different, but equivalent sets $W'_i$.

We may not be happy with the $W'_i$ because there may be a way to increase
some of their volumes and make $F(\W)$ smaller, so we will replace the $W'_i$
with possibly larger ones. We could try to make $F(\W)$ as large as possible
(given the natural constraints), but in fact making sure that the volumes 
of the $W^{i,k}$ go to the limit along a subsequence will be enough for our purposes.
Here is the place where we shall use the special form of the $F_k$.

Let us choose a new subsequence $\{ k_j \}$, which we even extract from 
the previous one, so that
\begin{equation} \label{e14.18}
l_i = \lim_{j \to +\infty} |W_{i,k_j}|
\ \text{ exists for each $i$;}
\end{equation}
recall that we assumed that $|\O| < +\infty$.
With this new information, \eqref{e14.17} just says that $|W'_{i}| \leq l_i$. 

Recall that for each $k$, the $W_{i,k}$ are disjoint and contained in $\Omega_k$
(because $(\u_k,\W_k) \in \F(\O,\Omega_k)$); then 
$\sum_i |W_{i,k_j}| \leq |\Omega_{k_j}|$ for each $j$ and
\begin{equation} \label{e14.19}
\sum_i |W'_{i}| \leq \sum_i l_i = \lim_{j \to +\infty} \sum_i |W_{i,k_j}|
\leq \liminf_{j \to +\infty} |\Omega_{k_j}| = |\Omega|,
\end{equation}
by \eqref{e14.6}.  We claim that we can chose disjoint measurable sets $W_i$, 
$1 \leq i \leq N$,  so that
\begin{equation} \label{e14.20}
W'_i \i W_i \i \Omega \ \text{ and } \ |W_i| = l_i \ \text{ for } 1 \leq i \leq N.
\end{equation}
Indeed, the $W'_i$ are disjoint and contained in $\Omega$ 
(see \eqref{e14.15} and \eqref{e14.16}), so we just need to cut a part of 
$\Omega \sm (\cup_i W'_i)$ into pieces, and add them to the $W'_i$
as needed. We could do this with the $l_i$ replaced by any numbers such that
$l_i \geq |W'_i|$ and $\sum_i l_i \leq |\Omega|$, but the present choice will be enough.

This is how we define $\W = (W_1, \ldots, W_N)$. Notice that
\begin{equation} \label{e14.21}
(\u,\W) \i \F(\O,\Omega)
\end{equation}
by construction, and our next task is to show that it is a local minimizer. 
That is, we are given a competitor
$(\u^\ast,\W^\ast) \i \F(\O,\Omega)$ for $(\u,\W)$ in $\O$, and we want to prove
that \eqref{e14.1} holds. As usual, the idea is to modify 
$(\u^\ast,\W^\ast)$ into a competitor $(\u^\ast_k,\W^\ast_k)$
for $(\u_k,\W_k)$, $k$  large, and use the minimality of $(\u_k,\W_k)$
to get some estimates.

We denote by $K_0$ a compact subset of $\O$ such that
$(\u^\ast,\W^\ast)$ coincides with $(\u,\W)$ on $\OÊ\sm K_0$
(as in the definition of competitors). 
Our construction will depend on four small positive constants, $\varepsilon_0$,
$\varepsilon$, $\delta$, and $\eta$, that eventually will all tend to $0$.
Our first action is to replace $K_0$ with a larger compact set $K$, with
$K_0 \i K \i \O$, and so large that 
\begin{equation} \label{e14.22}
|\O \sm K| \leq \varepsilon_0.
\end{equation}
This is possible, because $|\O| < +\infty$ and by the regularity of the 
Lebesgue measure, and this will be helpful when we control the volume
terms of the functional, because whatever happens in $\O \sm K$ 
will not change this term much.

In the estimates that follow, we shall not mark the dependence of the various 
constants on $\varepsilon_0$ and $K$, but we will be more careful about $\varepsilon$, 
$\delta$, and $\eta$. We intend to choose $\varepsilon_0$ and $K$ first, then $\varepsilon$,
then $\delta$, then $\eta$, and our estimates will typically hold as soon as $k$ is large enough,
depending on all these constants.

Set $K^{\varepsilon} = \big\{ x\in \R^n \, ; \, \dist(x,K) \leq \varepsilon \big\}$;
we restrict to $\varepsilon > 0$ so small that $K^{\varepsilon}$ is a compact
subset of $\O$, and we decide to take
\begin{equation} \label{e14.23}
\u_k^\ast = \u_k \text{ and } \W_k^\ast = \W_k \text{ on } \O \sm K^{\varepsilon}.
\end{equation}
In the intermediate region $K^\varepsilon \sm K$, we want to do two things. First,
we want to use a smooth cut-off function $\varphi_\varepsilon$ such that
\begin{equation} \label{e14.24}
\begin{array}{cll}
\varphi_\varepsilon(x) &= 1 &\text{ for } x\in K^{\varepsilon/2}
\\
\varphi_\varepsilon(x) &= 0  &\text{ for } x\in \R^n \sm K^{2\varepsilon/3}
\\
0 \leq \varphi_\varepsilon(x) &\leq 1 &\text{ for } x\in K^\varepsilon \sm K^{\varepsilon/2}
\\
|\nabla \varphi_\varepsilon(x)| &\leq  10\varepsilon^{-1} &\text{ everywhere}
\end{array}
\end{equation}
to interpolate between $\u_k$ and $\u$. But also, we want to replace $\u$
with a slightly smaller function with coordinates $v_i = h \circ u_i$, 
where $h$ is a smooth function such that
\begin{equation}\label{e14.25}
\begin{array}{clll}
h(t) &= &t &\text{ for } |t| \geq 2\delta
\\
h(t) &= &0 &\text{ for } |t| \leq \delta
\\
0 \,\leq \, h'(t) &\leq &3 &\text{ for } |t| \leq 2\delta.
\end{array}
\end{equation}
We should observe now that 
\begin{equation} \label{e14.26}
\lim_{k \to +\infty} ||\u-\u_k||_{L^\infty(K^\varepsilon)} = 0,
\end{equation}
i.e., the $\u_k$ converge to $\u$ uniformly on $K^\varepsilon$;
indeed, we can cover $K^\varepsilon$ by a finite number of compact balls $B \i \O$,
and use \eqref{e14.10} to get the uniform convergence in each $B$. Because of
\eqref{e14.26}, we can decide to restrict to integers $k$ such that
\begin{equation} \label{e14.27}
||\u-\u_k||_{L^\infty(K^\varepsilon)} < \delta.
\end{equation}
Then we set
\begin{equation} \label{e14.28}
u_{k,i}^\ast(x) = \varphi_\varepsilon(x) h(u_i(x)) + (1-\varphi_\varepsilon(x)) u_{k,i}(x)
\ \text{ for $x\in K^\varepsilon \sm K$ and $1 \leq i \leq N$.}
\end{equation}
Let us abuse notation slightly, and rewrite this as
\begin{equation} \label{e14.29}
\u_{k}^\ast = \varphi_\varepsilon \, h \circ \u + (1-\varphi_\varepsilon) \u_{k}
\ \text{ on $x\in K^\varepsilon \sm K$,}
\end{equation}
where we now write $h \circ \u$ for the function whose coordinates are the $h\circ u_i$.
Observe that if $x\in K^\varepsilon \sm K$ is such that $h(u_i(x)) > 0$
for some $i$, then $|u_i(x)| \geq \delta$, hence $u_{k,i} \neq 0$ (by \eqref{e14.26}) and,
almost surely, $x\in W_{k,i}$. So we may take
\begin{equation} \label{e14.30}
\W_k^\ast \cap (K^\varepsilon \sm K)= \W_k \cap (K^\varepsilon \sm K),
\end{equation}
and we still get that $\u_{k,i}^\ast(x) = 0$ almost everywhere on 
$(K^\varepsilon \sm K) \sm W_{k,i}^\ast$. Let us record the fact that
\begin{equation} \label{e14.31}
\u_{k}^\ast =  h \circ \u = h \circ \u^\ast
\ \text{ on } K^{\varepsilon/2} \sm K
\end{equation}
because $\varphi_\varepsilon = 1$ there and by definition of $K \supset K_0$.

Next we want to define $\u_k^\ast$ on $K$. We would have liked to take
$\u_k^\ast = h \circ\u^\ast$ (that is, $u_{k,i}^\ast = h \circ u_i^\ast$
for $1 \leq i \leq N$), but it could be that $h \circ u^\ast \neq 0$ somewhere on 
$\O \sm\Omega_k$, and this is not allowed. [Recall that we only get 
from the definitions that $\u^\ast = 0$ on $\O \sm \Omega$.]
So we will have to kill $h \circ u_i^\ast$ near $\O \sm \Omega_k$
with another cut-off function.
Let $\eta \in (0,\varepsilon/100)$ be our third small number, 
and let $\psi$ be a smooth function such that
\begin{equation} \label{e14.32}
\begin{array}{clll}
\psi(x) &= &0 &\text{ when } \dist(x,\O \sm \Omega) \leq \eta
\\
\psi(x) &= &1 &\text{ when } \dist(x,\O \sm \Omega) \geq 2\eta
\\
0 \,\leq \, \psi(x)  &\leq &1 &\text{ when } \eta \leq \dist(x,\O \sm \Omega) \leq 2\eta
\\
|\nabla\psi(x)| &\leq & 2\eta^{-1} &\text{ everywhere.}
\end{array}
\end{equation}

We intend to take
\begin{equation} \label{e14.33}
\u_{k}^\ast(x) =  \psi(x) \, h \circ \u^\ast(x)
\ \text{ for } x\in K^{\varepsilon/2},
\end{equation}
with an overlap of domains that will be useful to prove that $W^{1,2}_{loc}(\O)$,
but then we shall need to check that the two definitions coincide 
on $K^{\varepsilon/2} \sm K$.
We know that $\u^\ast = \u$ on that set, by definition of $K$, so just need
to check that $\psi(x) \, h \circ \u(x) = h \circ \u(x)$ on $K^{\varepsilon/2} \sm K$. 
When $\dist(x,\O \sm \Omega) \geq 2\eta$, this is clear because $\psi(x)=1$.
Set
\begin{equation} \label{e14.34}
H = \big\{ x\in K^{\varepsilon/2} \, ; \, 
\dist(x,\O \sm \Omega) \leq 2\eta \big\}
\end{equation}
If we show that
\begin{equation} \label{e14.35}
|\u(x)| \leq \delta \text{ and $h \circ \u(x) = 0$ for } x \in H,
\end{equation}
the second part will give the desired result. 

The second part follows from the first one, because 
$h(t) = 0$ when $|t| \leq \delta$. For the first part, 
let us back up a little and do a construction that depends 
only on $\varepsilon$ (we just want to avoid any confusion about 
what our constants depend on).
Let $Y$ be a maximal subset of $K^{3\varepsilon/4}\sm\Omega$ 
whose points lie at mutual distances larger than $\varepsilon/100$. 
For each $y\in Y$, the ball $D_y = \overline B(y,\varepsilon/10)$ 
is contained in $K^\varepsilon$, so \eqref{e14.10} gives a constant 
$C_y$ such that the $\u_k$, and then $\u$  too, are $C_y$-Lipschitz on $D_y$. 
Denote by $C_\varepsilon$ the largest of these numbers; the notation is fair 
because we can compute $C_\varepsilon$ as soon as $\varepsilon$ is chosen,
and it will not depend on $\delta$ and $\eta$. 

Now let $x\in H$ be given. First assume that $B(x,\eta) \i \O \sm \Omega$;
then $\u = 0$ almost everywhere on $B(x,\eta)$. 
Since $x\in K^{3\varepsilon/4}\sm\Omega$ we can find
$y\in Y$ such that $|y-x| \leq \varepsilon/100$, so $B(x,\eta) \i D_y$
(recall that $\eta \leq \varepsilon/100$), $\u$ is Lipschitz near $x$, and
$\u(x) = 0$. Now suppose that $B(x,\eta)$ meets $\Omega$;
since $\dist(x,\O \sm \Omega) \leq 2\eta$, we can find 
$z\in \d\Omega \cap B(x,3\eta)$. 
By the weak regularity condition \eqref{e14.7}, applied with a sufficiently
small radius $r$, we can find a set of positive measure $A$
such that $A \i B(z,\eta) \sm \Omega$, and 
hence $\u = 0$ almost everywhere on $A$.
Also, $z\in K^{3\varepsilon/4}\sm\Omega$ because $\eta \leq \varepsilon/100$,
so we can find $y\in Y$ such that $|y-z| \leq \varepsilon/100$, and then
$A$ and $x$ both lie in $D_y$ where $\u$ is $C_{\varepsilon}$-Lipschitz.
Pick $w\in A$, with $\u(w) = 0$; then $\u(x) \leq C_\varepsilon |x-w|
\leq 4C\varepsilon \eta < \delta$ if $\eta$ is small enough, depending
on $\delta$.

This proves the first part of \eqref{e14.35}; as we saw before, the second part and
then the fact that the two definitions \eqref{e14.33} and \eqref{e14.31} coincide
on $K^{\varepsilon}\sm K$ follow.

We now have a complete definition of $\u_k^\ast$ on $\O$, and let us check that 
\begin{equation} \label{e14.36}
\u^\ast_k \in W^{1,2}_{loc}(\O).
\end{equation}
By definitions, it is enough to show that for each $x\in \O$ there is a small ball 
$B_x$ centered at $x$ such that $\u^\ast_k \in W^{1,2}(B_x)$;
the verification would involve covering any open ball $B$ such that
$\overline B \i \O$ by a finite number of balls $B_x$, and then using a partition 
of the function $1$ to compute $\langle \nabla u^\ast_k, \varphi \rangle
= -\int u^\ast_k \nabla \varphi$ locally in the small balls.

When $x \in \O \sm K^{2\varepsilon /3}$, we 
choose $B_x \i \i \O$ so that $B_x \i \O \sm K^{2\varepsilon /3}$,
observe that  $\u^\ast_k = \u^\ast$ on $B_x$, either by \eqref{e14.23} or 
by \eqref{e14.29} and the fact that $\varphi_\varepsilon = 0$
on $B_x$, and just use our assumption that $(\u^\ast,\W^\ast) \in \F(\O,\Omega)$
to get that $\u^\ast_k \in W^{1,2}(B_x)$.
When $x\in K^{2\varepsilon /3} \sm K$, we 
choose $B_x \i K^{\varepsilon} \sm K$, notice that \eqref{e14.29}
is valid on $B_x$, observe that the first piece $\varphi_\varepsilon(\cdot) h \circ u(\cdot)$
is even Lipschitz on $B_x$ (because of \eqref{e14.10} and \eqref{e14.11}),
and get that $\u^\ast_k \in W^{1,2}(B_x)$ too. 
When $x\in K$, we use \eqref{e14.33}, observe that 
$\u^\ast \in W^{1,2}(B_x)$ by assumption, and then 
that the composition with the smooth function $h$  
with a bounded derivative and then the multiplication with $\psi$, preserve this
(see for instance \cite{Z}).
So \eqref{e14.36} holds.

Next we complete the definition of $\W_k^\ast$ by taking
\begin{equation} \label{e14.37}
W_{k,i}^\ast \cap K = \Omega_k \cap W_{i}^\ast \cap K 
\end{equation}
and we want to check that
\begin{equation} \label{e14.38}
(\u_k^\ast,\W_{k}^\ast) \text{ is a competitor for $(u_k,\W_k)$ in $\O$,
relative to $\Omega_k$.}
\end{equation}
We already know that $(\u_k^\ast,\W_{k}^\ast)$ coincides with
$(u_k,\W_k)$ on $\O \sm K^{\varepsilon}$, so we just need
to check that $(\u_k^\ast,\W_{k}^\ast) \in \F(\O,\Omega_k)$.
The $W_{k,i}^\ast$ are disjoint: on $\O \sm K$,
this is because the  $W_{k,i}$ are disjoint 
(see \eqref{e14.23} and \eqref{e14.30}), and
on $K$ we use the fact that the $W_{i}^\ast$ are disjoint. 
Also, $W_{k,i}^\ast \i \Omega_k$ (on $\O \sm K$, use the fact that this is 
true with the $W_{k,i}$, and on $K$ we forced it in \eqref{e14.37}). 
So $\W_k^\ast \in {\cal W}(\Omega_k)$.
Since we know that $\u^\ast_k \in W^{1,2}_{loc}(\O)$, we just need
to check that 
\begin{equation} \label{e14.39}
\u^\ast_{k,i}(x) = 0 \text{ almost everywhere on } \O \sm W^\ast_{i,k}.
\end{equation}
In $\O \sm K^\varepsilon$, this comes from \eqref{e14.23} 
(because $(\u_k,\W_k) \in \F(\O,\Omega_k)$).
We checked this on $K^\varepsilon \sm K$, just below \eqref{e14.30}; so we are 
left with $x\in K \sm W^\ast_{i,k}$. 

If $x\in K \sm W_{i}^\ast$, then almost surely $u_i^\ast(x) = 0$
(because $(\u^\ast,\W^\ast) \in \F(\O,\Omega)$), 
hence $u_{i,k}(x) = 0$ by \eqref{e14.33}.
Otherwise, \eqref{e14.37} says that $x\in K \cap W_{i}^\ast \sm W^\ast_{i,k}
= K \cap W_{i}^\ast \sm \Omega_k \i K \cap \Omega \sm \Omega_k$.
By \eqref{e14.8}, $\dist(x,\O \sm \Omega) \leq \delta(K,k)$, which is less
than $\eta$ if $k$ is large enough (depending on $K$ and $\eta$).
By \eqref{e14.32}, $\psi(x)=0$, and by \eqref{e14.33}
$\u^\ast_{k,i}(x) = 0$.
This proves \eqref{e14.39}, and \eqref{e14.38} follows.

We shall now start comparing our various function; we start 
with an estimate of $\u_k^\ast-\u^\ast$ on $K^\varepsilon$.

\begin{lem}\label{l14.2} 
We have that
\begin{equation} \label{e14.40}
 \int_{K^{\varepsilon}} |\nabla(\u_k^\ast-\u^\ast)|^2  = o(1),
\end{equation}
where by convention $o(1)$ is a number that can be made as small as we want
(once $\varepsilon_0$ and $K$ are chosen) by choosing $\varepsilon$, then $\delta$, 
then $\eta$ small enough, and taking $k$ large enough. 
\end{lem}

\begin{proof}
We write $\int_{K^{\varepsilon}} |\nabla(\u_k^\ast-\u^\ast)|^2 = A_1+A_2$,
with
\begin{equation} \label{e14.41}
A_1 = \int_{K^{\varepsilon} \sm K} |\nabla(\u_k^\ast-\u^\ast)|^2
\ \text{ and } \ 
A_2 = \int_{K} |\nabla(\u_k^\ast-\u^\ast)|^2.
\end{equation}
We start with $A_1$. Recall from \eqref{e14.29} that on $K^{\varepsilon} \sm K$,
\begin{equation} \label{e14.42}
\u_k^\ast-\u^\ast 
= \varphi_\varepsilon \, h \circ \u + (1-\varphi_\varepsilon) \u_{k} -\u^\ast
= \varphi_\varepsilon [h \circ \u - \u] + (1-\varphi_\varepsilon) [\u_{k} -\u]
\end{equation}
because $\u^\ast = \u$ on $\O\sm K$ by definition of $K$. 
This naturally gives a decomposition of $\nabla(\u_k^\ast-\u^\ast)$, and then 
$A_1$, into four pieces. We start with
\begin{equation} \label{e14.43}
A_{1,1} = \int_{K^{\varepsilon} \sm K} |\nabla \varphi_\varepsilon|^2 |h \circ \u - \u|^2
\leq \int_{K^{\varepsilon} \sm K} (100 \varepsilon^{-2})(9 \delta^2) = o(1)
\end{equation}
by \eqref{e14.24} and \eqref{e14.25}, and because we can choose $\delta$ small,
depending on $K$ and $\varepsilon$. For the next piece, 
cover $K$ by a finite number of open balls $B_i$ such that $\overline B_i \i \O$, 
and then apply \eqref{e14.10} and \eqref{e14.11}; we get that 
the $\u_k$, and then also $\u$, are Lipschitz on the $B_i$, with uniform bounds,
and so there is a constant $C_K$ such that
\begin{equation} \label{e14.44}
|\nabla \u_k|^2 + |\nabla \u|^2 \leq C_K 
\ \text{ on } K^\varepsilon.
\end{equation}
Our next term is 
\begin{equation} \label{e14.45}
A_{1,2} = \int_{K^{\varepsilon} \sm K} \varphi_\varepsilon^2 |\nabla [h \circ \u - \u]|^2
\leq 9 \int_{K^{\varepsilon} \sm K} |\nabla \u|^2 
\leq 9 C_K |K^\varepsilon \sm K| = o(1)
\end{equation}
by the chain rule and \eqref{e14.25}, and because the monotone intersection of the 
$K^{\varepsilon} \sm K$, when $\varepsilon$ tend to $0$, is empty. We continue with
\begin{equation} \label{e14.46}
A_{1,3} = \int_{K^{\varepsilon} \sm K} |\nabla \varphi_\varepsilon|^2 |\u_{k} -\u|^2
\leq 100 \varepsilon^{-2} \int_{K^{\varepsilon} \sm K} |\u_{k} -\u|^2
\leq 100 \varepsilon^{-2} \delta |K^\varepsilon \sm K|
= o(1)
\end{equation}
by \eqref{e14.24} and because $||\u-\u_k||_{L^\infty(K^\varepsilon)} < \delta$
by \eqref{e14.27}. Finally,
\begin{equation} \label{e14.47}
A_{1,4} = \int_{K^{\varepsilon} \sm K} (1-\varphi_\varepsilon)^2 |\nabla (\u_{k} -\u)|^2
\leq 2 \int_{K^{\varepsilon} \sm K} |\nabla \u_{k}|^2  + |\nabla\u|^2
\leq 2 C_K |K^{\varepsilon} \sm K| = o(1)
\end{equation}
by \eqref{e14.44} and as in \eqref{e14.45}.

We are thus left with $A_2$. On the set $K$, \eqref{e14.33} says that
\begin{equation} \label{e14.48}
\u_k^\ast-\u^\ast = \psi \, h\circ \u^\ast - \u^\ast
= (\psi - 1) h\circ \u^\ast + [h\circ \u^\ast - \u^\ast],
\end{equation}
which gives a natural decomposition of $A_2$ into three terms. The most
interesting one is 
\begin{equation} \label{e14.49}
A_{2,1}= \int_{K} |\nabla \psi|^2   |h\circ \u^\ast|^2.
\end{equation}
Notice that $\u^\ast$ and $h\circ \u^\ast$ vanish almost everywhere 
on $\O \sm \Omega$, because $(\u^\ast,\W^\ast) \in \F(\O,\Omega)$,
so we only need to integrate on $K \cap \Omega$. In addition, 
\eqref{e14.32} says that $\nabla \psi = 0$ unless
$\eta \leq \dist(x,\O \sm \Omega) \leq 2\eta$, so we just need to integrate on 
\begin{equation} \label{e14.50}
H_\eta = \big\{ x\in K \cap \Omega \, ; \,  
\dist(x,\O \sm \Omega) \leq 2\eta \big\}.
\end{equation}

Cover $H_\eta$ by balls $B_j$ of radius $4\eta$
centered on $H_\eta$ and such the $B_j$ have bounded overlap.
Obviously
\begin{equation} \label{e14.51}
\begin{aligned}
A_{2,1} &\leq \int_{H_\eta} |\nabla \psi|^2   |h\circ \u^\ast|^2
\leq 4 \eta^{-2} \int_{H_\eta}  |h\circ \u^\ast|^2
\\
&\leq 4 \eta^{-2} \sum_i \int_{B_j} |h\circ \u^\ast|^2
\leq 36 \eta^{-2} \sum_i \int_{B_j} |\u^\ast|^2
\end{aligned}
\end{equation}
by \eqref{e14.32} and \eqref{e14.25}, and now we shall use the weak regularity
assumption \eqref{e14.7} to estimate $|\u^\ast|^2$.
For each $j$, $B_j = B(x_j,4\eta)$ for some $x_j \in H_\eta$.
By definition of $H_\eta$, $x\in \Omega$ but
$\dist(x_j,\O \sm \Omega) \leq 2\eta$,
so we can find $y_j \in \d \Omega$ such that $|y_j-x_j| \leq 2\eta$.
Notice that $y_j \in K^\varepsilon \i \i \Omega$ because $x\in K$ and 
$\eta \leq \varepsilon/100$. If $\eta$ is small enough (depending on $K^\varepsilon$)
we can apply \eqref{e14.7} to $B(y_j,\eta)$ and get that 
$|B(y_j,\eta) \sm \Omega| \geq c_\varepsilon \eta^n$,
where $c_\varepsilon$ depends on $\varepsilon$ through $K^\varepsilon$.

Notice that this is the only place where we seriously use \eqref{e14.7}
(the previous time, we just needed to say that $\u(x) = 0$ on $\d \Omega$,
in a place where $\u$ was in fact Lipschitz).

Anyway, return to $A_{2,1}$ and the analogue of \eqref{e4.6} to $\u^\ast$
on  $B_j$, with $E = B_j \cap \Omega$. We get that
\begin{equation} \label{e14.52}
\int_{B_j} |\u^\ast|^2 = \int_{E} |\u^\ast|^2
\leq C (4\eta)^2 {|B_j| \over |B_j \cap E|} \, \int_{B_j} |\nabla \u^\ast|^2
\leq C \eta^2 c_\varepsilon^{-1} \int_{B_j} |\nabla \u^\ast|^2.
\end{equation}
We now return to \eqref{e14.51}, use the fact that the $B_j$ have bounded overlap,
and get that
\begin{equation} \label{e14.53}
A_{2,1} \leq C \eta^{-2} \sum_i \int_{B_j} |\u^\ast|^2
\leq C c_\varepsilon^{-1} \sum_i \int_{B_j} |\nabla\u^\ast|^2
\leq C c_\varepsilon^{-1} \int_{\cup_j B_j} |\nabla\u^\ast|^2.
\end{equation}
Recall that each $B_j$ contains a point $y_j \in \d\Omega$, and 
is centered on $K$; thus $\cup_j B_j \i Z(\eta)$, with
\begin{equation} \label{e14.54}
Z(\eta) = \big\{ x\in K^{\varepsilon/2} \, ; \,  \dist(x,\d\Omega) \leq 4\eta \big\}
\end{equation}
In addition, we claim that $\nabla \u^\ast = 0$
almost-everywhere on $K^{\varepsilon/2} \cap \d\Omega$.
Indeed, the Rademacher-Calder\'on theorem says that for almost every
$x\in K^{\varepsilon/2} \cap\d\Omega$, $\u^\ast$ is differentiable at $x$, 
with a differential that coincides with the distributional derivative. 
It is easy to compute that $\nabla \u^\ast(x) = 0$ for such an $x$, because by \eqref{e14.7}
every small ball centered at $x$ contains a fixed proportion of points of $\O\sm \Omega$,
where $\u^\ast = 0$ almost everywhere because $(\u^\ast,\W^\ast) \in \F(\O,\Omega)$.
Because of this,
\begin{equation} \label{e14.55}
A_{2,1} \leq C c_\varepsilon^{-1} \int_{\cup_j B_j} |\nabla\u^\ast|^2
\leq C c_\varepsilon^{-1} \int_{Z(\eta) \sm \d \Omega} |\nabla\u^\ast|^2 = o(1)
\end{equation}
because $|\nabla\u^\ast|^2$ is integrable near $K^{\varepsilon/2}$
and when $\varepsilon$ is fixed and $\eta$ decreases to $0$,
the set $Z(\eta)\sm \d\Omega$ decreases to the empty set.

The next term is
\begin{equation} \label{e14.56}
A_{2,2}= \int_{K} (1-\psi)^2   |\nabla (h\circ \u^\ast)|^2
\leq 9 \int_{K} (1-\psi)^2   |\nabla \u^\ast|^2
\end{equation}
by the chain rule and \eqref{e14.25}.
By \eqref{e14.32}, we integrate on the set 
$\big\{ x\in K \, ; \,  \dist(x,\O \sm\Omega) \leq 2\eta \big\}$,
but $\nabla \u^\ast$ almost everywhere on $\O \sm\Omega$, because $\u^\ast=0$
there, and also almost everywhere on $\O \cap \d\Omega$, by the same argument as for
$A_{2,1}$. Thus we may integrate on the smaller set 
$\big\{ x\in K \, ; \,  0 < \dist(x,\d\Omega) \leq 2\eta \big\}$,
and $A_{2,2}=o(1)$ by the proof of \eqref{e14.55}. We are left with 
\begin{equation} \label{e14.57}
A_{2,3}= \int_{K} |\nabla [h\circ \u^\ast - \u^\ast]|^2.
\end{equation}
By \eqref{e14.25}, we only integrate on the set 
$\big\{ x\in K \, ; \,  |\u^\ast(x)| \leq 2N\delta\big\}$
(we added $N$ to account for the different coordinates of $\u^\ast$),
and we can even remove the set where $\u^\ast(x) = 0$, because 
$\nabla [h\circ \u^\ast - \u^\ast] = 0$ almost everywhere on that set.
We are left with 
\begin{equation} \label{e14.58}
A_{2,3} = \int_{\big\{ x\in K \, ; \,  0 < |\u^\ast(x)| \leq N\delta\big\}} 
|\nabla [h\circ \u^\ast - \u^\ast]|^2 
\leq 9 \int_{\big\{ x\in K \, ; \,  0 < |\u^\ast(x)| \leq N\delta\big\}} 
|\nabla \u^\ast |^2 = o(1)
\end{equation}
by \eqref{e14.25} and because the intersection, when $\delta$ tends to $0$,
of the domains of integration is empty. This completes our proof of Lemma \ref{l14.2}.
\qed
\end{proof}

\ms
Our next estimate is on the difference $\u_k^\ast-\u^\ast$ itself. 
We claim that
\begin{equation} \label{e14.59}
\int_{K^\varepsilon} |\u_k^\ast-\u^\ast|^2 = o(1).
\end{equation}

We want the be a little less brutal that we have been so far, with respect
to estimates on $K^\varepsilon$.
Choose a new constant $\varepsilon_1$, that depends on $K$
but not on $\varepsilon$, $\delta$, or $\eta$, so small that $K^{\varepsilon_1} \i \O$,
and observe that $\u$ is bounded on $K^{\varepsilon_1}$
(cover $K^{\varepsilon_1}$ by a finite collection of balls $B$
such that $2B \i \O$, and apply \eqref{e14.10} and \eqref{e14.11}.
Then $\u$ is also bounded on $K^\varepsilon$ for $\varepsilon \leq \varepsilon_1$,
with a bound that does not depend on $\varepsilon$, and by \eqref{e14.26}
we also get that $|\u_k| \leq C$ for $k$ large, with a bound that does not
depend on $\varepsilon$, $\delta$, and $\eta$. The fact that how large $k$  
may depend on these constants does not matter.

We start our proof of \eqref{e14.59} with the contribution of $K^\varepsilon \sm K$.
On this set, $\u^\ast = \u$ by definition of $K$, and $\u_k^\ast$ is given by \eqref{e14.29},
so it is bounded as well. Thus 
\begin{equation} \label{e14.60}
\int_{K^\varepsilon \sm K} |\u_k^\ast-\u^\ast|^2 \leq C |K^\varepsilon \sm K| = o(1)
\end{equation}
and we are left with 
\begin{equation} \label{e14.61}
\int_{K} |\u_k^\ast-\u^\ast|^2 
\leq 2 \int_{K} \big| (\psi - 1) h\circ \u^\ast \big|^2 
+ 2 \int_{K} \big| h\circ \u^\ast - \u^\ast \big|^2 
\end{equation}
by \eqref{e14.33}. The second integral is easily estimated, since 
$|h\circ \u^\ast - \u^\ast| \leq C\delta$ everywhere, $|K| < +\infty$,
and we can take $\delta$ small.
For the first one, we observe that it is enough to integrate on $\Omega$
(because $h\circ \u^\ast = \u^\ast = 0$ on $\O\sm \Omega$), and even on 
$H_\eta = \big\{ x\in K \cap \Omega \, ; \,  \dist(x,\O\sm \Omega)\leq 2\eta \big\}$,
because otherwise $\psi(x) = 1$ by \eqref{e14.32}. This is the same set as in
\eqref{e14.50}, so the second parts of \eqref{e14.51} and of \eqref{e14.53},
and then \eqref{e14.55}, show that
\begin{eqnarray} \label{e14.62}
 \int_{K} \big| (\psi - 1) h\circ \u^\ast \big|^2
 &\leq& \int_{H_\eta} \big| h\circ \u^\ast \big|^2
 \leq C \sum_i \int_{B_j} |\u^\ast|^2
 \leq Cc_\varepsilon^{-1} \eta^2 \int_{\cup_j B_j} |\nabla \u^\ast|^2
 \nonumber\\
 &\leq& C c_\varepsilon^{-1} \eta^2 \int_{Z(\eta) \sm \d \Omega} |\nabla\u^\ast|^2 
 = o(1);
\end{eqnarray}
either for the same reason as in \eqref{e14.55}, or because 
$Z(\eta) \i K^{\varepsilon/2}$ and we can use the extra $\eta^2$ to make 
the right-hand side small. This completes our proof of the claim \eqref{e14.59}.

\ms
We are now ready to take limits. Let us assume that $\varepsilon$ is chosen
so that the boundary of $K^\varepsilon$ has vanishing Lebesgue measure.
Since the boundaries $\d K^\varepsilon = \big\{ x \, ; \, \dist(x,K) = \varepsilon \big\}$,
$\varepsilon > 0$, are all disjoint, this is the case for almost every $\varepsilon$.
Recall that by \eqref{e14.10} and \eqref{e14.11}, the $\u_k$ converge to $\u$
uniformly on $K^\varepsilon$. By the lowersemicontinuity of the homogeneous $W^{1,2}$ 
norm on the interior of $K^\varepsilon$,
\begin{equation} \label{e14.63}
\int_{K^\varepsilon} |\nabla u_i|^2 
= \int_{K^\varepsilon \sm \d K^{\varepsilon}} |\nabla u_i|^2
\leq \liminf_{k \to +\infty} \int_{K^\varepsilon} |\nabla u_{i,k}|^2
\end{equation}
for each $i$ (recall that we can evaluate these norm by duality with compactly
supported smooth functions, and then use the uniform convergence
to control integrals like $\int_K u_{i,k} \d_j\varphi$).

Let us also look at the convergence of the $M$-terms. Write
\begin{equation} \label{e14.64}
\Big|\int_{K^\varepsilon} u_i g_i - u_{i,k} g_{i,k} \Big|
\leq \Big|\int_{K^\varepsilon} u_i (g_{i}-g_{i,k}) \Big| 
+ \Big|\int_{K^\varepsilon} (u_i - u_{i,k}) g_{i,k} \Big|,
\end{equation}
observe the second term tends to $0$ because $u_{i,k}$ tends to $u_i$ uniformly
on $K$ and $||g_{i,k}||_\infty \leq C_0$, and the first term tends to $0$ by
\eqref{e14.3} and because $||u_i - u_{i,k}||_\infty \leq 2C_0$ by \eqref{e14.2}
and $u_i \1_{K^\varepsilon}$ can be approximated in $L^1(\O)$ by continuous compactly 
supported function $\varphi$. Thus
\begin{equation} \label{e14.65}
\lim_{k \to +\infty}\Big|\int_{K^\varepsilon} u_i g_i - u_{i,k} g_{i,k} \Big| = 0.
\end{equation}
The same argument works for the $\int u_i^2 f_i$ and yields
\begin{equation} \label{e14.66}
\lim_{k \to +\infty}\Big|\int_{K^\varepsilon} u_i^2 f_i - u_{i,k}^2 f_{i,k} \Big| = 0.
\end{equation}
Set 
\begin{equation} \label{e14.67}
J_-(\u) = \int_{K^\varepsilon} \Big[|\nabla \u|^2 + \sum_i u_i^2 f_i - \sum_i u_i g_i  \Big]
\end{equation}
(we leave the volume terms for later); we just checked that
\begin{equation} \label{e14.68}
J_-(\u) \leq \liminf_{k \to +\infty} J_{k,-}(\u_k),
\end{equation}
where
\begin{equation} \label{e14.69}
J_{k,-}(\u_k) = \int_{K^\varepsilon} 
\Big[|\nabla \u_k|^2 + \sum_i u_{i,k}^2 f_{i,k} - \sum_i u_{i,k} g_{i,k}  \Big].
\end{equation}
Next, by the minimizing property of $(\u_k,\W_k)$, \eqref{e14.38}, 
and \eqref{e14.23}
\begin{equation} \label{e14.70}
J_{k,-}(\u_k) \leq J_{k,-}(\u_k^\ast) - F_k(\W_k) + F_k(\W_k^\ast).
\end{equation}
We may soon replace $\u_k^\ast$ with $\u^\ast$, because
\begin{eqnarray} \label{e14.71}
|J_{k,-}(\u_k^\ast) - J_{k,-}(\u^\ast)| &\leq& \int_{K^\varepsilon} |\nabla (\u_k^\ast-\u^\ast)|^2
+ C \int_{K^\varepsilon} \sum_i \big[ |u_{k,i}^\ast-u_i^\ast| 
+ |(u_{k,i}^\ast)^2-(u_i^\ast)^2|\big] 
\nonumber
\\
&\leq& o(1) + C  \int_{K^\varepsilon} |\u_k^\ast-\u^\ast| (1 + |\u^\ast|+|\u_k^\ast-\u^\ast|)
= o(1)
\end{eqnarray}
by Lemma \ref{l14.2}, because $|(u_{k,i}^\ast)^2-(u_i^\ast)^2|
= |u_{k,i}^\ast-u_i^\ast| \, |u_{k,i}^\ast+u_i^\ast|$, and by \eqref{e14.59} and 
Cauchy-Schwarz. In turn,
\begin{equation} \label{e14.72}
J_{k,-}(\u^\ast) - J_-(\u^\ast)  = \int_{K^\varepsilon} \sum_i
\Big[ (u_i^\ast)^2 [f_i - f_{i,k}] -   u_i^\ast [g_i-g_{i,k}]  \Big]
= o(1)
\end{equation}
because only the $M$-term changes, by the weak limit assumption \eqref{e14.3},
and because $\1_{K^\varepsilon} \u^\ast$ and $\1_{K^\varepsilon} (\u^\ast)^2$
can be approximated in $L^1(\O)$ by continuous functions with compact support in $\O$
(see the proof of \eqref{e14.65}).

By \eqref{e14.68}, \eqref{e14.70}, \eqref{e14.71}, and \eqref{e14.72},
\begin{equation} \label{e14.73}
J_-(\u) \leq J_-(\u^\ast) - F_k(\W_k) + F_k(\W_k^\ast)  + o(1),
\end{equation}
and we now need to worry about the $F$-terms. 
We want to use the continuity of $F$, so let us estimate some symmetric
differences between sets; for $1 \leq i \leq N$,
\begin{eqnarray} \label{e14.74}
|W_{k,i}^\ast \Delta W_i^\ast| 
&\leq& |(\W_{k,i}^\ast \cap K) \Delta (W_i^\ast \cap K)| + |\O \sm K|
\nonumber
\\
&=& |(\Omega_k \cap W_{i}^\ast \cap K) \Delta (W_i^\ast \cap K)| + |O \sm K|
\nonumber
\\
&\leq& |\Omega \sm \Omega_k| + \varepsilon_0 = o(1) + \varepsilon_0 = o'(1)
\end{eqnarray}
by \eqref{e14.37}, because $W_i^\ast \i \Omega$,  by \eqref{e14.22}
and \eqref{e14.6}, and with the convention that $o'(1)$ is a number that can
be made as small as we want, by choosing $\varepsilon_0$ small, then
choosing $K$ and the other constants, and finally restricting to $k$ large enough.

We are ready to use the special form of $F_k$  in \eqref{e14.4}
for the first time. Recall also that by
\eqref{e14.5}, the $\wt F_k$ converge uniformly to the continuous
function $\wt F$, so that
\begin{eqnarray} \label{e14.75}
F_k(\W_k^\ast) &=& \wt F_k(|W_{1,k,}^\ast|, \ldots, |W_{N,k}^\ast|)
= \wt F(|W_{1,k}^\ast|, \ldots, |W_{N,k}^\ast|) + o(1)
\nonumber
\\
&=& \wt F(|W_{1}^\ast|, \ldots, |W_{N}^\ast|) + o'(1)
= F(\W^\ast) + o'(1)
\end{eqnarray}
by \eqref{e14.74}.
We shall restrict our attention to $k = k_j$, where $\{ k_j \}$ is 
the subsequence that we chose to get \eqref{e14.18} and to define $\W$.
We get that
\begin{eqnarray} \label{e14.76}
F_{k_j}(\W_{k_j}) &=& \wt F_{k_j}(|W_{1,{k_j}}|, \ldots, |W_{N,{k_j}}|)
= \wt F(|W_{1,{k_j}}|, \ldots, |W_{N,{k_j}}|) + o(1)
\nonumber
\\
&=&  \wt F(l_1, \ldots, l_N) + o(1) = \wt F(|W_{1}|, \ldots, |W_{N}|) + o(1)
= F(\W) + o(1)
\end{eqnarray}
by \eqref{e14.18} and \eqref{e14.20}. Thus \eqref{e14.73}, restricted to the subsequence 
$\{ k_j \}$, yields
\begin{equation} \label{e14.77}
J_-(\u) \leq J_-(\u^\ast) - F(\W) + F(\W^\ast) + o'(1).
\end{equation}
We may now let $\varepsilon_0$, $\varepsilon$, $\delta$, $\eta$
tend to $0$ in the prescribed order, let $k_j$ tend to $+\infty$,
and get that $J_-(\u) \leq J_-(\u^\ast) - F(\W) + F(\W^\ast)$.
This is the same thing as \eqref{e14.1}, with the compact set $K^\varepsilon$
(recall the definition \eqref{e14.67} of $J_-$, and that
the pairs $(\u,\W)$ and $(\u^\ast,\W^\ast)$ coincide outside of $K$
by definition of $K$).

\ms
We just completed the proof of minimality for our pair $(\u,\W)$, but we still
need to check the strong limit property in \eqref{e14.12}.
Let $B$ be a ball in $\O$, with $\overline B \i \O$, and observe that
for each $i$, 
\begin{equation} \label{e14.78}
\int_{\overline B} |\nabla u_i|^2 = \int_{B} |\nabla u_i|^2
\leq \liminf_{k \to +\infty} \int_{B} |\nabla u_{i,k}|^2
\end{equation}
because $\u$ is Lipschitz near $\overline B$ and by the lower semicontinuity
of $\int_{B} |\nabla u_i|^2$ (see the proof of \eqref{e14.63}). 
Suppose that for some $i$, 
\begin{equation} \label{e14.79}
\int_{\overline B} |\nabla u_i|^2 
- \liminf_{k \to +\infty} \int_{B} |\nabla u_{i,k}|^2 = \alpha > 0,
\end{equation}
and run all the proof above with $(\u^\ast,\W^\ast) = (\u,\W)$,
and $K \supset \overline B$. We can improve \eqref{e14.63}
and write instead
\begin{eqnarray} \label{e14.80}
\int_{K^\varepsilon} |\nabla u_i|^2 
&=& \int_{K^\varepsilon \sm \d K^{\varepsilon}} |\nabla u_i|^2
= \int_{K^\varepsilon \sm (\d K^{\varepsilon} \cup \overline B)} |\nabla u_i|^2
+ \int_{\overline B} |\nabla u_i|^2 
\nonumber
\\
&\leq& \liminf_{k \to +\infty} 
\int_{K^\varepsilon \sm (\d K^{\varepsilon} \cup \overline B)} |\nabla u_{i,k}|^2
+ \liminf_{k \to +\infty} \int_{B} |\nabla u_{i,k}|^2 -\alpha
\nonumber
\\
&\leq& \liminf_{k \to +\infty} \int_{K^\varepsilon} |\nabla u_{i,k}|^2 -\alpha,
\end{eqnarray}
where we have applied our lower semicontinuity estimate on the open set
$K^\varepsilon \sm (\d K^{\varepsilon} \cup \overline B)$.

We follow the rest of the proof, and eventually get an improved version of 
\eqref{e14.77}, that says that
\begin{equation} \label{e14.81}
J_-(\u) \leq J_-(\u^\ast) - \alpha - F(\W) + F(\W^\ast) + o'(1)
= J_-(\u) - \alpha + o'(1)
\end{equation}
(because $(\u^\ast,\W^\ast) = (\u,\W)$). This is impossible, so
$\int_{B} |\nabla u_i|^2 = \liminf_{k \to +\infty} \int_{B} |\nabla u_{i,k}|^2$.
We already know that $\nabla u_{i,k}$ converges weakly to $\nabla u_i$ in $L^2(B)$,
and the fact that the norms converge now imply that the convergence is  strong.
This proves \eqref{e14.12}, and completes our proof of Theorem~\ref{t14.1}.
\qed
\end{proof}

\ms
The following remark is probably more amusing than useful.

\begin{rem} \label{t14.3} 
In Theorem~\ref{t14.1}, if the volumes $|W_{k,i}|$ do not have a limit
for each $i\in [1,N]$, then there are more than one minimizers, associated 
to the same function $\u$, but with different $N$-uples $\W$. 
The values of $F$ for all these $\W$ is the same.
\end{rem}

Let us be more specific. To each $k$, we associate the
$N$-uple ${\bf V}_k = (|W_{1,k}|, \ldots , |W_{N,k}|)$
of volumes. If the sequence $\{ {\bf V}_k \}$ does not have a 
limit in $[0,|\O|]$, then for each point of accumulation 
${\bf V}_\infty$ of this sequence, we constructed a minimizer
$(\u,\W)$ such that $|W_i| = V_{i,\infty}$ for $1 \leq i \leq N$.
If $\W^\sharp$ is another $N$-uple of ${\cal W}(\Omega)$ for which 
$(\u,\W^\sharp)$ is a local minimizer, we can use $W$ to construct
competitors for $(\u,\W^\sharp)$ and show that $F(\W) \leq F(\W^\sharp)$:
we just keep $\W$ in some very large compact set $K$, replace
$\W$ by $\W^\sharp$ on $\O \sm K$, keep the same $\u$, notice that this still
gives a pair in $\F(\O,\Omega)$, and then let $K$ tend to $\O$.
The same argument shows that $F(\W^\sharp) \leq F(\W)$ too.

Also recall that we defined sets $W'_i$, $1 \leq i \leq N$. 
We first chose a sequence $\{ k_j \}$ so that the  $|W_{i,k_j}|$ converge
to a limit $l_i$, and given a point of accumulation ${\bf V}_\infty$ as above,
we could easily choose the sequence so that $l_i = V_{i,\infty}$. 
Then we constructed a pair $(\u,\W') \in \F(\O,\Omega)$, 
such that $|W'_i| \leq l_i$. Finally we chose $\W$ above so that $W_i \supset W'_i$
and $|W_i| = l_i$, and this was useful to prove that $(\u,\W)$ is a local minimizer.
but we could have tried other choices, with different measures.

Let $\cal V$ denote the set of $N$-uples ${\bf V} =(v_1, \ldots, v_N)$
such that $|W'_i| \leq v_i$ for $1 \leq i \leq N$ and $\sum_i v_i \leq |\Omega|$.
As was explained below \eqref{e14.19}, for each ${\bf V} \in \cal V$
we can find $\W^\sharp \in {\cal W}(\Omega)$ such that 
$W^\sharp_i = v_i$ and $W'_i \i W^\sharp_i$ for $1 \leq i \leq N$;
then $(\u,\W^\sharp) \in \F(\O,\Omega)$, and the local minimality
of $(\u,\W)$ shows that $F(\W) \leq F(\W^\sharp)$ (just use $\W$
to produce local competitors for $(\u,\W^\sharp)$, as above. In other
words, we get that
\begin{equation} \label{e14.82}
F(\W) = \wt F({\bf V}_\infty) \leq \wt F({\bf V})
\ \text{ for } {\bf V} \in \cal V.
\end{equation}
But recall that the class $\cal V$ depends on the point of accumulation
${\bf V}_\infty$ through the the $W'_i$ and the subsequence $\{ k_j \}$.

\begin{rem} \label{t14.4} 
The conclusion of Theorem~\ref{t14.1}
still holds, with the same proof, if instead of assuming that
$(\u_k,\W_k)$ is a local minimizer for $J_k$ in $\O$, as in \eqref{e14.9},
we only assume that this is asymptotically true on compact sets, i.e., 
that for each compact set $K \i \O$, there is a sequence 
$\{ \alpha_k \}$ that tends to $0$, such that
\begin{equation} \label{e14.83}
J_k(\u_k,\W_k) \leq J_k(\u_k^\ast,\W_k^\ast) + \alpha_k
\end{equation}
for every competitor $(\u_k^\ast,\W_k^\ast)$ for $(\u_k,\W_k)$
in $\O$, relative to $\Omega_k$, which coincides with $(\u_k,\W_k)$ on $\O \sm K$.
\end{rem}

\begin{proof}
Indeed nothing changes until \eqref{e14.70}, when we apply the minimality
of $(\u_k,\W_k)$. What we get instead is
\begin{equation} \label{e14.84}
J_{k,-}(\u_k) \leq J_{k,-}(\u_k^\ast) - F_k(\W_k) + F_k(\W_k^\ast)+ \alpha_k,
\end{equation}
where the sequence $\{ \alpha_k \}$ is associated to the compact set $K^\varepsilon$
(see \eqref{e14.23}). Then we continue the argument as above, with an extra
term $\alpha_k$ that does not disturb because it tends to $0$.
\qed
\end{proof}

\ms
We promised to say a few words about the analogue of 
Theorem~\ref{t14.1} in the context of Alt, Caffarelli, and Friedman.
Suppose, instead of \eqref{e14.4}, that the $F_k$ are given by 
\begin{equation} \label{e14.85}
F_k(W_{1}, \ldots, W_{N}) = \sum_{i=1}^N \int_{W_{i}} q_{i,k}(x) dx,
\end{equation}
where the $q_{i,k}$ are nonnegative measurable functions. 
We shall assume that the $q_{i,k}$ are locally bounded, with uniform estimates. 
That is, for each compact set $K \i \O$, we assume that there is a constant $C(K)$ 
such that
\begin{equation} \label{e14.86}
0 \leq q_{i,k}(x) \leq C(K) \ \text{ for $k \geq 0$, $1 \leq i \leq N$, and $x\in K$.}
\end{equation}
We also suppose that the $q_{i,k}$ converge weakly to locally bounded functions $q_i$, 
by which we mean that for each compact set $K \i \O$ and each measurable set $E \i K$,
\begin{equation} \label{e14.87}
\int_{E} q_{i}(x) dx = \lim_{k \to +\infty} \int_{E} q_{i,k}(x) dx.
\end{equation}
We could have chosen compact sets $E$, or integrated the $q_{i,k}$ and $q_i$ against
continuous function with compact support, without changing the notion (because of our
$L^\infty$ bounds). We just picked the definition that we shall use.

Of course in many cases, $|\O|$ will be finite, the functions $q_{i,k}$ will be uniformly
bounded (regardless of compact support), and the functional $F_k$ defined above
will take finite values. 
In the other cases, \eqref{e14.85} and its analogue for the $q_i$ may be infinite, 
but we can define local minimizers as we did before,
except that we replace \eqref{e14.1} by the simpler condition
\begin{equation} \label{e14.88}
\sum_i \int_{KÊ\cap W_i} |\nabla u_i|^2 + u_i^2 f_i - u_i g_i + q_i 
\leq \sum_i \int_{KÊ\cap W_i^\ast} |\nabla u_i^\ast|^2 + (u_i^\ast)^2 f_i - u_i^\ast g_i + q_i .
\end{equation}
That is, we say that $(\u,\W)$ is a local minimizer for $J$ in $\O$
if $(\u,\W) \in \F(\O,\Omega)$ and \eqref{e14.88} holds for every compact set
$K \i \O$ and every pair $(\u^\ast,\W^\ast) \in \F(\O,\Omega)$ that coincides 
with $(\u,\W)$ in $\O \sm K$. We define local minimizers for the $J_k$
similarly, with $\Omega_k$, the $f_{i,k}$, $g_{i,k}$, and $q_{i,k}$.

\begin{cor}\label{t14.5} 
Assume that the hypotheses of Theorem~\ref{t14.1} are satisfied,
except that we replace \eqref{e14.4} and \eqref{e14.5}
with \eqref{e14.85}, \eqref{e14.86}, and \eqref{e14.87}, and we modify
the definition of local minimizers as above.
Then we can find $\W$ such that $(\u,\W) \in \F(\O,\Omega)$ and $(\u,\W)$ 
is a local minimizer for $J$ in $\O$.
In addition, \eqref{e14.12} holds.
\end{cor}  

\ms
Our proof will really use the fact that the $q_{i,k}$ are nonnegative
(otherwise, we get the problem that we may not extract a sequence for which
the $\1_{W_{i,k}}$ converge).

We shall no longer need the assumption that $|\O| < +\infty$. This is an advantage
of choosing a local functional $F$, and the improvement is not a real one, because
we can easily reduce to the case when $\O$ is bounded by observing that 
$(\u,\W)$ is a local minimizer for $J$ in $\O$ if and only if its restriction
to $\O \cap B(0,R)$ is a local minimizer for $J$ in $\O \cap B(0,R)$ for every 
$R > 0$. We could not do this trick with our non local functional $F$
(but we could try to use Remark \ref{t14.4} in some cases).
 
\begin{proof}
The proof is even simpler. We are happy to take $\W = (W'_1, \ldots, W'_N)$,
where the $W'_i$ are defined near \eqref{e14.15}, because taking larger sets
may only make $F(\W)$ larger anyway.

Then we repeat the same proof as before, except that we need to 
rewrite \eqref{e14.70} and modify the limiting argument \eqref{e14.75}-\eqref{e14.76}. 
Instead of \eqref{e14.70}, the minimality of $(\u_k,\W_k)$ is now expressed by
\begin{equation} \label{e14.89}
J_{k,-}(\u_k) \leq J_{k,-}(\u_k^\ast) - \sum_i\int_{K^\varepsilon \cap W_{i,k}} q_{i,k}(x) dx
+ \sum_i\int_{K^\varepsilon \cap W_{i,k}^\ast} q_{i,k}(x) dx
\end{equation}
(compare with \eqref{e14.88}). That is, we replace 
$- F_k(\W_k) + F_k(\W_k^\ast)$ with the expression 
\begin{equation} \label{e14.90}
 - \sum_i\int_{K^\varepsilon \cap W_{i,k}} q_{i,k}(x) dx
+ \sum_i\int_{K^\varepsilon \cap W_{i,k}^\ast} q_{i,k}(x) dx
\end{equation}
in \eqref{e14.70}, and then in \eqref{e14.73}.

Instead of \eqref{e14.75}, we restrict to $K^\varepsilon$ and say that
\begin{eqnarray} \label{e14.91}
\sum_{i=1}^N  \int_{K^\varepsilon \cap W_{i,k}^\ast} q_{i,k}(x) dx
&=& \sum_{i=1}^N  \int_{K^\varepsilon \cap W_{i}^\ast} q_{i,k}(x) dx + o'(1)
\nonumber
\\
&=&  \sum_{i=1}^N  \int_{K^\varepsilon \cap W_{i}^\ast} q_{i}(x) dx + o'(1)
\end{eqnarray}
by \eqref{e14.74} and \eqref{e14.86}, and then \eqref{e14.87}.
For the analogue of \eqref{e14.76},
we use the subsequence $\{ k_j \}$ that defines $W_i = W'_i$ and first say that
\begin{equation} \label{e14.92}
\sum_{i=1}^N  \int_{K^\varepsilon \cap W_{i,k_j}} q_{i,k}(x) dx
\geq  \sum_{i=1}^N  \int_{K^\varepsilon \cap W_{i,k_j} \cap W_{i}} q_{i,k}(x) dx
\end{equation}
because $q_{i,k}(x) \geq 0$. Next observe that
\begin{equation} \label{e14.93}
\1_{W_{i}} = \1_{W'_{i}} \leq \liminf_{j \to +\infty} \1_{W'_{i,k_j}}
\leq \liminf_{j \to +\infty} \1_{W_{i,k_j}}
\end{equation}
by the remark above \eqref{e14.17} and because $W'_{i,k_j} \i W_{i,k_j}$
(see below \eqref{e14.16}). So 
$\1_{W_{i}} = \lim_{j \to +\infty} \1_{W_{i,k_j} \cap W_{i}}$ and
the right-hand side of \eqref{e14.92} is
\begin{eqnarray} \label{e14.94}
\sum_{i=1}^N  \int_{K^\varepsilon \cap W_{i,k_j} \cap W_{i}} q_{i,k}(x) dx
&=& \sum_{i=1}^N  \int_{K^\varepsilon \cap W_{i}} q_{i,k}(x) dx + o(1)
\nonumber
\\
&=& \sum_{i=1}^N  \int_{K^\varepsilon \cap W_{i}} q_{i}(x) dx + o(1)
\end{eqnarray}
by \eqref{e14.87}. We insert \eqref{e14.91}, \eqref{e14.92}, and \eqref{e14.94}
in the expression \eqref{e14.90} as we did for \eqref{e14.75} and \eqref{e14.76}, 
and we get the following replacement \eqref{e14.77}:
\begin{equation} \label{e14.95}
J_-(\u) \leq J_-(\u^\ast) - \sum_i \int_{K^\varepsilon \cap W_i} q_i
+ \sum_i \int_{K^\varepsilon \cap W_i^\ast} q_i + o'(1).
\end{equation}
This is then enough to conclude as before.
\qed
\end{proof}

\section{Blow-up Limits are minimizers}   \label{blowup}

We shall now apply the results of the previous section to describe blow-up limits 
of minimizers of our main functional. In this section, we show that they are 
themselves local minimizers of a simpler functional; in later sections, 
we shall use previous work in the context of \cite {AC} and \cite{ACF} 
to describe the blow-up limits more precisely, and then obtain some information on 
the initial minimizers themselves.

In this section we start from a minimizer $(\u,\W)$ for the usual functional $J$ 
(see Section~\ref{intro}), and choose an origin $x_0 \in \R^n$ where we do the blow up. 
Here we shall just take $x_0 = 0$ to save some notation. 
We give ourselves a sequence $\{ r_k \}$, with 
\begin{equation} \label{e15.1}
\lim_{k\to +\infty} r_k = 0
\end{equation}
and consider the pairs $(\u_k,\W_k)$ given by
\begin{equation} \label{e15.2}
\u_{k}(x) = r_k^{-1} \u(r_k x)
\ \text{ and } \ 
\W_{k} = r_k^{-1} \W
\end{equation}
(i.e., $W_{i,k} = r_k^{-1}W_{i}$ for $1 \leq i \leq N$). 
Notice the normalization of $\u_k$, which keeps the Lipschitz constants intact. 

Let us describe our main assumptions now, that will allow us to apply
the results of the previous section. First of all, we assume that
\begin{equation} \label{e15.3}
\u(0) = 0.
\end{equation}
This is needed because we will require the $u_{k,i}$ to have pointwise
limits $u_{\infty,i}$, in particular at the origin, and this will not hurt because 
otherwise (under weak assumptions that make $\u$ H\"older continuous) 
only one component $W_i$ is present near $0$, and there is no free boundary to study. 

We also assume that there is a ball $B(0,\rho_0)$ such that for $1 \leq i \leq N$
\begin{equation} \label{e15.4}
f_i \geq 0 \text{ on } B(0,\rho_0), \ \text{ and } \ 
||f_i||_{L^\infty(B(0,\rho_0))}+||g_i||_{L^\infty(B(0,\rho_0))} < +\infty
\end{equation}
and 
\begin{equation} \label{e15.5}
\text{$\u$ is Lipschitz in $B(0,\rho_0)$.}
\end{equation}
It is just as simple here not to say how we know that \eqref{e15.5}
holds, but recall that we could obtain it by applying Theorem \ref{t10.1} 
(if $0$ is an interior point of $\Omega$) or \ref{t11.1} (if $0\in \d\Omega$).

Next let $R>0$ be given, and assume that there is a Lipschitz function
$\u_\infty$ on $B(0,R)$ such that
\begin{equation} \label{e15.6}
\u_\infty(x) = \lim_{k \to +\infty} \u_k(x)
\ \text{ for } x\in B(0,R).
\end{equation}
Notice that if $\u$ is $C_0$-Lipschitz on $B(0,\rho_0)$, then
$\u_k$ is $C_0$-Lipschitz on $B(0,R)$ as soon as $r_k \leq R^{-1} \rho_0$.
Since $\u(0)=0$, it is easy to extract subsequences for which \eqref{e15.6} 
holds for any given $R$, or even all integers $R$ at the same time. 
Because of this, we shall often be able to assume that
$\u_\infty$ is defined on the whole $\R^n$ and that
$\u_\infty(x) = \lim_{k \to +\infty} \u_k(x)$ for every $x\in \R^n$.

Since $\u_k$ is $C_0$-Lipschitz on $B(0,R)$ for $k$ large,
$\u$ also is Lipschitz, and the convergence in \eqref{e15.6} 
is automatically uniform.

In some cases the origin lies on $\d \Omega$, and then we also need to worry
about a limit for the sets
\begin{equation} \label{e15.7}
\Omega_k = r_k^{-1} \Omega.
\end{equation}
We shall essentially keep the mild assumptions of the previous section.
Let us assume that there is a measurable set 
$\Omega_\infty  = \Omega_{\infty,R}\i B(0,R)$ such that
\begin{equation} \label{e15.8}
\1_{\Omega_\infty} = \lim_{k \to +\infty} \1_{B(0,R) \cap \Omega_k} 
= \lim_{k \to +\infty} \1_{B(0,R) \cap r_k^{-1}\Omega} 
\ \text{ in } L^1(B(0,R)),
\end{equation}
as in \eqref{e14.6}, that for $0 < T < R$, there exist small constants $r_T > 0$ 
and $c_T > 0$ such that
\begin{equation} \label{e15.9}
|B(x,r) \sm \Omega_\infty| \geq c_T r^n
\ \text{ for $x\in B(0,T) \cap \d \Omega_\infty$ and } 0 < r \leq r_T,
\end{equation}
as in \eqref{e14.7}, and that for $0 < T < R$
\begin{equation} \label{e15.10}
\lim_{k \to+\infty} \delta(k,T) = 0,
\ \text{ where }
\delta(k,T) = \sup\big\{ \dist(x,B(0,R) \sm \Omega_\infty) \, ; \, 
x\in B(0,T) \sm \Omega_k\big\},
\end{equation}
as in \eqref{e14.8} and to avoid tiny little islands of $B(0,R) \sm \Omega_k$
in the middle of $\Omega_k$ that \eqref{e15.8} would not detect.
Take $\delta(k,T)=0$ when $B(0,T) \sm \Omega_k = \emptyset$.

If $\Omega$ is very general, we cannot guarantee that we can find 
sequences for which such an $\Omega_\infty$ exists, but fairly weak regularity 
properties of $\Omega$ will ensure all this. We give in Lemma~\ref{t16.1}
below a sufficient condition for this to happen, where we ask $\d\Omega$
to be porous near $0$ (so that we can find a subsequence and $\Omega_\infty$ 
such that \eqref{e15.8} holds), and $\Omega$ to satisfy \eqref{e15.9} near $0$ 
(so that \eqref{e15.10} holds for $\Omega_\infty$). 

In the mean time, notice that if $\Omega$ has a $C^1$ boundary near the origin, 
we don't even need to extract a subsequence and the limit is a half space.
Notice also that the special case when $0$ is an interior point of $\Omega$
is still included here; the set $\Omega_{\infty,R} = B(0,R)$ satisfies the conditions above
because $\Omega_k \supset B(0,r)$ for $k$ large.

Next we describe our assumptions on the functional $F$. 
We assume that there exist real constants $\lambda_i$, $1 \leq i \leq N$, and a function
$\varepsilon(r)$ defined for $r$ small and such that $\lim_{r \to 0} \varepsilon(r) = 0$,
such that 
\begin{equation} \label{e15.11}
\Big| F(\W)-F(\W') - \sum_{i=1}^N \lambda_i \big[|W_i \cap B(0,r)| -|W'_i \cap B(0,r)|\big]
\Big| \leq r^n \varepsilon(r)
\end{equation}
for every $N$-uple $\W' = (W'_1, \ldots, W'_N) \in {\cal W}(\Omega)$
such that $W_i \sm B(0,r) = W'_i \sm B(0,r)$ for $0 \leq i \leq N$.
Recall that $\W$ is the $N$-uple that comes from our minimizer $(\u,\W)$.

Thus, very near $0$, we require $F$ to look a lot like the function $F$
defined by \eqref{e1.7}, with constant functions $q_i = \lambda_i$.
Slightly surprisingly, we shall be able to content ourselves with the error
$r^n \varepsilon(r)$, rather than $\varepsilon(r) |\W \Delta \W'|$
(which would vaguely correspond to requiring a derivative of the special
form suggested by \eqref{e15.11} in some directions.

Let us give two examples where this condition is satisfied. Naturally the first one is when
$F$ is given by the formula \eqref{e1.7}, with functions $q_i(x)$ that are 
continuous at $x=0$. In fact, it is even enough to take the $q_i(x)$ locally integrable,
and to assume that $0$ is a Lebesgue point for each $q_i$, in the sense that
\begin{equation} \label{e15.12}
\lim_{r \to 0} \fint_{B(0,r)} |q_i(x) - q_i(0)| = 0.
\end{equation}
Indeed, if $\W' = (W'_1, \ldots, W'_N) \in {\cal W}(\Omega)$
such that $W_i \sm B(0,r) = W'_i \sm B(0,r)$ for $0 \leq i \leq N$,
\begin{eqnarray} \label{e15.13}
F(\W)-F(\W') &=& \sum_{i=1}^N \int_{B(x,r)} [\1_{W_i}(x)-\1_{W'_i}(x)] q_i(x) dx
\nonumber\\
&=& \sum_{i=1}^N q_i(0) \big[|W_i \cap B(0,r)| -|W'_i \cap B(0,r)|\big] + E,
\end{eqnarray}
with
\begin{eqnarray} \label{e15.14}
|E| &=& \Big|\sum_{i=1}^N \int_{B(x,r)} [\1_{W_i}(x)-\1_{W'_i}(x)] (q_i(x)-q_i(0)) dx\Big|
\nonumber\\
&\leq& \sum_{i=1}^N \int_{B(x,r)} |q_i(x)-q_i(0)| dx = o(r^n).
\end{eqnarray}
Thus, \eqref{e15.11} holds, with $\lambda_i = q_i(0)$.
Notice that we do not require the $q_i$ to be nonnegative.

Our second example is when $F$ is a function of the volumes, i.e.,
$F({\bf Z}) = \wt F(|Z_1|, \ldots, |Z_N|)$ 
for ${\bf Z} \in {\cal W}(\Omega)$,
with a function $\wt F$ of $N$ variables  
which is differentiable at $(|W_1|, \ldots, |W_N|)$.
That is, assume that near $(|W_1|, \ldots, |W_N|)$, 
\begin{equation} \label{e15.15}
\wt F(v_1, \ldots, v_N) - \wt F(|W_1|, \ldots, |W_N|)
= \sum_{i=1}^N  \lambda_i (v_i-|W_i|) + o\big(\sum_i |v_i-|W_i||\big);
\end{equation}
then \eqref{e15.11} holds, with the same numbers $\lambda_i$,
because $\sum_i |v_i-|W_i|| \leq C r^n$ when ${\bf Z} = \W$
outside of $B(0,r)$.

So our last assumption \eqref{e15.11} is reasonably mild.
Associated to $\Omega_\infty$ and the $\lambda_i$, there is a functional
$J_\infty$, defined on $\F(B(0,R),\Omega_\infty)$ by
\begin{equation} \label{e15.16}
J_\infty({\bf v},{\bf Z}) = \sum_{i=1}^N \int_{B(0,R)} |\nabla v_i|^2
+ \sum_{i=1}^N \lambda_i |Z_i|.
\end{equation}

\begin{thm}\label{t15.1} 
Let $(\u,\W) \in \F$ be a minimizer for the functional $J$ of Section \ref{intro},
and let $\{ r_k \}$ be a sequence such that $\lim_{k\to +\infty} r_k = 0$. Assume
that \eqref{e15.3}-\eqref{e15.5} hold for some $R >0$ and some Lipschitz function 
$\u_\infty$ on $B(0,R)$, that \eqref{e15.8}-\eqref{e15.10} hold for some
$\Omega_\infty \i B(0,R)$, and that $F$ satisfies \eqref{e15.11}.
Then we can find $\W_\infty$ such that 
$(\u_\infty,\W_\infty) \in \F(B(0,R), \Omega_\infty)$ and $(\u_\infty,\W_\infty)$ 
is a local minimizer for $J_\infty$ in $B(0,R)$. We also have that
\begin{equation} \label{e15.17}
\lim_{k \to +\infty} \u_k = \u_\infty \text{ in $W^{1,2}(B(0,R))$.}
\end{equation}
\end{thm}

\ms
See near \eqref{e14.1} for the definition of $\F(B(0,R),\Omega_\infty)$ and
local minimizers for $J_\infty$; in the present case, $J_\infty$ has no $M$-term
and $\int_{B(0,R)} |\nabla U_\infty| < +\infty$, so \eqref{e14.1} can be rewritten as
\begin{equation} \label{e15.18}
J_\infty({\bf u}_\infty,\W_\infty)  \leq J_\infty({\bf v},{\bf Z}),
\end{equation}
and the local minimality of $(\u_\infty,\W_\infty)$ means that 
\eqref{e15.18} holds for every pair $({\bf v},{\bf Z}) \in \F(B(0,R),\Omega_\infty)$
that coincides with $(\u_\infty,\W_\infty)$ on some $B(0,R)\sm B(0,r)$, $r < 1$.

See Corollary \ref{t15.3} for the more natural rephrasing of this theorem with $R=+\infty$.

\begin{proof}
Naturally we want to deduce this from Theorem \ref{t14.1} and Remark \ref{t14.4}. 
Ironically, we cannot use Corollary \ref{t14.5}, even when $F$ is given by \eqref{e1.7},
because we assumed that the $q_i$ are nonnegative there.
We first compute the functional $J_k$ of which the pair $(\u_k,\W_k)$
defined in \eqref{e15.2} is a minimizer.

\begin{lem}\label{t15.2} 
If $(\u,\W) \in \F$ is a minimizer for $J$ in the class $\F(\Omega)$ of
Definition \ref{d1.1}, then the pair $(\u_k,\W_k)$ is a minimizer for 
$J_k$ in $\F(\Omega_k)$, where $\Omega_k = r_k^{-1} \Omega$
and $J_k$ is defined like $J$ in \eqref{e1.5}, 
except that we use the functions $f_{i,k}$ and $g_{i,k}$ given by
\begin{equation} \label{e15.19}
f_{i,k}(x) = r_k^2 f_i(r_k x) \ \text{ and } \ g_{i,k}(x) = r_k g_i(r_k x)
\end{equation}
and the functional $F_k$ defined by
\begin{equation} \label{e15.20}
F_k(H_1, \ldots, H_N) = r_k^{-n} \, F(r_k H_1, \ldots, r_k H_N)
= : r_k^{-n} \, F(r_k {\bf H}).
\end{equation}
\end{lem}

\begin{proof}
Notice the extra powers in \eqref{e15.19}, which come from the scaling
and the fact that we base our normalization of $J_k$ on the energy.
For a general pair $(\u^\ast,\W^\ast)$, define
$\u^\ast_k$ and $\W^\ast_k$ as we did for 
$\u_k$ and $\W_k$ (in \eqref{e15.2}).
It is clear that $(\u^\ast,\W^\ast) \in \F(\Omega)$ if and only if
$(\u_k^\ast,\W_k^\ast)$ lies in $\F(\Omega_k)$, and that
$F_k(\W_k^\ast)$ is well defined by \eqref{e15.20} (because 
$r_k {\bf H} \in {\cal W}(\Omega)$). Let us show that 
\begin{equation} \label{e15.21}
J_k(\u^\ast_k,\W^\ast_k) = r_k^{-n} J(\u^\ast,\W^\ast);
\end{equation}
the conclusion will easily follow.
For the last term, this is just a consequence of the definition, since
$F_k(\W^\ast_k) = r_k^{-n} F(\W^\ast)$ by \eqref{e15.20}.
For the energy, we just change variables and get that
\begin{eqnarray} \label{e15.22}
E(\u_k^\ast) &=& \int |\nabla u_k^\ast|^2 = \int |\nabla [r_k^{-1} \u^\ast(r_k \,\cdot) ]|^2
\nonumber
\\
&=& \int |(\nabla \u^\ast)(r_k x)|^2 dx = r_k^{-n} \int |\nabla \u^\ast|^2 = r_k^{-n} E(\u^\ast).
\end{eqnarray}
The $M$-term of the functional is computed similarly:
\begin{eqnarray} \label{e15.23}
M_k(\u_k^\ast) &=:& \sum_{i=1}^N \int [u_{i,k}^\ast(x)^2 f_{i,k}(x) - u_{i,k}^\ast(x) g_{i,k}(x)] dx
\nonumber
\\
&=& \sum_{i=1}^N \int [u_{i}^\ast(r_k x)^2 f_{i}(r_k x) - u_{i}^\ast(r_k x) g_{i}(r_k x)] dx
= r_k^{-n} M(\u^\ast)
\end{eqnarray}
because $u_{i,k}^\ast(x) = r_k^{-1}u_i(r_k x)$ by \eqref{e15.2}, and
by \eqref{e15.19}. Hence \eqref{e15.21} holds and the lemma follows.
\qed
\end{proof}

\ms
We return to the proof of Theorem \ref{t15.1}. 
Let $R > 0$ be as in the statement. We want to apply Theorem \ref{t14.1}
and Remark \ref{t14.4} to the restriction to $\O = B(0,R)$ of our sequence
of minimizers $(\u_k,\W_k)$. We use the domains $\Omega_k \cap \O$
as our measurable subsets of $\O$, and just the fixed $J_\infty$ as 
our variable functional; of course the $(\u_k,\W_k)$ will only be approximate
local minimizers for $J_\infty$, and this is why we shall use Remark \ref{t14.4}.
Then \eqref{e14.2} and \eqref{e14.3} are trivial (we take $f_i = f_{i_k}=g_i = g_{i,k}=0$),
\eqref{e14.4} and \eqref{e14.5} hold with 
$\wt F_k(v_1, \ldots , v_N) = \sum_i \lambda_i v_i$
and $\wt F = \wt F_k$, \eqref{e14.6} is just a translation of \eqref{e15.8},
\eqref{e14.7} is the same as \eqref{e15.9}, and 
\eqref{e14.8} follows from \eqref{e15.10}.

We will need to replace \eqref{e14.9} by \eqref{e14.83}
but the uniform Lipschitz bound in \eqref{e14.10} holds by \eqref{e15.5},
and \eqref{e14.11} follows from \eqref{e15.6}. This completes our list of
easy verifications; as soon as we check \eqref{e14.83}, 
Remark \ref{t14.4} will say that the conclusion of Theorem \ref{t14.1}
holds for the limit function $\u_\infty$,
and Theorem \ref{t15.1} will follow. Notice that the conclusion \eqref{e14.12} 
looks a little weaker than \eqref{e15.17}, but here $\u_\infty$ is Lipschitz
on $B(0,R)$ (by \eqref{e15.5} and \eqref{e15.6}), so
$\lim_{k \to +\infty} \u_k = \u_\infty$ in $W^{1,2}(B(0,R))$
as soon as this happens in $B(0,T)$ for every $T<R$.

So we check \eqref{e14.83}, which means that we take a
competitor $({\bf v},{\bf H})$ (previously called $(\u_k^\ast,\W_k^\ast)$)
for $(\u_k,\W_k)$ in $\O = B(0,R)$, relative to $\Omega_k \cap \O$, 
and we want to show that
\begin{equation} \label{e15.24}
J_\infty(\u_k,\W_k) \leq J_\infty({\bf v},{\bf H}) + \alpha_k,
\end{equation}
for some $\alpha_k$ that tends to $0$; we shall not need the dependence
on the compact set $K$ here.

Notice that the pair $(\u_k,\W_k)$ was actually defined on the larger
set $\R^n$, but $({\bf v},{\bf H}) \in \F(B(0,R),\Omega_k)$ 
is only defined on $B(0,R)$. This is easy to fix: set
${\bf v}(x) = \u_k(x)$ on $\R^n \sm B(0,R)$ and define 
${\bf H}'$ by $H'_i = H_i \cup (W_{i,k} \sm B(0,R)$ for $1 \leq i \leq N$. 
It is easy to see that $({\bf v},{\bf H}') \in \F(\Omega_k) = \F(\R^n,\Omega_k)$;
in particular, the fact that ${\bf v} \in W^{1,2}(\R^n)$ is trivial, because
${\bf v}$ coincides with $\u_k$ in $B(0,R) \sm B(0,T)$ for some $T<R$,
which gives enough room to glue $v\in W^{1,2}_{loc}(B(0,R)$ with $\u \in W^{1,2}(\R^n)$.
Notice also that $({\bf v},{\bf H}')$ is a competitor for $(\u_k,\W_k)$, which
implies that
\begin{equation} \label{e15.25}
J_k(\u_k,\W_k) \leq J_k({\bf v},{\bf H}').
\end{equation}
Recall from Lemma \ref{t15.2} that 
\begin{equation} \label{e15.26}
J_k({\bf v},{\bf H}') =  \int_{\Omega_k} |\nabla {\bf v}|^2 + M_k({\bf v}) 
+ r_k^{-n} \, F(r_k {\bf H}),
\end{equation}
where 
\begin{equation} \label{e15.27}
M_k({\bf v}) = \sum_{i=1}^N \int_{\Omega_k} [v_i(x)^2 f_{i,k}(x) - v_i(x) g_{i,k}(x)] dx
\end{equation}
Let us remove the large constant
\begin{equation} \label{e15.28}
A_k = \int_{\Omega_k \sm B(0,R)} |\nabla \u_k|^2 
+ \int_{\Omega_k \sm B(0,R)} [u_{i,k}(x)^2 f_{i,k}(x) - u_{i,k}(x) g_{i,k}(x)] dx
+ r_k^{-n} \, F(\W)
\end{equation}
from this; we get that
\begin{equation} \label{e15.29}
J_k({\bf v},{\bf H}') - A_k =  \int_{B(0,R)} |\nabla {\bf v}|^2 + M'_k({\bf v}) 
+ r_k^{-n} \, \big(F(r_k {\bf H}) - F(\W) \big),
\end{equation}
where 
\begin{equation} \label{e15.30}
M'_k({\bf v}) = \sum_{i=1}^N \int_{B(0,R)} [v_i(x)^2 f_{i,k}(x) - v_i(x) g_{i,k}(x)] dx.
\end{equation}
Also, the special case when $({\bf v},{\bf H}') = (\u_k,\W_k)$ yields
\begin{equation} \label{e15.31}
J_k(\u_k,\W_k) - A_k =  \int_{B(0,R)} |\nabla \u_k|^2 + M'_k(\u_k)
\end{equation}
because $r_k \W_k = \W$.
We now estimate various terms. First observe that for $x\in B(0,R)$,
\begin{equation} \label{e15.32}
|u_k(x)| = r_k^{-1} |\u(r_k x)| = r_k^{-1} |\u(r_k x)-\u(0)| \leq C
\end{equation}
by \eqref{e15.2}, \eqref{e15.3}, and \eqref{e15.5}; similarly,
$\u_k$ is $C$-Lipschitz on $B(0,R)$, so we easily get
that 
\begin{equation} \label{e15.33}
|M'_k(\u_k)| \leq C |B(0,R)| (||f_{i,k}||_\infty + ||g_{i,k}||_\infty) \leq C r_k
\end{equation}
by \eqref{e15.19} and \eqref{e15.4}, and 
with a constant $C$ that does not depend on $k$.

Notice that we may assume that
$J_\infty({\bf v},{\bf H}) \leq J_\infty(\u_k,\W_k)$, because otherwise
\eqref{e15.24} is satisfied for any $\alpha_k \geq 0$. This yields
\begin{equation} \label{e15.34}
\int_{B(0,R)} |\nabla {\bf v}|^2 
\leq J_\infty({\bf v},{\bf H}) + C \leq J_\infty(\u_k,\W_k) + C
\int_{B(0,R)} |\nabla \u_k|^2 + 2C \leq C'
\end{equation}
where our constants $C$ and $C'$ depend on the $\lambda_i$ and $R$, 
but again not on $k$, and we used again the fact that $\u_k$ is $C$-Lipschitz.
Recall also that the Sobolev function ${\bf v} - \u_k$ vanishes outside of $B(0,R)$,
so the Poincar\'e inequality \eqref{e3.7} yields
\begin{equation} \label{e15.35}
\int_{B(0,R} |{\bf v} - \u_k|^2 \leq C R^2 \int_{B(0,R} |\nabla({\bf v} - u_k|^2 \leq C
\end{equation}
hence $\int_{B(0,R} |{\bf v}|^2 \leq C$, by \eqref{e15.32}. Thus
\begin{equation} \label{e15.36}
|M'_k({\bf v})| \leq  (||f_{i,k}||_\infty + ||g_{i,k}||_\infty) 
\int_{B(0,R} (|{\bf v}| + |{\bf v}|^2) \leq C r_k.
\end{equation}
Finally notice that ${\bf H}'$ coincides with $\W_k$ outside of $B(0,R)$,
which implies that $r_k {\bf H}' = r_k \W_k = \W$ outside of $B(0,r_k R)$.
So we can apply \eqref{e15.11}, and we get that
\begin{equation} \label{e15.37}
r_k^{-n} \, \big(F(r_k {\bf H}') - F(\W) \big)
= r_k^{-n} \Big\{\sum_{i} \lambda_i 
\big[|r_k H'_i \cap B(0,r_k R)| -|W_i \cap B(0,r_k R)|\big]
+ error \Big\}
\end{equation}
with $|error| \leq (r_kR)^n \varepsilon(r_kR)$.
Since $W_i \cap B(0,r_k R) = r_k (W_{i,k} \cap B(0,R))$
and  $r_k H'_i \cap B(0,r_k R) = r_k(H'_i \cap B(0, R)) = r_k H_i$
by definition of ${\bf H}'$, \eqref{e15.37} is the same as
\begin{equation} \label{e15.38}
r_k^{-n} \, \big(F(r_k {\bf H}') - F(\W) \big)
= \sum_{i} \lambda_i \big[ |H_i | -|W_{i,k} \cap B(0,R)|\big]
+  r_k^{-n} error .
\end{equation}
We may now put things together:
\begin{eqnarray} \label{e15.39}
J_\infty(\u_k,\W_k) - J_\infty({\bf v},{\bf H})
&=&  \int_{B(0,R)} |\nabla \u_k|^2 - |\nabla {\bf v}|^2
- \sum_{i=1}^N \lambda_i \big[ |H_i | -|W_{i,k} \cap B(0,R)|\big]
\nonumber\\
&=& \int_{B(0,R)} |\nabla \u_k|^2 - |\nabla {\bf v}|^2
- r_k^{-n} \, \big(F(r_k {\bf H}') - F(\W) \big) + r_k^{-n} error
\nonumber\\
&=& J_k(\u_k,\W_k) - J_k({\bf v},{\bf H}') - M'_k(\u_k) + M'_k({\bf v})
+ r_k^{-n} error
\\
&\leq& - M'_k(\u_k) + M'_k({\bf v}) + r_k^{-n} error
\leq C r_k + \varepsilon(r_kR)
\nonumber
\end{eqnarray}
by \eqref{e15.16}, \eqref{e15.29}, \eqref{e15.31}, \eqref{e15.25},
\eqref{e15.33}, and \eqref{e15.36}.
This proves \eqref{e15.24}, and we know that Theorem~\ref{t15.1} follows.
\qed
\end{proof}

\ms
Usually it does not hurt to treat all the balls $B(0,R)$ at the same time, 
because extracting a sequence from an original sequence $\{ r_k \}$
so that \eqref{e15.6} holds for every $R > 0$ is just as easy, if we
know that $\u$ is Lipschitz near $0$ (as in \eqref{e15.5}), as getting
it for a single $R$. Here is the corresponding statement.

\begin{cor}\label{t15.3} 
Let $(\u,\W) \in \F$ and the sequence $\{ r_k \}$ satisfy
the assumptions of Theorem~\ref{t15.1} for each integer $R > 0$.
Then the $\u_k$ converge uniformly on compact subsets of $\R^n$ to a 
Lipschitz function $\u_\infty : \R^n \to \R^N$, 
there is a measurable set $\Omega_\infty \i \R^n$ such that 
\eqref{e15.8} holds for every $R > 0$, and we can find disjoint sets
$W_{i,\infty} \i \Omega_\infty$ with the following property. Set
$\W_\infty = (W_{1,\infty}, \ldots, W_{N,\infty})$;
then $(\u_\infty,\W_\infty)$ lies in $\F(\R^n,\Omega_\infty)$
and is a local minimizer for $J_\infty$ in $\R^n$.
In addition, \eqref{e15.17} holds for all $R > 0$.
\end{cor}

See the beginning of Section \ref{limits} for the definition $\F(\R^n,\Omega_\infty)$.
The functional is still defined as in $\eqref{e15.16}$, 
except that we would now integrate on $\R^n$; so, when we say 
that $(\u_\infty,\W_\infty)$ is a local minimizer for $J_\infty$ in $\R^n$,
we mean as in \eqref{e14.1}, or rather \eqref{e14.88} that if 
$({\bf v},{\bf H}) \in \F(\R^n,\Omega_\infty)$
coincides with $(\u_\infty,\W_\infty)$ outside of a ball $B(0,R)$, then
\begin{equation} \label{e15.40}
\int_{B(0,R)} |\nabla u_{i,\infty}|^2 
+ \sum_{i=1}^N \lambda_i |W_{i,\infty} \cap B(0,R)|
\leq
\int_{B(0,R)} |\nabla v_{i}|^2 + \sum_{i=1}^N \lambda_i |H_{i}\cap B(0,R)|.
\end{equation}

We did not require directly the existence of a limit $\u_\infty$ that
works for all $R$, but it is easy to see that if the $\u_k$ converge
in $B(0,R)$ for every integer $R$, the limit does not depend on $R$.
This is our function $\u_\infty$, and the fact that the convergence
is uniform on compact subsets of $\R^n$ follows easily from the 
fact that for each $R$, the $\u_k$, $k$ large enough, are uniformly
Lipschitz by \eqref{e15.5}.

\begin{proof}
Compared to Theorem \ref{t15.1}, there is just a small amount of new information here.
We already discussed the fact that the limit of the $\u_k(x)$ does not depend on the  
radius $R$, and that this limit is uniform on compact sets.
We also need to say that, modulo sets of measure $0$, the limit set
$\Omega_\infty$ on \eqref{e15.8} does not depend on $R$ either. 
It could be that if we are too clumsy with the gluing of the various 
$\Omega_\infty$, the condition \eqref{e15.9} does not hold any more
(on the whole $\R^n$). We do not need this to prove the corollary, and we could
probably easily fix this problem if it became an issue for some other question. 

Finally, we also need to show that we can find a fixed $\W_\infty$ that works for all $R$.
One way we can do this is to modify slightly our choice of $\W$ (starting a little
above \eqref{e14.18}).
Indeed the proof of \eqref{e14.17} also yields that for $A_R = B(0,R+1) \sm B(0,R)$,
$R \in \mathbb N$,
\begin{equation} \label{e15.41}
|W'_i \cap A_R| \leq \liminf_{j\to +\infty} |W_{i,k_j} \cap A_R|.
\end{equation}
Then we extract a new subsequence $\{ k_j \}$, so that for each integer $R$,
the limits $l_{i,R} = \lim_{j \to +\infty} |W_{i,k_j} \cap A_R|$ exist (as in \eqref{e14.17}),
we still have an analogue of \eqref{e14.20} in each annulus, and this allows us to
complete the $W'_i$, independently on each annulus, so that we have a stronger form
of \eqref{e14.20} with $|W_i \cap A_R| = l_{i,R}$. With this construction, the restriction
of $\W$ to a given ball does not depend on $R$.

We could also use the fact that $J_\infty$ has a special form to replace directly 
$\W'$ with an optimal choice, where we keep $W_i = W'_i$ for all $i$,
except perhaps for one $i_0$ for which $\lambda_{i_0}$ is the smallest.
If $\lambda_{i_0} \geq 0$ we still keep $W_{i_0} = W'_{i_0}$, but otherwise
we take $W_{i_0} = W'_{i_0} \cap \big[\Omega_\infty \sm 
\big(\cup_{i \neq i_0} W'_i\big)\big]$.
This candidate works at least as well as our old one, and it is easy to see that it does not
depend on $R$ (if we always choose the same $i_0$).
\qed
\end{proof}

\section{Blow-up limits with 2 phases}  \label{2phases} 

Our interest in the blow-up limits of our minimizers $(\u,\W)$ comes from the
fact that they should be simpler to study, and their description should still provide 
useful information on $(\u,\W)$. We also intend to use the fact that 
since our blow-up limits are given by a more standard Alt, Caffarelli, and Friedman 
formula, they were studied intensively and we can use some of the results.

In this Section we keep the same assumptions as in the previous one,
add a minor regularity property for $\Omega$ to make sure that we can
apply Corollary \ref{t15.3}, and study the blow-up limits of $\u$ 
when we can find two different choices 
of pairs $(i,\varepsilon)$, where $i \in [1,N]$ and $\varepsilon$ is a sign,
such that the functional $\Phi(r)$ associated in Section \ref{mono} to the two 
functions $(\varepsilon u_i)_+$ has a nonzero limit when $r$ tends to $0$.

We will see that in this case, all the blow-up limits $\u_\infty$ of $\u$ at the origin 
are composed of just two non-trivial affine functions defined on the two components
of the complement of some hyperplane $H$, and which vanish on $H$; the other
components of $\u_\infty$ are null.

It will follow that $0$ lies in the interior of $\Omega$, and that
the other components of $\u$ are small near the origin.
See Corollary \ref{t16.4}.

We will also see that in this case the natural free boundaries associated
to $\u$ stay quite close to hyperplanes in the small balls $B(0,r)$.
See Corollary \ref{t16.5}. This is not such an impressive regularity
result, but we can get it without nondegeneracy assumptions like
\eqref{e13.1}, or the size of the $q_i(0)$ when $F$ is given
by \eqref{e1.7} with functions $q_i$ that are nearly continuous at
$0$ (as in \eqref{e15.12}).

\ms
The main assumptions for this section are almost the same as for Corollary \ref{t15.3}. 
We consider a minimizer $(\u,\W)$ for $J$ and a sequence $\{ r_k \}$ that tend to $0$,
we suppose that the $f_i$ are nonnegative and bounded and the $g_i$
are bounded (as in \eqref{e15.4}),
that $\u(0) = 0$,  $\u$ is Lipschitz near $0$, and 
the $\u_k$ defined by \eqref{e15.2} converge pointwise 
(or uniformly on compact sets, this is the same) to a function $\u_\infty$, 
as in \eqref{e15.3}, \eqref{e15.5}, and \eqref{e15.6}, that we assume for all $R>0$.

We also systematically assume that the volume functional
$F$ satisfies the regularity condition \eqref{e15.11}.

When we want information for a specific blow-up limit of $(\u,\W)$
at the origin, we shall assume that that for each $R >0$, there is a limit 
$\Omega_\infty= \Omega_{\infty,R}$, as in \eqref{e15.8} and \eqref{e15.10}, 
and that satisfies the weak regularity assumption \eqref{e15.9}.
Then Corollary \ref{t15.3} will give a limiting domain $\Omega_\infty$
(that does not depend on $R$, but this is not a surprise), 
and a $N$-uple $\W_\infty$ such that
$(\u_\infty,\W_\infty)$ is a  local minimizer for the functional
$J_\infty$ in $\R^n$, relative to $\Omega_\infty$ (i.e., in the class
$\F(\R^n, \Omega_\infty)$).
We recall what this means near \eqref{e15.40}.

But for Corollary \ref{t16.4} we will need to know that for each sequence
$\{ r_k \}$ that tends to $0$, we can find a subsequence that
satisfies the assumptions above. There is no problem with the convergence
of $\u_k$, since $\u$ is Lipschitz near $0$, but if $\Omega$ is
too ugly near the origin, it may be hard to find the $\Omega_{\infty,R}$.
For instance, the functions ${\bf 1}_{\Omega_k}$ may converge weakly to the
constant $1/2$. Let us give a condition which will prevent that. 

\begin{lem} \label{t16.1}
Suppose that there is a radius $\rho > 0$ 
and a constant $\tau > 0$ such that, for $x\in \d\Omega \cap B(0,\rho)$ 
and $0 < r \leq \rho$, 
\begin{equation}\label{e16.1}
|B(x,r) \sm \Omega| \geq \tau r^n
\end{equation}
and we can find $y\in B(x,r)$ such that
\begin{equation}\label{e16.2}
\dist(y,\d\Omega) \geq \tau r.
\end{equation}
Then for each sequence $\{ r_k \}$ that
tends to $0$, we can find a subsequence and a measurable set
$\Omega_\infty$ for which \eqref{e15.8}-\eqref{e15.10} hold for each
$R > 0$.
\end{lem}

It is amusing that we do not need to know on which side of
$\d \Omega$ the point $y$ lies. Of course if we assume that
$\dist(y,\Omega) \geq \tau r$, we get \eqref{e16.1} for free.

We included the lemma mostly for fun, and the the reader that would not be 
convinced can skip and assume that $\Omega$
is a Lipschitz (or even $C^1$) domain near the origin; then the existence of a good
subsequence is really easy.

\begin{proof}
First we take a subsequence (which as usual we still
denote by $\{ r_k \}$) for which the boundaries $\d\Omega_k$ 
of converge to a closed set $Z$ locally in $\R^n$, for the Hausdorff distance.
This just means that for each $R>0$, the numbers $d_R$ tend to
$0$, where
\begin{equation} \label{e16.3}
d_R = \sup\big\{ \dist(z,\d\Omega_k) \, ; \, z\in Z \big\}
+\sup\big\{ \dist(w,Z) \, ; \, w\in \d\Omega_k \big\},
\end{equation}
and the existence of such a subsequence comes from the standard compactness 
property of the Hausdorff distance. 

Then we show that $Z$ is porous. Let $z\in Z$ and $r > 0$ be given.
Choose $R$ such that $B(z,r) \i B(0,R)$, and then $k$ so large that
$d_{2R} \leq \tau r/10$ and $r_k R < \rho$.
Pick $w\in \d\Omega_k$ such that $|w-z| \leq \tau t/10$,
and apply our second assumption to $x=r_k w \in \d\Omega \cap B(0,\rho)$ 
and the radius $r_k r < \rho$; this gives a point $y \in B(x,r_k r)$
such that $\dist(y,\d\Omega) \geq \tau r_k r$. Then $y'=r_k^{-1} y$
lies in $B(w,r) \i B(z,2r)$ and $\dist(y',\d\Omega_k) \geq \tau r$.
Since $d_{2R} \leq \tau r/10$, we get that $\dist(y',Z) \geq \tau r/2$.
Thus every ball $B(z,2r)$ centered on $Z$ contains a ball of radius
$\tau r/2$ that does not meet $Z$, our definition of porous.

It follows that $|Z| = 0$ (if $Z$ is porous, it cannot have a Lebesgue point of
density). We now need to say which part of $\R^n \sm Z$ lies in $\Omega_\infty$
and which part lies in $\R^n \sm \Omega_\infty$, and for this we shall extract
a subsequence again. Let $\{ y_j \}$, $j \geq 0$, be a dense sequence in
$\R^n \sm Z$, and cover $\R^n \sm Z$ by the balls $B_j = B(y_j,\dist(y_j,Z)/2)$.
For each $j \geq 0$, we define a sequence $\{m_{j,k} \}$, $k \geq 0$.
Set $m_{j,k}=1$ when $B_j \i \Omega_k$, $m_{j,k}=-1$ when 
$B_j \i \R^n\sm\Omega_k$, and $m_{j,k}=0$ otherwise. Since 
$\dist(B_j,Z) > 0$, we get that $B_j \cap \d \Omega_k$ for $k$ large,
so $m_{j,k} \neq 0$ for $k$ large. We extract our new subsequence 
so that $\{m_{j,k} \}$ has a limit $l_j$ for each $j$ (which is therefore either
$1$ or $-1$). Then we set 
$\Omega_\infty = \bigcup_{j \geq 0; l_j = 1} B_j$ and
$\Omega^\sharp =  \bigcup_{j \geq 0; l_j = -1} B_j$. Let us check that
\begin{equation} \label{e16.4}
\R^n \sm Z \text{ is the disjoint union of $\Omega_\infty$ and }
\Omega^\sharp.
\end{equation}
Both sets are contained in $\R^n \sm Z$ because $B_j \i \R^n \sm Z$ for $j \geq 0$.
If $x \in \R^n \sm Z$, then $x\in B_j$ for some $j$, and then
$x$ lies in the corresponding set. But if $x \in B_j \cap B_i$,
then for $k$ large, $B_i$ and $B_j$ are contained in $\Omega_k$
or $\R^n \sm \Omega_k$, and this has to be the same set for both balls
(the set that contains $x$). So $l_i = l_j$, and our two sets are disjoint.
So \eqref{e16.4} holds.

Now we have a candidate $\Omega_\infty$, and we just need to check 
\eqref{e15.8}-\eqref{e15.10}. 
We start with \
Let $R > 0$ and $\varepsilon > 0$ be given, and choose a compact set
$K \i B(0,R) \sm Z$ such that 
$|K| \geq |B(0,R) \sm Z| -\varepsilon = |B(0,R)| - \varepsilon$.
Cover $K$ by a finite number of balls $B_j$, $j\in J$.
Notice that for $k$ large, $B_j$ is either contained in $\Omega_k$
(if $l_j = 1$) or in $\R^n \sm \Omega_k$ (if $l_j = -1$).
Now for each $x\in K$, choose $j\in J$ such that $x\in B_j$;
if $l_j = 1$, then $x\in \Omega_k \cap \Omega_\infty$. If
$l_j = -1$, then $x\in (\R^n \sm \Omega_k) \cap \Omega^\sharp
\i (\R^n \sm \Omega_k) \cap(\R^n \sm \Omega_\infty)$
(by \eqref{e16.4}). That is $K \cap \Omega_\infty = K \cap \Omega_k$,
and hence $||\1_{\Omega_\infty}-\1_{\Omega_k}||_{L^1(B(0,R))}
\leq |B(0,R) \sm K| \leq \varepsilon$. This holds for $k$ large; 
\eqref{e15.8} follows.

Next we want to deduce from \eqref{e15.8} and \eqref{e16.1} that
\begin{equation} \label{e16.5}
|B(x,r) \sm \Omega_\infty| \geq \tau 2^{-n} r^n
\text{ for $x\in \d\Omega_\infty$ and $r > 0$;}
\end{equation}
obviously \eqref{e15.9} will follow (we even get some additional
scale invariance).

Let $x\in \d\Omega_\infty$ and $r > 0$ be given.
Observe that $x\in Z$, because otherwise $x$ would lie in one of the 
open balls $B_i$, which cannot meet $\d\Omega_\infty$ because they are 
contained in $\Omega_\infty$ or in $\Omega^\sharp \i \R^n \sm \Omega_\infty$.
For $k$ large, we can find $w \in \d\Omega_k$ such that
$|w-x| \leq r/2$. Then \eqref{e16.1}, applied to the ball
$B(r_k w,r_k r/2)$ (which is contained in $B(0,\rho)$ for $k$ large), says that
\begin{equation} \label{e16.6}
|B(x,r) \sm \Omega_k| 
\geq |B(w,r/2) \sm \Omega_k| = r_k^{-n} |B(r_k w,r_k r/2) \sm \Omega|
\geq \tau 2^{-n} r^n.
\end{equation}
But \eqref{e15.8} says that  
$|B(x,r) \sm \Omega_\infty| = \lim_{k \to +\infty} |B(x,r) \sm \Omega_k|$
so we take a limit in \eqref{e16.6} and get that
$|B(x,r) \cap \sm \Omega_\infty| \geq \tau 2^{-n} r^n$; \eqref{e16.5}
and \eqref{e15.9} follow.

For \eqref{e15.10} we fix $0 < T < R$ and $\varepsilon > 0$, and show
that for $k$ large, $\delta(k,T) 
:= \sup\big\{ \dist(x,B(0,R) \sm \Omega_\infty) \, ; \, 
x\in B(0,T) \sm \Omega_k\big\},\leq \varepsilon$. 
We may safely assume that
$\varepsilon < R-T$. Pick $x\in B(0,T) \sm \Omega_k$,
and first assume that $\dist(x,\Omega_k) \leq \varepsilon/2$.
Then $\dist(x,\d\Omega_k) \leq \varepsilon/2$ too, and if
$k$ is large enough (depending on $\varepsilon$ and $R$), we also get that 
$\dist(x,\d\Omega_k) \leq 2\varepsilon/3$,
because $Z$ is the limit of the $\d\Omega_k$.
By \eqref{e16.5} (applied to a small ball centered on $Z\cap B(x,\varepsilon)$, 
$\Omega_\infty$ meets $B(x,\varepsilon)$, and 
$\dist(x, B(0,R) \sm \Omega_\infty) < \varepsilon$,
because $x\in B(0,T)$ and $\varepsilon < R-T$.

When $\dist(x,\Omega_k) \geq \varepsilon/2$, we just say that
\begin{eqnarray} \label{e16.7}
|\Omega_\infty \cap B(x,\varepsilon/2)|
&\geq& |\Omega_k \cap B(x,\varepsilon/2)|
- ||\1_{\Omega_\infty}-\1_{\Omega_k}||_{L^1(B(0,R))}
\nonumber \\
&=&
|B(x,\varepsilon/2)| - ||\1_{\Omega_\infty}-\1_{\Omega_k}||_{L^1(B(0,R))}
> 0
\end{eqnarray}
if $k$ is small enough (again depending on $\varepsilon$ and $R$),
then $B(x,\varepsilon/2)$ meets $\Omega_\infty$ and we can conclude as above.
So $\delta(k,T) \leq \varepsilon$, and \eqref{e15.10} follows.

Notice that our proof of \eqref{e15.8} only uses the second condition
\eqref{e16.2}, and that our proof of \eqref{e15.9} and \eqref{e15.10} 
only uses the first one and \eqref{e15.8}.
\qed
\end{proof}

\ms
Return to the general case (without the assumptions of the lemma).

When $\u_\infty$ is the pointwise limit of a sequence $\{\u_k\}$
as above (and this includes the existence, for $R > 0$, of a limit
$\Omega_\infty$ such that \eqref{e15.8}-\eqref{e15.10} hold)
we shall say that it is a \underbar{regular blow-up limit of} $\u$ 
(at the origin, associated to the sequence $\{ r_k \}$);
by a slight abuse of notation (due to the fact that $\W_\infty$ is not
really determined by the sequence $\{ r_k \}$), we also say that
$(\u_\infty,\W_\infty)$ is a \underbar{regular blow-up limit} of $(\u,\W)$.
The reader should not worry, we shall not use this terminology too much;
we just want to insist on the fact that Lemma~\ref{t16.2} below, for instance,
only works for regular blow-up limits.

Our first result says that under the (other) assumptions above, the functionals $\Phi$
introduced in Section~\ref{mono} are monotone, and their normalized versions for the
$\u_k$ go to the limit, so that their analogues for $\u_\infty$ are constant 
on $(0,+\infty)$, with the value $L = \lim_{r \to 0} \Phi(r)$. We shall then 
study more carefully the case when $L >0$, which is somewhat easier.

Let us use again, and even expand slightly, the ugly notation of Section \ref{mono}.
The point is that in the discussion below, the main objects will not be the functions
$u_i$ themselves, but their positive or negative part.
Denote by $I$ the set of pairs $\varphi = (i,\varepsilon)$, where $i \in [1,N]$ as usual,
and $\varepsilon \in \{ -1, +1 \}$ is a sign. 
For $\varphi \in I$, we define the function $v_\varphi$ by
\begin{equation}\label{e16.8}
v_\varphi(x) = [\varepsilon u_{i}(x)]_+ = \max(0,\varepsilon u_{i}(x)) \in [0,+\infty)
\ \text{ for $x\in \R^n$.}
\end{equation}
We shall often refer $\varphi$ as a \underbar{phase}, and to $v_\varphi$ 
as a phase of $\u$. 
Hopefully, the reader will not be too disturbed by this additional notion, but
otherwise the trick that we used for Lemma~\ref{t10.2} can be applied here, 
to reduce to the situation where all the $u_i$ are required to be nonnegative, and the
notion of phase is useless (i.e., we could take $I =[1,N]$).
Incidentally, the reader may also want to remove from $I$ the phases
$\varphi = (i,\varepsilon)$ for which we demanded that $\varepsilon u_i \leq 0$
in the definition of $\F(\Omega)$ (the class of admissible pairs). These phases
will not really disturb, but they are useless because $v_\varphi = 0$.

We want to take blow-up limits, so we set
\begin{equation} \label{e16.9}
v_{\varphi,k}(x) = r_k^{-1} v_\varphi(r_k x) = [\varepsilon u_{i,k}(x)]_+ 
\end{equation}
for $k \geq 0$ and $x\in \R^n$ and, corresponding to $k=+\infty$,
\begin{equation} \label{e16.10}
v_{\varphi,\infty} =  [\varepsilon u_{i,\infty}]_+ .
\end{equation}
We shall now denote by $\Phi_{\varphi}^0$ the function
that we defined in \eqref{e9.4} in terms of $v_\varphi$, with $x_0=0$. 
That is, we set
\begin{equation} \label{e16.11}
\Phi_{\varphi}^0(r) = {1 \over r^2} \int_{B(0,r)} {|\nabla v_{\varphi}|^2 \over |x|^{n-2}} \, dx
\ \text{ for } r > 0.
\end{equation}
Notice that by \eqref{e15.5} (our Lipschitz condition),
\begin{equation} \label{e16.12}
\Phi_{\varphi}^0(r) 
\leq {1 \over r^2} \int_{B(0,r)} {C\over |x|^{n-2}} \, dx
\leq C'  \ \text{ for } 0 < r < \rho_0.
\end{equation}
Also define the analogue of $\Phi_{\varphi}^0$ for $k \geq 0$
and $k=\infty$, by
\begin{equation} \label{e16.13}
\Phi_{\varphi,k}(r) 
= {1 \over r^2} \int_{B(0,r)} {|\nabla v_{\varphi,k}|^2 \over |x|^{n-2}} \, dx.
\end{equation}

\begin{lem}\label{t16.2} 
For each $r > 0$,
\begin{equation} \label{e16.14}
\Phi_{\varphi,\infty}(r) = \lim_{k \to +\infty} \Phi_{\varphi,k}(r)
= \lim_{k \to +\infty} \Phi_{\varphi}^0(r_k r).
\end{equation}
\end{lem}

\begin{proof}
Fix $r >0$ and for each $\varepsilon > 0$, cut an integral as
\begin{eqnarray} \label{e16.15}
r^2 |\Phi_{\varphi,k}(r) - \Phi_{\varphi,\infty}(r)| 
&=& \int_{B(0,r)} {\big||\nabla v_{\varphi,k}|^2 
- |\nabla v_{\varphi,\infty}|^2\big|\over |x|^{n-2}} \, dx
\nonumber
\\
&&\hskip-2.2cm
\leq  \int_{x\in B(0,\eta r)} {|\nabla v_{\varphi,k}|^2 
+ |\nabla v_{\varphi,\infty}|^2\over |x|^{n-2}} 
+ (\eta r)^{2-n	} \int_{B(0,r) \sm B(0,\eta r)} 
\big||\nabla v_{\varphi,k}|^2 - |\nabla v_{\varphi,\infty}|^2\big| 
\nonumber
\\
&&\hskip-2.2cm
\leq C \int_{x\in B(0,\eta r)} {1\over |x|^{n-2}} 
 + (\eta r)^{2-n} \int_{B(0,r) \sm B(0,\eta r)} 
\big||\nabla v_{\varphi,k}|^2 - |\nabla v_{\varphi,\infty}|^2\big|
\end{eqnarray}
because our functions are uniformly Lipschitz.
The first term can be made as small as we want, by choosing $\eta$ small,
and then the second term also, because \eqref{e15.17} holds for every $R > 0$, and
says that $\nabla v_{\varphi,k}$ converges to $\nabla v_{\varphi,\infty}$ in $L^2(B(0,r))$. This is where we use the fact that $\u_\infty$ is a regular blow-up
limit. This proves the first part of \eqref{e16.14}.

For the second part we just compute.
By \eqref{e16.9}, \eqref{e15.2}, and \eqref{e9.3}, 
$v_{\varphi,k}(x) = r_k^{-1} v_\varphi(r_k x)$, 
so $\nabla v_{\varphi,k}(x) = \nabla v_\varphi(r_k x)$ and hence 
(setting $y = r_k x$ in the integral)
\begin{equation} \label{e16.16}
\Phi_{\varphi,k}(r) 
= {1 \over r^2} \int_{B(0,r)} {|\nabla v_{\varphi,k}(x)|^2 \over |x|^{n-2}} \, dx
= {1 \over r^2} \int_{B(0,r_k r)} 
{|\nabla v_{\varphi}(y)|^2 \over |r_k^{-1}y|^{n-2}} \, r_k^{-n} dy
= \Phi_{\varphi}^0(r_k r)
\end{equation}
for $k \geq 0$ and $r > 0$. Lemma \ref{t16.2} follows.
\qed
\end{proof}

\ms
We now come to the nearly monotone functional of Section \ref{mono}. 
Pick two different indices $\varphi_1, \varphi_2 \in I$,
and consider the product of the corresponding functions $\Phi_{\varphi}$.
Our function $\Phi$ from \eqref{e9.4} is now called
\begin{equation} \label{e16.17}
\Phi^0_{\varphi_1,\varphi_2} = 
\Phi_{\varphi_1}^0 \Phi_{\varphi_2}^0
\end{equation}
and we also set, for $k \geq 0$ and $k=\infty$
\begin{equation} \label{e16.18}
\Phi_{\varphi_1,\varphi_2,k} = 
\Phi_{\varphi_1,k} \Phi_{\varphi_2,k}.
\end{equation}
It was observed in the proof of Theorem \ref{t9.1} that the functions 
$v_{\varphi_1}$ and $v_{\varphi_2}$ satisfy the assumptions of Theorem 1.3 in \cite{CJK},
and this was even the main ingredient in the proof. Now we also know that
$\u(0) = 0$ (by \eqref{e15.3}), and since $\u$ is Lipschitz near the origin
(by \eqref{e15.5}), the additional size assumption in Theorem 1.6 of \cite{CJK}
is satisfied. Then that theorem says that the limit
\begin{equation} \label{e16.19}
L(\varphi_1,\varphi_2) =
\lim_{\rho \to 0} \Phi^0_{\varphi_1,\varphi_2}(\rho)
= \lim_{\rho \to 0} \Phi^{0}_{\varphi_1}(\rho)\Phi^{0}_{\varphi_2}(\rho)
\text{ exists.}
\end{equation}
It follows from Lemma \ref{t16.2} that the function 
$\Phi_{\varphi_1,\varphi_2,\infty}$ of \eqref{e16.18} is constant, with 
\begin{equation} \label{e16.20}
\Phi_{\varphi_1,\varphi_2,\infty}(r) =  L(\varphi_1,\varphi_2)
\ \text{ for } r>0.
\end{equation}

In this section we shall concentrate on the case when the limit
$L(\varphi_1,\varphi_2)$ is positive for some choice of phases
$\varphi_1 \neq \varphi_2$. This will be made easier, because
the following theorem gives a very good description of $\u_\infty$
when this happens.
For this theorem we forget a little minimizers and blow-up limits, 
and concentrate on pairs of harmonic functions with essentially 
disjoint supports.

\begin{thm} \label{t16.3}
Let $v_1$ and $v_2$ be two nonnegative Lipschitz functions on $\R^n$ such that
$v_1v_2 = 0$ everywhere, each $v_j$ is harmonic on the open set 
$\O_j = \big\{ x\in \R^n \, ; \, v_j(x) > 0 \big\}$, and there is 
a constant $L > 0$ such that 
\begin{equation} \label{e16.21}
\Phi_1(r)\Phi_2(r) = L
\ \text{ for } r > 0,
\end{equation}
where we set
\begin{equation} \label{e16.22}
\Phi_j(r) = {1 \over r^2} \int_{B(0,r)} {|\nabla v_j|^2 \over |x|^{n-2}}.
\end{equation}
Then there is a unit vector $e \in \R^n$, and two positive constants $a_1$ and $a_2$,
such that 
\begin{equation} \label{e16.23}
v_1(x) = a_1 \max(0,\langle x, e \rangle)
\text{ and } 
v_{2}(x) = a_2 \max(0,\langle x, -e \rangle)
\text{ for } x\in \R^n. 
\end{equation}
\end{thm}

\ms 
Of course we shall use this when the two $v_j$ are phases of our blow-up
limit $\u_\infty$; they are harmonic because they minimize $\int |\nabla v_i|^2$
locally; see the proof of \eqref{e9.6} with $f_i=g_i=0$.

\begin{proof}
The fact that $\Phi_1(r)\Phi_2(r)$ is nondecreasing when $v_1$ and $v_2$ 
are nonnegative Lipschitz functions such that $v_1v_2 = 0$ and $v_i$ 
is harmonic on $\O_i$ was proved in \cite{ACF}, and we shall follow their
proof (Lemma 5.1 in \cite{ACF}), see where equality occurs in the argument, 
and try to conclude from there.
So we need to recall how the proof of \cite{ACF} goes.

We start with some notation. Set $\rho(x) = |x|$ for $x\in \R^n$, and 
\begin{equation} \label{e16.24}
A_j(r) = r^2 \Phi_j(r) = \int_{B(0,r)} {|\nabla v_j(x)|^2 \over |x|^{n-2}} \, dx
= \int_{B(0,r)} \rho^{2-n} |\nabla v_j|^2\,dx
\end{equation}
for $j=1,2$. Set $B_r = B(0,r)$ and $S_r = \d B(0,r)$,
and denote by $\sigma$ the surface measure on $S_r$.

Notice that $S_r \cap \O_j$ is not empty, because 
otherwise $v_j = 0$ on $S_r$, hence also on $B_r$ (by the maximum principle), 
and then $\Phi_1(r)\Phi_2(r) = 0$ for $r$ small; we excluded this
in \eqref{e16.21}. Now let $r^2\alpha_j(r)\in (0,+\infty]$ 
be the square root of the Sobolev constant on $S_r \cap \O_j$, i.e., 
the smallest constant such that
\begin{equation} \label{e16.25}
\int_{S_r \cap \O_j} |u|^2 d\sigma 
\leq r^2\alpha_j(r) \int_{S_r \cap \O_j} |\nabla_t u|^2 d\sigma
\end{equation}
for a function $u \in W^{1,2}(S_r)$ with compact support in $S_r \cap \O_j$,
and where $\nabla_t$ denotes the gradient on the sphere.
This is the same thing as (5.4) in \cite{ACF}, except that we intend to keep the dependence on $r$ apparent, and we normalize $\alpha_j(r)$
so that it is dimensionless. Notice that $\alpha_j(r) < +\infty$
because the other domain meets $S_r$, so there is a nontrivial
Sobolev inequality.
Naturally we want to apply \eqref{e16.25} to the restriction of $v_j$ to $S_r$, 
which in our case even lies in $W^{1,2}(S_r)$ for every $r$ because 
it is Lipschitz. Of course it is not compactly supported in $S_r \cap \O_j$,
but this is easy to fix, because  $v_j$ is easy to approximate by such functions 
(for instance, try $u_\varepsilon = [v_j - \varepsilon]_+$),
so we get \eqref{e16.25} for $v_j$ as well. 

The proof will use the fact that
\begin{eqnarray} \label{e16.26}
2\int_{B_r} \rho^{2-n} |\nabla v_j|^2 
&\leq& 2 r^{2-n} \int_{S_r} v_j {\d v_j \over d\rho}
+ (n-2) r^{1-n} \int_{S_r} v_j^2
\nonumber\\
&=&  r^{3-n} \int_{S_r} \Big[2 {v_j\over r} {\d v_j \over d\rho}
+ (n-2) {v_j^2 \over r^2} \Big]
\end{eqnarray}
(see (5.2) in \cite{ACF}). The reader may be surprised that 
this is only an inequality, but this is because its proof uses the 
fact that $\Delta v_i \geq 0$ (as a distribution). 
Also, we only get it for almost every $r$ because some limit of integrals on
thin annuli near $S_r$ is taken. 

So we consider $\Phi_1(r)\Phi_2(r) = r^{-4} A_1(r) A_2(r)$ and differentiate it. 
Here we do not even care that $\Phi_1\Phi_2$ is the integral of its derivative;
we know that the derivative is zero, and we just need to compute it at almost
every $r$. We get that
\begin{equation} \label{e16.27}
0 = -4 r^{-5} A_1(r) A_2(r) + r^{-4} A'_1(r) A_2(r) + r^{-4} A_1(r) A'_2(r)
\end{equation}
with,
\begin{equation} \label{e16.28}
A'_j(r) = \int_{S_r} \rho^{2-n} |\nabla v_j|^2
= r^{2-n}\int_{S_r} |\nabla v_j|^2
\end{equation}
by \eqref{e4.3}. We divide by 
$r^{-5} A_1(r) A_2(r) = r^{-1} L \neq 0$ and get that
\begin{equation} \label{e16.28}
{r A'_1(r) \over A_1(r)} + {r A'_2(r) \over A_2(r)} = 4.
\end{equation}
Then we evaluate each ${r A'_j(r) \over A_j(r)}$.
We choose numbers $\beta_j(r) \in (0,1)$ such that
\begin{equation} \label{e16.30}
{1-\beta_j(r)^2 \over \alpha_j(r)} = (n-2) {\beta_j(r) \over \sqrt{\alpha_j(r)}}
\end{equation}
(see (5.4) in \cite{ACF}; there is only one solution $\beta_j(r) \in (0,1)$, 
which is even computed seven lines later in \cite{ACF}, but this is not the point yet).
We write that
\begin{equation} \label{e16.31}
 {2\beta_j(r) \over \sqrt{\alpha_j(r)}} \int_{S_r} {v_j\over r} {\d v_j \over d\rho}
 \leq {2\beta_j(r) \over \sqrt{\alpha_j(r)}} 
 \Big\{ \int_{S_r} {v_j^2 \over r^2} \Big\}^{1/2}
 \Big\{ \int_{S_r} \big|{\d v_j \over d\rho}\big|^2 \Big\}^{1/2}
\end{equation}
by Cauchy-Schwarz, then use \eqref{e16.25} to get that
\begin{equation} \label{e16.32}
{\beta_j(r) \over \sqrt{\alpha_j(r)}} 
\Big\{ \int_{S_r} {v_j^2 \over r^2}\Big\}^{1/2}
 \leq \Big\{ \int_{S_r} \beta_j(r)^2 |\nabla_t v_j|^2 \Big\}^{1/2},
\end{equation}
and then use the fact that $2AB = A^2 + B^2$, with
\begin{equation} \label{e16.33}
A = \Big\{ \int_{S_r} \big|{\d v_j \over d\rho}\big|^2 \Big\}^{1/2}
\ \text{ and } \ 
B = \Big\{ \int_{S_r} \beta_j(r)^2 |\nabla_t v_j|^2 \Big\}^{1/2},
\end{equation}
to get that
\begin{equation} \label{e16.34}
 {2\beta_j(r) \over \sqrt{\alpha_j(r)}} \int_{S_r} {v_j \over r} {\d v_j \over d\rho}
 \leq 2AB \leq A^2 + B^2
 = \int_{S_r} \big|{\d v_j \over d\rho}\big|^2 + \beta_j(r)^2 |\nabla_t v_j|^2.
\end{equation}
We use \eqref{e16.25} a second time, to say that
\begin{equation} \label{e16.35}
{1-\beta_j(r)^2 \over \alpha_j(r)} \int_{S_r} {v_j^2 \over r^2}
\leq \int_{S_r} (1-\beta_j(r)^2) |\nabla_t v_j|^2,
\end{equation}
and then we add this to \eqref{e16.34} to get that
\begin{equation} \label{e16.36}
{2\beta_j(r) \over \sqrt{\alpha_j(r)}} \int_{S_r} {v_j\over r} {\d v_j \over d\rho}
+{1-\beta_j(r)^2 \over \alpha_j(r)} \int_{S_r}{v_j^2 \over r^2}
\leq \int_{S_r} |\nabla v_j|^2
= r^{n-2} A'_j(r)
\end{equation}
by \eqref{e16.28}.
Because of \eqref{e16.30},  the left-hand side is equal to
\begin{equation} \label{e16.37}
{\beta_j(r) \over \sqrt{\alpha_j(r)}} \int_{S_r} 
\Big[ 2{v_j\over r} {\d v_j \over d\rho} + (n-2) {v_j^2 \over r^2}\Big]
\geq 2 {\beta_j(r) \over \sqrt{\alpha_j(r)}} \, r^{n-3} A_j(r) ,
\end{equation}
where the last part comes from\eqref{e16.26}. Thus
\begin{equation} \label{e16.38}
2 {\beta_j(r) \over \sqrt{\alpha_j(r)}} A_j(r) \leq  r A'_j(r).
\end{equation}
Hence
\begin{equation} \label{e16.39}
{r A'_1(r) \over A_1(r)} + {r A'_2(r) \over A_2(r)}
\geq 2 {\beta_1(r) \over \sqrt{\alpha_j(r)}}
+2 {\beta_2(r) \over \sqrt{\alpha_j(r)}}.
\end{equation}
Then we compute the constants, and find out that
\begin{equation} \label{e16.40}
{\beta_1(r) \over \sqrt{\alpha_j(r)}} = \gamma_j(r),
\end{equation}
where the numbers $\gamma_j(r)$ are defined by
\begin{equation} \label{e16.41}
\gamma_j(r) (\gamma_j(r) + n - 2) = {1 \over \alpha_j(r)}\, ;
\end{equation}
see (5.8) and (5.9) in \cite{ACF}, and observe again that these
numbers are dimensionless. Now the situation is that for
all choices of disjoint domains $\O_1$ and $\O_2$ that intersect the
sphere, the numbers $\gamma_j(r)$ are such that
\begin{equation} \label{e16.42}
\gamma_1(r) + \gamma_2(r) \geq 2,
\end{equation}
from which \cite{ACF} deduces that 
${rA'_1(r) \over rA_1(r)} + {A'_2(r) \over A_2(r)} \geq 4$
and (returning to variable $r$ and integrating) that $\Phi_1\Phi_2$
is nondecreasing. This is roughly how one proceeds in \cite{ACF}.

In the situation of Theorem \ref{t16.3}, we know that
${A'_1(r) \over A_1(r)} + {A'_2(r) \over A_2(r)} = 4$
(by \eqref{e16.40}), so all the inequalities above are in fact
identities (the quantities in play are all positive), and we 
need to derive information from that.

The fastest route would use an unpublished paper
of W. Beckner, C. Kenig, and J. Pipher \cite{BKP} which says that
when $\gamma_1(r) + \gamma_2(r) = 2$, the two domains
$\O_j \cap S_t$ are complementary hemispheres. This helps greatly,
but let us see what we can get easily.

Because we have equality in \eqref{e16.31} (Cauchy-Schwarz),
we get that the two functions ${v_j \over r}$ and ${\d v_j \over \d \rho}$
are proportional. That is, there is a constant $c_j(r)$ such that
\begin{equation} \label{e16.43}
{\d v_j \over \d \rho} = c_j(r){v_j \over r}
\ \text{ on } \O_j \cap S_r.
\end{equation}
We don't need to say almost everywhere, because both functions are smooth
on $\O_j$, and we know that $0 <  c_j(r) < +\infty$
because otherwise the left-hand side of \eqref{e16.31}
would not be positive like the right-hand side.
We can even compute $c_j(r)$; indeed \eqref{e16.32} is an
equality, and $A=B$ in \eqref{e16.33}
(because \eqref{e16.34} is an equality), so
\begin{equation} \label{e16.44}
{\beta_j(r) \over \sqrt{\alpha_j(r)}} 
\Big\{ \int_{S_r} {v_j^2 \over r^2}\Big\}^{1/2}
= B = A = \Big\{ \int_{S_r} \big|{\d v_j \over d\rho}\big|^2 \Big\}^{1/2}
\end{equation}
and hence $c_j(r) = {\beta_j(r) \over \sqrt{\alpha_j(r)}}= \gamma_j(r)$
(by \eqref{e16.40}).
Notice that \eqref{e16.43} holds for almost every $r > 0$;
let us use this to check that
\begin{equation} \label{e16.45}
\O_j \text{ is a cone.}
\end{equation}
Pick $x \in \O_j$, write $x = r_0 \xi$ for some $\xi \in S_1$,
and consider $h(r) = \log(v_j(r \xi))$; we know that $h$ is defined
and locally Lipschitz as long as $r\xi \in \O_j$, and 
\eqref{e16.43} yields $h'(r) = r^{-1} c_j(r) = r^{-1} \gamma_j(r)$ 
almost everywhere. Since $\gamma_j(r) \leq 2$, we get that
$h(r) \geq h(r_0)| - 2 |\log(r/r_0)|$ as long as $r\xi \in \O_j$.
This gives a lower bound for $v_j(r\xi)$ and proves that in fact 
$r\xi \in \O_j$ for all $r > 0$; \eqref{e16.45} follows.

By \eqref{e16.45}, the numbers $\alpha_j(r)$, and then
$\beta_j(r)$ and $\gamma_j(r)$, do not depend on $r$.
Then we can solve the differential equation \eqref{e16.43},
and we get that $v_j(tx) = t^{\gamma_j} v_j(x)$
for $x\in \O_j$ and $t > 0$. Since $v_j$ is Lipschitz near the origin
and $\nabla_t v_j \neq 0$ somewhere,
$\gamma_j \leq 1$. Since $\gamma_1+\gamma_2 = 2$,
$\gamma_1 = \gamma_2 = 1$.

Now we really need some information on Sobolev constants, and we
shall use results from \cite{BZ} (that are in fact a little posterior to \cite{ACF}).
Fix a pole in $z_0 \in S_1$, and denote by 
$\Gamma_j$ the spherical cap centered at $z_0$
(meaning, the intersection $S_1 \cap B$, where $B$ is a ball
centered at $z_0$) such that $\sigma(\Gamma_j) =\sigma(\O_j \cap S_1)$.
Denote by $\alpha_j^\ast$ and $\gamma_j^\ast$ the constants
associated to $\Gamma_j$ as above.
Then Theorem 5.1 on page 175  of \cite{BZ}, with $A(t) = t^2$, 
says the following things.

First, if $u \in W^{1,2}(S_1)$ and $u^\ast$ denotes its 
symmetric rearrangement, then $u^\ast \in W^{1,2}(S_1)$ and 
\begin{equation} \label{e16.46}
\int_{S_r} |\nabla_t u^\ast|^2 \leq \int_{S_r} |\nabla_t u|^2.
\end{equation}
Notice that if $u$ is compactly supported in $\O_j \cap S_1$,
then $u^\ast$ is compactly supported in $\Gamma_j$;
since $\int |u^\ast|^2 = \int |u|^2$, the definition \eqref{e16.25} 
says that $\alpha_j^\ast \leq \alpha_j$,
and then $\gamma_j^\ast \geq \gamma_j = 1$
by \eqref{e16.41}.

It is easy to check that $\alpha_j^\ast$ is a (strictly)
decreasing function of (the volume of) $\Gamma_j$, hence
$\gamma_j^\ast$ is increasing. 
Also, $\gamma_j^\ast = 1$ when $\Gamma_j$
is a hemisphere (see later), hence $\sigma(\Gamma_j) \geq \sigma(S_1)/2$.
But $\sigma(\Gamma_j) =\sigma(\O_j \cap S_1)$
and the $\O_j$ are disjoint, so 
$\sigma(\Gamma_j) =\sigma(\O_j \cap S_1)=1$
for both $j$, $\gamma_j = \gamma^\ast_j=1$, and
$\alpha_j = \alpha_j^\ast$.

Notice that $\int_{S_1} |v_j|^2 = \alpha_j \int_{S_1} |\nabla_t v_j|^2$
because \eqref{e16.35} is an equality for almost
every $r$ and $v_1$ is homogeneous of degree $1$, hence
by \eqref{e16.46} and \eqref{e16.25} for $\Gamma_j$,
\begin{equation} \label{e16.47}
\int_{S_r} |\nabla_t u^\ast|^2 \leq \int_{S_r} |\nabla_t u|^2
= \alpha_j^{-1}\int_{S_1} |v_j|^2 = \alpha_j^{-1}\int_{S_1} |v_j^\ast|^2
\leq  \alpha_j^{-1}\alpha_j^\ast \int_{S_r} |\nabla_t u^\ast|^2;
\end{equation}
we know that $\alpha_j = \alpha_j^\ast$, so \eqref{e16.25}
is an equality for $v_j$. So $v_j^\ast$ is a Sobolev minimizer
in a half sphere. It is well known that then $v_j^\ast$ is a first
eigenfunction for the Laplacian on the sphere, and (because we can
use a symmetry argument to reduce to the sphere) that it is
the restriction to $S_1$ of an affine function. This is also how one
computes that $\gamma_j^\ast = 1$.

Return to Theorem 5.1 of \cite{BZ}; its most important part is
that its says that since \eqref{e16.46} is an equality for $v_j$, 
$v_j$ is of the form $v_j^\ast \circ R$, where $R$ is a rotation. 
There are just two assumptions to check: 
first, that $A(t)$ is increasing (this is trivial because $A(t)=t^2$), 
and also that
\begin{equation} \label{e16.48}
\big|\big\{ x\in S_1 \, ; \, v_1^\ast(x) > 0 \text{ and } 
\nabla_t v_j^\ast = 0 \big\}\big| = 0,
\end{equation}
which holds because $v_j^\ast$ comes from an affine function.
So Brothers and Ziemer's theorem applies and says that both $v_j$
are also equal to affine functions on the sphere;
we easily deduce the representation formula
\eqref{e16.23} for the $v_j$ from this, because they
are homogeneous of degree~$1$, and Theorem \ref{t16.3} follows.
\qed
\end{proof}

\ms
Let us now return to our minimizer $(\u,\W)$ and its regular blow-up
limits. As was said above, if $\u_\infty$ is as in the beginning
of the section and we can find different indices $\varphi_1, \varphi_2 \in I$
such that $L(\varphi_1,\varphi_2) > 0$ in \eqref{e16.20},
the corresponding functions $v_j = v_{\varphi_j,\infty}$ satisfy the 
hypotheses of Theorem~\ref{t16.3} (because they are harmonic,
see \eqref{e9.6}), and so they are described by \eqref{e16.23}.

Because of this, all the other functions $v_{\varphi,\infty}$,
$\varphi \neq \varphi_1, \varphi_2$ are null because, by definition of
the class $\F(\R^n,\Omega_\infty)$, they vanish almost everywhere on 
the sets where $v_{\varphi_j,\infty} > 0$.

But Corollary \ref{t15.3} also says that $(\u_\infty,\W_\infty)$ is a 
minimizer for the functional $J_\infty$, which means that
\eqref{e15.40} holds for all competitors of $(\u_\infty,\W_\infty)$
in a ball $B(0,R)$. We could simplify this and write it in terms of
the single function $v = v_1-v_2$ and usual Alt, Caffarelli, and
Friedman minimizers in $\R^n$, but let us not do this for the moment.
Write $\varphi_1 = (i_1,\varepsilon_1)$ and $\varphi_2 = (i_2,\varepsilon_2)$.
We claim that because of this additional minimality property (where we may 
move the free boundary),
\begin{equation} \label{e16.49}
a_1^2 - a_2^2 = \lambda_{i_1} - \lambda_{i_2}.
\end{equation}  
This can be proved easily, with a computations of first variation where
we modify $v_1-v_2$. The proof is classical, but we shall do it (later)
for the convenience of the reader. See near \eqref{e20.10}
We can also check whether making $v_1 = v_2 = 0$ near the separating
hyperplane would make things better, and the first variation computation for this 
gives the additional constraints that
\begin{equation} \label{e16.50}
a_1^2 \geq \lambda_{i_1} -\min(0, \lambda_1, \ldots, \lambda_N)
\ \text{ and } \ 
a_2^2 \geq \lambda_{i_2} -\min(0, \lambda_2, \ldots, \lambda_N).
\end{equation} 
See near \eqref{e20.12}, and notice that most often all the
$\lambda_i$ are nonnegative, so we get the simpler relations
$a_j^2 \geq \lambda_{i_j}$.

There is another relation, which is more a matter of definitions:
\begin{eqnarray} \label{e16.51}
L(\varphi_1,\varphi_2) &=& \Phi_{\varphi_1,\varphi_2,\infty}(1)
= \Phi_{\varphi_2,\infty}(1)\Phi_{\varphi_1,\infty}(1)
\nonumber\\
&=& \Big\{\int_{B(0,1)} {|\nabla v_{\varphi_1,\infty}(x)|^2 \over |x|^{n-2}} \Big\}
\Big\{\int_{B(0,1)} {|\nabla v_{\varphi_2,\infty}(x)|^2 \over |x|^{n-2}} \Big\}
\nonumber\\
&=& {a_1 a_2 \over 4} 
\big\{\int_{B(0,1)}  |x|^{2-n} \big\}^2
= {a_1 a_2 \over 16} \sigma(S_1)^2 =  {n^2 a_1 a_2 \over 4} |B(0,1)|^2
\end{eqnarray}  
by \eqref{e16.20}, \eqref{e16.18}, \eqref{e16.16}, and \eqref{e16.23}.
So we can always compute $a_1$ and $a_2$, if we know $L(\varphi_1,\varphi_2)$
and the $\lambda_i$.

\ms
Let us gather a few easy consequences of this discussion on
the minimizer $(\u,\W)$ itself.

In the next statement, a blow-up limit of $\u$ is any function 
$\u_\infty$ such that $\u_\infty(x) = \lim_{k \to +\infty} r_k^{-1}\u(r_kx)$
for $x\in \R^n$, for some sequence $\{ r_k \}$ that tends to $0$.
Since $\u$ is assumed to be Lipschitz near $0$, such limits exist and the
convergence is uniform on compact sets of $\R^n$. We won't need to
say that $\u_\infty$ is a regular blow-up limit, because this will follow
from the assumptions of Lemma~\ref{t16.1}.
  
\ms
\begin{cor}\label{t16.4}
Let $(\u,\W)$ be a minimizer for $J$.
Suppose that $\u$ satisfies \eqref{e15.3} and \eqref{e15.5}
that the $f_i$ and $g_i$ satisfy \eqref{e15.4},
that the domain $\Omega$ satisfies the two assumptions of Lemma~\ref{t16.1},
and that $F$ satisfies \eqref{e15.11}. Suppose in addition that
there are phases $\varphi_1 = (i_1,\varepsilon_1) \in I$ and 
$\varphi_2 = (i_2,\varepsilon_2) \in I \sm \{ \varphi_1 \}$ 
such that
\begin{equation} \label{e16.52}
L(\varphi_1,\varphi_2) 
= \lim_{\rho \to 0} \Phi^{0}_{\varphi_1}(\rho)\Phi^{0}_{\varphi_2}(\rho)
= \lim_{\rho \to 0} \Big\{{1 \over \rho^4} 
\int_{B(0,\rho)} {|\nabla v_{\varphi_1}|^2 \over |x|^{n-2}}
\int_{B(0,\rho)} {|\nabla v_{\varphi_2}|^2 \over |x|^{n-2}} \Big\}
\neq 0
\end{equation}
(see \eqref{e16.19} and \eqref{e16.11}). Then 

\noindent (i) For each blow-up limit $\u_\infty$ of $\u$ at $0$,
the functions $v_j = v_{\varphi_j,\infty}= [\varepsilon_j u_{\varphi_j,\infty}]_+$ 
take the form given by \eqref{e16.23}, with constants $a_j > 0$ such
that \eqref{e16.49}, \eqref{e16.50}, and \eqref{e16.51} hold, and all the other
components $v_{\varphi,\infty}$, $\varphi \in I \sm \{ \varphi_1, \varphi_2 \}$,
of $\u_\infty$ are null.

\noindent (ii) For each $\varphi \in I \sm \{ \varphi_1, \varphi_2 \}$,
\begin{equation} \label{e16.53}
\lim_{\rho \to 0} \Phi^{0}_{\varphi}(\rho) := 
\lim_{\rho \to 0} {1 \over \rho^2} \int_{B(0,\rho)} 
{|\nabla v_{\varphi}|^2 \over |x|^{n-2}} = 0.
\end{equation}

\noindent (iii) the origin is an interior point of $\Omega$.
\end{cor}

\ms
In (i), the coefficients $a_j$ in the description \eqref{e16.23} 
do not depend on the blow-up sequence $\{ r_k \}$ because they
can be computed in terms of $L(\varphi_1,\varphi_2)$ and the
$\lambda_i$, but in principle the unit vector $e$ does. 
Unless we prove a better regularity result for $\u$ near $0$.

Also, (iii) only says that $0 \in {\rm int}(\Omega)$ under our assumptions,
which includes some weak regularity assumptions on $\Omega$. 
Probably \eqref{e16.52} can happen when $\Omega$ has an inward 
cusp at $0$, with a free boundary that continues in $\Omega$ in
the direction of the cusp.

\begin{proof}
We included the assumptions of Lemma \ref{t16.1} to make sure that
when $\u_\infty$ is a blow-up limit associated to any sequence $\{ r_k \}$
that tends to $0$, the assumptions of the beginning of this section 
(we were missing the conditions on $\Omega_\infty$) are satisfied for
some subsequence. Then (i) is a consequence of the discussion above.

Suppose that \eqref{e16.53} fails for some 
$\varphi \in I \sm \{ \varphi_1, \varphi_2 \}$, and choose 
a sequence $\{ r_k \}$ such that 
\begin{equation} \label{e16.54}
\lim_{k \to +\infty} \Phi^{0}_{\varphi}(r_k)
= \limsup_{k \to \infty}\Phi^{0}_{\varphi}(\rho) > 0.
\end{equation}
Then use Lemma \ref{t16.1} to replace $\{ r_k \}$ with  a subsequence
that satisfies the assumptions at the beginning of this section.
Of course $\u_\infty$ stays the same for this subsequence, and
(i) gives two nontrivial components $v_{\varphi_j,\infty}$ of $\u_\infty$.
Notice that
\begin{eqnarray} \label{e16.55}
\liminf_{k \to +\infty} \Phi^{0}_{\varphi_1}(r_k)
&\geq& {\liminf_{k \to \infty}
\big[\Phi^{0}_{\varphi_1}(r_k)\Phi^{0}_{\varphi_2}(r_k)\big]
\over
\limsup_{k \to \infty} \Phi^{0}_{\varphi_2}(r_k)}
\nonumber \\
&=& {L(\varphi_1,\varphi_2) \over
\limsup_{k \to \infty} \Phi^{0}_{\varphi_2}(r_k)}
\geq C^{-1} L(\varphi_1,\varphi_2) > 0
\end{eqnarray}
by \eqref{e16.52} and \eqref{e16.12}, so
\begin{eqnarray} \label{e16.56}
L(\varphi,\varphi_1) 
&=& \lim_{\rho \to 0} \Phi^{0}_{\varphi}(\rho)\Phi^{0}_{\varphi_1}(\rho)
= \lim_{k \to +\infty} \Phi^{0}_{\varphi}(r_k)\Phi^{0}_{\varphi_1}(r_k)
\nonumber\\
&\geq& C^{-1} L(\varphi_1,\varphi_2) 
\limsup_{k \to \infty}\Phi^{0}_{\varphi}(\rho) > 0
\end{eqnarray}
by \eqref{e16.19}, \eqref{e16.54}, and \eqref{e16.55}. 
Then the component $v_{\varphi,\infty}$ of $\u_\infty$
is also given by a formula like \eqref{e16.23}, which is impossible
because the other two don't leave any room for its support.
This proves (ii).

Finally suppose that $0$ is an interior point of $\d\Omega$.
Then Lemma \ref{t16.1} gives a limiting domain $\Omega_\infty$
such that in particular $|B(x,r) \sm \Omega_\infty| > 0$
for every ball $B(x,r)$ centered on $\d\Omega_\infty$;
see \eqref{e16.5}. But $\u_\infty = 0$ almost everywhere on 
$\R^n \sm \Omega_\infty$, by definition of $\F(\R^n,\Omega_\infty)$
(see the lines above Lemma \ref{t16.1}, and the first lines of Section \ref{limits}.
This contradicts that description \eqref{e16.23} of the $v_{\varphi_j,\infty}$,
proves (iii), and completes the proof of Corollary \ref{t16.4}.
\qed
\end{proof}

\ms
We can find sufficient conditions for the condition \eqref{e16.52}
to hold. The main one will use good domains and the nondegeneracy condition
of Section \ref{good}, and we shall discuss it in Section \ref{good2}. 
Let us just say now that if
\begin{equation} \label{e16.57}
\liminf_{\rho \to 0} \Phi^{0}_{\varphi_j}(\rho) > 0
\end{equation}
for $j=1, 2$ (and $\u$ is Lipschitz near $0$), then \eqref{e16.52}
holds because we already knew from \eqref{e16.19} that the limit
of the product exits. In general, if we merely assume that
\begin{equation} \label{e16.58}
\limsup_{\rho \to 0} \Phi^{0}_{\varphi_j}(\rho) > 0,
\end{equation}
this should not be enough to conclude. 
A priori it can happen that $v_{\varphi_1}$ is dormant at some scales 
(i.e., $\nabla v_{\varphi_1}$ is very small), and revives at smaller scales, 
but never at the same time as for the other phase $\varphi_2$; then $L(\varphi_1,\varphi_2)$ may be null, and we won't find a blow-up 
sequence that would show both phases.

\ms
The following proposition can be seen as a weak regularity result for the
free boundaries associated to phases that satisfy the conditions of
Corollary \ref{t16.4}. Set
\begin{equation} \label{e16.59}
\Omega_\varphi = \big\{ x\in \R^n \, ; \, v_\varphi(x) > 0 \big\}
= \big\{ x\in \R^n \, ; \, \varepsilon u_i(x) > 0 \big\}
\end{equation}
for $\varphi = (i,\varepsilon) \in I$. Then let $\varphi_1$ and $\varphi_2 \in I$
be as in Corollary \ref{t16.4}; we want to measure the flatness of 
$\d\Omega_{\varphi_1} \cup \d\Omega_{\varphi_2}$ in small
balls $B(0,r)$.

\begin{pro} \label{t16.5}
Let $(\u,\W)$ and the pairs $\varphi_1$ and $\varphi_2 \in I$
satisfy the hypotheses of Corollary \ref{t16.4}. Then there exists 
numbers $\beta(r) \in [0,1]$ such that
\begin{equation} \label{e16.60}
\lim_{r\to 0} \beta(r) = 0
\end{equation}
and, for $r > 0$, a unit vector $e = e(r)$ such that
\begin{equation} \label{e16.61}
v_{\varphi_1}(x) > 0 \ \text{ for $x\in B(0,r)$ such that } 
\langle x, e(r) \rangle \geq \beta(r) r,
\end{equation}
\begin{equation} \label{e16.62}
v_{\varphi_2}(x) > 0 \ \text{ for $x\in B(0,r)$ such that } 
\langle x, e(r) \rangle \leq -\beta(r) r,
\end{equation}
\begin{equation} \label{e16.63}
(\d\Omega_{\varphi_1} \cup \d\Omega_{\varphi_2})\cap B(0,r)
\i \big\{ x\in B(x,r) \, ; \, |\langle x, e(r) \rangle| \leq \beta(r) r \big\},
\end{equation}
and
\begin{equation} \label{e16.64}
\Omega_{\varphi} \cap B(0,r) \i \big\{ x\in B(x,r) \, ; \, 
|\langle x, e(r) \rangle| \leq \beta(r) r \big\}
\ \text{ for } \varphi \in I \sm \{ \varphi_1, \varphi_1 \}.
\end{equation}
\end{pro}

\ms
Thus all the action is on the thin band where $|\langle x, e(r) \rangle| \leq \beta(r) r$.
There may be a lot of action though, because if volume for the other $W_i$ is
cheap, lots of them may find it convenient to squat some of that band.
Even when we assume that all the regions are good, there may be a part of the
band that lies on $\Omega \sm \cup_{i} W_i$.

As for the blow-up limits, we do not know whether the direction $e(r)$ really depends
on $r$, but by lack of a better regularity theorem (which may be hard
with the present degree of generality), we have to assume that it does.

\begin{proof}
Let us first observe that 
$\Omega_{\varphi_1} \cap \Omega_{\varphi_2}= \emptyset$,
by definition if $\F$ if $i_1 \neq i_2$ and by definition of $v_{\varphi_i}$
otherwise. So \eqref{e16.63} is an 
immediate consequence of \eqref{e16.61} and \eqref{e16.62}.
Similarly, $\Omega_{\varphi}$ does not meet the $\Omega_{\varphi_j}$
and \eqref{e16.64} follows from \eqref{e16.61} and \eqref{e16.62}.

Notice also that \eqref{e16.61} and \eqref{e16.62} are trivial when
$\beta(r) = 1$, so the point is to show that we may make it tend to $0$.

We shall prove the proposition by contradiction and compactness.
Suppose that we cannot find the $\beta(r)$ and $e(r)$ as above;
then we can find a positive number $\alpha > 0$, and a sequence
$\{ r_k \}$ that tends to $0$, with the following property. 
For each $k$, there is no choice of a unit vector $e$ such that
\eqref{e16.61} and \eqref{e16.62} hold with $r=r_k$ and 
$\beta(r) = \alpha$.

Extract from $\{ r_k \}$ a subsequence for which the assumptions
of Corollary \ref{t15.3} are satisfied; we have seen at the beginning of
the proof of Corollary \ref{t16.4} that Lemma \ref{t16.1} allow us to do this.
Then we get a blow-up limit $\u_\infty$, associated to our subsequence,
which admits the description of \eqref{e16.23}. This gives a unit
vector $e$, which we may try in \eqref{e16.61} and \eqref{e16.62}.
By definition of $\{ r_k \}$, this does not work well, so we get a
point $x_k \in B(0,r_k)$ such that 
\begin{equation} \label{e16.65}
\langle x_k, e(r) \rangle \geq \alpha r_k \ \text{ but }
v_{\varphi_1}(x_k) = 0,
\end{equation}
or (a point $x_k \in B(0,r)$ such that)
\begin{equation} \label{e16.66}
\langle x_k, e(r) \rangle \leq - \alpha r_k \ \text{ but }
v_{\varphi_2}(x_k) = 0.
\end{equation}
We extract a new subsequence so that we get \eqref{e16.65}
for all $k$, or we get \eqref{e16.65} for all $k$, and in addition
$y_k = r_k^{-1} x_k$ has a limit $y_\infty \in \overline B(0,1)$.

Suppose for instance that \eqref{e16.65} holds for all $k$.
Then $v_{\varphi_1,k}(y_k) = r_k^{-1} v_{\varphi_1}(r_k y_k) 
r_k^{-1} v_{\varphi_1}(x_k) = 0$ by \eqref{e16.9}, and
\begin{equation} \label{e16.67}
|v_{\varphi_1,\infty}(y_\infty)| 
\leq |v_{\varphi_1,\infty}(y_\infty)-v_{\varphi_1,k}(y_\infty)|
+ |v_{\varphi_1,k}(y_\infty)-v_{\varphi_1,k}(y_k)|.
\end{equation}
The first term tends to $0$ because $v_{\varphi_1,\infty}$
is the limit of the $v_{\varphi_1,k}$ in $\R^n$, and the
second term because the $v_{\varphi_1,k}$ are uniformly Lipschitz
and $y_k$ tends to $y_\infty$. So $v_{\varphi_1,\infty}(y_\infty) = 0$.
But $\langle y_k, e(r) \rangle = r_k^{-1} \langle x_k, e(r) \rangle
\geq \alpha$ by \eqref{e16.65}, so $\langle y_\infty, e(r) \rangle \geq \alpha$,
which contradicts \eqref{e16.23} (recall that we set $v_j = v_{\varphi_j,\infty}$
there). The case when \eqref{e16.66} holds for all $k$ would be treated
the same way, and Proposition \ref{t16.5} follows from the contradiction.
\qed
\end{proof}

\section{Blow-up limits with one phase}   \label{1phase}

In this section we keep the same general assumptions as in Section \ref{2phases} 
and study the blow-up limits of $\u$
when all the limits $L(\varphi_1,\varphi_2)$ of \eqref{e16.19}
are null, but, say 
\begin{equation} \label{e18.1}
\limsup_{r \to 0} \Phi_{1,1}^0(r) 
:= \limsup_{r \to 0}{1 \over r^2} 
\int_{B(0,r)} {|\nabla u_{1,+}|^2 \over |x|^{n-2}} \, dx > 0.
\end{equation}
In fact, we shall rapidly assume that $F$ is Lipschitz and $i=1$ is a good index, 
i.e., that
\begin{equation} \label{e18.2}
\text{\eqref{e10.2} holds and there exists $\lambda > 0$ and $\varepsilon > 0$
such that \eqref{e13.1} holds.}
\end{equation}
The two go together because we want to apply Proposition \ref{t17.1},
for instance. At this stage, this is not so much to assume, because 
we shall see that if \eqref{e18.1} holds and $0$ is an interior point
of $\Omega$, then
\begin{equation} \label{e18.3}
\lambda_1 > \min(0, \lambda_2, \ldots, \lambda_N),
\end{equation}
where the $\lambda_i$ are as in our assumption \eqref{e15.11}.
In this case, \eqref{e18.2} just amounts to requiring a little bit more regularity 
on $F$ than we do in \eqref{e15.11}.

Then, if $0$ is an interior point of $\Omega$, or at least $\Omega$
looks enough like cones near $0$,  we shall be able to show that some 
blow-up limits $\u_\infty$ of $\u_1$ at the origin are nontrivial, 
one-phase minimizers of the standard Alt, Caffarelli, and Friedman functional 
(in $\R^n$ or in a cone), that are also homogeneous of degree $1$. 
See Theorem \ref{t18.2} below.
When $n \leq 3$ and the cone is $\R^n$, 
it was proved in \cite{CJK2} that the first coordinate of such functions 
$u_{\infty}$ is of the form $u(x) = a(x)_+ = \max(a(x),0)$, 
where $a$ is affine, so this will give a good description of the corresponding 
blow-up limits of $\u$ at $0$ when $n \leq 3$ and $0$ is an interior point of $\Omega$.
See Corollary \ref{t18.3}.

The main ingredient for this section is a functional introduced by
G. S. Weiss \cite{We}, its monotonicity properties, and what happens
when it is constant.

We shall try to add assumptions as they are used.
The initial assumptions for this section are, as for most of Section \ref{1phase},
that
\begin{equation} \label{e18.4}
\text{$\u$ satisfies \eqref{e15.3} and \eqref{e15.5}, the $f_i$ and $g_i$ satisfy \eqref{e15.4}, $F$ satisfies \eqref{e15.11},}
\end{equation}
\begin{equation} \label{e18.5}
\text{the domain $\Omega$ satisfies the two assumptions of Lemma~\ref{t16.1},}
\end{equation}
but for the main result we shall assume that $0$ is an interior 
point of $\Omega$, in which case the issue does not arise.
We also assume that
\begin{equation} \label{e18.6}
L(\varphi_1,\varphi_2) := \lim_{\rho \to 0} 
\Phi^{0}_{\varphi_1}(\rho)\Phi^{0}_{\varphi_2}(\rho) = 0 
\ \text{ for $\varphi_1 \in I$ and $\varphi_1 \in I\sm \{ \varphi_1 \}$}
\end{equation}
(see the definitions \eqref{e16.19} and \eqref{e16.11}),
and \eqref{e18.1} or \eqref{e18.2}.

\ms
Let us define the Weiss functional associated to $u_{1,+}$.
Set $v = u_{1,+} = \max(0,u_1)$ and 
$\Omega_1 = \big\{ x\in \R^n \, ; \, u_1(x) > 0 \big\}$ to save notation,
and then define $\Psi$ by 
\begin{eqnarray} \label{e18.7}
\Psi(r) &=& r^{-n} \int_{B(0,r)} |\nabla v|^2 
+r^{-n} \lambda_1 |\Omega_1 \cap B(0,r)| 
- \fint_{0 \leq t \leq r} t^{1-n}\int_{\d B(0,t)} \Big|{\d v \over d\rho}\Big|^2
d\sigma dt
\nonumber \\
&=& r^{-n} \int_{B(0,r)} |\nabla v|^2 
+r^{-n} \lambda_1 |\Omega_1 \cap B(0,r)| 
- {1 \over r} \int_{B(0,r)} |x|^{1-n}\Big|{\d v \over d\rho}(x)\Big|^2 dx
\end{eqnarray}
for $0 < r < \rho_0$ (where $\rho_0$ comes from our Lipschitz assumption
\eqref{e15.5}). Here $\lambda_1$ comes from \eqref{e15.11} and
$\d u \over d\rho$ denotes the radial derivative of
$u$ (also written $\langle \nabla u, \nu \rangle$ in \cite{We}),
and there is no convergence problem for the integrals, since $\u$
is Lipschitz near $0$.
This is the same function as $p$ on page 319 of \cite{We}, with
$Q(0) = \lambda_1$.

Next let $\u_\infty$ be any blow-up limit of $\u$ at $0$.
That is, assume that there is a sequence $\{ r_k \}$ that tends to $0$, such that
$\u_\infty(x) = \lim_{k \to +\infty} r_k^{-1}\u(r_k x)$ on $\R^n$.
As usual, set
\begin{equation} \label{e18.8}
v_k(x) = r_k^{-1} v(r_k x) = r_k^{-1} [u_1(r_k x)]_+ \ \text{ for } x\in \R^n,
\end{equation}
\begin{equation} \label{e18.9}
v_\infty(x) = [\u_{\infty,1}(x)]_+ = \lim_{k \to +\infty} r_k^{-1} v(r_k x)
\end{equation}
for $x\in \R^n$, and 
\begin{equation} \label{e18.10}
\Omega_{1,k} = \big\{ x\in \R^n \, ; \, v_k(x) > 0\big\}
\ \text{ for $k \geq 0$ and $k=+\infty$.}
\end{equation}
We shall soon use the limit Weiss function
$\Psi_\infty$ defined on $(0,+\infty)$ by
\begin{equation} \label{e18.11}
\Psi_\infty(r) = r^{-n} \int_{B(0,r)} |\nabla v_\infty|^2 
+r^{-n} \lambda_1 |\Omega_{1,\infty} \cap B(0,r)| 
- {1 \over r} \int_{B(0,r)} |x|^{1-n}\Big|{\d v_\infty \over d\rho}(x)\Big|^2 dx
\end{equation}
but let us first talk about the minimizing properties of $\u_\infty$.

Because of Lemma \ref{t16.1}, we can replace 
$\{ r_k \}$ with some subsequence for which \eqref{e15.8}-\eqref{e15.10}
hold, and the we can apply Corollary \ref{t15.3}. Thus there is a domain 
$\Omega_\infty$ and a $N$-uple $\W_\infty$ such that
\begin{equation} \label{e18.12}
(\u_\infty,\W_\infty) \text{ is a local minimizer for $J_\infty$
in } \F(\R^n,\Omega_\infty); 
\end{equation}
see near \eqref{e15.40} for the definitions.
We claim that because of \eqref{e18.6}
\begin{equation} \label{e18.13}
u_{1,\infty} \geq 0 \ \text{ and } 
u_{i,\infty} = 0 \text{ for } i \geq 2
\end{equation}
if we chose $\{ r_k \}$ such that 
\begin{equation} \label{e18.14}
\liminf_{k \to +\infty} \Phi_{1,1}^0(r_k) > 0.
\end{equation}
Of course we can find a sequence like this, by \eqref{e18.1}.
And indeed, suppose that \eqref{e18.14} holds. 
Set $\varphi_1 = (1,1)$ and let $\varphi \in I\sm \{ \varphi_1 \}$ be given.
By \eqref{e18.14}, $\liminf_{k \to +\infty} \Phi_{1,1}^0(r_k R) > 0$
for each $R > 1$ as well (just because 
$\Phi_{1,1}^0(r_k R) \geq R^{-2}\Phi_{1,1}^0(r_k)$
by \eqref{e16.11}); by \eqref{e18.6},
$\Phi_{\varphi_1}^0(r_k R)\Phi_{\varphi}^0(r_k R)$ tends to $0$,
so $\Phi_{\varphi}^0(r_k R)$ tends to $0$, and then 
Lemma \ref{t16.2} says that
\begin{equation} \label{e18.15}
\Phi_{\varphi,\infty}(R) = \lim_{k \to +\infty} \Phi_{\varphi}^0(r_k R) = 0.
\end{equation}
Then $\nabla v_{\varphi} = 0$ almost everywhere, and \eqref{e18.13}
follows.

Next we check that \eqref{e18.3} holds if $0$ is an interior point
of $\Omega$. Suppose not.
Choose a sequence $\{ r_k \}$ such that \eqref{e18.14} holds, 
and then use Lemma~\ref{t16.1} to get a subsequence such that 
Corollary \ref{t15.3} applies, 
so that we get a local minimizer $(\u_\infty,\W_\infty)$ as in 
\eqref{e18.12}. Notice that here $\Omega_\infty = \R^n$.
Set $\varphi = (1,1)$, and notice that 
$\Phi_{\varphi,\infty}(1) = \lim_{k \to +\infty} \Phi_{\varphi}^0(r_k) >0$
by Lemma \ref{t16.2}, so $u_{1,\infty} > 0$ somewhere.

We build a competitor for $(\u_\infty,\W_\infty)$. Pick a ball $B = B(0,r)$
such that $u_{1,\infty} > 0$ somewhere on $\d B$, then denote
by $u_1^\ast$ the harmonic extension to $B$ of the restriction of
$u_{1,\infty}$ to $\d B(0,r)$. Keep $u_1^\ast = u_{1,\infty}$
on $\R^n \sm B$, and also set $u_i^\ast = 0$ for $i \geq 2$.
Also set $W_1^\ast = W_{1,\infty} \cup B$ and 
$W_i^\ast = W_{i,\infty} \cup B$ for $i \geq 2$. It is easy to see
that $(\u^\ast,\W^\ast) \in \F(\R^n, \Omega_\infty)$
(because $\Omega_\infty = \R^n$ and by \eqref{e18.13}). So \eqref{e15.40} holds,
even with $R = r$. Since
$\sum_i  \lambda_i |W_1^\ast \cap B| = \lambda_1 |B|
\leq \sum_i  \lambda_i |W_{i,\infty} \cap B|$ because 
\eqref{e18.3} fails, we do not lose anything on the volume term, and
\eqref{e15.40} yields $\int_{B} |\nabla u_{1,\infty}|^2 
\leq \int_{B} |\nabla u_1^\ast|^2$. But $u_1^\ast$ is the only
minimizer of $\int_{B} |\nabla u|^2$ with the given boundary values on $\d B$,
so $u_{1,\infty} = u_1^\ast$ on $B$. This is not possible, because
$\u(0) = 0$ and $u_1^\ast(0) > 0$
(recall that $u_{1,\infty} \geq 0$ by \eqref{e18.13}, and that it is positive
somewhere on $\d B$).

When $0$ lies on the boundary, it is harder to say much; $\Omega_1$
may fill the entire region $\Omega$, be harmonic (or satisfy the
equation \eqref{e9.4}) there, and naturally
vanish at the boundary because this was our initial constraint. In this case,
\eqref{e18.1} seems to reflect more on the shape of the boundary than on
the $\lambda_i$. But even in this case 
we shall decide to assume \eqref{e18.2}.

We return to the Weiss functional of \eqref{e18.7} and \eqref{e18.11}
and prove that it goes to the limit.

\ms
\begin{lem} \label{t18.1}
Let $\u_\infty$ be a blow-up limit of $\u$ at $0$,
associated to the sequence $\{ r_k \}$. Suppose, in addition to
\eqref{e18.4} and \eqref{e18.5}, that \eqref{e18.2} holds. 
Then, maybe after replacing $\{ r_k \}$ by some subsequence, 
\begin{equation} \label{e18.16}
\Psi_\infty(r) = \lim_{k \to +\infty} \Psi(r_k r)
\ \text{ for } r >0.
\end{equation}
\end{lem}

\ms
\begin{proof}
Because of Lemma \ref{t16.1}, we can replace 
$\{ r_k \}$ with some subsequence for which \eqref{e15.8}-\eqref{e15.10}
hold, and then we can apply Corollary \ref{t15.3}. We get that there is a domain 
$\Omega_\infty$ and a $N$-uple $\W_\infty$ such that
\eqref{e18.12} holds, but this is not what we care about here, we just
need to know that for $r > 0$, 
\begin{equation} \label{e18.17}
\nabla v_\infty = \lim_{k \to +\infty} \nabla v_k
\ \text{ in } L^2(B(0,r)),
\end{equation}
as in \eqref{e15.17}. Then
\begin{eqnarray} \label{e18.18}
r^{-n} \int_{B(0,r)} |\nabla v_\infty|^2
&=& \lim_{k \to +\infty} r^{-n} \int_{B(0,r)} |\nabla v_k(x)|^2 dx
= \lim_{k \to +\infty} r^{-n} \int_{B(0,r)} |\nabla v(r_k x)|^2 dx
\nonumber \\
&=& \lim_{k \to +\infty} \int_{B(0,r_k r)} |\nabla v(y)|^2 dy
\end{eqnarray}
by \eqref{e18.8} and a change of variable, and similarly,
for each small $\varepsilon > 0$,
\begin{eqnarray} \label{e18.19}
{1 \over r}\int_{B(0,r) \sm B(0,\varepsilon)} 
|x|^{1-n}\Big|{\d v_\infty \over d\rho}(x)\Big|^2 dx
&=& \lim_{k \to +\infty}
{1 \over r} \int_{B(0,r) \sm B(0,\varepsilon)} 
|x|^{1-n}\Big|{\d v_k \over d\rho}(x)\Big|^2 dx
\nonumber \\
&=& \lim_{k \to +\infty}
{1 \over r} \int_{B(0,r) \sm B(0,\varepsilon)} 
|x|^{1-n}\Big|{\d v \over d\rho}(r_k x)\Big|^2 dx
\nonumber \\
&=&\lim_{k \to +\infty}
{1 \over r_k r}\int_{B(0,r_k r) \sm B(0,r_k\varepsilon)} 
|x|^{1-n}\Big|{\d v \over d\rho}(x)\Big|^2 dx.
\end{eqnarray}
The small missing pieces are estimated with the Lipschitz norms, i.e.,
\begin{equation} \label{e18.20}
{1 \over r}\int_{B(0,\varepsilon)} 
|x|^{1-n}\Big|{\d v_\infty \over d\rho}(x)\Big|^2 dx
\leq {C \over r} \int_{B(0,\varepsilon)} |x|^{1-n}
\leq {C \varepsilon \over r},
\end{equation}
and 
\begin{equation} \label{e18.21}
{1 \over r_k r}\int_ {B(0,r_k\varepsilon)} 
|x|^{1-n}\Big|{\d v \over d\rho}(x)\Big|^2 dx
\leq {C \over r_k r} \int_{B(0,r_k\varepsilon)} |x|^{1-n} \leq {C \varepsilon \over r}.
\end{equation}
Thus the last term of the functional $\Psi$ goes to the limit too, and 
the lemma will follow as soon as we prove that
\begin{equation} \label{e18.22}
r^{-n} \lambda_1 |\Omega_{1,\infty} \cap B(0,r)| 
= \lim_{k \to +\infty}
(r_k r)^{-n} \lambda_1 |\Omega_{1} \cap B(0,r_k r)|.
\end{equation}
This is where the nondegeneracy assumption \eqref{e18.2}
will be useful.
Set $\Omega_{1,k} = r_k^{-1} \Omega_1 = \big\{x \, ; \,  v_k > 0 \big\}$
and simplify \eqref{e18.22}; we just need to check that
\begin{equation} \label{e18.23}
|\Omega_{1,\infty} \cap B(0,r)| 
= \lim_{k \to +\infty} |\Omega_{1,k} \cap B(0,r)|.
\end{equation}
If $x\in \Omega_{1,\infty} \cap B(0,r)$, then $v_\infty(x) > 0$,
hence $v_k(x) > 0$ for $k$ large, and 
$x\in \Omega_{1,k} \cap B(0,r)$ for $k$ large. So 
$\1_{\Omega_{1,\infty} \cap B(0,r)} \leq
\liminf_{k \to +\infty} \1_{\Omega_{1,\infty} \cap B(0,r)}$ and
by Fatou
\begin{equation} \label{e18.24}
|\Omega_{1,\infty} \cap B(0,r)| 
\leq \liminf_{k \to +\infty} |\Omega_{1,k} \cap B(0,r)|.
\end{equation}
Next set $\varepsilon_k = ||v_\infty - v_k||_{L^\infty(B(0,r))}$ and
\begin{equation} \label{e18.25}
\O_k = \big\{ x\in B(0,r) \, ; \,  0 < v_k(x) \leq 2\varepsilon_k \big\}.
\end{equation}
By definition, $\Omega_{1,k} \cap B(0,r) \i \Omega_{1,\infty} \cup \O_k$;
if we prove that
\begin{equation} \label{e18.26}
\lim_{k \to +\infty} |\O_k| = 0,
\end{equation}
we will get that
\begin{equation} \label{e18.27}
\limsup_{k \to +\infty} |\Omega_{1,k} \cap B(0,r)|   
\leq  |\Omega_{1,\infty} \cap B(0,r)|+ \limsup_{k \to +\infty} |\O_k|
= |\Omega_{1,\infty} \cap B(0,r)|,
\end{equation}
and \eqref{e18.23} and the lemma will follow.

Let us first use Theorem \ref{t13.3} to show that 
for $k$ large,
\begin{equation} \label{e18.28}
\O_k \i A_k := \big\{ y\in B(0,r) \, ; \,  
\dist(y,\d\Omega_{1,k}) \leq C \varepsilon_k \big\}.
\end{equation}
Pick any $x \in \Omega_1 \cap B(0,\rho_0/2)$, 
where $\rho_0$ is still such that $\u$ is Lipschitz
on $B(0,\rho_0)$. Set $\delta(x) = \dist(x,\R^n \sm\Omega_1)$ as 
in \eqref{e13.39}. Since $\u(0) = 0$, we get that $\delta(x) < \rho_0/2$,
hence $\u$ is Lipschitz on $B(x,\delta(x)/2)$. The other assumptions of
Theorem \ref{t13.3} follow from \eqref{e18.2} and \eqref{e18.4},
so \eqref{e13.40} holds. That is,
\begin{equation} \label{e18.29}
u_1(x) \geq c_5 \min(\delta(x), \varepsilon^{1/n}, 1)
\ \text{ for } x\in B(0,\rho_0/2),
\end{equation}
where $\varepsilon$ is a constant that comes from \eqref{e13.1}.
Set $y = r_k^{-1}x$; this yields that for $y \in \Omega_{1,k} \cap B(0,r_k^{-1}\rho_0/2)$,
\begin{equation} \label{e18.30}
v_k (y) =   r_k^{-1} v(r_k y)  = r_k^{-1} u_1(r_k y) = r_k^{-1} u_1(x)
\geq c_5 r_k^{-1} \min(\delta(r_ky), \varepsilon^{1/n}, 1)
\end{equation}
(by \eqref{e18.8} and because $u_1(r_k y)$).
If $k$ is large enough, this holds for $y\in \Omega_{1,k} \cap B(0,r)$, and in addition 
$\delta(r_ky) = \dist(r_k y,\R^n\sm\Omega_1) \leq |r_k y|
< \min(\varepsilon^{1/n}, 1)$.
That is,
\begin{equation} \label{e18.31}
v_k(y) \geq c_5 r_k^{-1} \dist(r_k y,\R^n \sm\Omega_1)
= c_5 \dist(y,\R^n \sm\Omega_{1,k})
\end{equation}
for $y\in \Omega_{1,k} \cap B(0,r)$; \eqref{e18.28} follows because 
$y\in \Omega_{k,1}$ and $v_k(y) \leq 2 \varepsilon_k$
when $y\in \O_k$.

Now we shall use the uniform local Ahlfors regularity of the $\d \Omega_k$
to estimate $|A_k|$.
Let us apply Proposition \ref{t17.1} to some the ball 
$B_0 = B(0,r_0)$, where $r_0$ is chosen small so that
the assumptions of the proposition are satisfied. 
Notice that we did not forget to add \eqref{e10.2} in \eqref{e18.2}. 
We get a measure $\mu$ such that the local Ahlfors regularity condition \eqref{e17.6}
is satisfied for all balls $B$ centered on $\d\Omega_1$
such that $2B \i B_0$. Set $\mu_k(A) = r_k^{1-n}\mu(r_k A)$
for Borel sets $A$ (to preserve the homogeneity); then
by \eqref{e17.6}
\begin{equation} \label{e18.32}
C_1^{-1} \rho^{n-1} \leq \mu_k(B(y,\rho)) 
= r_k^{1-n}\mu(B(r_k y,r_k\rho)) \leq C_1 \rho^{n-1}
\end{equation} 
for $y\in \d\Omega_{1,k} = r_k^{-1}\d\Omega_1$ and 
$\rho > 0$ such that $B(y,2\rho) \i B(0,r_k^{-1} r_0)$.
For $k$ large enough, this includes all the balls $B(y,\rho)$ such that
$y\in \d \Omega_k \cap B(0,2r)$ and $0 \leq \rho \leq 3r$.
We shall gladly restrict to such $k$.

For each $t \in (0,r)$, denote by $Y_{k,t}$ a subset of 
$\d\Omega_k \cap B(0,2r)$, whose points lie at distances at least
$t$ from each other, and which is maximal with this property. The 
number of points of $Y_{k,t}$ is easy to estimate, as 
\begin{eqnarray} \label{e18.33}
\sharp Y_{k,t} &\leq& C\sum_{y\in Y_{k,t}} t^{-d}\mu_k(B(y,t))
\leq C t^{-d} \mu_k\big(\bigcup_{y\in Y_{k,t}} B(y,t)\big)
\nonumber\\
&\leq& C t^{-d} \mu_k(B(0,3r)) \leq C t^{1-n} r^{n-1}
\end{eqnarray}
 by \eqref{e18.32} and because the balls have bounded overlap. 
 We apply this with $t = C \varepsilon_k$, where
 $C$ is as in \eqref{e18.28}, notice that the balls $B(y,2t)$,
 $y\in Y_{k,t}$, cover $A_k$, and deduce from \eqref{e18.28} that
 \begin{equation} \label{e18.34}
|\O_k| \leq |A_k| \leq \sum_{y\in Y_{k,t}} |B(y,2t)|
\leq C t^n \sharp Y_{k,t}
\leq C t r^{n-1} = C \varepsilon_k r^{n-1}
\end{equation}
for $k$ large; Since $\varepsilon_k$ tends to $0$ because the $v_k$
converge to $v_{\infty}$ uniformly in $B(0,r)$, \eqref{e18.26}, 
then \eqref{e18.23} and Lemma \ref{t18.1}, follow.
\qed
\end{proof}

\ms
The next stage is to apply the monotonicity argument. 
The wiser thing to do would probably be to restrict to the case when
$0$ is an interior point of $\Omega$, but we shall try to include boundary
points for some time. We shall assume that
\begin{equation} \label{e18.35}
\begin{aligned}
&\text{each limit set $\Omega_\infty$ that can be obtained from $\Omega$
by applying Lemma \ref{t16.1} to}
\\
&\text{a sequence $\{ r_k \}$  that tend to $0$ 
is equal a.e. to an open cone centered at $0$;}
\end{aligned}
\end{equation}
this is not very explicit because we did not really choose our notion
of convergence for the domains, but fairly weak conditions of approximation
of $\Omega$ by cones at $0$ would imply this.

\begin{thm} \label{t18.2}
Assume that $0 \in \d \Omega_1 = \d \big\{ x\in \R^n \, ; \ u_1(x) >0 \big\}$,
and that \eqref{e18.2}, \eqref{e18.4}, \eqref{e18.5}, \eqref{e18.6},
and \eqref{e18.35} hold. Then there is a sequence $\{ r_k \}$  that tend to $0$ 
such that the blow-up limit $\u_\infty$ defined by $\{ r_k \}$ exists, is non trivial,
gives a minimizer of $J_\infty$ as in \eqref{e18.12}, and is homogeneous
of degree 1.
\end{thm}

By ``exists'', we would just mean that the $r^{-k} \u(r_k x)$ have a limit,
but \eqref{e18.12} asks for more anyway (a domain $\Omega_\infty$ and a partner 
$\W_\infty$).
Since all the components other than $(u_{1,\infty})_+$ vanish by \eqref{e18.13},
by nontrivial we just mean that $u_{1,\infty}(x) > 0$ somewhere.

Note that when $0 \in \d\Omega$, the assumptions of Theorem \ref{t18.2} 
put a nontrivial restriction on the shape of $\Omega$ near $0$, 
because the limit cone $\Omega_\infty$ needs
to be large enough to host a subdomain $\Omega_{1,\infty}$ and a nontrivial
positive harmonic function on $\Omega_{1,\infty}$ that vanishes on the boundary
and is homogeneous of degree $1$. For instance, if $\d\Omega$ is flat at $0$,
$\Omega_\infty$ is a half space, and there is just enough room to put such a harmonic
function on $\Omega_\infty$. In this case, we could prove that  $(u_{1,\infty})_+$
coincides with an affine function on $\Omega_\infty$. The point is that the
Poincar\'e constant for $\Omega_{1,\infty} \cap \d B(0,1)$ is at least as small
as for a half sphere, which gives a control on the first eigenvalue of the Laplacian
on $\Omega_{1,\infty} \cap \d B(0,1)$, and then on the existence of homogeneous 
harmonic functions on $\Omega_{1,\infty}$; we would get that the Poincar\'e
constant is the same as for the half sphere, then that $\Omega_{1,\infty} \cap \d B(0,1)$
is a half sphere, $\Omega_{1,\infty}$ is a half space, and $(u_{1,\infty})_+$ is affine on 
$\Omega_{1,\infty} = \Omega_\infty$.

\begin{proof}
Let us first check that \eqref{e18.1} follows from \eqref{e18.2}.
We claim that
\eqref{e13.1} and \eqref{e18.4} even imply that
\begin{equation} \label{e18.36}
\liminf_{r \to 0} \Phi_{1,1}^0(r) 
= \liminf_{r \to 0}{1 \over r^2} 
\int_{B(0,r)} {|\nabla u_{1,+}|^2 \over |x|^{n-2}} \, dx > 0,
\end{equation}
where the first part comes from the definition \eqref{e16.11}. Indeed if 
$0 \in \d\Omega_1$ and \eqref{e13.1} holds, Theorem \ref{t13.2}
applies to $B(0,r)$ for $r$ small, and says that
\begin{equation} \label{e18.37}
\liminf_{r \to 0} \fint_{B(0,r)} |\nabla u_{1,+}|^2 \geq c_3
\end{equation}
for some constant $c_3 > 0$. But
\begin{eqnarray} \label{e18.38}
\Phi_{1,1}^0(r) &=& 
{1 \over r^2} \int_{B(0,r)} {|\nabla u_{1,+}|^2 \over |x|^{n-2}} \, dx 
\geq {1 \over r^2} \int_{B(0,r) \sm B(0, \eta r)} {|\nabla u_{1,+}|^2 \over |x|^{n-2}}
\nonumber\\
&\geq& {\eta^{n-2}  \over r^n} \int_{B(0,r) \sm B(0, \eta r)} |\nabla u_{1,+}|^2 
= {\eta^{n-2}  \over r^n} \Big\{\int_{B(0,r)} |\nabla u_{1,+}|^2 
-  \int_{B(0,\eta r)} |\nabla u_{1,+}|^2 \Big\}
\nonumber\\
&\geq& {\eta^{n-2}  \over r^n} \int_{B(0,r)} |\nabla u_{1,+}|^2 
-C \eta^{n-2} > c 
\end{eqnarray}
for $r$ small if we choose $\eta>0$ small enough, depending on $c_3$ and 
the Lipschitz constant in \eqref{e15.5}. This proves \eqref{e18.36}.

Now we select a first sequence. Notice that our function $\Psi$
from \eqref{e18.7} is bounded, so 
\begin{equation} \label{e18.39}
L = \limsup_{\rho \to 0} \Psi(\rho)
\end{equation}
is finite. We choose $\{ r_k \}$, tending to $0$, such that
\begin{equation} \label{e18.40}
\lim_{k \to +\infty} \Psi(r_k) = L.
\end{equation}
Then we find a subsequence for which we can apply Corollary 
\ref{t15.3}; this gives a domain $\Omega_\infty$ and a pair
$(\u_\infty,\W_\infty)$ that satisfies \eqref{e18.12}.
By Lemma \ref{t18.1}, 
\begin{equation} \label{e18.41}
\Psi_\infty(1) = \lim_{k \to +\infty} \Psi(r_k r) = L
\ \text{ and } \ 
\Psi_\infty(r) = \lim_{k \to +\infty} \Psi(r_k r) \leq L 
\text{ for } r > 0.
\end{equation}
Now we want to use the proof of Theorem 1.2 in \cite{We} to show that
\begin{equation} \label{e18.42}
\Psi_\infty \text{ is a nondecreasing function on $(0,+\infty)$.}
\end{equation}
There are a few differences with his statement and the present
situation that we need to discuss. First we need to modify a little our definition
of minimizer for $J_\infty$. Set $w = \u_{1,\infty}$. We know from
\eqref{e18.13} that $w \geq 0$ and all the other components are null.
Let $Q \geq 0$ be defined by
\begin{equation} \label{e18.43}
Q^2 = \big[\lambda_1 - \min(0, \lambda_2, \ldots, \lambda_N) \big]_+.
\end{equation}
Thus $Q > 0$ when \eqref{e18.3} holds and $Q = 0$ otherwise.
It turns out that the assumptions \eqref{e13.1} and \eqref{e15.11}
imply \eqref{e18.3}, so $Q > 0$, but we do not need to know this
for the moment.

We claim that $w$ is a \underline{one-phase
ACF minimizer} in $\Omega_\infty$ with coefficient $Q^2$,
and we mean by this that if $w^\ast \in W^{1,2}_{loc}(\R^n)$ is 
such that $w^\ast = 0$ almost everywhere on $\R^n \sm \Omega_\infty$,
and there is a ball $B$ such that $w^\ast = w$ almost everywhere on 
$\R^n \sm B$, then
\begin{equation} \label{e18.44}
\int_B |\nabla w|^2 + Q^2 \big|\big\{x \in B \, ; \, w(x) > 0\big\}\big|
\leq \int_B |\nabla w^\ast|^2 + Q^2 \big|\big\{x \in B \, ; \, w^\ast(x) > 0\big\}\big|.
\end{equation}
The verification is easy. Given a competitor $w^\ast$ for $w$,
we construct a competitor for $(\u_\infty,\W_\infty)$, use the minimality
of this pair, and conclude. That is, we observe that $w^\ast_+$ is
at least as good as $w^\ast$, take $u_1^\ast = w^\ast_+$ and keep
$u_i^\ast = 0$ for $i \geq 2$. We keep $\W^\ast = \W$ on $\R^n \sm B$, 
and on $B$ we distinguish cases. 
If $Q > 0$ and $\min(0, \lambda_2, \ldots, \lambda_N) = \lambda_j$ for some
$j \geq 2$, we take $B \cap W_1^\ast = B \cap \{ w^\ast > 0\}$,
$W_j^\ast = B \sm W_1^\ast$, and all the other $B \cap W_i^\ast$
empty. If $Q > 0$ and $\min(0, \lambda_2, \ldots, \lambda_N) = 0$
we take $B \cap W_1^\ast = B \cap \{ w^\ast > 0\}$ and all the other ones empty,
and if $Q=0$ we take  $B \cap W_1^\ast = B$ and all the other ones empty.
It is easy to see that we could hardly do better, and that the minimality
of $(\u_\infty,\W_\infty)$ yields \eqref{e18.44};
we spare the details.

This definition is almost the same as the one used in \cite{We},
and we are happy because $Q$ is constant and our functional will
be nondecreasing.
It is true that Theorem 1.2 in \cite{We} is stated when $\Omega_\infty = \R^n$,
but its proof also works when $\Omega_\infty$ is an open cone (and changing
$\Omega$ on a set of measure $0$ does not change the fact that $w$
is an AFC minimizer). The main ingredient of the proof consists in taking
a ball $B$ centered at the origin, and testing the competitor $w^\ast$
which is equal to $w$ on $\R^n \sm B$, is continuous across $\d B$,
and is homogeneous of degree $1$ on $B$. The proof uses the fact that
$w$ is Lipschitz (this simplifies the computations), that $w(0)=0$,
and now we need to notice that $w^\ast(x) = 0$ on $\R^n \sm \Omega_\infty$
because $\Omega$ is a cone. The claim \eqref{e18.42} follows.

From \eqref{e18.41} and \eqref{e18.42} we easily deduce that
\begin{equation} \label{e18.45}
\Psi_\infty(r) = L \text{ for } r \geq 1;
\end{equation}
we want to use this to show that
%
\begin{equation} \label{e18.46}
\u_\infty(\lambda x) = \lambda \u_\infty(x) 
\ \text{ for $x\in \R^n$ and } r > 0.
\end{equation}
In fact, it is enough to prove this for $w = (u_{1,\infty})_+$ and
$x$ in the cone $\Omega_\infty$ (because $\u_\infty = 0$
almost everywhere on $\R^n \sm \Omega_\infty$).

Again we just follow the proof given in \cite{We}.
The proof of monotonicity shows that for $0 < s < r$,
\begin{equation} \label{e18.47}
\Psi(r) - \Psi(s) \geq A(s,r),
\end{equation}
where the quantity
\begin{equation} \label{e18.48}
A(s,r) = \int_{t=s}^r t^{-3} \int_{\xi \in \Omega_\infty \cap \d B(0,1)}
\Big[ t \int_{a=0}^t \Big|{\d w \over \d \rho}(a\xi)\Big|^2 da
- \Big\{ \int_{a=0}^t {\d w \over \d \rho}(a\xi) da \Big\}^2
\Big] d\sigma(\xi) dt
\end{equation}
is nonnegative by an application of Cauchy-Schwarz in the $a$-integral. 
Here we know that $\Psi(s) =\Psi(r)$ for $1 < s< r$, and we get that
\begin{equation} \label{e18.49}
t \int_{a=0}^t \Big|{\d w \over \d \rho}(a\xi)\Big|^2 da
= \Big\{ \int_{a=0}^t {\d w \over \d \rho}(a\xi) da \Big\}^2
\end{equation}
for almost every $t \in (s,r)$ and almost every $\xi \in \Omega_\infty \cap \d B(0,1)$.
Thus for such $t$ and $\xi$,
${\d w \over \d \rho}(a\xi)$ is (almost everywhere) constant 
on $[0,t]$; \eqref{e18.46} easily follows from this (recall that $w$ is Lipschitz).
Theorem \ref{t18.2} follows.
\qed
\end{proof}

When $n\leq 3$ and $\Omega_\infty = \R^n$, 
the main theorem of \cite{CJK2} says that the homogeneous
minimizer $\u_\infty$ that we produced for Theorem \ref{t18.2}
has the following simple form: there is a unit vector $e \in \R^n$
such that 
\begin{equation} \label{e18.50}
u_{1,\infty} = Q \max(0,\langle x,e \rangle)
\ \text{ for } x\in \R^n,
\end{equation}
where $Q$ is given by \eqref{e18.43}. Notice that this forces $Q > 0$
(which we could also have obtained by comparing \eqref{e13.1} with
\eqref{e15.3}), because $\u_\infty$ is not null.
Thus we obtained the following result.

\begin{cor} \label{t18.3}
Suppose that $n = 2$ or $3$, that $0$ is an interior point of $\Omega$,
that $0 \in \d \Omega_1 = \d \big\{ u_1(x) > 0 \big\}$, 
and that \eqref{e18.2}, \eqref{e18.4}, and \eqref{e18.6} hold.
Then there is a sequence $\{ r_k \}$  that tend to $0$ 
such that the blow-up limit $\u_\infty$ defined by $\{ r_k \}$ exists
and is such that $\u_{i,\infty} = 0$ for $i > 0$, and \eqref{e18.50} holds
for some unit vector $e$, with
\begin{equation} \label{e18.51}
Q^2 = \lambda_1 - \min(0, \lambda_2, \ldots, \lambda_N) > 0.
\end{equation}
\end{cor}

\ms
Recall that the $\lambda_i$ come from \eqref{e15.11}.
We removed the assumptions \eqref{e18.5} and \eqref{e18.35} because 
they are trivial when $0$ is an interior point of $\Omega$.

The description of $\u_\infty$ by \eqref{e18.50} implies that the free boundary
$\d\Omega_1$ has some flat blow-up limits at the origin, as in the following.

\begin{lem} \label{t18.4}
If $\{ r_k \}$ is as in Corollary \ref{t18.3}, then for each 
$R<0$ there exist numbers $\beta_k > 0$ such that 
$\lim_{k \to +\infty} \beta_k = 0$ and
\begin{equation} \label{e18.52}
|\langle x, e \rangle| \leq \beta_k r_k 
\ \text{ for $x\in B(0,r_k R) \cap \d\Omega_1$.}
\end{equation}
\end{lem}

\begin{proof}
Let $R > 0$ be given, and set 
$\varepsilon_k = ||u_{1,\infty} - u_{1,k}||_{L^\infty(B(0,2R)}$
for $k \geq 0$. Thus $\varepsilon_k$ tends to $0$.
Let $k \geq 0$ and $x\in B(0,r_k R) \cap \d\Omega_1$
be given. Set $y= r_k^{-1} y \in B(0,R)$, and observe that $u_{1,k}(y) 
= r_k^{-1} u_1(x) = 0$. Thus $u_{1,\infty}(y) \leq \varepsilon_k$,
and by \eqref{e18.50} $\langle y,e \rangle \leq Q^{-1} \varepsilon_k$.
Thus $\langle x,e \rangle \leq Q^{-1} \varepsilon_k r_k$.

Now let $\eta > 0$ be small, and apply Theorem \ref{t13.1} to the ball 
$B(x, \eta r_k)$. This is possible, because we assumed that \eqref{e13.1}
holds and as soon as $r_k$ is small enough. We get that
$\fint_{B(x,\eta r)} |u_{1,+}|^2 \geq c_1 (\eta r_k)^2$,
by \eqref{e13.6}. By Chebyshev, we can choose $z\in B(x,\eta r_k)$ such
that $u_{1}(z) \geq c_1^{1/2} \eta r_k$. Set $w = r_k^{-1} y$;
then $z\in r_k^{-1} B(x,\eta r_k) \i B(0,2R)$ and
$u_{1,k}(w) = r_k^{-1}u_1(z) \geq c_1^{1/2} \eta$.
If $k$ is so large (depending on $\eta$) that 
$\varepsilon_k < c_1^{1/2} \eta$, we get that
$u_{1,\infty}(w) > 0$, hence $\langle w,e \rangle >0$ by \eqref{e18.50} and
$\langle z,e \rangle >0$ too. Then $\langle x,e \rangle \geq - \eta r_k$ for $k$
large. This holds for every $\eta > 0$; the lemma follows.
\qed
\end{proof}

Of course it would be better to know that
all the blow-up limits of $\u$ are as $\u_\infty$ above, or that
all the blow-up limits of $\d\Omega_1$ are hyperplanes, but here the
presence of the other components $\Omega_\varphi$ seems to make it harder
to prove better estimates, even though their contribution is small
by \eqref{e18.6}. The situation will be better in 
Section \ref{good2}, because none of the other $\Omega_\varphi$
are allowed to touch $0$.

Notice that the flatness of $\u$ and $\d \Omega_1$ at some small
scales is often the entry point for further regularity results, but we do not
know whether we still can prove such regularity results without further 
assumptions on the other components. See the comments in Section~\ref{good2}.

\ms
There is another case when we can get the same control on 
the blow-up limit $\u_\infty$ as in Corollaries \ref{t16.4}
and \ref{t18.3}, and this is when we know that the free boundary is flat
at the origin.

\begin{pro} \label{t18.5}
Suppose that $0 \in \d \Omega_1$, and that \eqref{e18.2}, \eqref{e18.4}, 
and \eqref{e18.5} hold. Let $\{ r_k \}$ be a sequence that tends to $0$ and such that 
\begin{equation} \label{e18.53}
v(x) = \lim_{k \to +\infty} [u_{1,k}(x)]_+
\ \text{ exists for $x\in \R^n$.}
\end{equation}
Also suppose that there is a unit vector $e \in \R^n$ and, for each
$R > 0$, numbers $\beta_k > 0$ 
such that $\lim_{k \to +\infty} \beta_k = 0$ and \eqref{e18.52} holds.
Then there is a constant $a \neq 0$ such that
\begin{equation} \label{e18.54}
v(x) = \max(0,\langle x, ae \rangle)
\ \text{ for } x\in \R^n.
\end{equation}
If in addition $0$ is an interior point of $\Omega$ and the $\u_k(x)$ converge to some
limit $\u_\infty(x)$, then conclusion of either Corollary~\ref{t16.4}
or Corollary~\ref{t18.3} holds.
\end{pro}

\ms\begin{proof}  
Let $\{ r_k \}$ be as in the statement, and as usual we use Lemma \ref{t16.1} 
to replace it with a subsequence for which we can apply Corollary \ref{t15.3}
(see the discussion at the beginning of Section \ref{2phases}).
This gives a limit $\u_\infty$ and a $N$-uple $\W_\infty$
such that $(\u_\infty,\W_\infty)$ is a local minimizer for $J$,
as in \eqref{e18.12} and \eqref{e15.40}. 
Naturally, $v = [u_{1,\infty}]_+$, so we look for a good description
of $\u_\infty$.

We want to use \eqref{e18.52} and the nondegeneracy results of Section \ref{good}
to show that $v > 0$ on one side of the hyperplane 
$H = \big\{ x\in \R^n \, ; \,  \langle x, e \rangle = 0 \big\}$, and
$v=0$ on the other side.

Let $x\in \R^n \sm H$ be given, and choose $R > 0$ such that
$x\in B(0,R/2)$ and $\delta > 0$ such that $B(x,2\delta) \i B(0,R)\sm H$.
Then set $x_k = r_k y$. Thus $x_k \in B(0,r_k r)$ and 
$\dist(x_k,H) \geq 2\delta r_k$. If $k$ is so large
that $\beta_k < \delta$, \eqref{e18.52} says that $B(x_k,\delta r_k)$
does not meet $\d \Omega_1$.
One possibility is that $B(x_k,\delta r_k) \i \R^n \sm \Omega_1$;
then $u_{1} \leq 0$ on $B(x_k,\delta r_k)$ and 
$u_{1,k} \leq 0$ on $B(x,\delta)$. If this happens for an infinite number
of values of $k$, we get that $u_{1,\infty} \leq 0$ and $v=0$ on $B(x,\delta)$.

Otherwise, for an infinite number of values of $k$, we get that 
$B(x_k,\delta r_k) \i \Omega_1$.
We want to apply Theorem \ref{t13.3} to the point $x_k$.
The assumption \eqref{e13.1} comes from \eqref{e18.2},
$\delta(x_k) = \dist(x_k, \R^n \sm \Omega_1) \leq r_k R$
(see \eqref{e13.39} and recall that $0 \notin \Omega_1$),
so $\u$ is Lipschitz on $B(x,\delta(x_k))$, with a constant that
does not depend on $k$, by \eqref{e16.5}. 
Since $\delta(x_k) \leq r_k R \leq \min(\delta^{1/n},1)$
for $k$ large (recall that $\varepsilon$ is a constant that 
comes from \eqref{e13.1}), \eqref{e13.40} says that
$u_1(x_k) \geq c_5 \delta(x_k) \geq c_5  r_k \delta$. Then
$u_{1,k}(x) = r_k^{-1} u_1(x_k) \geq c_5 \delta$.
Since this happens for infinitely many $k$, we get that
$u_{1,\infty}(x) \geq c_5 \delta$ as well. In this second
case, there is a small ball centered at $x$ on which 
$u_{1,\infty} > 0$. So, on each component of $\R^n \sm H$,
$v = [u_{1,\infty}]_+$ is either always positive, always $0$.

We claim that the case when $u_{1,\infty} \leq 0$ on both sides of $H$ is impossible.
Indeed $0 \in \d \Omega_1$, so for $k$ large we can apply Theorem \ref{t13.1} to the
ball $B(x,r_k)$, and use \eqref{e13.6} and Chebyshev to find 
$y_k \in B(x,r_k)$ such that $u_1(y_k) \geq c_1^{1/2} r_k$;
we can then extract a new sequence so that $r_k^{-1} y_k$ has a limit
$y_\infty \in \overline B(0,1)$, and then
$u_{1,k}(y_\infty) \geq u_1(y_k) - C |y_\infty - y_{k}|
\geq c_1^{1/2} r_k/2$ for $k$ large, hence $u_{1,k}(y_\infty) > 0$.

The case when $u_{1,\infty} > 0$ on both sides of $H$ is impossible as
well. Indeed, if this happens, $W_{1,\infty} = \R^n$ modulo a set of measure
$0$, then $0$ is an interior point of $\Omega$ (because otherwise \eqref{e16.1}
would give big chunks of $\R^n \sm \Omega$ in each small ball $B(0,r_k)$,
and this would stay true for the limit $\Omega_\infty$, by \eqref{e15.8}.
Then replacing $\u_{1,\infty}$ by the harmonic extension of its values 
on $\d B(0,1)$ would be licit, and make a strictly better competitor than 
$\u_{1,\infty}$ (notice that the two are really different, because 
$\u_{\infty}(0) = 0$), a contradiction with the minimality of $(\u_\infty,\W_\infty)$.

Denote by $U_+$ the component of $\R^n \sm H$ where $\u_{1,\infty}(x) >0$.
We know that $v = [\u_{1,\infty}]_+$ is harmonic on $U_+$, because 
$(\u_\infty,\W_\infty)$ is a local minimizer for $J$, and also that it
is Lipschitz and vanishes at $0$. It is then easy to show that 
$v$ is affine on $U_+$ (use a reflection to get a harmonic function
$\wt v$ in $\R^n$, and then write the Poisson formula for $\nabla \wt v$ on huge balls 
to find out that $\nabla^2 \wt v = 0$). So $v$ is given by \eqref{e18.54} for some
$a \neq 0$.

We continue a little further. If \eqref{e16.52} holds, we can apply Corollary \ref{t16.4},
$\u_\infty$ satisfies its conclusions, and we are happy (because the function
$\u_\infty$ that we study now coincides with the $\u_\infty$ that appears
in the statement). Otherwise, for 
$\varphi_1 = (1,1)$ and each choice of $\varphi_2 \neq \varphi_1$
and $R > 0$,
\begin{eqnarray} \label{e18.55}
0 &=& L(\varphi_1,\varphi_2) 
= \lim_{\rho \to 0} \Phi^{0}_{\varphi_1}(\rho)\Phi^{0}_{\varphi_2}(\rho)
\nonumber \\
&=&\lim_{k \to +\infty} \Phi^{0}_{\varphi_1}(r_k R) \, \Phi^{0}_{\varphi_2}(r_k R)
= \Phi_{\varphi_1,\infty}(R) \, \Phi_{\varphi_2,\infty}(R)
\end{eqnarray}
by \eqref{e16.19} and \eqref{e16.14}.
Now $\Phi_{\varphi_1,\infty}(R) = C > 0$, by direct computation with 
\eqref{e18.54}, so we are left with $\Phi_{\varphi_2,\infty}(R)=0$,
and hence $u_{\varphi, \infty} = 0$ (see the definition \eqref{e16.13}).
That is, $\u_\infty$ is given by the same sort of formula as in 
Corollary \ref{e18.3}, we just need to check that $|a|$ in \eqref{e18.54}
is the same as $Q$ in \eqref{e18.51}, and this is where we need our extra
assumption that $0$ is an interior point of $\Omega$.

Indeed, otherwise it could be that $\d\Omega$ is smooth near $0$,
$\Omega_1 \cap B(0,r) = \Omega \cap B(0,r)$ for $r$ small, and the
normal derivative of $u$ along $\d \Omega$ is very large, due to a
large pressure coming from outside of $B(x,r)$. In other words, we are not
free to add space to $\Omega_1$ on the other side of $\Omega$, so the
first variation computation that leads to \eqref{e18.51} is not available
when $0 \in \d\Omega$.

So we assume that $0$ is an interior point of $\Omega$, and then 
fact that $(\u_\infty,\W_\infty)$ is a local minimizer for the functional $J$
(as in \eqref{e18.12} and \eqref{e15.40}), and the discussion
near \eqref{e18.43}, show that $v$ is a one-phase ACF minimizer, 
associated to the possibly different $Q_+ = \max(0,Q)$ (see \eqref{e18.43}), 
and to the domain $\Omega_\infty = \R^n$.  
In addition, $Q_+ > 0$ because \eqref{e18.54} holds for some
$a \neq 0$, and if $Q_+ = 0$ we can replace $v$ by its harmonic extension near
$0$, win on the energy, and not lose on the volume. So $Q > 0$ too. Finally, the 
fact that $|a| = Q$ comes from the fact that $v$ is a one-phase ACF minimizer
and a classical first variation computation. See Section \ref{vari}, near \eqref{e20.10}.
\qed
\end{proof}

\ms
\begin{rem} \label{t18.6}
An obvious defect of Proposition \ref{t18.5} (compared with Corollary \ref{t18.3}
for instance) is that we cannot be sure in advance that the assumption will
be satisfied for a given origin $x_0$ (so far called $0$). But Proposition \ref{t17.3}
gives a lot of points $x_0$ that it could be applied to. Indeed, under the current assumptions
(including \eqref{e10.2} and \eqref{e13.1} for $i=1$), Proposition~\ref{t17.3}
implies that $\d\Omega_1$ has a tangent plane at $\H^{n-1}$-almost every 
point $x_0$. This is not hard; it uses the fact that locally uniformly rectifiable sets
have tangent planes at almost every point. We could also use 
\eqref{e17.26}, and the fact that if a locally Ahlfors-regular set
has an approximate tangent plane at some point, it has a true tangent plane at
that point. See for instance Exercise 41.21 in \cite{D}.
Anyway, Proposition \ref{t18.5} applies to almost every point 
$x_0 \in \d\Omega_1 \sm \d \Omega_0$, and with any sequence $\{ r_k \}$
for which the $\u_k$ converges. 
\end{rem}

\section{Local regularity when all the indices are good}   \label{good2}

In this section we suppose that all the concerned indices are good,
and reduce the study of our minimizer $(\u,\W)$ near a point (say, the origin) to
the study of minor variants of the Alt, Caffarelli, and Friedman's 
free boundary problems, with just one or two phases. Most of the section here
will follow at once from results of Sections \ref{2phases} and \ref{1phase}.

More precisely, we shall still assume that 
\begin{equation} \label{e19.1}
\text{$\u$ satisfies \eqref{e15.3} and \eqref{e15.5}, the $f_i$ and $g_i$ satisfy \eqref{e15.4}, and $F$ satisfies \eqref{e15.11},}
\end{equation}
as in the previous sections, and now that, for each index $i$ such that
\begin{equation} \label{e19.2}
\text{$0$ lies in the boundary of } \big\{ x\in \R^n \, ; \, u_i(x) \neq 0 \big\},
\end{equation}
there exist $\lambda > 0$ and $\varepsilon > 0$ such that
\begin{equation} \label{e19.3}
\text{the analogue of \eqref{e13.1} for the index $i$ holds.}
\end{equation}
This is more brutal than in the previous sections, but
not a shocking thing to ask; for instance, this holds if
$F$ is given by \eqref{e1.7} with $q_i \geq c > 0$,
or by \eqref{e1.6} with $a > 0$ and $b\geq 0$, where in both cases
we can trade against the empty set. Of course this assumption will
simplify our life, because we won't have to worry about phases of $\u$
that may be small near $0$, but not really vanish.

When $0$ is not an interior point of $\Omega$, we shall also assume that
\begin{equation} \label{e19.4}
\Omega \text{ satisfies the assumptions of Lemma \ref{t16.1}.}
\end{equation}
With these assumptions, we can already reduce the number of phases 
that live near the origin. 
For $\varphi = (i,\varepsilon)\in I = [1,N] \times \{ -1, +1 \}$, set
\begin{equation} \label{e19.5}
\Omega_\varphi = \big\{ x\in \R^n \, ; \, \varepsilon u_i(x) > 0 \big\},
\end{equation}
and the denote by $I(0)$ the set of $\varphi \in I$ such that
$0 \in \d \Omega_\varphi$. 

\begin{lem} \label{t19.1}
Under the assumptions above, $I(0)$ has at most two elements 
if $0$ lies in the interior of $\Omega$, and at most one if $0 \in \d \Omega$.
\end{lem}

As we were preparing this document, we were informed of a recent result
of D. Bucur and B. Velichkov \cite{BV} , which contains a result very similar to 
Lemma \ref{t19.1}, although in a slightly different context. Their paper
is based on a $3$-phase monotonicity formula, which they manage to prove
and use without knowing that their analogue of $\u$ is H\"older-continuous
(left alone, Lipschitz).

\begin{proof}  
Because of \eqref{e19.3}, we know from \eqref{e18.36} that 
\begin{equation} \label{e19.6}
\liminf_{r \to 0} \Phi_{\varphi}^0(r) > 0 \ \text{ for } \varphi \in I(0).
\end{equation}
We may assume that $I(0)$ contains (at least) two elements 
$\varphi_1$ and $\varphi_2$. Notice that $L(\varphi_1,\varphi_2) > 0$
(see the definition \eqref{e16.19}); then we can apply Corollary \ref{t16.4},
and point (ii) says that $\lim_{r \to 0} \Phi_{\varphi}^0(r)= 0$
for $\varphi \in I \sm \{ \varphi_1, \varphi_2 \}$. By \eqref{e19.6} again, such
$\varphi$ lie out of $I(0)$; hence $I(0)$ has at most $2$ elements.
In addition, part (iii) of Corollary \ref{t16.4} says that $0$ lies in
the interior of $\Omega$; this completes our proof of Lemma \ref{t19.1}.
\qed
\end{proof}

\ms
Notice that since the sets $\overline \Omega_\varphi$ are closed, and
touch $0$ only when $\varphi \in I(0)$, there is a small
radius $r_0 > 0$ such that
\begin{equation} \label{e19.7}
u_{\varphi}(x) = 0 \ \text{ for $x\in B(0,r_0)$ and } \varphi \in I \sm I(0),
\end{equation}
where $u_{\varphi} = (\varepsilon u_i)_+$ when
$\varphi = (i,\varepsilon)$. Thus we get a small ball $B_0 = B(0,r_0)$
where $\u$ has at most two nonzero phases. We shall soon see that in this
ball, $(\u,\W)$ solves a simpler free boundary problem.

\begin{rem} \label{t19.2}
Our proof of Lemma \ref{t19.1} does not seem to give lower bounds
for $r_0$ above. But the following scheme, that occurred to us {\rm after}
discussing the results of \cite{BV} with B. Velichkov, seems to give
such lower bounds. Suppose that $\Omega_{\varphi_j}$ meets
$B(0,\tau)$ for three different phases $\varphi_j$; by \eqref{e13.7},
$|\Omega_{\varphi_3} \cap B(0,r)| \geq c_2 r^n$ for
$2\tau \leq r \leq C^{-1}\rho_0$, where $\rho_0$ is as in \eqref{e15.5}.
We can use this to show that for many radii $r$, the two open sets
$\Omega_{\varphi_1} \cap \d B(0,r)$ and $\Omega_{\varphi_2} \cap \d B(0,r)$
have a significantly smaller joint measure than $\d B(0,r)$, which leads to a strict 
increase (with estimates) for the corresponding functional 
$\Phi^0_{\varphi_1,\varphi_2} = \Phi^0_{\varphi_1}\Phi^0_{\varphi_2}$
of Section \ref{mono}. If $\tau$ is small enough, we get that
$\Phi^0_{\varphi_1,\varphi_2}(2\tau)$ is so small that this 
contradicts \eqref{e13.10}. Notice that for this, the fact that $\u$ is Lipschitz
and the quantitative nondegeneracy estimates of Section \ref{mono} would
be needed.
\end{rem}

\begin{rem} \label{t19.3}
In the last sections, we have taken the fact that $\u$ is Lipschitz
as a assumption, but Theorem \ref{t10.1} (when $0$ is an interior
point) and Theorem \ref{t11.1} (when $0 \in \d\Omega$) give sufficient 
conditions for this to happen.
\end{rem}

\ms
Let us now formalize our claim that 
the restriction of $(\u,\W)$ to $B_0$ comes from a simpler
free boundary problem. We start with the simpler case when there
are two true phases.

\ms
\noindent{\bf Case 1.} Let us assume, in addition to the hypotheses
above, that $I(0) = \{\varphi_1,\varphi_2\}$, with $\varphi_j = (i_j,\varepsilon_j)$.
Recall that in this case $0$ is an interior point of $\Omega$,
and choose $B_0$ so small that $B_0 \i \Omega$.

\begin{lem} \label{t19.4} 
Set
\begin{equation} \label{e19.8}
v = u_{\varphi_1} - u_{\varphi_2} 
= [\varepsilon_1 u_{i_1}]_+ - [\varepsilon_2 u_{i_2}]_+
\end{equation}
and denote by $\F_v$ the class of functions $w \in W^{1,2}_{loc}(\R^n)$
such that $w = v$ almost everywhere on $\R^n \sm B_0$.
Then $v$ is a minimizer in $\F_v$ of the functional
\begin{equation} \label{e19.9}
J(w) = G(w) + \int_{B(0,r)} |\nabla w|^2 + (w_+)^2 f_{i_1} + (w_-)^2 f_{i_2}
- w_+ g_{i_1} -  w_- g_{i_2},
\end{equation}
where the $f_i$ and the $g_i$ are as in the definition of $M$
in \eqref{e1.4}, $w_{\pm} = \max(0,\pm w)$ as usual, and 
the volume term $G(w)$ can be computed in terms of the sets 
\begin{equation} \label{e19.10}
A_{w,\pm} = \big\{ x\in B_0 \, ; \, \pm w(x) > 0 \big\}.
\end{equation}
See \eqref{e19.13} below for the formula.
\end{lem}

Notice that the real-valued function contains all
the information about $\u$ in $B_0$, by definition of $I(0)$.
The condition that $w=v$ on $\R^n \sm B_0$ is our way of 
stating a Dirichlet constraint on $\d B_0$.
We left the computation of $G$ for the proof, because the formula
is a little ugly. But let us say now that when $F$ is given by \eqref{e1.7}
with nonnegative functions $q_i$, we can take
\begin{equation} \label{e19.11}
G(w) = \int_{A_{w,+}} q_{i_1} + \int_{A_{w,-}} q_{i_2}.
\end{equation}

\begin{proof}  
Let us prove the lemma, and at the same time define $G$.
The idea is simple: we associate to each $w\in \F_v$ a
competitor $(\u_w,\W_v)$ for $(\u,\W)$, test the minimality of 
$(\u,\W)$ on this competitor, and hopefully we shall get
the desired inequality. 

Due to our definition of phases, we shall need to distinguish
between two main cases. First assume that $i_1 \neq i_2$, and to
simplify the notation, that $\varphi_1 = (1,+1)$ and $\varphi_2 = (2,+1)$.
The $N$-uple of functions associated to $w$
is just  $\u_w = (w_+,w_-, u_3, \ldots, u_N)$
(notice that $w$ is also defined on $\R^n \sm B_0$), and
a simple $N$-uple of sets $\W_w$ that we can take is given by
\begin{equation} \label{e19.12}
\begin{aligned}
W_{w,1} &= A_{w,+} \cup (W_1 \sm B_0),
\cr
W_{w,2} &= A_{w,-} \cup (W_2 \sm B_0),
\cr
W_{w,i} &= W_i \sm B_0  \hskip0.5cm\text{ for } i > 2.
\end{aligned}
\end{equation}
It is easy to see that $(\u_w,\W_v)$ is an acceptable pair
(i.e., that $(\u_w,\W_v) \in \F(\Omega)$); in particular
there is no gluing problem because $u_{w,i} = u_i$ for $i > 2$,
and the $W_{w,i}$ are contained in $\Omega$ because $B_0 \i \Omega$.

If $F$ is a nondecreasing function of the $W_i$, this is the best
that we could do. But in some cases, we may prefer to use 
another element of the class ${\cal H}(w)$, where ${\cal H}(w)$ is the set of 
$N$-uples $\W^\ast= (W^\ast_1, \ldots, W^\ast_N)$,
where the $W^\ast_i$ are disjoint, and each $W^\ast_i$ 
contains $W_{w,i}$ and coincides with $W_i$ and $W_{w,i}$ on $\R^n \sm B_0$. 
Notice that $(\u_w,\W^\ast_v)$ also lies in $\F(\Omega)$ when
$\W^\ast \in {\cal H}(w)$. We set
\begin{equation} \label{e19.13}
G(w) = \inf\{ F(\W^\ast) \, ; \, \W^\ast \in {\cal H}(w) \big\}.
\end{equation}
Notice that $G(v) = F(\W)$, because $(\u,\W)$ is a minimizer,
hence $F(\W)$ also minimizes $F$ in the class ${\cal H}(v)$.

We are ready to check the minimality of $v$. Let
$w \in \F_v$ be given, and let $(\u_w,\W^\ast_v)$ be as above.
Call $M_0$ the part of $M(\u)$ that comes from integrating
outside of $B_0$; then
\begin{eqnarray} \label{e19.14}
J(v) &=& G(v) + \int_{B(0,r)} |\nabla u_1|^2 
+ u_1^2 f_1 - u_1 g_1+ u_2^2 f_2 - u_2^2 g_2
\nonumber \\
&=& F(\W) + \int_{B(0,r)} |\nabla u_1|^2 
+ u_1^2 f_1 - u_1 g_1+ u_2^2 f_2 - u_2^2 g_2
 \\
&=& J(\u,\W) - M_0 \leq J(\u_w,\W_w) -M_0 = J(w)
\nonumber
\end{eqnarray}
because $(\u,\W)$ minimizes $J$, all the other components of $\u$ and $\u_w$ 
vanish on $B_0$, and ${\bf u}_w = \u$ on $\R^n \sm B(0,R)$.
This proves the minimality of $v$ when $\varphi_1 = (1,+1)$ and 
$\varphi_2 = (2,+1)$.

When $\varphi_1 = (1,+1)$ and $\varphi_2 = (1,-1)$, we just need to define
$\u_w$ and $\W_w$ slightly differently. We set
$\u_w = (w,u_2, \ldots, u_N)$,
\begin{equation} \label{e19.15}
W_{w,1} = A_{w,+} \cup A_{w,-} \cup(W_1 \sm B_0)
\ \text{ and } \ 
W_{w,i} = W_i \sm B_0 \text{ for } i > 1.
\end{equation}
We then define ${\cal H}(w)$ and $G$ and complete the argument as above.
All the other cases can be treated like one of these two,
and the lemma follows.
\qed
\end{proof}

\ms
The formula \eqref{e19.13} is not very beautiful, but in many cases we may
use the simpler formula 
\begin{equation} \label{e19.16}
G(w) = F(\W_w), \text{ with $\W_w$ as in
\eqref{e19.12} or \eqref{e19.15}.}
\end{equation}
This is the case when 
$F$ is given by \eqref{e1.7} with nonnegative functions $q_i$, and this
is why we get \eqref{e19.11} (we may add a constant to $G$
without changing the result). But for instance if $q_N$ is negative
and smaller than the all the other $q_i$, the good choice of $G$ is really
$G(w) = \int_{A_{w,+}} (q_{i_1}-q_N) + \int_{A_{w,-}} (q_{i_2}-q_N)$.

The proof above shows that Lemma \ref{t19.4} is still true 
(but may be less precise) with $G$ given by \eqref{e19.16} as soon
as $F(\W) = F(\W_v)$, and we claim that this is the case as soon as
\begin{equation} \label{e19.17}
|W_i \cap B_0| = 0 \ \text{ for } i \neq i_1, i_2.
\end{equation}
To check this, let us even show that for each good index $i$,
\begin{equation} \label{e19.18}
W_i= \big\{ x\in \R^n \, ; \, u_i(x) \neq 0 \big\}
\end{equation}
modulo negligible sets. Indeed, call $V$ the set on the right;
we know that $V \i W_i$ almost everywhere, because 
$u_i = 0$ almost everywhere on $\R^n \sm W_i$.
If $|W_i \sm V| > 0$,  we can take a subset of positive measure in 
$W_i \sm V$, use \eqref{e13.1} to sell part of it to some other region,
throw the rest to the trash, and make a profit. This is impossible
because $(\u,\W)$ is a minimizer. So \eqref{e19.18} holds.

We use this, with $i = i_1$ or $i_2$, and \eqref{e19.17}, to see
that $\W_v = \W$ (almost everywhere) when we take $w=v$
in \eqref{e19.12} or \eqref{e19.15}. The claim follows.

We also deduce from \eqref{e19.18} that \eqref{e19.17} is automatic
when all the indices $i$ (such that $W_i$ meets $B_0$) satisfy \eqref{e13.1}.

\ms
In this case 1, the results of Section \ref{2phases} are available. 
That is, all the blow-up limits of $\u$ at the origin are given by pairs 
of affine functions, as in \eqref{e16.23} and Corollary~\ref{t16.4},
and the free boundary $\d \Omega_{\varphi_1} \cup \d \Omega_{\varphi_2}$
is flat near the origin, as in Proposition~\ref{t16.5}.

It can be expected that, modulo additional regularity and
nondegeneracy assumptions on $F$, the method of \cite{C1}, \cite{C2} 
yields a much stronger regularity result, namely that both free boundaries
$\d \Omega_{\varphi_1}$ and $\d \Omega_{\varphi_2}$ coincide with $C^{1+\alpha}$
submanifolds in a neighborhood of the origin. In the special case when
$\u$ (or equivalently, $v$ above) is harmonic on the $\Omega_{\varphi_1}$, 
$F$ is given by \eqref{e1.7} with nonnegative functions $q_i$ such that $q_{i_1}$ 
and $q_{i_2}$ are both Lipschitz near $0$, and, say $q_{i_1}(0) > q_{i_2}(0) > 0$, 
we can apply directly the results of \cite{C1} and \cite{C2}, and the description 
of the blow-up limits given by Corollary~\ref{t16.4} is more than enough to say that 
$v$ above is a weak solution.
Here we have a slightly different situation in a few respects, because 
$v$ satisfies a slightly different equation (coming from \eqref{e9.6}),
and $F$ is a little more general than in \eqref{e1.7}. We do not expect major differences
to come from our different equation \eqref{e9.6}, and it can be imagined that if we
replace \eqref{e15.11} with a stronger form of approximation by volume functionals
of type \eqref{e1.7}, Caffarelli's regularity results may go through.
We shall not do the verification here, but at least we know that the difficulties,
if they exist, do not come from the fact that $J$ is a functional with many phases,
or the fact that we would not have enough control in the Lipschitz properties
of $\u$ or the description of the blow-up limits.

Even if this works, we may still wonder about what happens when we do not assume
the additional condition that $q_{i_1}(0) > q_{i_2}(0) > 0$ (or similar ones on the
$\lambda_i$ from \eqref{e15.11}), but just that all the indices are good (as in \eqref{e15.11}). Also, we do not know whether the presence of additional
phases $\varphi_i$, with low energy near $0$ as in \eqref{e16.53}, may disturb
the results above, even when $F$ is as in \eqref{e1.7} with smooth nonnegative
(but not positive) coefficients $q_i$.

\ms
\noindent{\bf Case 2}. Let us now assume that $I(0)$ has a unique element
$\varphi$, and let us assume for definiteness that $\varphi = (1,+1)$.
For the moment, let us authorize the case when $0 \in \d \Omega$, and let
$B_0$ be as in \eqref{e19.7}. We start with an analogue of Lemma \ref{t19.4}.

Set $v = (u_{1})_+$ and
\begin{equation} \label{e19.19}
\F(B_0,\Omega,v) = \big\{ w\in W^{1,2}(\R^n) \, ; \,
w = v \text{ a.e. on } \R^n \sm B_0
\text{ and $w= 0$ a.e. on } \R^n \sm \Omega
 \big\}.
\end{equation}
We also add in \eqref{e19.19} the requirement that $w \geq 0$, if we required
that $u_1 \geq 0$ in the definition of $\F$ (see Definition \ref{d1.1}); otherwise,
this is not needed.

\begin{lem} \label{t19.5} 
Suppose as above that \eqref{e19.1}-\eqref{e19.5} hold, and
let $B_0$ and $\F(B_0,\Omega,v)$ be as in \eqref{e19.7} and \eqref{e19.19}.
Then $v$ is a minimizer in $\F_v$ of the functional
\begin{equation} \label{e19.20}
J_+(w) = G_+(w) + \int_{B(0,r)} |\nabla w|^2 + w^2 f_{1} - w g_1 
\end{equation}
where the $f_i$ and the $g_i$ are as in \eqref{e1.4} and 
the volume term $G_+(w)$, which can be computed in terms of 
\begin{equation} \label{e19.21}
A_w = \big\{ x\in B_0 \, ; w(x) > 0 \big\},
\end{equation}
is given by \eqref{e19.22} below.
\end{lem}

\ms\begin{proof}  
The proof is essentially the same as for Lemma \ref{t19.4}.
To each $w \in \F(B_0,\Omega,v)$, we associate
$\u_w = (w, u_2, \ldots, u_N)$ and 
$\W_w = (A_w \cup (W_1 \sm B_0), W_2 \sm B_0, \ldots, W_N \sm B_0)$
(as we did near \eqref{e19.12}), and observe that
$(\u_w,\W_w) \in \F(\Omega)$, where this time we also need to say that
$w \geq 0$ if we required that $u_1 \geq 0$ in the definition of $\F(\Omega)$,
and $\u_w=0$ almost everywhere on $\R^n \sm \Omega$, by definition
of $\F(B_0,\Omega,v)$. Then we denote by ${\cal H}(w)$ the class of 
$N$-uples $\W^\ast = (W_1^\ast, \ldots, W_N^\ast)$ such that
the $W_i^\ast$ are disjoint and contained in $\Omega$, and each
$W_i^\ast$ coincides with $W_i$ on $\R^n \sm \Omega$ and 
contains $W_{w,i}$. Finally we set
\begin{equation} \label{e19.22}
G(w) = \inf\{ F(\W^\ast) \, ; \, \W^\ast \in {\cal H}(w) \big\}
\end{equation}
as we did before. Then we follow the proof of Lemma \ref{t19.4}
and get the result. Also notice that when $|W_i \cap B_0| = 0$
for $i > 1$, we may simplify the definition of $G$, and merely
take $G(w) = F(\W_w)$, as in \eqref{e19.16}. 
\qed
\end{proof}

\ms
Again we are in a relatively good situation to continue our investigation, 
with the functional $J_+$ that has only one phase.

If we also assume that $F$ is Lipschitz (as in \eqref{e10.2})
and (if $0 \in \d\Omega$) that the blow-up limits of $\Omega$ 
are open cones $\Omega_\infty$, as in \eqref{e18.35}, then Theorem \ref{t18.2}
says that some blow-up limits of $\u$ at $0$ are nontrivial homogeneous minimizers of a
simple Alt-Caffarelli-Friedman functional ($J_\infty$ as in \eqref{e15.40}),
but in the perhaps complicated cone $\Omega_\infty$.

If in addition $n \leq 3$ and $0$ is an interior point of $\Omega$,
this blow-up limit of $\u_\infty$ is given by an affine function, as
in Corollary \ref{t18.3} and \eqref{e18.50}, and we also get that the 
free boundary $\d \Omega_1$ is flat at the origin (but only along the corresponding
sequence), as in \eqref{e18.52}.

Again, the situation is not perfect, because in addition to the potential difficulties
of Case 1 (connected to our slightly different equations for $\u$, the 
more general $F$, and other non-good phases that may float around),
we may have complications due to the shape of $\Omega$, and the fact that
one-phase minimizers of the Alt-Caffarelli functional in $\R^n$, $n \geq 4$
are not well understood yet. But we shall nonetheless quit here for the moment
and pretend that this is a general problem about free boundaries, and not a
specific problem about $N \geq 3$.

\ms
Let us also mention that if we do not insist on studying $\u$ and the free
boundary near a specific point, Proposition \ref{t17.3} gives lots of points
of the $\d \Omega_i$ where the blow-up limits of $\u$ and $\d\Omega_i$
are controlled as in Corollaries \ref{t16.4} and \ref{t18.3}.
See Remark \ref{t18.6} and Proposition~\ref{t18.5}. This is true in all
dimensions, and the points where we can do this are also good candidates for
further local regularity, depending for instance on how the techniques
of \cite{C1} and \cite{C2} adapt. 

\begin{rem} \label{t19.6}
Lemmas \ref{t19.1}, \ref{t19.4}, and \ref{t19.5} seem to be give a fast
shortcut to many of the results above, especially if the technique of \cite{BV}
allows to prove Lemma \ref{t19.1} here. But this would not give exactly the
same estimates. First we would still need to check that the special form of $F$
that we have does not disturb the usual Lipschitz and nondegeneracy estimates
in the one- or two-phase problems (they are badly needed), 
but also our estimates would depend on the small radius $r_0$ which, 
with fast proofs, we won't be able to estimate. Finally, we can always hope
to understand better the situation when some of the indices $i$ do not
satisfy \eqref{e13.1}, and some of the long proofs above may prepare the way.
\end{rem}

\section{First variation and the normal derivative}   \label{vari}

We start this section with some first variation computations and the verification 
of the formulas \eqref{e16.49}, \eqref{e16.50}, and \eqref{e18.50}.
We first do a computation with one phase, consider 
\begin{equation} \label{e20.1}
v(x) = a\langle x, e \rangle_+
\end{equation}
for some unit vector $e$ and some positive constant
$a$, and try to find necessary conditions for $v$ to 
define a local minimizer of our functional $J_\infty$ in $\R^n$,
or more simply to be a one-phase ACF minimizer in $\R^n$,
as defined near \eqref{e18.44}. 

The typical small variations that we would try in general would be to add
a small function to $v$, but here it will be simpler to compose with 
a one parameter family of diffeomorphism, because it will be easier to
keep track of the domains. Let us choose coordinates in 
$\R^n$ so that $e = (0, \ldots, 1)$, and write the generic
point of $\R^n$ as $x = (x',y)$, with $x' \in e^\perp \simeq \R^{n-1}$
and $t \in \R e \simeq \R$.

We pick a nonnegative smooth bump function $\varphi$ and for small $t \in \R$ 
define a diffeomorphism $\Phi_t$ of $\R^n$ by
$\Phi_t(x) = x + t \varphi(x) e$. Then we set 
\begin{equation} \label{e20.2}
v_t(x) = v(\Phi_t(x)) = a\langle \Phi_t(x), e \rangle_+
\end{equation}
and compute the derivative of $v_t$. Set $\O = \big\{ x\in \R^n \, ; \, 
\langle x, e \rangle > 0 \big\}$ and 
$\O_t = \Phi_t^{-1}(\O) =
\big\{ x\in \R^n \, ; \, \langle \Phi_t(x), e \rangle > 0 \big\}$;
we just need to compute $\nabla v_t$ on $\O_t$, because 
it vanishes on $\R^n \sm \O_t$. And for $x\in \O_t$,
\begin{equation} \label{e20.3}
{\d v_t \over dx_j} = a \langle e, {\d \Phi_t \over dx_j}\rangle
= a \delta_{j,n} + at {\d \varphi \over dx_j},
\end{equation}
where we used the Kronecker symbol $\delta_{j,n}$.
Let $B$ be a ball that contains the support of $\varphi$;
observe that by \eqref{e20.3}, 
$|\nabla v_t(x)|^2 = a^2 [1+ 2 t {\d \varphi \over \d x_n}(x) + O(t^2)]$,
where $O(t^2)$ is function which is less than $Ct^2$ and is supported
in $B$. Then 
\begin{equation} \label{e20.4}
\int_{2B} |\nabla v_t|^2  = a^2 \int_{2B \cap \O_t} |\nabla v_t|^2
= a^2 \int_{2B \cap \O_t} 1 + 2t {\d \varphi \over \d x_n}(x) + O(t^2).
\end{equation}
Notice that for $t$ small, $2B = \varphi_t(2B)$;
then we set $y = \Phi_t(x)$, with $y \in 2B \cap \O$, notice that
$dy = J_{\Phi_t}(x) dx = 1 + {\d \varphi \over dx_n}(x)$
(because the matrix of $D\Phi_t$ is an identity matrix, plus a last
column composed of the ${\d \varphi \over dx_j}$), so
\begin{eqnarray} \label{e20.5}
\int_{2B} |\nabla v_t|^2  
&=& a^2 \int_{2B \cap \O} [1 + 2t{\d \varphi \over \d x_n}(\Phi_t^{-1}(y)) + O(t^2)]
[1 + t{\d \varphi \over \d x_n}(\Phi_t^{-1}(y)]^{-1} dy
\nonumber\\
&=& a^2 \int_{2B \cap \O} [1 + t{\d \varphi \over \d x_n}(\Phi_t^{-1}(y)) + O(t^2) ]dy
= a^2 \int_{2B \cap \O} [1 + t{\d \varphi \over \d x_n}(y) + O(t^2) ]dy
\nonumber\\
&=& a^2 |2B \cap \O| + a^2 t \int_{2B \cap \O} {\d \varphi \over \d x_n} + O(t^2).
\end{eqnarray}
Set $A = \int_{2B \cap \O} {\d \varphi \over \d x_n}$; notice that
$A = - \int_{2B \cap \d O} \varphi(x') dx'$, and so we can choose $\varphi$
so that $A < 0$. The same computation as above also shows that
\begin{eqnarray} \label{e20.6}
|2B \cap \O_t| &=& \int_{2B \cap \O_t} dx 
= \int_{2B \cap \O} [1 + t{\d \varphi \over \d x_n}(\Phi_t^{-1}(y)]^{-1} dy
\nonumber\\
&=& |2B \cap \O_t| - t \int_{2B \cap \O} {\d \varphi \over \d x_n} + O(t^2)
= |2B \cap \O_t| - t A + O(t^2).
\end{eqnarray}

If $v$ is a one-phase ACF minimizer in $\R^n$, associated to the constant $Q^2$,
\eqref{e18.44} says that for $t$ small
\begin{equation} \label{e20.7}
\int_{2B} |\nabla v|^2 + Q^2 |2B \cap \O|
\leq \int_{2B} |\nabla v_t|^2 + Q^2 |2B \cap \O_t|,
\end{equation}
and by \eqref{e20.5} and \eqref{e20.6}, this means that
\begin{equation} \label{e20.8}
0 \leq a^2 t A - Q^2t A + O(t^2).
\end{equation}
Since this holds for $t$ small (of both signs), we get that $a^2=Q^2$.

Thus, in \eqref{e18.50}, we had no choice about the value of $Q$,
it had to be given by \eqref{e18.43}, where the $\lambda_j$ are the same 
as in the definition of $J_\infty$ (see \eqref{e15.40}), 
and the story is the same for the end of the proof of Proposition \ref{t18.5}.

We can do a similar computation for $2$-phase functions. That is,
let $\u_\infty$ be a local minimizer in $\R^n$ of the functional $J_\infty$
that shows up in \eqref{e15.40}, and suppose that we have two phases
$\varphi_1 = (i_1,\varepsilon_1)$ and $\varphi_2 = (i_2,\varepsilon_2)$ such that 
\begin{equation} \label{e20.9}
v_1(x) = a_1 \langle x, e \rangle_+
\ \text{ and } \ 
v_2(x) = a_2 \langle x, e \rangle_-
\ \text{ for $x\in \R^n$,}
\end{equation}
where we set $v_j= [\varepsilon_j u_{i_j,\infty}]_+$,
and with coefficients $a_1, a_2 > 0$. Then there is no place left
for the other phases of $\u_\infty$ (so they are null), and also our only choice is to take
$W_{i_j} = \big\{ x\in \R^n \, ; \,  (-1)^j\langle x, e \rangle > 0 \big\}$ 
if $i_1 \neq i_2$, and $W_{i_1} = W_{i_2} = \R^n$ if $i_1 = i_2$.

We can compare $\u_\infty$ with $\u_{\infty}\circ \Phi_t$, where $\Phi_t$
is as above, and then use the definition \eqref{e15.40} of a local minimizer. 
Again we have no other choice than taking the sets $\Phi_t^{-1}(W_{i_j})$.
We can compute both pieces
$\int_{2B} |\nabla v_j|^2 + \lambda_{i_j} |2B \cap \Phi_t^{-1}(W_{i_j})|$
as we did before, and then \eqref{e15.40} yields
\begin{equation} \label{e20.10}
0 \leq a_1^2 A t - \lambda_{i_1} A t + a_2^2 A t - \lambda_{i_2} A t + O(t^2),
\end{equation}
just as we obtained \eqref{e20.8} above. That is, we obtain the necessary
condition 
\begin{equation} \label{e20.11}
a_1^2 - a_2^2 = \lambda_{i_1} - \lambda_{i_2},
\end{equation}
which is the same as  \eqref{e16.49}.

But we can also let $\Phi_t$ operate on $v_1$ alone, and let $v_2$ as it is,
provided that the domain $\O_t = \Phi_{-1}(\O)$ associated to the modification
of $v_1$ stays inside of $\O$, so that $(v_1)_t$ does not interfere with $v_2$.
We can get that if $\varphi \geq 0$, by restricting to $t < 0$.
This even leaves some free space (the sets $\O \sm \O_t$), which we can attribute
to the empty set, or any index $i$ that we may find suitable. That is, set
$\lambda_0 = \min(0, \lambda_1, \ldots, \lambda_N)$; when we add
$\O \sm \O_t$ to some other $W_i$, or just drop it, we win
$(\lambda_{i_1}-\lambda_0)|\O \sm \O_t|$ in the volume term.

The same computation as for \eqref{e20.8} now yields
\begin{equation} \label{e20.12}
0 \leq a_1^2 t A - (\lambda_{i_1} -\lambda_0) t A + O(t^2).
\end{equation}
Now recall that $A < 0$, and also that we are only allowed to take $t < 0$;
then we only get the inequality
\begin{equation} \label{e20.13}
a_1^2 \geq \lambda_{i_1} -\lambda_0 = \lambda_{i_1} -\min(0, \lambda_1, \ldots, \lambda_N).
\end{equation}
The same argument also yields $a_2^2 \geq \lambda_{i_2} -\lambda_0$.
Those are the constraints that were noted in \eqref{e16.50}.

\ms
It is possible to do the same sort of first variation computations
near a minimizer $(\u,\W)$, assuming that everything is sufficiently smooth
for us to make the computations. This leads to Euler-Lagrange equations at the
boundary, that can be stated as follows. On a (smooth enough) piece of boundary
that would separate two regions $\Omega_{\varphi_i}$, the normal derivatives
${\d v_j \over \d n}$ satisfy
\begin{equation} \label{e20.14}
\big({\d v_1 \over \d n}\big)^2 - \big({\d v_2 \over \d n}\big)^2 
= \lambda_{i_1} - \lambda_{i_2},
\end{equation}
as in \eqref{e20.11}, where as usual we write $\varphi_j = (i_j,\varepsilon_j)$ and
$v_j = [\varepsilon_j u_{i_j}]_+$. Similarly, one-sided variations lead to the constraint
\begin{equation} \label{e20.15}
\big({\d v_j \over \d n}\big)^2 \geq \lambda_{i_j} -\min(0, \lambda_1, \ldots, \lambda_N),
\end{equation}
as in \eqref{e20.13}.
And along a nice piece of $\d \Omega_{\varphi_1}$, that would lie inside
of $\Omega$ and separate $\Omega_{\varphi_1}$ from a region where $\u = 0$, 
we get that 
\begin{equation} \label{e20.16}
\big({\d v_1 \over \d n}\big)^2 = \lambda_{i_1} -\min(0, \lambda_1, \ldots, \lambda_N),
\end{equation}
as in our estimate below \eqref{e20.13}. 

Finally, along a (smooth enough) piece of $\d\Omega \cap \d\Omega_{\varphi_1}$ 
(where we would have $\Omega_{\varphi_1}$ on one side and $\R^n \sm \Omega$
on the other side), we would get that
$\big({\d v_1 \over \d n}\big)^2 \geq \lambda_{i_1} 
-\min(0, \lambda_1, \ldots, \lambda_N)$, by the same one-sided variations as
for \eqref{e20.15}.

We do not do the computation here, because they are similar, but just more complicated
than the computations above. But the following discussion will justify the equations
\eqref{e20.14}-\eqref{e20.16} on an almost everywhere pointwise level.

\begin{pro} \label{t20.1}
Assume that $0$ is an interior point of $\Omega$, that $0 \in \d \Omega_1$,
and that \eqref{e18.2} and \eqref{e18.4} hold; in particular, $i=1$ is a good index.
Assume that we can find a blow-up limit $\u_\infty$ of $\u$ at the origin
such that 
\begin{equation} \label{e20.17}
u_{1,\infty}(x) = a \langle x,e \rangle
\ \text{ for } x\in \R^n
\end{equation}
for some choice of $a>0$ and $e \in \d B(0,1)$, and that the limit
\begin{equation} \label{e20.18}
h(0) = {1 \over \omega_{n-1}} \lim_{r \to 0} \, r^{1-n}
\langle \Delta u_{1,+}, \1_{B(0,r)}\rangle
\end{equation}
exists. Then $h(0) = a$.
\end{pro}

\ms
Recall that $h(0)$ is the same as in \eqref{e17.30}, that the limit $h(x)$ 
exists for $\H^{n-1}$-every point $\d \Omega_1 \cap {\rm int}(\Omega)$,
and that $h$ is the Radon-Nikodym density of $\mu$, the restriction
of $\Delta u_{1,+}$ to $\d \Omega_1$, with respect with the restriction
of $\H^{n-1}$ to $\d \Omega_1$. See Proposition \ref{t17.2}.

The first assumption says that there is a sequence $\{ r_k \}$ that tends 
to $0$ such that the $\u_k$ defined by \eqref{e15.2} tend to $\u_\infty$.
In general, $\u_\infty$ and $e$ may depend on the sequence $\{ r_k \}$, but
the proposition says that $a$ does not. The case when $a=0$
would not be allowed by our assumption \eqref{e18.2},
because \eqref{e18.36} says that $\liminf_{r \to 0} \Phi^0_{1,1}(r) > 0$, 
and then $\Phi_{1,1,\infty}(1) = \lim_{k \to +\infty} \Phi^0_{1,1}(r_k) > 0$
by Lemma \ref{t16.2}.

Finally observe that if \eqref{e18.2} holds, the assumptions of 
Proposition \ref{t20.1} (with $0$ replaced by $x_0$)
are satisfied at $\H^{n-1}$-every point 
$x_0 \in \d \Omega_1 \cap {\rm int}(\Omega)$ 
for which \eqref{e18.4} holds, because $\d\Omega_1$
has a tangent plane at $x_0$; 
see Remark \ref{t18.6} and Proposition~\ref{t18.5}.
In this case, since $\d\Omega_1$ has a tangent plane at $x_0$,
we even get that $e$ does not depend on the blow-up sequence,
and that $u_{1,+}$ has a normal derivative at $x_0$, equal to $a$.
[The verification would need a little bit of playing with \eqref{e20.17},
to get a coherent coherent choice of $e$, but this is easy and we leave the details.]

\begin{proof} 
Since we intend to take limits, let us replace $\1_{B(0,r)}$ by smoother functions.
For $\tau > 0$ small, choose a smooth radial function $\varphi_\tau$
such that 
\begin{equation} \label{e20.19}
\1_{B(0,1-\tau)} \leq \varphi_\tau \leq \1_{B(0,1)},
\end{equation}
and then define $\varphi_{\tau,k}$ by 
$\varphi_{\tau,k}(x) = \varphi_{\tau}(r_k^{-1}x)$, where 
$r_k$ comes from our blow-up sequence. 
We will use $\varphi_{\tau,k}$ to approximate $\1_{B(0,r_k)}$.
Set $v = u_{1,+}$ and $v_k(x) = r_k^{-1} v(r_k x)$ to save notation, 
and compute
\begin{eqnarray} \label{e20.20}
A_k :&=&  r_k^{1-n} \langle \Delta u_{1,+}, \varphi_{\tau,k}\rangle
= - r_k^{1-n} \langle \nabla u_{1,+}, \nabla \varphi_{\tau,k}\rangle
= - r_k^{1-n} \langle \nabla v, \nabla \varphi_{\tau,k}\rangle
\nonumber \\
&=& - r_k^{-n} \int \langle \nabla v(x), (\nabla \varphi_{\tau})(r_k^{-1}x)\rangle dx
= - \int \langle \nabla v(r_k y), \nabla \varphi_{\tau}(y)\rangle dy
\nonumber \\
&=& - \int \langle \nabla v_k(y), \nabla \varphi_{\tau}(y)\rangle dy.
\end{eqnarray}
Then we apply Corollary \ref{t15.3} (no need to extract a subsequence or apply
Lemma \ref{t18.1} here, because $0$ is an interior point of $\Omega$), and 
get that \eqref{e15.17} holds for all $R$. Hence
\begin{equation} \label{e20.21}
\lim_{k \to +\infty} A_k = - \int \langle \nabla v_\infty, \nabla \varphi_{\tau}\rangle 
= - \int \langle \nabla u_{1,\infty,+} , \nabla \varphi_{\tau}\rangle.
\end{equation}
Set $H = \big\{ x\in \R^n \, ; \, \langle x, e\rangle\big\} = 0$
and $H_+ = \big\{ x\in \R^n \, ; \, \langle x, e\rangle\big\} > 0$.
Then \eqref{e20.17} yields
\begin{equation} \label{e20.22}
\lim_{k \to +\infty} A_k =  - a \int_{H_+} {\d \varphi_{\tau} \over \d e}
= a \int_{H} \varphi_{\tau},
\end{equation}
and by \eqref{e20.19} 
\begin{equation} \label{e20.23}
(1-\tau)^{n-1} a\leq {1 \over \omega_{n-1}} \, \lim_{k \to +\infty} A_k \leq a.
\end{equation}
Now we estimate 
\begin{equation} \label{e20.24}
\delta_k = r_k^{1-n} \langle \Delta u_{1,+}, \1_{B(0,r_k)}\rangle - A_k
= r_k^{1-n} \langle \Delta u_{1,+}, (\1_{B(0,r_k)}- \varphi_{\tau,k})\rangle.
\end{equation}
Recall from Proposition \ref{t17.2} that 
$\Delta u_{1,+} = \mu + [f_1 v - {1\over 2}g_1] \1_{\Omega_1}$,
where $\mu$ is a positive measure on $\d\Omega_1$ and (we just need
to know that) $w = [f_1 v - {1\over 2}g_1] \1_{\Omega_1}$ is bounded.
Thus
\begin{eqnarray} \label{e20.25}
|\delta_k| &\leq& r_k^{1-n} \int (\1_{B(0,r_k)}- \varphi_{\tau,k}) d\mu
+ r_k^{1-n} \int_{B(0,r_k)} |w|
\nonumber \\
&\leq& r_k^{1-n} \mu(B(0,r_k) \sm B(0,(1-\tau) r_k)) + C r_k
\end{eqnarray}
by \eqref{e20.19}. Also notice that for $r$ small,
\begin{equation} \label{e20.26}
\big| r^{1-n} \langle \Delta u_{1,+}, \1_{B(0,r)}\rangle - r^{1-n} \mu(B(0,r)) \big|
= r^{1-n} \Big| \int_{B(0,r)} w \Big| \leq C r
\end{equation}
for the same reason, which means that 
\begin{equation} \label{e20.27}
h(0) = {1 \over \omega_{n-1}} \lim_{r \to 0} r^{1-n} \mu(B(0,r))
\end{equation}
by \eqref{e20.18}. Then $r_k^{1-n} \mu(B(0,r_k) \sm B(0,(1-\tau) r_k))$
tends to $[1 - (1-\tau)^{n-1}] \omega_{n-1} h(0)$, and
$|\delta_k| \leq C \tau$ for $k$ large. We may now put things together.
For $k$ large,
\begin{equation} \label{e}
\omega_{n-1} |h(0) - a|  \leq \big| \omega_{n-1} h(0) - r_k^{1-n} 
\langle \Delta u_{1,+}, \1_{B(0,r_k)}\rangle\big|
+ |\delta_k| + |A_k - \omega_{n-1} a| \leq C \tau,
\end{equation}
by \eqref{e20.18} and \eqref{e20.23}. Since $\tau$ is as small as we want,
we get that $h(0) = a$; Proposition~\ref{t20.1} follows.
\qed
 \end{proof}

\Addresses


\begin{thebibliography}{AAA}
\bibitem[AC]{AC} H. W. Alt and L. Caffarelli, Existence and regularity for a minimum 
problem with free boundary. J. Reine Angew. Math. 325 (1981), 105--144.
\bibitem[ACF]{ACF} H. W. Alt, L. Caffarelli, and  A. Friedman,
Variational problems with two phases and their free boundaries. 
Trans. Amer. Math. Soc. 282 (1984), no. 2, 431--461.
\bibitem[Bo]{Bo} A. Bonnet, On the regularity of edges in image segmentation. 
Ann. Inst. H. Poincar\'e Anal. Non Lin\'eaire 13 (1996), no. 4, 485--528.
\bibitem[BKP]{BKP} W. Beckner, C. Kenig, and J. Pipher,  A convexity property 
of eigenvalues with applications, private communication 
\bibitem[BV]{BV} D. Bucur and B Velichkov, Multiphase shape optimization problems,
preprint, arXiv: 1310.2448v1
\bibitem[BZ]{BZ} J. Brothers and W. Ziemer, Minimal rearrangements of Sobolev functions, 
J. Reine Angew. Math, 384 (1988), 153-179.
\bibitem[C1]{C1} L. Caffarelli, A Harnack inequality approach to the regularity of free 
boundaries. I. Lipschitz free boundaries are $C^{1,\alpha}$. 
Rev. Mat. Iberoamericana 3 (1987), no. 2, 139--162.
\bibitem[C2]{C2} L. Caffarelli, A Harnack inequality approach to the regularity of free 
boundaries. II. Flat free boundaries are Lipschitz. Comm. Pure Appl. Math. 42 (1989), 
no. 1, 55--78.
\bibitem[CJK]{CJK} L. Caffarelli, D. Jerison, and C. Kenig, Some new monotonicity theorems 
with applications to free boundary problems. Ann. of Math. (2) 155 (2002), no. 2, 369--404.
\bibitem[CJK2]{CJK2} L. Caffarelli, D. Jerison, and C. Kenig, Global energy minimizers 
for free boundary problems and full regularity in three dimensions. 
Noncompact problems at the intersection of geometry, analysis, and topology, 83--97, 
Contemp. Math., 350, Amer. Math. Soc., Providence, RI, 2004.
\bibitem[Daa]{Daa} G. David,  Morceaux de graphes lipschitziens et int\'egrales singuli\`eres 
sur une surface. Rev. Mat. Iberoamericana 4 (1988), no. 1, 73--114. 
\bibitem[D]{D}  G. David, \underbar{Singular sets of minimizers 
for the Mumford-Shah functional}. Progress in Mathematics, 233. Birkh\"auser Verlag, 
Basel, 2005. xiv+581 pp.
\bibitem[DJ]{DJ} G. David and D. Jerison, Lipschitz approximation to hypersurfaces, 
harmonic measure, and singular integrals. Indiana Univ. Math. J. 39 (1990), no. 3, 831--845.
\bibitem[DT]{DT} G. David and T. Toro, Regularity of almost minimizers with free boundary, preprint arXiv:1306.2704.
\bibitem[DS]{DS} G. David and S. Semmes, \underbar{Analysis of and on uniformly 
rectifiable sets}. Mathematical Surveys and Monographs, 38. 
American Mathematical Society, Providence, RI, 1993. xii+356 pp.
\bibitem[F]{F}  H. Federer, \underbar{Geometric measure theory}. 
Die Grundlehren der mathematischen Wissenschaften, 
Band 153 Springer-Verlag New York Inc., New York 1969 xiv+676 pp.
\bibitem[FM]{FM} M. Filoche, and S. Mayboroda, Universal mechanism for Anderson and weak localization. Proc. Natl. Acad. Sci. USA 109 (2012), no. 37, 14761--14766. 
\bibitem[FH]{FH} A. Friedman and W. K. Hayman, Eigenvalue inequalities for the Dirichlet 
problem on spheres and the growth of subharmonic functions, 
Comment. Math. Helv. 51 (1976), 131--161.
\bibitem[Gi]{Gi} E. Giusti, \underbar{Minimal surfaces and functions of bounded 
variation}, Monographs in Mathematics, 80. Birkh\"auser Verlag, Basel-Boston, 
Mass., 1984.
\bibitem[HP]{HP} A. Henrot and M. Pierre, \underbar{Variation et optimization de formes.  
Une analyse g\'eom\'etrique}. Math\'ematiques \& Applications 48. Springer, 
Berlin, 2005. xii+334 pp.
\bibitem[M]{M} P. Mattila, \underbar{Geometry of sets and measures in Euclidean space}, 
Cambridge Studies in Advanced Mathematics 44, Cambridge University Press l995.
\bibitem[P]{P} C. Pommerenke, \underbar{Boundary behavior of conformal maps}, 
Grundslehren der Mathematischen Wissenchaften 299, Springer-Verlag 1992.
\bibitem[Se]{Se} S. Semmes, A criterion for the boundedness of singular integrals on hypersurfaces, Trans. Amer. Math. Soc. 311 (1989), no. 2, 501--513.
\bibitem[S]{S} E. Stein, \underbar{Singular integrals and differentiability properties 
of functions}. Princeton Mathematical Series, No. 30 Princeton University Press, 
Princeton, N.J. 1970. xiv+290 pp.
\bibitem[We]{We} Partial regularity for a minimum problem with free boundary.  
J. Geom. Anal. 9 (1999), no. 2, 317Ð326.
\bibitem[Wi]{Wi} Kjell-Ove Widman, Inequalities for the Green function 
and boundary continuity of the gradient of solutions 
of elliptic differential equations, Math. Scand. 21 1967 17Ð37 (1968).
\bibitem[Z]{Z} Ziemer, William P. \underbar{Weakly differentiable functions. 
Sobolev spaces and functions of} \underbar{bounded variation}. 
Graduate Texts in Mathematics, 120. Springer-Verlag, New York, 1989. xvi+308 pp.

\end{thebibliography}
\end{document}